\documentclass{article}
\usepackage{arxiv}
\usepackage[utf8]{inputenc}
\usepackage[ngerman,english]{babel}
\usepackage{csquotes}
\usepackage{amssymb,amsthm,amsmath,amsfonts}
\usepackage{bm}
\usepackage{bbm}
\usepackage{esint}
\usepackage{setspace}
\usepackage{mathtools}
\usepackage{booktabs, makecell, tabularx}
\usepackage{multirow}
\usepackage{placeins}
\usepackage{pgfplots}
\usepackage{pgfplotstable}
\pgfplotsset{compat=newest}
\usepgfplotslibrary{colormaps}
\usepackage{subfig}
\usepackage{enumitem}
\usepackage{tikz}
\usepackage{url}
\usetikzlibrary{
	decorations.pathreplacing,
    decorations.pathmorphing,
	automata,
	3d,
	fit, 
	graphs,
	plotmarks, 
	external,
	positioning,
	arrows,
	calc,
	spy,
	chains,
	shapes,
    fadings,
	backgrounds,
	quotes
	}
\usepackage[
    style=numeric-comp,
    sorting=nyt,
    backend=biber,
    defernumbers,
    doi=true,
    isbn=false,
    url=false,
    eprint=false,
    giveninits=true,
    maxnames=6, 
    block=none]{biblatex}
    
\usepgfplotslibrary{fillbetween}
\usepackage{float}
\usepackage{graphicx}
\usepackage{hyperref}
\usepackage{algorithm}
\usepackage{algpseudocode}  
\usepackage{newtxtext,newtxmath}
\usepackage{textcomp}
\newcommand{\E}{\mathbb{E}}
\newcommand{\Var}{\mathbb{V}}

\newcommand{\hdiv}{L^2_{\mathrm{div}}(\dom;U)}
\newcommand{\Div}{\mathrm{div}}

\newcommand{\dom}{\Omega}
\newcommand{\vis}{\varepsilon}

\newcommand{\Grad}{\bm{\nabla}}
\newcommand{\Laplace}{\bm{\Delta}}

\newcommand{\Hn}{L^2_{\Div}}

\newcommand{\eps}{\epsilon}
\allowdisplaybreaks
\def\addlegendimage{\csname pgfplots@addlegendimage\endcsname}
\pgfplotsset{select coords between index/.style 2 args={
    x filter/.code={
        \ifnum\coordindex<#1\fi
        \ifnum\coordindex>#2\fi
    }
}}

\newtheorem{theorem}{\bf Theorem}[section]

\newtheorem{corollary}{\bf Corollary}[section]
\newtheorem{proposition}{\bf Proposition}[section]
\newtheorem{hypothesis}{\bf Hypothesis}[section]

\newtheorem{lemma}{\bf Lemma}[section]
\newtheorem{assumption}{\bf Assumption}[section]
\theoremstyle{definition}
\newtheorem{definition}{\bf Definition}[section]

\theoremstyle{remark}
\newtheorem{remark}{\bf Remark}[section]

\title{Computing statistical Euler limits of the Navier--Stokes equations in three dimensions}

\author{
    Johannes L.\ Grafen\thanks{Equal contribution.}\\[.5em]
    \footnotesize LBRG Mathematical Modeling and Numerics Lab\\
    Lattice Boltzmann Research Group\\
    Institute of Mechanical Process Engineering and Mechanics\\
    Karlsruhe Institute of Technology (KIT)\\
    76131 Karlsruhe, Germany\\
    \And
    Tobias Rohner\\[.5em]
    \footnotesize Seminar for Applied Mathematics \\
    ETH Zurich\\
    8092 Z\"urich, Switzerland\\
    \And
    \href{https://orcid.org/0000-0001-8555-4245}{\includegraphics[scale=0.06]{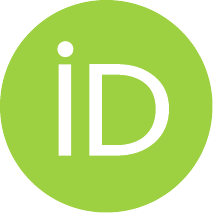}\hspace{1mm}Stephan Simonis}$^{*,}$\thanks{Corresponding author, email: \href{ssimonis@ethz.ch}{ssimonis@ethz.ch}, \href{stephan.simonis@kit.edu}{stephan.simonis@kit.edu}}\\[.5em]
	\footnotesize Seminar for Applied Mathematics \\
    ETH Zurich\\
    8092 Z\"urich, Switzerland \\[.5em]
    LBRG Mathematical Modeling and Numerics Lab\\
    Lattice Boltzmann Research Group\\
    Institute for Applied and Numerical Mathematics\\
    Karlsruhe Institute of Technology (KIT)\\
    76131 Karlsruhe, Germany\\
}

\renewcommand{\headeright}{J.\L.\ Grafen, T.\ Rohner, S.\ Simonis}
\renewcommand{\shorttitle}{Computing statistical Euler limits of the Navier--Stokes equations in 3D}

\hypersetup{
pdftitle={ComputingNSEtoEuler3D},
pdfsubject={},
pdfauthor={Stephan Simonis}
}

\begin{document}

\maketitle

\begin{abstract}
We develop a Monte Carlo lattice Boltzmann method to compute statistical solutions to the three-dimensional incompressible Navier--Stokes and Euler equations. 
Entropic space-time adaptive relaxation of the higher order kinetic moments yields stable numerical solutions with decreasing viscosity. 
We provide a convergence analysis that is conditional on four explicitly stated assumptions regarding the discrete dynamics. 
Under diffusive scaling, the laws of the discrete ensemble converge along a subsequence to a limit satisfying the Foias--Temam Liouville formulation of the Navier--Stokes equations. 
Consequently, provided the structure function scaling holds uniformly, the vanishing viscosity limit of these measures satisfies the multi-point statistical Euler hierarchy of Fjordholm, Mishra, and Weber. 
The limit measures inherit the known weak-strong uniqueness principle on the interval of existence of a strong Euler solution. 
Under explicit scaling assumptions, a Kuznetsov-type argument yields a fractional 1-Wasserstein convergence rate. 
We present three-dimensional computations of time-dependent statistical solutions along the inviscid limit of the incompressible Navier--Stokes equations together with convergence measurements in the Wasserstein metric. 
Numerical experiments on a randomized Taylor--Green vortex with 24-dimensional initial uncertainty recover Kolmogorov's K41 scaling for energy spectra and structure functions, exhibit the failure of pathwise strong convergence, and yield Wasserstein convergence rates of about $0.5$ at the onset of turbulence.
Finally, error measurements with respect to spectral hyperviscosity computations indicate that the computed limit measure is independent of the numerical regularization. 
\end{abstract}



\clearpage
\newpage
\renewcommand{\baselinestretch}{0.75}\normalsize
\tableofcontents
\renewcommand{\baselinestretch}{1.0}\normalsize

\clearpage
\newpage
   
\section{Introduction}

In the force-free, viscous case (viscosity \(\vis >0\)), incompressible Newtonian fluid flows can be described by the incompressible Navier--Stokes equations (NSE)
\begin{align}\label{eq:incNSE}
    \begin{cases}
        \mathrm{div}_{\bm{x}}\left(\bm{u}\right) = 0 
        	\quad & \text{in } \Omega_{T}, \\
        \partial_{t} \bm{u} + \mathrm{div}_{\bm{x}} \left( \bm{u} \otimes \bm{u}\right) - \vis \Laplace_{\bm{x}} \bm{u} + \Grad_{\bm{x}} p = \bm{0} 
        	\quad & \text{in } \Omega_{T}, \\ 
        \bm{u}\vert_{t = 0} = \bm{u}_{0} 
        	\quad & \text{in } \Omega, \\
    \end{cases}
\end{align}
where \(\mathrm{div}_{\bm{x}} ( \bm{u} ) = \Grad_{\bm{x}} \cdot \bm{u}\) refers to the divergence operator, \(\bm{u}\) is the velocity field, \(p\) is the pressure, and \(\Omega_{T} \coloneqq \Omega \times I \subseteq \mathbb{R}^{3} \times \mathbb{R}_{\geq 0}\).  
In the inviscid limit \(1/R\!e \sim \vis \searrow 0\), \eqref{eq:incNSE} formally becomes the incompressible Euler equations (EE), where \(R\!e = U_{\mathrm{c}} l_{\mathrm{c}}/\vis\) is the Reynolds number defined by the characteristic flow velocity \(U_{\mathrm{c}}\) and the characteristic domain length \(l_{\mathrm{c}}\).  
Below, we refer to \eqref{eq:incNSE} in both cases. 
The existence of weak solutions to \eqref{eq:incNSE} for \(d\in \{2,3\}\) has been pioneered by Leray~\cite{leray1934mouvement} and Hopf~\cite{hopf1950anfangswertaufgabe}. 
Although uniqueness has been proven for \(d=2\), the uniqueness of such solutions to the incompressible NSE for \(d=3\) is still open~\cite{fefferman2000existence}. 
Many techniques and results to approach well-posedness questions for NSE and EE can be found in the literature: exponential attractors~\cite{eden1993exponential,eden1994exponential,eden1995exponential}, nonuniqueness for finite kinetic energy weak solutions~\cite{buckmaster2019nonuniqueness}, numerical investigations of blowup solutions~\cite{hou2024nearly,wang2023discovering}, and analytical and numerical local-in-space estimates near initial time~\cite{jia2014local,guillod2023numerical}.

Recent efforts to address the possible non-uniqueness in three dimensions are based on the concept of statistical solutions of the incompressible NSE~\cite{foias1976solutions,foias2001navier} and of the incompressible EE (for \(\vis \searrow 0 \) in \eqref{eq:incNSE})~\cite{fjordholm2021vanishing,chae1991vanishing,glimm2023smooth}.  
In this setting, Fjordholm~\textit{et al.}~\cite{fjordholm2021vanishing} have proved that a scaling assumption on statistical turbulence objects (e.g., energy spectra and structure functions) is sufficient for the statistical solutions of the incompressible NSE to converge for \(\vis \searrow 0\) to a statistical solution of the EE~\cite[Theorem 4.8]{fjordholm2021vanishing}. 
Computing time-dependent statistical solutions in three spatial dimensions is expensive, since a large number of samples has to be evolved in time at high resolution. 
Single level Monte Carlo (MC) methods evolve a large number of random samples in time with a deterministic solver, so that the cost grows with the number of samples and with the resolution. 
Several approaches have been proposed to meet this challenge. 
Combined with multi level MC methods, a finite difference (FD) discretization of the vorticity form of the incompressible NSE for \(d=2\) with periodic boundaries has been proposed in~\cite{leonardi2016numerical}. 
In~\cite{bansal2021numerical}, statistical solutions of the incompressible NSE for \(d=2\) with boundaries are computed using single level MC and an \(H(\mathrm{div})\)-based finite element method. 
Statistical solutions to the incompressible EE for \(d=2\) with periodic boundaries have been approximated successfully in~\cite{lanthaler2021statistical} by combining MC with a deterministic spectral hyperviscosity method. 
The implementation of these methods for periodic incompressible flows has been extended to \(d=3\) by Rohner and Mishra~\cite{rohner2024efficient}. 
To the knowledge of the authors, at this time, statistical solutions to the EE in three dimensions have been computed only by Rohner and Mishra~\cite{rohner2024efficient} based on an MC spectral hyperviscosity method proposed in~\cite{lanthaler2021statistical}. 

The computational demand of statistical solutions calls for highly parallelizable numerical schemes. 
The lattice Boltzmann method (LBM) combines a local collision step with a linear streaming step, which makes it well suited for parallel hardware, and its entropic variants provide stability at high Reynolds numbers. 
For these reasons, LBMs have become an established alternative to conventional approximation tools for the NSE, where scalability to HPC is crucial~\cite{lallemand2021lattice}. 
Many extensions for unsteady computer simulations of turbulent fluid flow have been established, such as assistive numerical diffusion~\cite{simonis2021linear}, large eddy simulation (LES) based on filtering in space~\cite{siodlaczek2021numerical,jahanshaloo2013review,malaspinas2012consistent}, or in time~\cite{simonis2022temporal}. 
Moreover, Simonis \textit{et al.}~\cite{simonis2024spectral} have numerically observed up to second order convergence of Karlin--Chikatamarla--B\"{o}sch (KBC) LBMs toward the incompressible NSE under diffusive scaling. 
On an industrial scale, the application of LBM to the LES methodology provides significant speedup over traditional methods (e.g., finite volume methods)~\cite{kajzer2014large,haussmann2020evaluation,lohner2019towards}. 
Moreover, the kinetic derivation of LBM allows for thermodynamically consistent extensions toward multi-physics modeling (e.g., see~\cite{krause2020openlb,aidun2010lattice}) and enables simulating compressible flows with strong discontinuities~\cite{dorschner2018particles,coreixas2020compressible,wilde2021high,kallikounis2022particles}. 
Combinations of LBM with intrusive and nonintrusive uncertainty quantification (UQ) methods have rarely been studied so far. 
To the knowledge of the authors, the available results are the following.  
An LBM is proposed in~\cite{zhao2021lattice} to approximate Galerkin-projected stochastic convection--diffusion equations. 
The scheme offers rigorous weighted \(L^{2}\)-stability and applicability to stochastic problems with complex boundaries. 
In~\cite{zhao2019stochastic}, a probabilistic collocation method has been combined with LBM to obtain statistical properties for fluid flows through porous media. 
A reduction of the computational effort by \(\geq \mathcal{O}(100)\) compared to single level MC LBM is reported for 2D flow simulations. 
Further, in~\cite{vandenbos2017nonIntrusive}, newly developed sparse quadrature and cubature rules are used in stochastic collocation methods combined with LBM for 2D cavity flows. 
Compared to single level MC, the reduced parameter grids show spectral convergence, achieving comparable performance to Smolyak sparse grid procedures. 
Recently, Zhong~\textit{et al.}~\cite{zhong2024stochastic} have established a stochastic Galerkin LBM for the simulation of fluid flows with uncertainty, where an average speedup greater than five is obtained compared to single level MC LBMs in multidimensional configurations in position and stochastic space. 
Further, Zhong~\textit{et al.}~\cite{zhong2025openlbuquncertaintyquantificationframework} have developed and validated a dedicated UQ-module for the scalable LBM library OpenLB~\cite{krause2020openlb}. 
OpenLB-UQ has been applied to simulate uncertain measurement-data assimilated wind flow over real urban geometries~\cite{zhong2025uncertaindataassimilationurban}. 
In summary, the combination of stability and scalability on current HPC hardware makes LBM a candidate for computing statistical solutions, which motivates the present work. 

In the present work, we develop a probabilistic MC LBM to approximate statistical solutions to the three-dimensional incompressible EE. 
Apart from the preliminary tests of the algorithm reported by Simonis~\cite{simonis2023lattice} and by Simonis and Mishra~\cite{simonis2024computing}, we are not aware of computations of three-dimensional statistical solutions to the NSE with a global-scale viscosity, nor of LBMs used for this purpose. 

We approximate the EE by computing a sequence of sample solutions to the incompressible NSE for decreasing but fixed values of viscosity $\vis >0$. 
The entropy-controlled, space-time dependent relaxation of the kinetic moments keeps the computations stable as the viscosity decreases. 
We complement the computations with a convergence analysis. 
Since the required stability and consistency properties of the fully discrete entropic scheme are not available as theorems in three dimensions, we isolate them in four explicitly stated standing assumptions and prove the limit statements conditional on them. 
First, we show that under diffusive scaling with fixed macroscopic viscosity, the bounds provided by the standing assumptions imply that the laws of the numerical ensemble converge along a subsequence to a limit satisfying the Foias--Temam Liouville formulation of the Navier--Stokes equations. 
Second, provided the structure-function scaling holds uniformly along the sequence, the vanishing viscosity theorem of Fjordholm, Mishra, and Weber (FMW)~\cite{fjordholm2021vanishing} applies to these limit families, so that the vanishing viscosity limit satisfies the FMW multi-point statistical Euler hierarchy. 
We further argue that the direct diagonal scaling of the LBM, in which the viscosity is coupled to the lattice spacing, leads to the same hierarchy without the coercivity assumption needed for the iterated limit. 
The limit measures inherit the weak-strong uniqueness principle for admissible measure-valued and statistical solutions~\cite{brenier2011weak,lanthaler2021statistical}: Whenever a strong Euler solution exists, an admissible statistical solution from the same initial data coincides with it. 
We recall the argument since it is used to interpret the computations.
For the computations, the approximate sample solutions to the NSE are obtained by statistically perturbing periodic Taylor--Green vortex (TGV) initial conditions~\cite{rohner2024efficient,simonis2024computing} and evolving them in time via the KBC LBM~\cite{simonis2024spectral} implemented in OpenLB-UQ~\cite{zhong2025openlbuquncertaintyquantificationframework}. 
First, we motivate the use of statistical solutions by demonstrating numerically the failure of pathwise strong convergence of the computed deterministic TGV solutions under the inviscid diagonal scaling. 
Moreover, we observe the scaling laws for the energy spectra and structure functions predicted by Kolmogorov's K41 theory~\cite{kolmogoroff1941local,kolmogorov1991local}, which are sufficient conditions for the convergence of statistical NSE solutions to statistical Euler solutions in three dimensions~\cite{fjordholm2021vanishing}. 
Next, we measure experimental orders of convergence of the computed statistical solutions in the Wasserstein metric (EOWC) along the diagonal scaling. 
Finally, we compare the computed measures with a reference produced by the hyperviscosity MC method from~\cite{rohner2024efficient}, extending the cross-solver comparison reported in two dimensions in~\cite{lanthaler2021statistical} to three dimensions and to the vanishing viscosity limit of the NSE.

This paper is structured as follows. 
In Section~\ref{sec:mathmodels} we recall the mathematical models used here, along with the respective statistical solutions and the kinetic models employed for their approximation. 
In Section~\ref{sec:methodology} we recall the entropic multi-relaxation LBM used as a deterministic solver for each sample in the present context. 
In Section~\ref{sec:approximation} we propose the approximation of statistical solutions based on the probabilistic MC LBM. 
In addition, we state the standing assumptions on the discrete dynamics, prove the convergence of the scheme toward Navier--Stokes solutions conditional on these assumptions, apply the vanishing viscosity theory of~\cite{fjordholm2021vanishing} to the limit families, recall the weak-strong uniqueness principle for the limit measures, and derive a convergence rate in the Wasserstein metric under explicit scaling assumptions.
Section~\ref{sec:numerics} documents the numerical experiments and discusses the findings. 
In Section~\ref{sec:conclusion} we conclude our work and suggest future research directions. 
The appendices collect the supporting material. 
Appendix~\ref{sec:summary} summarizes in Table~\ref{tab:overview} which statements are proven outright, which are known results applied to the limit objects, which are conditional on the standing assumptions, and which quantities are computed. 
Appendix~\ref{sec:appendix-algorithms} documents the implementation in OpenLB-UQ, the consumed computational resources, and the algorithms used to evaluate energy spectra, structure functions, and Wasserstein distances. 
Appendix~\ref{appsec:refSol} describes the spectral hyperviscosity scheme used to produce the independent reference solution, and Appendix~\ref{appsec:flowfields} provides further visualizations of the computed statistical flow fields.

\section{Mathematical models}
\label{sec:mathmodels}

We recall the concepts of statistical solutions to \eqref{eq:incNSE} for fixed viscosity and its vanishing limit, as well as a kinetic perspective to approximate deterministic solutions (here, single level MC samples) based on discretizing the BGK--Boltzmann equation in the diffusion limit.

\subsection{Incompressible Navier--Stokes equations and statistical solutions}

Given \eqref{eq:incNSE}, let \(\bm{u} \colon \Omega_{T} \to U\coloneqq \mathbb{R}^{d}, ( \bm{x}, t) \mapsto \bm{u}(\bm{x},t)\) denote the flow velocity as a function of space and time, where \(d\in \{2,3\}\), and let \(p \colon \Omega_{T} \to \mathbb{R}, (\bm{x}, t) \mapsto p (\bm{x}, t) \) denote the pressure acting as a Lagrange multiplier. 
The density is assumed to be constant, and \(p\) is rescaled accordingly.  
The initial data is defined by \(\bm{u}_{0}\colon \Omega \to \mathbb{R}^{d}\), which is assumed to be weakly divergence-free and in \(L^2 \left( \Omega; U\right) \), i.e.,
\begin{align}
    \bm{u}_{0}\in L^{2}_{\mathrm{div}}\left( \Omega ; U \right) = \left\{ \bm{u} \in L^{2}\left( \Omega; U \right) \; \left\vert \;	  \mathrm{div}_{\bm{x}} \bm{u} = 0 \text{ in the sense of distributions} \right. \right\} . 
\end{align}
The kinematic viscosity \(\vis>0 \) is given and is finite but arbitrarily small. 
Let \(\Omega = \mathbb{T}^{3}\). 
The system \eqref{eq:incNSE} is supplied only with an initial condition, i.e., we consider a Cauchy problem. 
As is known (see, e.g.,~\cite{buckmaster2019nonuniqueness,hou2024nearly,wang2023discovering,jia2014local,guillod2023numerical} and references therein), the well-posedness of \eqref{eq:incNSE} is questionable due to the lack of proven uniqueness. 
However, given the existence, we conjecture that, in the statistical framework, the multitude of solutions collapses to a unique statistical Euler solution as \(\vis \searrow 0\). 
To further support our approach, we briefly recall the concept of statistical solutions from~\cite{fjordholm2021vanishing} (and references therein) below. 

Considering statistical solutions (defined as families of probability measures on the tensor products \(U^{k}\) for \(k \in \mathbb{N}\)), let \(\bm{u}_{0} \in L^{2}_{\mathrm{div}}(\Omega ; U)\). 
We can interpret \eqref{eq:incNSE} as a Liouville equation on a function space, which defines the solution as a mapping of time \(t\in I \) to a probability measure on \(\hdiv\)~\cite{foias1976solutions,foias2001navier}. 
\begin{definition}[Statistical solutions]\label{def:foiastemamstatsol}
    A family of probability measures
    \begin{align}\label{eq:statsolCont}
        \bm{\mu}^{\vis}=(\mu_{t}^{\vis})_{0\leq t\leq T}, \quad \text{on }\hdiv
    \end{align}
    is a statistical solution of \eqref{eq:incNSE} with initial data $\mu_{0}^{\vis}$ and fixed viscosity \(\vis>0\), if the function 
    \begin{align}
        t\mapsto \int_{\hdiv} \beth(\bm{u})\,\mathrm{d}\mu_t^{\vis}(\bm{u})
    \end{align}
    is measurable on $[0,T]$ for every $\beth\in C_{\mathrm{b}}(\hdiv)$ (bounded continuous functions on \(\hdiv\)), and \(\bm{\mu}^{\vis}\) satisfies the additional conditions listed in~\cite[Definition 3.6 (b--d)]{fjordholm2021vanishing}, namely, a weak functional formulation of \eqref{eq:incNSE} for cylindrical test functions, a strengthened mean energy inequality, and that the function 
    \begin{align}
        t\mapsto \int_{\hdiv} \daleth \left( \left\| \bm{u} \right\|^{2}_{L^{2}(\Omega)} \right) \,\mathrm{d}\mu_{t}^{\vis}\left( \bm{u}\right) 
    \end{align}
    is continuous at \(t=0\) from the right for any nonnegative, nondecreasing \(\daleth\in C^{1}\left( \mathbb{R}; \mathbb{R}\right)\) with a bounded derivative. 
\end{definition}
\begin{remark}[Young measures]
    In our computations, the initial measure \(\mu_{0}^{\vis}\) is generated by an additive random perturbation \(\bm{u}_{0} = \bm{u}_{0}^{\mathrm{det}} + \bm{\mathfrak{s}}\) of a deterministic field \(\bm{u}_{0}^{\mathrm{det}}\), where \(\bm{\mathfrak{s}}\) is a random field such that $\bm{u}_{0} \sim \mu_{0}^{\vis}$. 
    In the sense of Foias and Prodi~\cite{foias1976solutions}, a statistical solution is a family $\bm{\mu}^\vis=(\mu^{\vis}_t)_{0\leq t\leq T}$ of probability measures on $\hdiv$ parametrized by time, i.e., a time-dependent Young measure. 
    See also~\cite[Definition 3.6]{fjordholm2021vanishing} and the references given there. 
\end{remark}

\subsection{Statistical vanishing viscosity limit and Kolmogorov scaling laws}

For \(\vis \searrow 0\), \eqref{eq:incNSE} formally passes into the incompressible EE. 
Deterministic weak solutions to the 3D Euler equations are susceptible to non-uniqueness and dissipative anomalies. 
In practice, access to initial conditions is often limited to a single realization only. 
To capture the underlying probabilistic nature of the flow, we still need to compute a statistical solution. 
By tracking a probability measure $\mu_t^\vis$ on the space of divergence-free vector fields, the Foias--Temam~\cite{foias2001navier} and FMW~\cite{fjordholm2021vanishing} frameworks provide a setting in which these flows can be described.
For the incompressible EE, a decay of the structure functions has been proven to be implied by the decay of time-averaged energy spectrum functions~\cite{lanthaler2021statistical}. 
Further, if a scaling assumption on the structure functions is fulfilled, it has been proven that the statistical solutions of the incompressible NSE converge with \(\vis \searrow 0\) to the statistical solutions of the EE~\cite[Theorem 4.8]{fjordholm2021vanishing}. 
Based on that, the evaluation of energy spectra in the sense of Kolmogorov's K41 theory~\cite{kolmogoroff1941local,kolmogorov1991local} can be used to at least indicate whether a computed sample is an approximation of a weak solution to the incompressible NSE for \(d=3\). 
Moreover, if the sequence of expectations of the energy spectra of several samples shows an asymptotic power law in the sense of K41 theory for \(\vis \searrow 0\), we deduce that a statistical solution of the incompressible EE is approximated~\cite{fjordholm2021vanishing}. 
We recall the necessary definitions for this reasoning below. 

The vanishing viscosity theorem \cite[Theorem 4.8]{fjordholm2021vanishing} assumes a weak statistical scaling relation between the second- and third-order longitudinal structure functions \cite[Assumption 1]{fjordholm2021vanishing}, whose only role in the proof is to provide a bound \(S^{2} (\tau, r) \leq C r^{\beta - 1/2}\) with \(\beta \in (1/2, 3/2)\) on the second-order structure functions, uniformly in \(\vis\) \cite[Remark 4.6 and Eq.~(4.35)]{fjordholm2021vanishing} (cf.\ \eqref{eq:energyScaling}).
Under this bound, the statistical solutions of the NSE converge to a statistical solution of the EE as \(\vis \searrow 0\). 
Here, the \(p\)th order time-integrated structure function for \(\tau \in (0,T]\), \(r>0\) reads
\begin{align} \label{eq:structures}
    S^p(\tau,r):= \left( \int_0^\tau S^{p}(r,t)\,\mathrm{d}t \right)^{\frac{1}{p}} 
\end{align}
and 
\begin{align} \label{eq:structuresLocal}
    S^p(r, t) = \int_{\Hn}\!\!\fint_{\mathbb{S}^2}\int_{\dom}\big\vert(\bm{u}(\bm{x}+r \bm{n})-\bm{u}(\bm{x}))\cdot \bm{n}\big\vert^p \,\mathrm{d}\bm{x}\,\mathrm{d}S(\bm{n})\,\mathrm{d}\mu_{t}^{\vis}(\bm{u})
\end{align}
denotes the time-local structure function. 
The global~\eqref{eq:structures} and time-local~\eqref{eq:structuresLocal} objects are distinguished by the order of their arguments. 
Note that for the signed odd-order structure functions of K41 phenomenology, e.g., the $4/5$-law, the modulus is omitted. 
Here only $p=2$ is used, where both conventions coincide. 
For small \(\vis\), Lanthaler, Mishra, and Par\'{e}s-Pulido~\cite{lanthaler2021statistical} prove the scaling implication
\begin{align}\label{eq:energyScaling}
    E_{T}\left( \bm{\mu}^{\vis}, \kappa\right) \lesssim \kappa^{-2\beta}  
    \quad 
    \Rightarrow \quad S^{2}\left( T, r\right) \lesssim r^{\beta- 1/2},
\end{align}
for \(1<2\beta <3\), where 
\begin{align}
    E_{T}\left( \bm{\mu}^{\vis}, \kappa\right) 
    \coloneqq \int_{0}^{T} \int_{\hdiv} E(\kappa, t; \bm{u}) \,\mathrm{d}\mu_{t}^{\vis}(\bm{u}) \,\mathrm{d}t
\end{align}
denotes the global energy spectrum,
\begin{align}\label{eq:energySpec}
E \left(\kappa, t; \bm{u}\right) 
=  
\oiint\limits_{S(\kappa)} \frac{1}{2} \Phi  (\bm{k}, t; \bm{u}) \,\mathrm{d} S(\bm{k}) 
\end{align}
is the energy contained in the wavenumber shell \(S(\kappa) = \{ \bm{k} \in \mathcal{K} : \vert \bm{k} \vert = \kappa\} \) of radius \(\kappa\), where \(\mathcal{K}\) is the set of admissible wavenumbers. 
We use \(\mathcal{K} = \mathbb{Z}^{d}\) for the periodic domain considered here, in which case the surface integral is approximated by shell averages over the integer lattice, see Appendix~\ref{sec:appendix-energySpec}.
Here  
\begin{align}\label{eq:fourierVelocityEnergy}
    \Phi (\bm{k}, t; \bm{u}) = \left\| \tilde{\bm{u}}(\bm{k}, t) \right\|_{2}^{2}
\end{align}
is the squared modulus of the Fourier coefficient 
\begin{align}
    \tilde{\bm{u}} \left( \bm{k}, t \right) = \frac{1}{\sqrt{2\pi}^{d}}\int_{\Omega} \bm{u} \left(\bm{x}, t \right) \exp\left(-\mathsf{i}\bm{k}\cdot \bm{x} \right)\,\mathrm{d} \bm{x} , \quad \bm{k} \in \mathbb{Z}^{d} . 
\end{align}

In the present work, we limit the discussion to homogeneous isotropic turbulence, i.e., \(\beta = \frac{5}{6}\) in \eqref{eq:energyScaling}~\cite{lanthaler2021statistical}. 
In addition, since we explicitly consider statistically nonstationary problems, we omit the time-integration in the computation of energy spectra and structure functions in Section~\ref{sec:numerics}. 

To quantify the weak-strong uniqueness of the statistical solutions, we rely on the optimal transport topology.
\begin{definition}[$p$-Wasserstein distance]\label{def:wasserstein}
    Let $X$ be a separable Banach space. For $p \geq 1$, the time-local $p$-Wasserstein distance between a (computed) statistical solution $\mu^{\vis}_{t}$ and a reference statistical solution $\mu^{\mathrm{ref}}_{t}$, both with finite $p$th moments ($\mu^{\vis}_{t}, \mu^{\mathrm{ref}}_{t} \in \mathcal{P}_{p}(X)$), is defined by
    \begin{equation}\label{eq:wasserstein}
        W_{p}(\mu^{\vis}_{t}, \mu^{\mathrm{ref}}_{t} )^p = \inf\limits_{\pi \in \Pi(\mu^{\vis}_{t}, \mu^{\mathrm{ref}}_{t} )} \int_{X^{2}} \| x - y \|_X^{p} \,\mathrm{d}\pi(x,y),
    \end{equation}        
    where the infimum is taken over all transport plans $\pi \in \Pi(\mu^{\vis}_{t}, \mu^{\mathrm{ref}}_{t}) \subset \mathcal{P}(X^{2})$. A probability measure $\pi$ is a valid transport plan if its marginals coincide with $\mu^{\vis}_{t}$ and $\mu^{\mathrm{ref}}_{t}$, meaning that
    \begin{equation}
        \int_{X^{2}} \left( F(x) + G(y) \right) \,\mathrm{d} \pi(x,y)  = \int_{X} F(x)\, \mathrm{d}\mu^{\vis}_{t}(x) + \int_{X} G(y) \,\mathrm{d} \mu^{\mathrm{ref}}_{t}(y)
    \end{equation}
    holds for all bounded continuous functions $F, G \in C_{b}(X)$. 
\end{definition}
In our subsequent derivation of numerical convergence rates via a Kuznetsov-type argument (see Section~\ref{sec:formalWrate}), we will specifically focus on the \(1\)-Wasserstein distance (\(p=1\)).
\begin{lemma}[Wasserstein metrization of weak convergence]\label{prop:weak-strong_continous}
	Let \(1 \leq p < 2\), and let the time-local measures \((\mu_{t}^{\vis})_{\vis > 0}\) and \(\mu_{t}^{\mathrm{ref}}\) satisfy the mean energy inequality of Definition~\ref{def:foiastemamstatsol}. 
    Then, for every \(t \in [0,T]\),
    \begin{align}
        \mu_{t}^{\vis} \rightharpoonup \mu_{t}^{\mathrm{ref}} \ (\vis \searrow 0)   \quad \Longleftrightarrow \quad W_{p} ( \mu_{t}^{\vis}, \mu_{t}^{\mathrm{ref}} ) \to 0 \ (\vis  \searrow 0 ).
    \end{align}
\end{lemma}
\begin{proof}
    This is \cite[Theorem 7.12]{Villani2003} (see also \cite{lanthaler2021conservation,lanthaler2021statistical,fjordholm2021vanishing}): Weak convergence is equivalent to convergence in the \(p\)-Wasserstein metric provided that the \(p\)th moments are uniformly integrable. 
    The mean energy inequality inherent to the Foias--Temam statistical solutions yields a uniform bound on the second moments, which implies uniform integrability of the \(p\)th moments for every \(p \in [1,2)\).
\end{proof}
\begin{remark}[Weak-strong uniqueness in the inviscid limit]
    If the energy spectra scale according to \eqref{eq:energyScaling}, then by \cite[Theorem 4.8]{fjordholm2021vanishing} the statistical solutions of the NSE converge to a statistical solution of the EE as \(\vis \searrow 0\). 
    Lemma~\ref{prop:weak-strong_continous} allows measuring this convergence equivalently in the \(p\)-Wasserstein metric, which underlies the numerical convergence studies in Section~\ref{sec:numerics}. 
    If, in addition, a strong Euler solution exists and the limit is admissible, the weak-strong uniqueness principle of \cite{brenier2011weak,lanthaler2021statistical}, recalled in Section~\ref{sec:approximation}, singles out the limit uniquely. 
    An illustration is provided in Figure~\ref{fig:limitIllustration}. 
\end{remark}
\begin{figure}[ht!]
    \centering
    \includegraphics[scale=1]{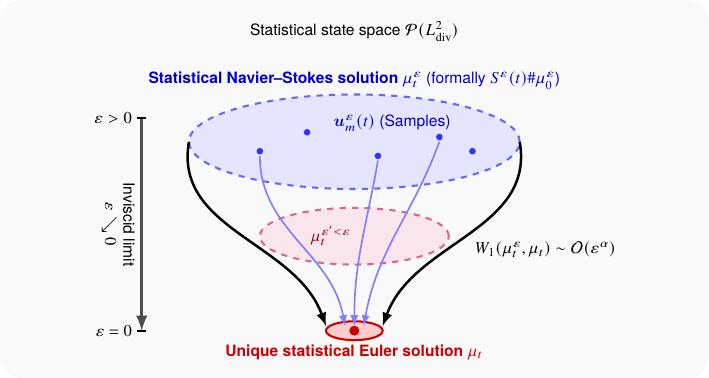}
    \caption{Schematic illustration of the weak-strong contraction: If the EE admit a strong solution, the family of Navier--Stokes weak statistical solutions contracts to it ($W_p(\mu_t^\vis, \mu_t) \to 0$ for $\vis \searrow 0$).}
    \label{fig:limitIllustration}
\end{figure}

\subsection{Boltzmann equation with Bhatnagar--Gross--Krook collision}

Let $\Omega \subseteq \mathbb{R}^d$ with \(d=3\) be a volume of rarefied gas which comprises many interacting particles. 
Assuming that all particles carry the same mass $m \in \mathbb{R}_{>0}$, we interpret them as point masses. 
The state of a one-particle system is assumed to depend on position $\bm{x} \in \Omega$ and velocity $\bm{v} \in \Xi$ at time $t \in I = [0,T]$, where $\Omega \subseteq \mathbb{R}^d$ denotes the positional space, $\Xi = \mathbb{R}^d$ is the velocity space, \(\mathfrak{P} \coloneqq \Omega \times \Xi\) is the phase space, and the Cartesian product \(\mathfrak{R} \coloneqq \Omega \times \Xi \times I \) defines the phase--time domain. 
\begin{definition}[Boltzmann equation]
    The one-particle phase-space density
    \begin{align} \label{eq:statisticPDF}
        f\colon\; \mathfrak{R} \to \mathbb{R}_{>0} ,\, (\bm{x},\bm{v},t) \mapsto f(\bm{x} , \bm{v}, t ) 
    \end{align}
    of the particles' positions \(\bm{x}\in \Omega\) and velocities \(\bm{v}\in \Xi\) at time \(t\in I\) defines the state of the dynamical system which is governed by the Boltzmann equation (BE)
    \begin{align} \label{eq:continuousBE}
        \left( \partial_{t} + \bm{v} \cdot \Grad_{\bm{x}} + \frac{\bm{F}}{m} \cdot \Grad_{\bm{v}} \right) f = \widetilde{J}(f,f ) \quad \text{in } \mathfrak{R}, 
    \end{align}
    supplemented with the initial condition \(f \vert_{t=0} = f_{0}\) in \(\mathfrak{P}\). 
    Here, \(\bm{F}\) denotes an external body force, which is set to \(\bm{0}\) in the force-free setting considered below, and the operator \(\widetilde{J}(f,f)\) models binary hard-sphere collisions. 
    Its explicit integral form is given, e.g., in~\cite{babovsky1998boltzmann}. 
\end{definition}
\begin{definition}[Kinetic moments]\label{def:moments}
    Let \(f\) be given in the sense of \eqref{eq:statisticPDF}. 
    Then, via weighted integration over \(\Xi = \mathbb{R}^{d}\), we define the moments 
    \begin{align}
        n_f &\colon \Omega \times I \to\mathbb{R}_{>0}, (\bm{x},t)\mapsto n_f(\bm{x},t) \coloneqq \int_{\mathbb{R}^d} f(\bm{x},\bm{v},t) \,\mathrm{d}\bm{v}, \label{eq:statisticPDFparticleDensity} \\
        \rho_f &\colon \Omega \times I \to\mathbb{R}_{>0}, (\bm{x},t)\mapsto \rho_f(\bm{x},t) \coloneqq mn_f(\bm{x},t),  \label{eq:statisticPDFmassDensity} \\
        \bm{u}_f &\colon \Omega \times I \to\mathbb{R}^{d}, (\bm{x},t)\mapsto \bm{u}_f(\bm{x},t) \coloneqq \frac{1}{n_f(\bm{x},t)} \int_{\mathbb{R}^d} \bm{v} f(\bm{x},\bm{v},t) \,\mathrm{d}\bm{v}, \label{eq:statisticPDFvelocity} \\
        p_f &\colon \Omega \times I  \to \mathbb{R}_{>0}, (\bm{x},t) \mapsto p_f(\bm{x},t) \coloneqq \frac{m}{d} \int_{\mathbb{R}^d} |\bm{v} - \bm{u}_f(\bm{x},t)|^2 f(\bm{x},\bm{v},t)\,\mathrm{d}\bm{v}, \label{eq:statisticPDFpressure}
    \end{align}
    respectively as particle density, mass density, velocity, and pressure. 
    Here and below, the moments of $f$ are indexed with \(\cdot_{f}\). 
\end{definition}
Notably, the absolute temperature $\theta$ is determined implicitly by an ideal gas assumption \(p_f=n_f R\theta\), where $R>0$ is the universal gas constant. 

To a controlled order in the characteristic scales, the above moments (Definition~\ref{def:moments}) approximate the macroscopic quantities conserved by the incompressible NSE~\cite{gorban2018hilbert}. 
Equilibrium states $f^{\mathrm{eq}}$, i.e., states with vanishing collision term \(\widetilde{J}(f^{\mathrm{eq}}, f^{\mathrm{eq}}) = 0\) in \(\mathfrak{R}\), exist~\cite{gorban2018hilbert}. 
In terms of the gas constant \(R = k_{\mathrm{B}}/m \in\mathbb{R}_{>0}\) (where \(k_{\mathrm{B}}\in\mathbb{R}_{>0}\) is the Boltzmann constant), the temperature $\theta\in\mathbb{R}_{>0}$, the particle density $n_f$, and the velocity $\bm{u}_f$, the equilibrium state is found to be of Maxwellian form
\begin{align} \label{eq:feq}
	f^{\mathrm{eq}} \colon \mathfrak{R} \to \mathbb{R}_{>0}, 
	(\bm{x},\bm{v},t) \mapsto \frac{n_f(\bm{x},t)}{\left(2 \pi R \theta \right)^{\frac{d}{2}}} \exp\left(-\frac{\left\vert \bm{v} - \bm{u}_f(\bm{x},t) \right\vert^2}{2 R \theta}\right) .
\end{align}
\begin{remark}[Normal distribution]
    We identify $f^{\mathrm{eq}}/n_f$ as a $d$-dimensional normal distribution for $\bm{v}\in\mathbb{R}^d$ with expectation $\bm{u}_f$ and covariance $R\theta \mathbf{I}_d$. 
    In this regard, the arguments of \(f^{\mathrm{eq}}\) regularly appear in terms of moments \(f^{\mathrm{eq}} ( n_{f}, \bm{u}_{f}, \theta)\) (see, for example,~\cite{he1997theory,junk2005asymptotic,lallemand2000theory}). 
\end{remark}
\begin{lemma}[Moment matching of the Maxwellian]
    The Maxwellian \eqref{eq:feq} reproduces the moments of \(f\), i.e., \(\rho_{f^{\mathrm{eq}}} = \rho_{f}\), \(\bm{u}_{f^{\mathrm{eq}}} = \bm{u}_{f}\), and \(p_{f^{\mathrm{eq}}} = p_{f}\). 
\end{lemma}
\begin{proof}
    From $f^{\mathrm{eq}}/n_f$ being a density function, we find
    \begin{align} \label{eq:boltzStatisticEquilibrium7}
      \rho_{f^{\mathrm{eq}}} &\stackrel{\eqref{eq:statisticPDFmassDensity}}{=} m \int_{\mathbb{R}^d} f^{\mathrm{eq}} (\bm{x},\bm{v},t)\,\mathrm{d}\bm{v}~= mn_f =\rho_f , \\
        \label{eq:boltzStatisticEquilibrium8}
      \bm{u}_{f^{\mathrm{eq}}} &\stackrel{\eqref{eq:statisticPDFvelocity}}{=} \frac{1}{n_{f^{\mathrm{eq}}}} \int_{\mathbb{R}^d} \bm{v} f^{\mathrm{eq}} (\bm{x},\bm{v},t)\,\mathrm{d}\bm{v}~= \bm{u}_f .
    \end{align}
    The covariance matrix of $f^{\mathrm{eq}}/n_f$ for a perfect gas verifies the matching of the pressure 
    \begin{align} \label{eq:boltzStatisticEquilibrium10}
      p_{f^{\mathrm{eq}}} &\stackrel{\eqref{eq:statisticPDFpressure}}{=}  \frac{1}{d}m \int_{\mathbb{R}^d} \left\vert\bm{v} - \bm{u}_{f^{\mathrm{eq}}}\right\vert^{2} f^{\mathrm{eq}} (\bm{x},\bm{v},t)\,\mathrm{d} \bm{v} = p_f. 
    \end{align}
\end{proof}
\begin{definition}[Bhatnagar--Gross--Krook collision]
    According to~\cite{bhatnagar1954model}, we simplify the collision operator $\widetilde{J}$ in \eqref{eq:continuousBE} to a BGK-type  
    \begin{align} \label{eq:collisionOperatorJ}%\label{eq: boltz statistic bgk 1}
        J(f) \coloneqq -\frac{1}{\tau_{\mathrm{rel}}}(f-{f}^{\mathrm{eq}}) \quad & \text{in } \mathfrak{R} ,  
    \end{align}
    where $\tau_{\mathrm{rel}} > 0$ denotes the relaxation time between collisions, and \(f^{\mathrm{eq}}(\bm{x},\bm{v},t)\) is now a formal particular Maxwellian determined by \(n_{f}\) and \(\bm{u}_{f}\).
\end{definition}
\begin{remark}[BGK moment conservation]
    Since the Maxwellian in \eqref{eq:collisionOperatorJ} is constructed from the moments of \(f\) itself, the moment matching \eqref{eq:boltzStatisticEquilibrium7} and \eqref{eq:boltzStatisticEquilibrium8} implies
    \begin{align}
        \int_{\mathbb{R}^{d}} \psi(\bm{v}) J(f) \,\mathrm{d}\bm{v} = 0, 
    \end{align}
    for \(\psi(\bm{v}) \in \{1, \bm{v}\}\). 
    Hence, \(\rho_{f}\) and \(\bm{u}_{f}\) are conserved during the BGK collision.
\end{remark}
\begin{definition}[BGK--Boltzmann equation]
    With $J$ from \eqref{eq:collisionOperatorJ} inserted into \eqref{eq:continuousBE}, the BGK--Boltzmann equation reads
    \begin{align}\label{eq:BGKBE}
        \left( \partial_t + \bm{v} \cdot \Grad_{\bm{x}} + \frac{\bm{F}}{m} \cdot \Grad_{\bm{v}}\right)  f = J(f) \quad & \text{in } \mathfrak{R} , 
    \end{align}
    where \( f (\cdot , \cdot , 0  ) = f_{0} (\cdot,\cdot) \coloneqq f^{\mathrm{eq}}(\cdot,\cdot,0) \) sets a suitable initial condition. 
    Here and below, we reuse the symbol \(f\) for the solution of \eqref{eq:BGKBE} instead of \eqref{eq:continuousBE}. 
\end{definition}
\begin{remark}[Global existence]
    The global existence of solutions to the BGK--Boltzmann equation \eqref{eq:BGKBE} has been rigorously proven in~\cite{perthame1989global}. 
    Weighted \(L^{\infty}\) bounds and uniqueness have later been established on bounded domains~\cite{perthame1993weighted} and in \(\mathbb{R}^{d}\)~\cite{mischler1996uniqueness}.  
\end{remark}
\begin{remark}[Diffusive limit to Navier--Stokes]\label{rem:convergenceBGKcontinuous}
    The BGK--Boltzmann equation \eqref{eq:BGKBE} is connected to the NSE \eqref{eq:incNSE} via the diffusive limit~\cite{saint-raymond2003bgk}. 
    To this end, a formal verification of the continuum balance equations \eqref{eq:incNSE} for the moments \(\rho_{f}\) and \(\bm{u}_{f}\) in Definition~\ref{def:moments} is conducted, e.g., in ~\cite{simonis2022limit}. 
    The vanishing of higher order terms in the diffusive limit \(K\!n \searrow 0\) is rigorously proven by Saint-Raymond~\cite{saint-raymond2003bgk}, where \(K\!n\) denotes the Knudsen number. 
    There, solutions \(f_{K\!n}\) to the \(K\!n\)-scaled BGK--Boltzmann equation are passed to the limit, and the corresponding velocity moments \(\bm{u}_{f_{K\!n}}\) are shown to converge to Leray weak solutions~\cite{leray1934mouvement} of the incompressible NSE~\cite[Theorem 1.2]{saint-raymond2003bgk}.  
    Here, we neglect the additional temperature equation appearing in the limit by imposing an ideal gas.  
    Transferring this continuous limit to the fully discrete scheme while sending the viscosity to zero requires discrete energy bounds.
    In the subsequent sections, these energy bounds are derived under explicit assumptions on the discrete dynamics.
\end{remark}
\begin{remark}[Low-resolution divergence]
    While the standard single-relaxation-time BGK operator~\eqref{eq:collisionOperatorJ} provides a foundational bridge to the NSE, once discretized, it is prone to numerical instabilities in the vanishing viscosity limit ($\vis \searrow 0$). 
    Because a single relaxation time couples the kinematic viscosity to the decay of all higher-order kinetic moments, under-resolved turbulent simulations typically suffer from spectral energy accumulation and blow-up. 
    This motivates the use of the entropic multi-relaxation KBC model in the discrete setting.
\end{remark}

\section{Deterministic numerical methodology}\label{sec:methodology}

Every sample of the MC method in Section~\ref{sec:approximation} is evolved with the same deterministic solver, which we fix in this section: a lattice Boltzmann scheme on the \(D3Q19\) velocity set with a second-order truncated equilibrium and the entropic multi-relaxation collision of KBC type~\cite{karlin2014gibbs,boesch2015entropic}, in the reduced form of~\cite{simonis2024spectral}. 
Preliminary tests of this algorithm as a sampler for statistical solutions are reported in~\cite{simonis2023lattice,simonis2024computing}. 
The randomness enters through the initial datum of \eqref{eq:incNSE} only. 
It is passed to the populations by the equilibrium initialization of Section~\ref{subsec:wellprepared}. 
Below, the moments of \(f\) are written without the index \(\cdot_{f}\).
\subsection{Entropic multi-relaxation lattice Boltzmann scheme}
Starting from the BGK--Boltzmann equation \eqref{eq:BGKBE}, the scheme is obtained in two steps: the velocity space \(\Xi\) is replaced by a finite set of \(q\) velocities, and space and time are discretized on a lattice with an implicit shift of the populations. 
Both steps are limit consistent with the incompressible NSE in the sense of~\cite{simonis2022limit}. 
The velocity-discrete system has commuting advection matrices~\cite{simonis2020relaxation,simonis2022constructing} and therefore exactly \(q\) scalar moments, one per velocity, i.e.,
\begin{align}
    \partial_{t} f_{i} + \bm{c}_{i} \cdot \Grad f_{i} = J_{i}(\bm{f}), 
\end{align}
for \(i=0,\ldots, q-1\), where \(\bm{c}_{i}\) are the velocities of the set \(DdQq\), the populations \(f_{i}\) are collected in \(\bm{f}\in \mathbb{R}^{q}\), and \(J_{i}\) is the velocity-discrete BGK collision. 
We use \(D3Q19\), i.e., \(q = 19\), which keeps the cost per lattice node low. 
The populations relax toward the second-order truncation of the discrete Maxwellian~\cite{bhatnagar1954model},
\begin{align}\label{eq:equilibrium}
    f_{i}^{\mathrm{eq}} \left( \rho, \boldsymbol{u} \right) =   \rho w_{i} \left[ 1 + \frac{1}{c_{s}^{2}}  c_{i\alpha} u_{\alpha} + \frac{1}{2c_{s}^{4}} (c_{i\alpha}c_{i\beta} - c_{s}^{2} \delta_{\alpha\beta}) u_{\alpha} u_{\beta} \right] ,
\end{align} 
with summation over repeated Greek indices, the lattice speed of sound \(c_{s} = 1/\sqrt{3}\), and the usual weights \(w_{i}\) of \(D3Q19\) (see, e.g.,~\cite{coreixas2019comprehensive}). 
Density and momentum are the conserved moments
\begin{align}
    \rho(\bm{x}, t) &= \sum_{i=0}^{q-1} f_i (\bm{x}, t),  \\
    \rho \bm{u} (\bm{x}, t) &= \sum_{i=0}^{q-1} \bm{c}_i f_i (\bm{x}, t).
\end{align}
\begin{remark}[Discrete velocity convergence]\label{rem:convergenceDVBGK}
    For the velocity-discrete BGK--Boltzmann equation, \cite{gu2024incompressible} constructs a local weak solution of the three-dimensional incompressible NSE as the diffusive limit and finds the limit to be of second order in the Knudsen number. 
    This agrees with the formal order of the velocity discretization derived in~\cite{simonis2022limit}. 
\end{remark}
Discretizing space and time with second order and shifting the populations implicitly~\cite{simonis2022limit} gives the lattice Boltzmann equation (LBE)
\begin{align}\label{eq:lbIteration}
    f_{i} \left( \bm{x} + \triangle t \bm{c}_i, t + \triangle t\right) \overset{\text{stream}}{=} f_{i}^{\star}(\bm{x}, t) \overset{\text{collide}}{=} f_{i}\left( \bm{x}, t\right) + \triangle t J_{i} \left(\bm{f}(\bm{x}, t)\right), 
\end{align}
for \(\left( \bm{x}, t\right) \in \Omega_{\triangle x} \times I_{\triangle t}\). 
Note that the limit toward the NSE survives the discretization as well~\cite{simonis2022limit}. 
\begin{remark}[Diffusive scaling]\label{rem:convergenceLBE}
    With the BGK collision, the mass and momentum moments of \(\bm{f}\) approximate \eqref{eq:incNSE} with first and second order in space, respectively~\cite{simonis2022limit,junk2003rigorous}. 
    Under the diffusive scaling \(\triangle t \sim \triangle x^{2}\) used throughout, the order in time is one. 
\end{remark}
For small \(\vis > 0\), the single relaxation time of the BGK collision is replaced by the entropy-controlled multi-relaxation of~\cite{karlin2014gibbs}. 
We take the reduced variant of~\cite{simonis2024spectral}, which pairs this collision with \(D3Q19\) and the equilibrium \eqref{eq:equilibrium}. 
The multi-relaxation collision reads
\begin{align}\label{eq:bgkCollisionOp}
    J_{i} (\bm{f}) = -  \bm{K}_{i} \left[ \bm{f} - \bm{f}^{\mathrm{eq}} \right], 
\end{align}
with the equilibrium vector \(\bm{f}^{\mathrm{eq}} = ( f_{i}^{\mathrm{eq}} )_{i} \in \mathbb{R}^{q}\) and the rows \(\bm{K}_{i} = ( K_{i,j} )_{j}\) of
\begin{align}
    \mathbf{K} = \mathbf{M}^{-1} \mathbf{S} \mathbf{M} \in \mathbb{R}^{q \times q}.
\end{align}
Here \(\mathbf{S} = \mathrm{diag}\left(\bm{s}\right)\) holds the relaxation frequencies \(\bm{s} = ( s_{i} )_{i}\), and \(\mathbf{M} \in \mathrm{GL}_{q}(\mathbb{R})\) maps populations to moments,
\begin{align}\label{eq:momentInnerProd}
    m_{i} = \langle \bm{\phi}^{i}, \bm{f} \rangle ,
\end{align}
with the Euclidean inner product \(\langle\cdot, \cdot\rangle\) on \(\mathbb{R}^{q}\). 
The rows \(\bm{\phi}^{i} = ( \phi^{i}_{j} )_{j}\) of \(\mathbf{M}\) are \(q\) linearly independent polynomials evaluated at the discrete velocities. 
The monomials give the raw moments
\begin{align}\label{eq:momentDefin}
    \rho \Gamma_{p_{1}, p_{2}, \ldots, p_{d}} = \bigl\langle \bigl( (c_{j})_{1}^{p_{1}} (c_{j})_{2}^{p_{2}} \cdots (c_{j})_{d}^{p_{d}} \bigr)_{j}, \bm{f} \bigr\rangle ,
\end{align}
and the tensors \(\mathbf{N}\), \(\mathbf{\Pi}\), \(\mathbf{T}\), \(\mathbf{Q}\), and \(\mathbf{A}\) of Table~\ref{tab:frequencySummary} are the combinations of raw moments listed in~\cite{simonis2024spectral}. 
Since
\begin{align}
    \mathbf{M} = \left(M_{i,j}\right)_{i,j} = \bigl( \phi_{j}^{i} \bigr)_{i,j} \in \mathbb{R}^{q\times q}
\end{align}
is invertible, \(\bm{m} = \mathbf{M} \bm{f}\) identifies population space with moment space. 
Table~\ref{tab:frequencySummary} sorts the moments into three groups, the conserved (kinematic) moments, the shear moments, and the remaining kinetic moments of order two to four, and assigns one relaxation frequency to each group. 
\begin{table}[ht!]
    \caption{Relaxation frequencies of moment tensors normalized with \(\rho\) (Table from~\cite{simonis2024spectral} with permission of the authors).}
    \label{tab:frequencySummary}
    \small
    \centering
        \begin{tabular}{c c c c} 
        \toprule
        Type & Tensor & Order & Relaxation frequency \\
        \midrule 
        \multirow{2}{*}{\fbox{\makecell{kinematic}}} 	& \(1\) & 0 & \multirow{2}{*}{\(\omega_{k} = 0\)} \\
            											& \( \bm{u}\) & 1 & \\
        \cmidrule{1-1}\cmidrule{4-4}
        \multirow{2}{*}{\fbox{\makecell{shear}}} 	& \(\mathbf{N}\) & 2 & \multirow{2}{*}{\makecell{\(\omega_{\vis} = \frac{2c_{s}^{2}}{2\vis + c_{s}^{2}}\)}} \\ 
                                                    & \(\mathbf{\Pi}\) & 2 & \\
        \cmidrule{1-1}\cmidrule{4-4}
        \multirow{3}{*}{\fbox{\makecell{kinetic \\ (hom)}}}		& \(\mathbf{T}\)	& 2 & \multirow{3}{*}{\makecell{\(\omega_{\eta}\)\\ \(\omega_{Q}\)\\ \(\omega_{A}\)}\makecell{\(\Bigg\}  \equiv \omega = \frac{\omega_{\vis}}{2} \gamma :\) \fbox{\makecell{entropy\\ controlled}}}} \\
        & \(\mathbf{Q}\) & 3 &  \\
    	& \(\mathbf{A}\) 			& 4 &  \\
        \bottomrule 
        \end{tabular}
\end{table}
The populations split accordingly,
\begin{align}
    \bm{f}^{\mathrm{neq}} \coloneqq  \bm{f} - \bm{f}^{\mathrm{eq}} = \bm{f}_{\vis}^{\mathrm{neq}} + \bm{f}^{\mathrm{neq}}_{\mathrm{hom}},
\end{align}
where \(\bm{f}_{\vis}^{\mathrm{neq}}\) carries the shear moments and \(\bm{f}^{\mathrm{neq}}_{\mathrm{hom}}\) the higher-order moments (hom). 
Note that the conserved part \(\bm{f}_{\mathrm{con}}\) has no non-equilibrium contribution. 
Two frequencies remain in \eqref{eq:lbIteration}, \(\omega_{\vis}\) for the shear moments and \(\omega(\bm{x}, t) = \frac{\omega_{\vis}}{2} \gamma (\bm{x}, t)\) for the higher-order moments (Table~\ref{tab:frequencySummary}), and the collision step becomes
\begin{equation}\label{eq:kbcCollision}
    f_i^\star (\bm{x}, t) = f_i(\bm{x}, t) - \omega_{\vis} f_{\vis,i}^{\mathrm{neq}} (\bm{x}, t) - \omega(\bm{x}, t) f_{\mathrm{hom},i}^{\mathrm{neq}} (\bm{x}, t). 
\end{equation}
Both frequencies are dimensionless, i.e., the factor \(\triangle t\) of \eqref{eq:lbIteration} is absorbed in lattice units, and \(\omega_{\vis} \in (0,2)\) is fixed by the viscosity of \eqref{eq:incNSE} in lattice units through the relation in Table~\ref{tab:frequencySummary}. 
For \(\gamma = 2\), both frequencies coincide, and \eqref{eq:kbcCollision} reduces to the BGK collision. 
The KBC model instead selects \(\gamma\) at every node and time step as the value for which the post-collision state has the smallest discrete \(\mathcal{H}\)-function,
\begin{align} 
    \gamma (\bm{x},t) 
     \coloneqq  \arg\min_{\gamma^{\prime}} \mathcal{H}(\bm{f}^\star(\bm{x}, t; \gamma^{\prime})) = \arg\min_{\gamma^{\prime}} \sum_{i=0}^{q-1} f_i^\star (\bm{x}, t; \gamma^{\prime}) \ln\left( \frac{f_i^\star(\bm{x}, t; \gamma^{\prime})}{w_i} \right), \label{eq:kbc_entropy}
\end{align}
where \(\bm{f}^{\star}(\cdot\,, \cdot\,; \gamma^{\prime})\) is \eqref{eq:kbcCollision} evaluated with \(\omega = \frac{\omega_{\vis}}{2} \gamma^{\prime}\). 
Solving \eqref{eq:kbc_entropy} exactly at every node is expensive. 
Expanding the stationarity condition to first order in the non-equilibrium gives the closed form of~\cite{boesch2015entropic}:
\begin{align}\label{eq:entropyControl}
    \gamma (\bm{x},t ) 
    \approx 
    \frac{2}{\omega_{\vis}}  - \left( 2  - \frac{2}{\omega_{\vis}} \right) \frac{\langle  \bm{f}_{\vis}^{\mathrm{neq}}(\bm{x},t) \vert \bm{f}_{\mathrm{hom}}^{\mathrm{neq}}(\bm{x}, t) \rangle}{\langle  \bm{f}_{\mathrm{hom}}^{\mathrm{neq}}(\bm{x}, t) \vert \bm{f}_{\mathrm{hom}}^{\mathrm{neq}} (\bm{x}, t)\rangle} .
\end{align}
Hence, the implementation evaluates \eqref{eq:entropyControl} and thus minimizes the \(\mathcal{H}\)-function in \eqref{eq:kbc_entropy} up to the truncation error of the expansion. 
The brackets denote the inner product on \(\mathbb{R}^{q}\) weighted by the equilibrium,
\begin{align}
    \langle \bm{X} \vert \bm{Y} \rangle = \sum_{i=0}^{q-1} \frac{ X_{i} Y_{i}}{ f_{i}^{\mathrm{eq}} } .
\end{align}
\begin{remark}[Entropic stabilization]
    In under-resolved turbulent flows, the frequency selected by \eqref{eq:kbc_entropy} removes energy from the highest resolved wavenumbers. 
    The numerical Fourier analysis in~\cite{simonis2024spectral} shows this effect on the spectra of the velocity field and of \(\gamma\) itself, so that the KBC collision acts as an implicit structural model. 
    With the exact minimizer of \eqref{eq:kbc_entropy}, the collision does not increase the discrete \(\mathcal{H}\)-function~\cite{karlin2014gibbs}. 
    With the approximation \eqref{eq:entropyControl}, this holds only up to the truncation error, which is why the corresponding property enters the analysis of Section~\ref{sec:approximation} as Assumption~\ref{ass:coercivity} (see also Definition~\ref{def:entropy_production}). 
    The entropic construction goes back to~\cite{karlin1999perfect} and stable computations at high Reynolds numbers are documented in~\cite{boesch2015entropic,simonis2024spectral}. 
\end{remark}
\begin{remark}[Convergence of the deterministic scheme]\label{rem:convergenceKBC}
    For the TGV flow, the integral turbulence quantities computed with this scheme converge under grid refinement in diffusive scaling with rates between first and second order, locally in time up to second order, and the frequencies \(\omega(\bm{x},t)\) approach the BGK value as the resolution increases~\cite{simonis2024spectral}. 
    The hydrodynamic limit of the KBC collision has been derived by a formal Chapman--Enskog expansion~\cite{karlin2014gibbs,boesch2015entropic}. 
    Note that, to the knowledge of the authors, a convergence proof for the fully discrete scheme is not available to date. 
\end{remark}
\begin{remark}[Appropriateness of the deterministic scheme]
    The formal and numerical results recalled above indicate that the KBC LBM approximates single solutions of the NSE, and its collision is local, so the parallel efficiency of LBM is retained. 
    Both properties are crucial for a sampler that is run thousands of times in three dimensions. 
\end{remark}

\section{Discrete approximation and limit theory for statistical solutions}\label{sec:approximation}

In the present work, the probability measure \(\bm{\mu}^{\vis}\) is approximated empirically using an ensemble of \(M\) discrete samples \(\bm{u}_{m}^{N}\) (\(m=1,2,\ldots, M\)). 
For a computational grid with resolution \(N\) per spatial dimension, the empirical measure at time \(t\in I\) is defined as
\begin{align}\label{eq:statSolDisc}
    \mu_{t}^{\vis} \approx \mu^{\vis, N,M}_{t} \coloneqq \frac{1}{M} \sum\limits_{m=1}^{M} \delta_{\bm{u}_{m}^{N} (t)} ,
\end{align}
where \(\delta_{\bm{u}_{m}^{N}(t)}\) denotes the Dirac measure centered at the state \(\bm{u}_{m}^{N}(t)\). 
Each sample \(\bm{u}_{m}^{N}\) is obtained by evolving independent and identically distributed (IID) initial data \(\bm{u}_{0,m}\) for \eqref{eq:incNSE} in time using the LBM at a fixed macroscopic viscosity \(\vis>0\). 
Thus, for a given viscosity \(\vis\) and spatial grid resolution \(N\), we draw \(M\) IID samples from the well-prepared initial probability measure
\begin{align}
    \{\bm{u}_{0,1}, \bm{u}_{0,2}, \ldots, \bm{u}_{0,M}\} \sim \mu_{0}^{\vis}.
\end{align}
These samples are independently evolved via the KBC LBM to construct the discrete ensemble \(\{ \bm{u}_{m}^{N}(t) \mid m=1, 2, \ldots, M\}\) for each \(\vis\).

We work on the lattice $\Lambda_\eps = \eps \mathbb{Z}^3 \cap \Omega$ with spacing $\eps = 2\pi/N$. 
In Section~\ref{sec:numerics}, lengths are normalized by the domain length $2\pi$, so that $\eps = 1/N$ in normalized units. 
Throughout this section, the dual space $H^{-3}(\Omega)$ is used for the temporal estimates. 
Thus, any Sobolev index strictly larger than $5/2$ is admissible, and we fix the value $3$ for definiteness. 
Moreover, $\tilde{\bm{u}}^{\eps} = P^{\eps} \bm{u}^{\eps}$ denotes the piecewise constant spatiotemporal extension of the discrete random velocity field 
\begin{align}
    \bm{u}^\eps = (1/\rho^{\eps})\sum_{i=0}^{q-1} \bm{c}_{i} f_{i}
\end{align}
obtained from the moments in Definition~\ref{def:moments}. 
The extension operator $P^{\eps}$ maps lattice fields to $L^\infty(0,T; L^2(\Omega))$.

\paragraph{Standing assumptions.}
The convergence statements of this section are conditional. 
The following assumptions isolate the properties of the fully discrete KBC dynamics that are supported by formal Chapman--Enskog analysis and the numerical evidence of Section~\ref{sec:numerics}, but that are not proven for the scheme. 
All subsequent results reference these assumptions explicitly. 
We write 
\begin{align}
    \langle g, h \rangle \coloneqq \eps^{3} \sum_{\bm{x} \in \Lambda_{\eps}} g(\bm{x}) h(\bm{x})
\end{align} 
for the discrete spatial duality pairing, extended to the piecewise constant representatives.
\begin{assumption}[Uniform low Mach regime, well-prepared evolution]\label{ass:lowMach}
    There exist deterministic constants $G > 0$ and $\eps_{0} > 0$ such that, $\mu_{0}^{\eps}$-almost surely, for all $\eps \in (0, \eps_{0}]$, all nodes $\bm{x} \in \Lambda_{\eps}$, all $t \in I_{\triangle t}$, and all $i$,
    \begin{enumerate}
        \item[(i)] the populations admit the representation $f_{i}^{\eps} = w_{i} \left( 1 + \eps g_{i}^{\eps} \right)$ with $\vert g_{i}^{\eps} \vert \leq G$, and
        \item[(ii)] the density fluctuation satisfies $\vert \rho^{\eps}(\bm{x}, t) - 1 \vert \leq G \eps^{2}$.
    \end{enumerate}
\end{assumption}
Assumption~\ref{ass:lowMach} encodes that the diffusive (low Mach) scaling of the initialization (Section~\ref{subsec:wellprepared}) persists under the discrete dynamics: (i) yields the uniform positivity 
\begin{align}
    f_{i}^{\eps} \geq w_{i}/2, \quad \text{ for } \eps \leq \min(\eps_{0}, 1/(2G))
\end{align}
and, via $\rho^{\eps} \tilde{u}^{\eps}_{\alpha} = \sum_{i} w_{i} g_{i}^{\eps} c_{i\alpha}$, the uniform velocity bound 
\begin{align}
    \|\tilde{\bm{u}}^{\eps}\|_{L^{\infty}((0,T) \times \Omega)} \leq C_{G} ;
\end{align} 
and (ii) states that no $\mathcal{O}(\eps)$ acoustic waves are generated, so that the rescaled pressure $\pi^{\eps} \coloneqq c_{s}^{2} \eps^{-2} (\rho^{\eps} - 1)$ is uniformly bounded.
\begin{assumption}[Entropic coercivity]\label{ass:coercivity}
    There exist constants $c > 0$ and $C \geq 0$, independent of $\eps$ and $\vis$, such that the global discrete entropy production of Definition~\ref{def:entropy_production} satisfies, $\mu_{0}^{\eps}$-almost surely,
    \begin{align}
        \mathcal{D}^\eps(\bm{f}(t)) \ge c \, \vis \, \eps^{2} \, \|\bm{\nabla}^\eps \tilde{\bm{u}}^\eps(t)\|_{L^2(\Omega)}^2 - C \eps^{5/2} \quad \text{for all } t \in I_{\triangle t}.
    \end{align}
\end{assumption}
Any exponent strictly larger than $2$ is admissible in the error term. 
We fix the value $5/2$ for definiteness. 
For the exact entropic selection of the relaxation parameter (solving the minimization \eqref{eq:kbc_entropy} exactly), the non-negativity $\mathcal{D}^{\eps} \geq 0$ holds by construction~\cite{karlin2014gibbs}. 
Assumption~\ref{ass:coercivity} additionally quantifies that the entropy production controls the shear dissipation, which is the leading-order content of the Chapman--Enskog expansion~\cite{boesch2015entropic,simonis2022limit}.
\begin{assumption}[Weak consistency of the discrete moment balances]\label{ass:consistency}
    There exist residual fields $\mathcal{R}_{0}^{\eps}$ and $\bm{\mathcal{R}}^{\eps}$ such that for every $\chi \in C^{\infty}(\Omega)$, every $\bm{\phi} \in C^{\infty}(\Omega; \mathbb{R}^{3})$, and every $t \in I_{\triangle t}$,
    \begin{align}
        \langle \partial_t^\eps \rho^{\eps}, \chi \rangle &= \langle \rho^{\eps} \tilde{\bm{u}}^{\eps}, \bm{\nabla} \chi \rangle + \langle \mathcal{R}_{0}^{\eps}, \chi \rangle , \label{eq:massConsistency}\\
        \langle \partial_t^\eps (\rho^{\eps} \tilde{\bm{u}}^{\eps}), \bm{\phi} \rangle &= \langle \rho^{\eps} \tilde{\bm{u}}^{\eps} \otimes \tilde{\bm{u}}^{\eps}, \bm{\nabla} \bm{\phi} \rangle + \vis \langle \tilde{\bm{u}}^{\eps}, \bm{\Delta} \bm{\phi} \rangle + \langle \pi^{\eps}, \bm{\nabla} \cdot \bm{\phi} \rangle + \langle \bm{\mathcal{R}}^{\eps}, \bm{\phi} \rangle , \label{eq:momentumConsistency}
    \end{align}
    where $\partial_{t}^{\eps} h(t) \coloneqq \eps^{-2}(h(t + \eps^{2}) - h(t))$, $\pi^{\eps} = c_{s}^{2} \eps^{-2} (\rho^{\eps} - 1)$, and the residuals are uniformly bounded and vanish in the limit,
    \begin{align}
        \sup_{\eps} \E_{\mu_{0}^{\eps}} \left[ \| \mathcal{R}_{0}^{\eps} \|_{L^{2}(0,T; H^{-3}(\Omega))}^{2} + \| \bm{\mathcal{R}}^{\eps} \|_{L^{2}(0,T; H^{-3}(\Omega))}^{2} \right] & < \infty , 
        \\ 
        \lim_{\eps \to 0} \E_{\mu_{0}^{\eps}} \left[ \| \mathcal{R}_{0}^{\eps} \|^{2} + \| \bm{\mathcal{R}}^{\eps} \|^{2} \right] &= 0 ,
    \end{align}
    uniformly for $\vis \in (0, \vis_{0}]$.
\end{assumption}
Assumption~\ref{ass:consistency} is the weak-form statement of the limit consistency of the scheme: The exact discrete moment balances of \eqref{eq:lbIteration} evaluate, through the Chapman--Enskog expansion, to the incompressible NSE fluxes up to residuals. 
For the underlying BGK dynamics, this chain is established in~\cite[Theorem 2 and Propositions 1 and 3]{simonis2022limit}, where the lattice Boltzmann equation is shown to be limit consistent with the discrete velocity Boltzmann equation of order $\mathcal{O}(\eps^{2})$, the second-order moment closes with the Newtonian stress $\bm{P} = p \bm{I} - 2 \vis \rho \bm{D}(\bm{u})$, and the resulting weak formulation converges to the incompressible NSE. 
There, the residuals $\mathcal{R}_{0}^{\eps}$ and $\bm{\mathcal{R}}^{\eps}$ collect the corresponding truncation terms, which are formally of order $\mathcal{O}(\eps^{2})$. 
A rigorous counterpart for smooth solutions of BGK-type LBM is given in~\cite{junk2003rigorous}. 
For the dynamically relaxed KBC operator employed here, the analogous statement is open. 
To this end, we impose it as an assumption. 
The uniformity in $\vis$ makes the assumption applicable along the diagonal path $\vis_{\eps} = c \eps$.
\begin{assumption}[Uniform structure-function scaling]\label{ass:scaling}
    There exist $C > 0$ and $\beta \in (1/2, 3/2)$ such that the second-order structure functions of the discrete statistical solutions satisfy
    \begin{align}
        S^{2}(\tau, r) \leq C r^{\beta - 1/2} \quad \text{for all } \tau \in (0,T], \ r > 0,
    \end{align}
    uniformly along the considered sequence of solutions (in $\vis$ for the iterated path of Section~\ref{subsec:vanishingViscosityLimit}, in $\eps$ for the diagonal path of Section~\ref{subsec:diagonalLimit}). 
    By \cite[Lemma 4.7, Eq.~(4.35), and Theorem 2.4(iii)]{fjordholm2021vanishing}, this implies that every weak-$\ast$ limit point of the associated correlation measures is diagonally continuous.
\end{assumption}
Assumption~\ref{ass:scaling} is the discrete counterpart of the structure-function bound \cite[Remark 4.6]{fjordholm2021vanishing} through which the scaling assumption enters the proof of \cite[Theorem 4.8]{fjordholm2021vanishing}, and it is consistent with the K41 scaling observed for the computed solutions in Section~\ref{sec:numerics}.
Since $S^{2}(\tau,r)$ is obtained from the time-local structure functions \eqref{eq:structuresLocal} by integration over $(0,\tau]$, a bound $S^{2}(r,t) \leq C r^{2\beta - 1}$ on the latter that is uniform in $t \in (0,T]$ implies the stated bound up to a factor $T^{1/2}$.
The numerical evaluation in Section~\ref{sec:numerics} is performed time-locally.

It is noteworthy that the passage from finite to infinite sample size is unconditional and can be condensed into the following statement. 
\begin{lemma}[Empirical consistency]\label{lem:empirical}
    Fix $\eps$ (equivalently $N$), $\vis > 0$, and $t \in I$. 
    Then, as $M \to \infty$, the empirical measures \eqref{eq:statSolDisc} converge weakly to $\mu_{t}^{\vis, N} = S^{\eps}(t) \# \mu_{0}^{\eps}$, almost surely with respect to the law of the IID sample sequence.
\end{lemma}
\begin{proof}
    The samples $\bm{u}_{m}^{N}(t)$ are IID with law $\mu_{t}^{\vis, N}$, since they arise from IID initial data under the deterministic, measurable solution map $S^{\eps}(t) \circ I^{\eps}$ (a finite composition of continuous maps, see Section~\ref{subsec:wellprepared}). 
    The claim is Varadarajan's theorem on the separable metric space $L^{2}(\Omega; U)$~\cite{varadarajan1958convergence}. 
\end{proof}
\begin{proposition}[Conditional statistical convergence to Euler]\label{prop:weak-strong_discrete}
    Let Assumptions~\ref{ass:lowMach}--\ref{ass:scaling} hold, and assume that the limit families provided by Theorem~\ref{thm:nse_limit} satisfy the remaining conditions of \cite[Definition 3.6]{fjordholm2021vanishing}. 
    Then the following holds.
    \begin{enumerate}
        \item[(i)] For every $\vis > 0$, $N$, and $t \in I$, the empirical measures converge weakly, $\mu^{\vis,N,M}_{t} \rightharpoonup \mu^{\vis,N}_{t}$ as $M \to \infty$, $\mu_{0}^{\eps}$-almost surely.
        \item[(ii)] For every $\vis > 0$, there is a sequence $\eps_{k} \to 0$ such that the laws of $\tilde{\bm{u}}^{\eps_{k}}$ on $L^{2}(0,T; L^{2}(\Omega))$ converge weakly as $k \to \infty$ to the law of a field inducing a family $\bm{\mu}^{\vis}$ that satisfies the Foias--Temam Liouville formulation of \eqref{eq:incNSE}.
        \item[(iii)] There is a sequence $\vis_{l} \to 0$ such that the correlation measures of $\bm{\mu}^{\vis_{l}}$ converge weakly-$\ast$ as $l \to \infty$ to a statistical solution of the incompressible Euler equations in the sense of \cite{fjordholm2021vanishing}.
    \end{enumerate}
    Along the diagonal $\vis_{\eps} = c\eps$, the first equation of the Euler hierarchy is obtained under Assumptions~\ref{ass:lowMach}, \ref{ass:consistency}, and \ref{ass:scaling} alone (Proposition~\ref{thm:diagonal_limit}). 
    Whenever time-local measures converge weakly in (i) or (iii), the convergence holds equivalently in $W_{p}$, $p \in [1,2)$, by Lemma~\ref{prop:weak-strong_continous}, since the $p$th moments are uniformly integrable by Assumption~\ref{ass:lowMach}. 
\end{proposition}
The proof is given at the end of Section~\ref{subsec:diagonalLimit}. 
\begin{remark}
    The entropic stabilization of the KBC LBM provides the bounds used along the inviscid limit, but it does not guarantee convergence to a \textit{strong} statistical solution of the Euler equations. 
    To the knowledge of the authors, in three dimensions, the global existence of strong Euler solutions is open. 
    What can be expected is convergence to an admissible (dissipative) statistical solution, i.e., a limit measure satisfying the mean energy inequality. 
    Under Assumption~\ref{ass:coercivity}, the entropy balance \eqref{eq:entropyTelescope} suggests that the computed limits are admissible, but we do not verify this here (see Remark~\ref{rem:afterThm4.1} and the end of the proof of Proposition~\ref{thm:diagonal_limit}). 
    For admissible limits, the weak-strong uniqueness principle of~\cite{brenier2011weak,lanthaler2021statistical}, recalled in Proposition~\ref{thm:weakStrong}, applies: If a strong Euler solution exists for the given initial data, the limit measure coincides with it on its interval of existence.
\end{remark}
Recall that \(\Omega = \mathbb{T}^{3}\) (Section~\ref{sec:mathmodels}). 
Due to the lack of global uniqueness of strong solutions in 3D, we utilize the measure-valued Foias--Temam framework~\cite{foias2001navier}, tracking a probability measure $\mu_t^\vis$ on the space of divergence-free vector fields. 
In the inviscid limit, we adopt the FMW multi-point correlation framework~\cite{fjordholm2021vanishing}. 
We consider two paths: First, the iterated limit (fixing $\vis > 0$ to recover the NSE statistical solutions, then taking $\vis \searrow 0$), and second, the direct diagonal path ($\vis_\eps \sim \eps \to 0$), which leads to the FMW statistical Euler hierarchy under Assumptions~\ref{ass:lowMach}, \ref{ass:consistency}, and \ref{ass:scaling}, i.e., without Assumption~\ref{ass:coercivity}.

\subsection{Discretization and well-prepared initial data}\label{subsec:wellprepared}

To prevent the formation of $\mathcal{O}(1)$ acoustic waves that would destroy the analytical compactness required for the incompressible diffusive limit, the initial data must be well-prepared. 
We formally initialize the LBM ensemble using a Leray-projected randomized TGV (RTGV). 
Let $\bm{u}_{0}^{\mathrm{det}} \colon \Omega \to \mathbb{R}^{3}$ denote the deterministic TGV flow (details are given in Section~\ref{sec:numerics}). 
The TGV initial condition is perturbed with IID random variables $X_{\alpha,i,j,k} \sim \mathcal{U}_{[-\mathfrak{w},\mathfrak{w}]}$ to obtain $\bm{u}_{0} = \bm{u}_{0}^{\mathrm{det}} + \mathbf{P}\bm{\mathfrak{s}}$, where $\mathbf{P}$ is the Leray projection. 
The pre-projected perturbation $\bm{\mathfrak{s}} = (\mathfrak{s}_{\alpha})_{1\leq\alpha\leq 3}$ is
\begin{align}\label{eq:alpha_perturbation}
\mathfrak{s}_{\alpha} = \frac{1}{8} \sum\limits_{(i,j,k) \in \{0,1\}^3} X_{\alpha,i,j,k} \mathfrak{a}_{i}(2x) \mathfrak{a}_{j}(2y) \mathfrak{a}_{k}(2z),
\end{align}
with $\mathfrak{a}_{0}(x) = \sin(x)$ and $\mathfrak{a}_{1}(x) = \cos(x)$. 
Note that $\bm{u}_{0} \in L^{\infty}(\Omega; U)$ holds almost surely with a deterministic bound, since $\bm{u}_{0}^{\mathrm{det}}$ is a fixed trigonometric field and the noise has compact support. 
The KBC LBM algorithm \eqref{eq:kbcCollision} is initialized with the equilibrium populations \eqref{eq:equilibrium} based on
\begin{align}\label{eq:initialPopulations}
    \bm{f}_{0} (\bm{x})= \bm{f}^{\mathrm{eq}}(1,\bm{u}_{0}(\bm{x})) . 
\end{align}

The initial measure and its evolution in time are both push-forwards, and we separate the two as follows. 
First, the initial probability measure $\mu_0^\eps$ is constructed as the push-forward of a base noise distribution, specifically, the independent tensor product of the $24$ uniform random variables $\bigotimes_{\alpha=1}^{3} \bigotimes_{(i,j,k) \in \{0,1\}^{3}} \mathcal{U}_{[-\mathfrak{w},\mathfrak{w}]}$, under an initialization mapping $I^\eps$. 
This mapping assembles the perturbation \eqref{eq:alpha_perturbation}, applies the Leray projection $\mathbf{P}$, restricts to the lattice $\Lambda_{\eps}$, and generates the initial lattice populations via the equilibrium initialization \eqref{eq:initialPopulations}. 
This yields
\begin{equation}
    \mu_0^\eps = I^\eps \# \left( \bigotimes_{\alpha=1}^{3} \bigotimes_{(i,j,k) \in \{0,1\}^{3}} \mathcal{U}_{[-\mathfrak{w},\mathfrak{w}]} \right).
\end{equation}
Subsequently, the time-dependent statistical solution $\mu_t^\eps$ is defined as the push-forward of this initial measure under the discrete solution operator $S^\eps(t)$ of the lattice Boltzmann scheme. 
Since $S^{\eps}(t)$ is a well-defined deterministic mapping, this is an exact identity rather than an approximation (the approximation enters only through the finite sample size $M$ in \eqref{eq:statSolDisc}). 
Consequently, the statistical flow state at any time $t$ is the mapping of the raw initial noise through the composed initialization and time-stepping operators
\begin{equation}
    \mu_t^\eps = S^\eps(t) \# \mu_0^\eps = \left( S^\eps(t) \circ I^\eps \right) \# \left( \bigotimes_{\alpha=1}^{3} \bigotimes_{(i,j,k) \in \{0,1\}^{3}} \mathcal{U}_{[-\mathfrak{w},\mathfrak{w}]} \right).
\end{equation}
This composition tracks how the initial randomness propagates through the deterministic dynamics of the LBM to form the time-evolving statistical ensemble.

The equilibrium initialization is well-prepared in the following quantitative, entropic sense. 
For lattice populations $\bm{f}$, define the local and global (Bregman) relative entropies
\begin{align}
    \mathcal{H}_{\mathrm{rel,loc}}(\bm{f}) &\coloneqq \sum_{i=0}^{q-1} \left[ f_i \ln\left(\frac{f_i}{w_i}\right) - f_{i} + w_{i} \right] \geq 0 , \label{eq:relEntropy1}
    \\
    \mathcal{H}_{\mathrm{rel}}(t) &\coloneqq \eps^{3} \sum_{\bm{x} \in \Lambda_\eps} \mathcal{H}_{\mathrm{rel,loc}}(\bm{f}(\bm{x},t)) ,
    \label{eq:relEntropy2}
\end{align}
respectively, where the non-negativity of the integrand follows from the convexity of $x \mapsto x \ln x$. 
On any interval $f_{i}/w_{i} \in [1/2, 3/2]$, the integrand is two-sided comparable to its quadratic expansion,
\begin{align}\label{eq:bregmanEquivalence}
    \frac{1}{3} w_{i} \left( \frac{f_{i}}{w_{i}} - 1 \right)^{2} \leq f_i \ln\left(\frac{f_i}{w_i}\right) - f_{i} + w_{i} \leq w_{i} \left( \frac{f_{i}}{w_{i}} - 1 \right)^{2} .
\end{align}
\begin{lemma}[Well-prepared initialization]\label{lem:wellprepared}
    There exist deterministic constants $C_{\mathrm{wp}} > 0$ and $\eps_{1} > 0$ such that, $\mu_{0}^{\eps}$-almost surely, $\mathcal{H}_{\mathrm{rel}}(0) \leq C_{\mathrm{wp}} \eps^{2}$ for all $\eps \in (0, \eps_{1}]$.
\end{lemma}
\begin{proof}
    By \eqref{eq:initialPopulations} and \eqref{eq:equilibrium} evaluated at unit density and lattice velocity $\eps \bm{u}_{0}$ (diffusive scaling, see also the scaled discrete Maxwellian in~\cite[Definition 13]{simonis2022limit}), we have 
    \begin{align}
        f_{0,i}/w_{i} - 1 = \eps c_{s}^{-2} \bm{c}_{i} \cdot \bm{u}_{0} + \eps^{2} Q_{i}(\bm{u}_{0})
    \end{align}
    with $Q_{i}$ quadratic. 
    Further, since $\|\bm{u}_{0}\|_{L^{\infty}} \leq C$ holds almost surely with a deterministic constant, we obtain 
    \begin{align}
        \vert f_{0,i}/w_{i} - 1 \vert \leq C' \eps \leq 1/2
    \end{align} 
    for $\eps \leq \eps_{1}$. 
    The upper bound in \eqref{eq:bregmanEquivalence} then yields
    \begin{align}
        \mathcal{H}_{\mathrm{rel}}(0) &\leq \eps^{3} \sum_{\bm{x}} \sum_{i} w_{i} \left( \eps c_{s}^{-2} \bm{c}_{i} \cdot \bm{u}_{0} + \eps^{2} Q_{i} \right)^{2} \\
        &\leq C \eps^{2} \, \eps^{3} \sum_{\bm{x}} \vert \bm{u}_{0}(\bm{x}) \vert^{2} \left( 1 + C \eps \right) \\
        &\leq C_{\mathrm{wp}} \eps^{2} ,
    \end{align}
    using $\sum_{i} w_{i} c_{i\alpha} c_{i\beta} = c_{s}^{2} \delta_{\alpha\beta}$ and the almost sure uniform bound on $\bm{u}_{0}$.
\end{proof}

\subsection{Convergence to statistical Navier--Stokes solutions (fixed viscosity)}

We analyze the numerical limit where the viscosity $\vis > 0$ is fixed and the lattice spacing $\eps$ (and with it the time step $\triangle t = \eps^2$ in diffusive scaling) is variable. 
Before passing to any limit, we establish stability bounds for the discrete scheme. 
The foundation of our analysis relies on the discrete entropy structure of the KBC collision operator.
\begin{definition}[Discrete entropy production] \label{def:entropy_production}
    Let $\bm{f}(\bm{x}, t) = (f_i(\bm{x}, t))_{i=0}^{q-1}$ denote the vector of discrete populations at a lattice node $\bm{x} \in \Lambda_\eps$. 
    The local discrete entropy production rate $\mathcal{D}^\eps_{\mathrm{loc}}$, driven by the multi-relaxation collision operator \eqref{eq:kbcCollision}, is defined as the rate of relative-entropy dissipation during the relaxation step, i.e.,
    \begin{equation} \label{eq:local_dissipation}
        \mathcal{D}^\eps_{\mathrm{loc}}(\bm{f}(\bm{x}, t)) \coloneqq - \frac{1}{\eps^2} \Big( \mathcal{H}_{\mathrm{rel,loc}}(\bm{f}^\star(\bm{x}, t)) - \mathcal{H}_{\mathrm{rel,loc}}(\bm{f}(\bm{x}, t)) \Big),
    \end{equation}
    where $\triangle t = \eps^2$ and $\bm{f}^\star$ is the post-collision state vector. 
    Since the collision conserves the local mass $\sum_{i} f_{i}$, the definition agrees with the rate computed from the local $\mathcal{H}$-function $\sum_{i} f_{i} \ln (f_{i}/w_{i})$, which differs from $\mathcal{H}_{\mathrm{rel,loc}}$ in \eqref{eq:relEntropy1} only by that conserved mass. 
    The exact entropic selection of \(\gamma\) via the minimization \eqref{eq:kbc_entropy} guarantees $\mathcal{D}^\eps_{\mathrm{loc}} \geq 0$~\cite{karlin2014gibbs}. 
    For the closed-form approximation \eqref{eq:entropyControl}, the corresponding quantitative statement is Assumption~\ref{ass:coercivity}. 
    The volume-weighted summation of the local rates yields the global discrete entropy production
    \begin{equation}
        \mathcal{D}^\eps(\bm{f}(t)) \coloneqq \eps^{3} \sum_{\bm{x} \in \Lambda_\eps} \mathcal{D}^\eps_{\mathrm{loc}}(\bm{f}(\bm{x}, t))
    \end{equation}
    over the lattice $\Lambda_\eps$.
\end{definition}
\begin{lemma}[Uniform discrete bounds] \label{lem:energy}
    Under Assumption~\ref{ass:lowMach}, the macroscopic fields satisfy $\mu_{0}^{\eps}$-almost surely and uniformly in $\eps \leq \min(\eps_{0}, \eps_{1}, 1/(2G))$ and $\vis \in (0, \vis_{0}]$,
    \begin{equation}\label{eq:sup_bound}
        \sup_{t \in [0,T]} \|\tilde{\bm{u}}^\eps(t)\|_{L^{\infty}(\Omega)} \leq C_{G} , 
        \qquad 
        \sup_{t \in [0,T]} \| \rho^{\eps}(t) - 1 \|_{L^{\infty}(\Omega)} \leq G \eps^{2} .
    \end{equation}
    If, in addition, Assumption~\ref{ass:coercivity} holds, then there exists $C > 0$, independent of $\eps$ and $\vis$, such that $\mu_{0}^{\eps}$-almost surely,
    \begin{equation} \label{eq:energy_bound_as}
        \vis \int_0^T \|\bm{\nabla}^\eps \tilde{\bm{u}}^\eps(t^{\prime})\|_{L^2(\Omega)}^2 \,\mathrm{d}t^{\prime} \leq C . 
    \end{equation}
    In particular, the expectation bound
    \begin{equation} \label{eq:energy_bound}
        \E_{\mu_0^\eps} \left[ \sup_{t \in [0,T]} \|\tilde{\bm{u}}^\eps(t)\|_{L^2(\Omega)}^2 + \vis \int_0^T \|\bm{\nabla}^\eps \tilde{\bm{u}}^\eps(t^{\prime})\|_{L^2(\Omega)}^2 \,\mathrm{d}t^{\prime} \right] \leq C
    \end{equation}
    holds.
\end{lemma}
\begin{proof}
    All estimates are pathwise, with deterministic constants, such that \eqref{eq:energy_bound} then follows by taking the expectation. 
    
    For \eqref{eq:sup_bound}, Assumption~\ref{ass:lowMach}(ii) is the second bound directly. 
    For the first, we write $f_{i} = w_{i}(1 + \eps g_{i}^{\eps})$ as in Assumption~\ref{ass:lowMach}(i). 
    Using $\sum_{i} \bm{c}_{i} w_{i} = \bm{0}$, the momentum moment reads 
    \begin{align}
        \rho^{\eps} \tilde{u}^{\eps}_{\alpha} = \sum_{i} w_{i} g_{i}^{\eps} c_{i\alpha} ,
    \end{align} 
    hence
    \begin{align}
        \vert \rho^{\eps} \tilde{u}^{\eps}_{\alpha} \vert \leq G \sum_{i} w_{i} \vert c_{i\alpha} \vert \leq C G ,
    \end{align} 
    and 
    \begin{align}
        \rho^{\eps} \geq 1 - G\eps^{2} \geq 1/2
    \end{align}
    yields 
    \begin{align}
        \|\tilde{\bm{u}}^{\eps}\|_{L^{\infty}} \leq C_{G}.
    \end{align}

    The kinetic energy bound in \eqref{eq:energy_bound} follows from 
    \begin{align}
        \|\tilde{\bm{u}}^{\eps}(t)\|_{L^{2}(\Omega)} \leq \vert \Omega \vert^{1/2} C_{G} .
    \end{align}

    For the dissipation bound \eqref{eq:energy_bound_as}, we use the exact entropy balance of the scheme. 
    The collision step changes $\mathcal{H}_{\mathrm{rel}}$ by exactly $-\eps^{2} \mathcal{D}^{\eps}(\bm{f}(t_{n}))$ per Definition~\ref{def:entropy_production}, while the streaming step merely permutes the population values over the periodic lattice and hence conserves $\mathcal{H}_{\mathrm{rel}}$ exactly. 
    Telescoping over the time steps $t_{n} = n \eps^{2}$ yields the identity
    \begin{align}\label{eq:entropyTelescope}
        \mathcal{H}_{\mathrm{rel}}(T) + \eps^{2} \sum_{n \,:\, n\eps^{2} < T} \mathcal{D}^{\eps}(\bm{f}(t_{n})) = \mathcal{H}_{\mathrm{rel}}(0) .
    \end{align}
    Since $\mathcal{H}_{\mathrm{rel}}(T) \geq 0$ by \eqref{eq:relEntropy1} and \eqref{eq:relEntropy2}, Lemma~\ref{lem:wellprepared} gives 
    \begin{align}
        \eps^{2} \sum_{n} \mathcal{D}^{\eps}(\bm{f}(t_{n})) \leq C_{\mathrm{wp}} \eps^{2} .
    \end{align}
    Inserting the lower bound of Assumption~\ref{ass:coercivity} and using that the piecewise constant time extension turns the sum into the exact integral, 
    \begin{align} 
        \eps^{2} \sum_{n} \|\bm{\nabla}^{\eps} \tilde{\bm{u}}^{\eps}(t_{n})\|_{L^{2}}^{2} = \int_{0}^{T} \|\bm{\nabla}^{\eps} \tilde{\bm{u}}^{\eps}(t^{\prime})\|_{L^{2}}^{2} \,\mathrm{d}t^{\prime}, 
    \end{align}
    we obtain
    \begin{align}
        c \, \vis \, \eps^{2} \int_{0}^{T} \|\bm{\nabla}^\eps \tilde{\bm{u}}^\eps(t^{\prime})\|_{L^2(\Omega)}^2 \,\mathrm{d}t^{\prime} \leq C_{\mathrm{wp}} \eps^{2} + C T \eps^{5/2} ,
    \end{align}
    and dividing by $c \eps^{2}$ yields \eqref{eq:energy_bound_as} with $C = (C_{\mathrm{wp}} + C T \eps_{0}^{1/2})/c$. 
    We remark that, by \eqref{eq:bregmanEquivalence} and Assumption~\ref{ass:lowMach}, the relative entropy is two-sided comparable to the discrete kinetic-energy-type quantity $\eps^{2} \cdot \eps^{3} \sum_{\bm{x}} \sum_{i} w_{i} (g_{i}^{\eps})^{2}$, which links the entropic and energetic viewpoints. 
\end{proof}
\begin{remark}
    The KBC collision bounds the non-hydrodynamic moments through the approximate entropy inequality of Definition~\ref{def:entropy_production}, which keeps the ghost moments from growing. 
    In the derivation above, this structural property enters through exactly two assumptions: the persistence of the low Mach regime (Assumption~\ref{ass:lowMach}) and the quantitative coercivity of the entropy production on the shear moments (Assumption~\ref{ass:coercivity}). 
    We point out that both are consequences of the exact entropic construction at the formal Chapman--Enskog level but neither is proven for the fully discrete KBC dynamics in 3D.
\end{remark}
While Lemma~\ref{lem:energy} provides spatial compactness via the discrete $H^1$-norm, applying classical compactness theorems requires simultaneous control over the temporal oscillations. 
We secure this by bounding the discrete time derivative of the momentum $\bm{m}^{\eps} \coloneqq \rho^{\eps} \tilde{\bm{u}}^{\eps}$ in a dual space. 
Note that 
\begin{align}
    \|\bm{m}^{\eps} - \tilde{\bm{u}}^{\eps}\|_{L^{\infty}(0,T; L^{2}(\Omega))} \leq \|\rho^{\eps} - 1\|_{L^{\infty}} \|\tilde{\bm{u}}^{\eps}\|_{L^{2}} \leq C \eps^{2}
\end{align}
by Lemma~\ref{lem:energy}, so the two fields are asymptotically equivalent. 
\begin{lemma}[Discrete time derivative] \label{lem:time_deriv}
    Let Assumptions~\ref{ass:lowMach} and \ref{ass:consistency} hold. 
    Then
    \begin{align}
        \E_{\mu_0^\eps} \left[ \| \partial_t^\eps \bm{m}^\eps \|_{L^2(0,T; H^{-3}(\Omega))}^2 \right] \leq C 
    \end{align}
    with $C$ independent of $\eps$ and $\vis \in (0, \vis_{0}]$.
\end{lemma}
\begin{proof}
    The zeroth and first discrete velocity moments of \eqref{eq:lbIteration} yield exact update rules for $\rho^{\eps}$ and $\bm{m}^{\eps}$. 
    Assumption~\ref{ass:consistency} states their weak-form evaluation \eqref{eq:momentumConsistency}. 
    We bound the $H^{-3}(\Omega)$ norm of each term on the right-hand side of \eqref{eq:momentumConsistency} by testing against $\bm{\phi} \in H^3(\Omega)$.
 
    \textit{1. Convective term:}
    By Lemma \ref{lem:energy}, $\rho^{\eps} \tilde{\bm{u}}^\eps \otimes \tilde{\bm{u}}^\eps \in L^\infty(0,T; L^1(\Omega))$ almost surely with 
    \begin{align}
        \|\rho^{\eps} \tilde{\bm{u}}^\eps \otimes \tilde{\bm{u}}^\eps\|_{L^{1}} \leq 2 \|\tilde{\bm{u}}^{\eps}\|_{L^{2}}^{2} , 
    \end{align} 
    hence
    \begin{align}\label{eq:discTimeConvective}
        |\langle \rho^{\eps} \tilde{\bm{u}}^\eps \otimes \tilde{\bm{u}}^\eps, \bm{\nabla} \bm{\phi} \rangle| \leq 2 \|\tilde{\bm{u}}^\eps\|_{L^2}^2 \|\bm{\nabla} \bm{\phi}\|_{L^\infty} \leq C \|\bm{\phi}\|_{H^{3}} ,
    \end{align}
    where the Sobolev embedding $H^{2}(\Omega) \hookrightarrow L^\infty(\Omega)$ in 3D is applicable because $2 > 3/2$. 
    We emphasize that this is the step that requires the Sobolev index to exceed $5/2$. 
    
    \textit{2. Viscous term:}
    Directly, 
    \begin{align}\label{eq:discTimeViscous}
        \vert \vis \langle \tilde{\bm{u}}^{\eps}, \bm{\Delta} \bm{\phi} \rangle \vert \leq \vis \|\tilde{\bm{u}}^{\eps}\|_{L^{2}} \|\bm{\Delta} \bm{\phi}\|_{L^{2}} \leq C \vis \|\bm{\phi}\|_{H^{2}} \leq C \|\bm{\phi}\|_{H^{3}} .
    \end{align}
    
    \textit{3. Pressure term:}
    By Assumption~\ref{ass:lowMach}(ii), the rescaled pressure is uniformly bounded by 
    \begin{align}
        \|\pi^{\eps}\|_{L^{\infty}((0,T) \times \Omega)} \leq c_{s}^{2} G , 
    \end{align}
    hence
    \begin{align}\label{eq:discTimePressure}
        \vert \langle \pi^\eps, \bm{\nabla} \cdot \bm{\phi} \rangle \vert \leq \|\pi^{\eps}\|_{L^{2}} \|\bm{\nabla} \bm{\phi}\|_{L^{2}} \leq C \|\bm{\phi}\|_{H^{1}} . 
    \end{align}
    
    \textit{4. Residual:} We observe that 
    \begin{align}
        \vert \langle \bm{\mathcal{R}}^{\eps}, \bm{\phi} \rangle \vert \leq \|\bm{\mathcal{R}}^{\eps}\|_{H^{-3}} \|\bm{\phi}\|_{H^{3}}
    \end{align} 
    and that the residual norm is uniformly square-integrable in expectation by Assumption~\ref{ass:consistency}.
    
    \textit{Conclusion:}
    Combining the four bounds, squaring, integrating over time, and taking the expectation $\E_{\mu_0^\eps}$, the almost sure uniform bounds of Lemma~\ref{lem:energy} (in particular $\|\tilde{\bm{u}}^{\eps}\|_{L^{2}}^{4} \leq C$ almost surely) and the residual bound of Assumption~\ref{ass:consistency} render the right-hand side bounded by a constant independent of $\eps$ and $\vis$.
\end{proof}

Equipped with uniform spatial and temporal bounds, we apply a discrete-in-time version of the Aubin--Lions--Simon compactness theorem \cite{simon1986,dreher2012compact} to pass to the limit in the nonlinear convective flux.
\begin{theorem}[Subsequential limit at fixed viscosity] \label{thm:nse_limit}
    Let Assumptions~\ref{ass:lowMach}, \ref{ass:coercivity}, and \ref{ass:consistency} hold and fix $\vis > 0$. 
    Then there exists a subsequence $\eps_{k} \to 0$ and a family of probability measures $\bm{\mu}^{\vis} = (\mu_{t}^{\vis})_{0 \leq t \leq T}$ on $\Hn$ with uniformly bounded mean energy such that the laws of $\tilde{\bm{u}}^{\eps_{k}}$ converge weakly on $L^{2}(0,T; L^{2}(\Omega))$ to the law of a limit field inducing $\bm{\mu}^{\vis}$, and $\bm{\mu}^{\vis}$ satisfies the Foias--Temam Liouville (weak) formulation of the 3D NSE \eqref{eq:incNSE}. 
\end{theorem}
\begin{proof}
    Let $B = L^2(\Omega)$ and $Y = H^{-3}(\Omega)$. 
    We realize that piecewise constant extensions of lattice fields do not belong to $H^{1}(\Omega)$. 
    Hence, the spatial compactness is instead obtained from the discrete gradient bound through the translation estimate
    \begin{align}\label{eq:translationEstimate}
        \|\tilde{\bm{u}}^{\eps}(\cdot + \bm{h}) - \tilde{\bm{u}}^{\eps}\|_{L^{2}(\Omega)}^{2} \leq C \, \vert \bm{h} \vert \left( \vert \bm{h} \vert + \eps \right) \|\bm{\nabla}^{\eps} \tilde{\bm{u}}^{\eps}\|_{L^{2}(\Omega)}^{2}, \qquad \bm{h} \in \mathbb{R}^{3},
    \end{align}
    which holds for piecewise constant functions on a lattice \cite{eymard2000finite} and whose right-hand side vanishes as $\vert \bm{h} \vert \to 0$ uniformly in $\eps$. 
    By Lemma~\ref{lem:energy} (with $\vis$ fixed), 
    \begin{align}
        \int_{0}^{T} \|\bm{\nabla}^{\eps} \tilde{\bm{u}}^{\eps}\|_{L^{2}}^{2} \,\mathrm{d}t    
    \end{align} 
    is almost surely bounded by a deterministic constant. 
    The same holds for $\bm{m}^{\eps} = \rho^{\eps} \tilde{\bm{u}}^{\eps}$, since the discrete product rule together with the inverse inequality (Lemma~\ref{lem:energy})
    \begin{align}
        \|\bm{\nabla}^{\eps} (\rho^{\eps} - 1)\|_{L^{2}} \leq (C/\eps) \|\rho^{\eps} - 1\|_{L^{2}} \leq C \eps
    \end{align}
    and the uniform velocity bound give 
    \begin{align}
        \|\bm{\nabla}^{\eps} \bm{m}^{\eps}\|_{L^{2}} \leq C (1 + \|\bm{\nabla}^{\eps} \tilde{\bm{u}}^{\eps}\|_{L^{2}}). 
    \end{align}
    Lemma~\ref{lem:time_deriv} bounds the discrete time derivative of $\bm{m}^{\eps}$ in $L^{2}(0,T; Y)$ in expectation. 
    By the Kolmogorov--Riesz theorem and the discrete-in-time Aubin--Lions--Simon theorem \cite{simon1986,dreher2012compact}, a family of piecewise constant (in space and time) functions with time-integrated translation moduli of the type \eqref{eq:translationEstimate} bounded uniformly in $\eps$ and with discrete time derivatives bounded in $L^{2}(0,T;Y)$ is relatively compact in $L^{2}(0,T; B)$. 
    Since the time-derivative bound holds in expectation only, we conclude by tightness. 
    For $R > 0$, let $K_{R}$ denote the set of such functions whose time-integrated translation modulus and discrete time-derivative norm are bounded by $R$. 
    Then $K_{R}$ is relatively compact in $L^{2}(0,T;B)$. 
    By Chebyshev's inequality and Lemma~\ref{lem:time_deriv}, we have 
    \begin{align}
        \mu_{0}^{\eps}(\bm{m}^{\eps} \notin K_{R}) \leq C/R^{2} 
    \end{align} 
    uniformly in $\eps$.
    Hence, the family of laws of $\bm{m}^{\eps}$ is tight on $L^{2}(0,T; L^{2}(\Omega))$. 
    Since 
    \begin{align}
        \|\bm{m}^{\eps} - \tilde{\bm{u}}^{\eps}\|_{L^{\infty}(0,T;L^{2})} \leq C \eps^{2} ,
    \end{align} 
    the same holds for $\tilde{\bm{u}}^{\eps}$. 
    By Prokhorov's theorem (see, e.g., \cite{Billingsley1999}), a subsequence of the laws converges weakly. 
    The Skorokhod representation theorem \cite{Skorokhod1956,Billingsley1999} provides that random variables $\hat{\bm{u}}^{\eps_k} \to \hat{\bm{u}}^\vis$ strongly in $L^2(0,T; L^2(\Omega))$ almost surely, and hence
    \begin{equation} \label{eq:strong_conv}
        \| \hat{\bm{u}}^{\eps_k} \otimes \hat{\bm{u}}^{\eps_k} - \hat{\bm{u}}^\vis \otimes \hat{\bm{u}}^\vis \|_{L^1(0,T; L^1(\Omega))} \to 0 \quad \text{a.s.}
    \end{equation}
    The limit is divergence-free: Testing the mass consistency \eqref{eq:massConsistency} with $\chi \in C^{\infty}_{c}((0,T) \times \Omega)$ and summing by parts in time, the left-hand side converges to
    \begin{align}
        -\int_{0}^{T} \langle \rho^{\eps}, \partial_{t} \chi \rangle \,\mathrm{d}t \to -\int_{0}^{T} \langle 1, \partial_{t} \chi \rangle \,\mathrm{d}t = 0
    \end{align} 
    by Lemma~\ref{lem:energy}, while $\mathcal{R}_{0}^{\eps} \to 0$. 
    Hence 
    \begin{align}
        \int_{0}^{T} \langle \hat{\bm{u}}^{\vis}, \bm{\nabla} \chi \rangle \,\mathrm{d}t = 0
    \end{align}
    for all such $\chi$, i.e., $\hat{\bm{u}}^{\vis}(t) \in \Hn$ for a.e.\ $t$. 
    We then test the momentum consistency \eqref{eq:momentumConsistency} against divergence-free cylindrical test functions (for which the pressure term vanishes identically), and pass to the limit via \eqref{eq:strong_conv} and Assumption~\ref{ass:consistency}. 
    Interchanging limits and expectations by Vitali's convergence theorem (uniform integrability is granted by the almost sure uniform bounds, see \cite[Theorem 16.14]{Billingsley1995}), the time slices of the limit law satisfy the Foias--Temam generalized Liouville equation \cite{foias2001navier}. 
    The mean energy bound follows from Fatou's lemma and \eqref{eq:energy_bound}. 
\end{proof}
\begin{remark}\label{rem:afterThm4.1}
    Theorem~\ref{thm:nse_limit} verifies the measurability condition, the Liouville weak formulation, and a uniform mean energy bound. 
    The remaining structural conditions of \cite[Definition 3.6]{fjordholm2021vanishing}, i.e., the sharp form of the mean energy inequality and the right-continuity at $t = 0$, are expected to follow from the entropy balance \eqref{eq:entropyTelescope} and the well-preparedness (Lemma~\ref{lem:wellprepared}), respectively. 
    Note that we do not verify these here. 
\end{remark}

\subsection{Vanishing viscosity limit to statistical Euler}\label{subsec:vanishingViscosityLimit}

Having established the consistency of the MC KBC LBM with the 3D NSE \eqref{eq:incNSE}, we now turn to the inviscid limit $\vis \searrow 0$. 
In this regime, the physical dissipation vanishes. 
From the uniform bounds established in Lemma~\ref{lem:energy}, the continuous kinetic energy inequality inherited by the Foias--Temam measure dictates via \eqref{eq:energy_bound} that 
\begin{align}
    \vis \int_0^T \|\bm{\nabla} \bm{u}^\vis\|_{L^2}^2 \,\mathrm{d}t \le C.
\end{align}
Consequently, as $\vis \searrow 0$, the bound on the spatial gradient degenerates, i.e., the inequality only yields
\begin{align}
    \int_{0}^{T} \|\bm{\nabla} \bm{u}^\vis\|_{L^2}^2 \,\mathrm{d}t \leq \frac{C}{\vis} ,
\end{align}
and strong spatial compactness in $L^2$ is no longer available.
Following~\cite{fjordholm2021vanishing}, we therefore pass to the multi-point correlation framework, in which the nonlinear convective term becomes linear in the two-point correlation measure.
\begin{definition}[Correlation measures]
    For any integer $k \ge 1$ and spatial coordinates $\bm{x} = (\bm{x}_1, \dots, \bm{x}_k) \in \Omega^k$, we define the $k$-point correlation measure $\boldsymbol{\nu}^{\vis, k}_{t} \in \mathcal{M}^+(\Omega^k \times \mathbb{R}^{3k})$ generated by the Foias--Temam measure $\mu^\vis_t$. 
    Let $\bm{\xi} \in \Hn$ denote the dummy variable of integration representing a single realization of the fluid velocity field. 
    For any bounded continuous test function $\psi \in C_b(\Omega^k \times \mathbb{R}^{3k})$, the measure is defined via the spatial integration of the velocity fields, i.e.,
    \begin{equation} \label{eq:correlation_def}
        \langle \boldsymbol{\nu}^{\vis, k}_{t}, \psi \rangle \coloneqq \int_{\Hn} \int_{\Omega^k} \psi(\bm{x}, \bm{\xi}(\bm{x}_1), \dots, \bm{\xi}(\bm{x}_k)) \, \mathrm{d}\bm{x} \, \mathrm{d}\mu^\vis_t(\bm{\xi}),
    \end{equation}
    which is well-defined for Lebesgue-a.e.\ $\bm{x} \in \Omega^{k}$ and $\mu_{t}^{\vis}$-a.e.\ $\bm{\xi}$ in the duality sense of \cite[Section 2]{fjordholm2021vanishing}. 
\end{definition}
Because the macroscopic kinetic energy is uniformly bounded by
\begin{align}\label{eq:expectationBound}
    \sup_{t \in [0,T]} \int_{\Hn} \|\bm{\xi}\|_{L^2(\Omega)}^2 \,\mathrm{d}\mu_{t}^{\vis}(\bm{\xi}) \le C ,
\end{align} 
the sequence of measures $\{ \boldsymbol{\nu}^{\vis, k}_{t} \}_{\vis > 0}$ has uniformly bounded second moments. 
Together with the uniform diagonal continuity provided by Assumption~\ref{ass:scaling}, this is the input of the compactness theory for correlation measures of \cite[Section 2]{fjordholm2021vanishing} (see also \cite[Theorem 2.4]{lanthaler2021statistical}). 
This then yields a subsequence $\vis \searrow 0$ along which the correlation measures converge weakly-$\ast$ to limit correlation measures $\boldsymbol{\nu} = \{\nu^k\}_{k=1}^\infty$ and along which one- and two-point observables of quadratic growth converge as well. 
In this framework, the nonlinear convective tensor product is evaluated via the diagonal trace of the $2$-point measure, which does not require strong spatial convergence. 
Note that the diagonal continuity granted by Assumption~\ref{ass:scaling} is what makes this trace well-defined.
\begin{corollary}[Consistency with statistical Euler]\label{thm:eulerConsistency}
    Let Assumption~\ref{ass:scaling} hold along the sequence $\vis \searrow 0$ of limit families provided by Theorem~\ref{thm:nse_limit}, and assume that these families satisfy the remaining conditions of \cite[Definition 3.6]{fjordholm2021vanishing} (cf.\ Remark~\ref{rem:afterThm4.1}). 
    Then, by \cite[Theorem 4.8]{fjordholm2021vanishing}, whose proof uses the scaling assumption only through the bound \cite[Eq.~(4.35)]{fjordholm2021vanishing} that Assumption~\ref{ass:scaling} provides via \cite[Lemma 4.7]{fjordholm2021vanishing}, the limit hierarchy $\boldsymbol{\nu}$ is a statistical solution of the incompressible Euler equations in the sense of \cite{fjordholm2021vanishing}, i.e., it satisfies the FMW multi-point statistical Euler hierarchy. 
    Moreover, by \cite[Remark 4.9, Eq.~(4.41)]{fjordholm2021vanishing}, its one-point marginal satisfies the energy admissibility \eqref{eq:admissibility}. 
\end{corollary}
We do not reproduce the proof of \cite[Theorem 4.8]{fjordholm2021vanishing}. 
The bounded support of the initial measure required there holds since $\|\bm{u}_{0}\|_{L^{\infty}(\Omega)} \leq C$ almost surely. 
We only indicate the two mechanisms in the first hierarchy equation that are reused in the discrete argument of Section~\ref{subsec:diagonalLimit}.

First, the viscous dissipation term vanishes in the limit. 
Applying the Cauchy--Schwarz inequality against the uniform $L^2$ kinetic energy bound \eqref{eq:expectationBound} yields
\begin{equation} \label{eq:viscous_vanishing}
    \left| \vis \int_{\Hn} \langle \bm{\xi}, \bm{\Delta} \bm{\phi} \rangle \, \mathrm{d}\mu^\vis_{t}(\bm{\xi}) \right| \leq \vis \left( \int_{\Hn} \|\bm{\xi}\|_{L^2}^2 \, \mathrm{d}\mu^\vis_{t} \right)^{1/2} \|\bm{\Delta} \bm{\phi}\|_{L^2} \leq C \vis \to 0
\end{equation}
for every smooth divergence-free test field $\bm{\phi}$. 

Second, the convective flux is linear in the 2-point correlation measure $\boldsymbol{\nu}^{\vis, 2}$ and passes to the limit as the diagonal trace, 
\begin{equation} \label{eq:fmw_diagonal}
    \lim_{\vis \searrow 0} \int_{\Hn} \langle \bm{\xi} \otimes \bm{\xi}, \bm{\nabla} \bm{\phi} \rangle \, \mathrm{d}\mu^\vis_{t}(\bm{\xi}) = \int_{\Omega} \int_{\mathbb{R}^{6}} (\bm{\xi}_1 \otimes \bm{\xi}_2) : \bm{\nabla} \bm{\phi}(\bm{x}) \, \mathrm{d}\nu^2_{t, \bm{x}, \bm{x}}(\bm{\xi}_1, \bm{\xi}_2) \, \mathrm{d}\bm{x}.
\end{equation}
Since the spatial diagonal $\{\bm{x}_{1} = \bm{x}_{2}\}$ is a Lebesgue null set in $\Omega^{2}$, weak-$\ast$ convergence of the correlation measures alone does not control the trace on it. 
The diagonal evaluation \eqref{eq:fmw_diagonal} is well-defined precisely under the diagonal continuity granted by Assumption~\ref{ass:scaling} \cite[Definition 2.1 and Theorem 2.4]{fjordholm2021vanishing}.
We emphasize that this is the property whose scaling signature is observed numerically in Section~\ref{sec:numerics} below.

\subsection{Direct diagonal convergence to statistical Euler}\label{subsec:diagonalLimit}

The iterated limit passes through the Navier--Stokes statistical solutions at fixed viscosity. 
The computations in Section~\ref{sec:numerics}, however, use a diagonal scaling in which the macroscopic viscosity is coupled to the grid resolution ($\vis_\eps \sim \eps$), so that the grid Reynolds number is constant. 
We now argue that the limit of the KBC LBM ensemble along this path satisfies the FMW statistical Euler hierarchy under Assumptions~\ref{ass:lowMach}, \ref{ass:consistency}, and \ref{ass:scaling}. 
Note that the coercivity Assumption~\ref{ass:coercivity} is not needed, since no gradient control is required on the diagonal path.
\begin{proposition}[Diagonal limit to statistical Euler] \label{thm:diagonal_limit}
    Let the macroscopic kinematic viscosity scale strictly with the grid resolution to maintain a constant grid Reynolds number, $\vis_\eps = c \eps$, and let Assumptions~\ref{ass:lowMach}, \ref{ass:consistency} (with $\vis = \vis_{\eps}$), and \ref{ass:scaling} hold. 
    Then there exists a subsequence $\eps_{j} \to 0$ such that the correlation measures $\boldsymbol{\nu}^{\eps_{j}, k}$ generated by the push-forward measures $\mu^{\eps_{j}}_{t}$ of the discrete KBC LBM fields converge weak-$\ast$ to a limit $\boldsymbol{\nu}$ satisfying the first equation of the FMW statistical Euler hierarchy.
\end{proposition}
\begin{proof}
    By Lemma~\ref{lem:energy} (first part, requiring only Assumption~\ref{ass:lowMach}), the kinetic energy bound
    \begin{align}
        \sup_{t \in [0,T]} \| \tilde{\bm{u}}^\eps(t) \|_{L^2(\Omega)}^2 \le C \quad \mu_{0}^{\eps}\text{-a.s.}
    \end{align}
    holds uniformly in $\eps$. 
    Note that no gradient control is available or needed along the diagonal path. 
    The discrete $k$-point correlation measures $\{ \boldsymbol{\nu}^{\eps, k} \}_{\eps > 0}$ are well-defined including their diagonal traces (since the discrete fields are piecewise constant) and uniformly bounded in the space of Radon measures with uniformly bounded second moments. 
    Together with the uniform diagonal continuity of Assumption~\ref{ass:scaling}, the compactness theorem for correlation measures \cite[Theorem 2.4]{fjordholm2021vanishing} (see also \cite[Theorem 2.4]{lanthaler2021statistical}) provides a subsequence $\eps_j \to 0$ along which $\boldsymbol{\nu}^{\eps_j} \rightharpoonup^\ast \boldsymbol{\nu}$ and along which the quadratic observables appearing below converge. 
    The uniform support bound \cite[Eq.~(2.12)]{fjordholm2021vanishing} required for the latter holds by \eqref{eq:sup_bound}.
    
    We employ the weak momentum consistency \eqref{eq:momentumConsistency} with $\vis = \vis_{\eps} = c \eps$, tested against a smooth, divergence-free, space-time test function $\bm{\phi} \in C_c^\infty([0,T) \times \Omega)$. 
    For such test functions, the pressure term $\langle \pi^{\eps}, \bm{\nabla} \cdot \bm{\phi} \rangle$ vanishes identically. 
    Applying discrete summation by parts in time and taking the expectation, the weak formulation reads
    \begin{align} \label{eq:diagonal_weak_form_proof}   
        & \E_{\mu_0^\eps} \left[ \int_0^T \left( \langle \bm{m}^\eps, \partial_t^{-\eps} \bm{\phi} \rangle + \langle \rho^{\eps} \tilde{\bm{u}}^\eps \otimes \tilde{\bm{u}}^\eps, \bm{\nabla} \bm{\phi} \rangle + c\eps \langle \tilde{\bm{u}}^\eps, \bm{\Delta} \bm{\phi} \rangle \right) \,\mathrm{d}t + \langle \bm{m}^{\eps}(0), \bm{\phi}(0, \cdot) \rangle \right] \nonumber \\
        & = - \E_{\mu_0^\eps} \left[ \int_{0}^{T} \langle \bm{\mathcal{R}}^\eps, \bm{\phi} \rangle \,\mathrm{d}t \right],
    \end{align}
    where $\partial_t^{-\eps}$ is the backward discrete time derivative. 
    We now pass to the limit $\eps \to 0$ term by term. 
    
    \textit{1. Time derivative, initial term, and residual:} 
    Since the test function $\bm{\phi}$ is smooth, $\partial_t^{-\eps} \bm{\phi} \to \partial_t \bm{\phi}$ uniformly. 
    Since (Lemma~\ref{lem:energy})
    \begin{align} 
        \|\bm{m}^{\eps} - \tilde{\bm{u}}^{\eps}\|_{L^{\infty}(0,T; L^{2})} \leq C \eps^{2} , 
    \end{align} 
    the time and initial terms may be expressed through $\tilde{\bm{u}}^{\eps}$ up to $\mathcal{O}(\eps^{2})$. 
    By the definition of the one-point correlation measure, the resulting time term equals
    \begin{align}
        \E_{\mu_{0}^{\eps}} \left[ \int_{0}^{T} \langle \tilde{\bm{u}}^{\eps}, \partial_{t} \bm{\phi} \rangle \,\mathrm{d}t \right] 
        = \int_{0}^{T} \int_{\Omega} \int_{\mathbb{R}^{3}} \bm{\xi} \cdot \partial_{t} \bm{\phi}(t,\bm{x}) \,\mathrm{d}\nu^{\eps,1}_{t, \bm{x}}(\bm{\xi}) \,\mathrm{d}\bm{x} \,\mathrm{d}t , 
    \end{align}
    and the initial term equals the same expression at $t = 0$, without the time integral and with $\bm{\phi}(0,\cdot)$ in place of $\partial_{t} \bm{\phi}$. 
    Both integrands are linear in $\bm{\xi}$ and involve the velocity at a single point only, so that the one-point measures suffice and no diagonal trace enters. 
    Their linear growth in $\bm{\xi}$ is controlled by the uniform second moment bound. 
    Specifically, truncating the integrand at $\vert \bm{\xi} \vert = R$, the truncated part passes to the limit along $\eps_{j} \to 0$ by the weak-$\ast$ convergence $\nu^{\eps_{j},1} \rightharpoonup^{\ast} \nu^{1}$, while the remainder is bounded by $C/R$ uniformly in $\eps$, since $\vert \bm{\xi} \vert \leq \vert \bm{\xi} \vert^{2}/R$ on $\{ \vert \bm{\xi} \vert > R \}$. 
    Letting $R \to \infty$ yields
    \begin{align}
        \lim_{j \to \infty} \E_{\mu_{0}^{\eps_{j}}} \left[ \int_{0}^{T} \langle \tilde{\bm{u}}^{\eps_{j}}, \partial_{t} \bm{\phi} \rangle \,\mathrm{d}t \right] 
        &= \int_{0}^{T} \int_{\Omega} \int_{\mathbb{R}^{3}} \bm{\xi} \cdot \partial_{t} \bm{\phi}(t, \bm{x}) \,\mathrm{d}\nu^{1}_{t, \bm{x}}(\bm{\xi}) \,\mathrm{d}\bm{x} \,\mathrm{d}t , \\
        \lim_{j \to \infty} \E_{\mu_{0}^{\eps_{j}}} \left[ \langle \tilde{\bm{u}}^{\eps_{j}}(0), \bm{\phi}(0, \cdot) \rangle \right] 
        &= \int_{\Omega} \int_{\mathbb{R}^{3}} \bm{\xi} \cdot \bm{\phi}(0, \bm{x}) \,\mathrm{d}\nu^{1}_{0, \bm{x}}(\bm{\xi}) \,\mathrm{d}\bm{x} , 
    \end{align}
    where, for the second identity, $\nu^{1}_{0,\bm{x}}$ denotes the law of $\bm{u}_{0}(\bm{x})$, i.e., the initial datum of the limit hierarchy, and the convergence holds since the discrete initial data are the lattice restrictions of the same random field $\bm{u}_{0}$.
    The residual vanishes in expectation by Assumption~\ref{ass:consistency}, whose bound is uniform in $\vis \in (0, \vis_{0}]$ and hence applies with $\vis = c\eps$.
    
    \textit{2. Numerical dissipation:} 
    The scaled viscous term now acts purely as artificial numerical dissipation. 
    Using the Cauchy--Schwarz inequality against the intact $L^2$ energy bound, it vanishes, 
        \begin{align}
            \left| \E_{\mu_0^\eps} \left[ \int_0^T c\eps \langle \tilde{\bm{u}}^\eps, \bm{\Delta} \bm{\phi} \rangle \,\mathrm{d}t \right] \right| \leq c\eps \, \E_{\mu_0^\eps} \left[ \| \tilde{\bm{u}}^\eps \|_{L^2(L^2)} \right] \| \bm{\Delta} \bm{\phi} \|_{L^2(L^2)} \to 0.
        \end{align}
    \textit{3. Convective flux:} 
    Because strong $L^2$ compactness is unavailable, the quadratic nonlinear flux is lifted to the $2$-point correlation measure $\boldsymbol{\nu}^{\eps, 2}$. 
    The density factor is removed first via 
    \begin{align}
        \|(\rho^{\eps} - 1) \tilde{\bm{u}}^{\eps} \otimes \tilde{\bm{u}}^{\eps}\|_{L^{1}} \leq C \eps^{2} .
    \end{align} 
    Using the convergence of quadratic observables established above, together with the diagonal continuity granted by Assumption~\ref{ass:scaling} \cite{fjordholm2021vanishing}, the tensor product passes to the limit as the diagonal trace of the multi-point measure
    \begin{align}
        \lim_{\eps \to 0}  \E_{\mu_0^\eps} \left[ \int_0^T \langle \tilde{\bm{u}}^\eps \otimes \tilde{\bm{u}}^\eps, \bm{\nabla} \bm{\phi} \rangle \,\mathrm{d}t \right] 
        = \int_0^T \int_\Omega \int_{\mathbb{R}^6} (\bm{\xi}_1 \otimes \bm{\xi}_2) : \bm{\nabla} \bm{\phi}(\bm{x}) \, \mathrm{d}\nu^2_{t,\bm{x},\bm{x}}(\bm{\xi}_1, \bm{\xi}_2) \, \mathrm{d}\bm{x} \,\mathrm{d}t.
    \end{align}

    \textit{Conclusion:} 
    Substituting these limits into \eqref{eq:diagonal_weak_form_proof} recovers the continuous FMW statistical Euler weak formulation for the first hierarchy equation. 
\end{proof}
\begin{remark}
    Note that the higher-order equations in the FMW hierarchy can be obtained in the same way from the $k$-point moment balances, which we do not carry out here.
    Note further that the dissipation anomaly itself resides in the energy balance rather than in the momentum equation. 
    Its sign in the limit would follow from a lower bound on the discrete entropy production (cf.\ Assumption~\ref{ass:coercivity}), and a rigorous identification of the anomaly would require passing to the limit in the energy equation, which we do not pursue here.
\end{remark}
\begin{proof}[Proof of Proposition~\ref{prop:weak-strong_discrete}]
    Based on the results above, we assemble the statements as follows: (i) is Lemma~\ref{lem:empirical}, (ii) is Theorem~\ref{thm:nse_limit}, and (iii) is Corollary~\ref{thm:eulerConsistency}. 
    The statement on the diagonal path is Proposition~\ref{thm:diagonal_limit}. 
    The uniform integrability of the $p$th moments follows from the uniform bound \eqref{eq:sup_bound}, so that Lemma~\ref{prop:weak-strong_continous} applies. 
\end{proof}

\subsection{Weak-strong uniqueness}

A generalized solution concept should reduce to the classical one whenever a classical solution is available. 
For admissible measure-valued solutions of the incompressible Euler equations, this weak-strong uniqueness principle was proven by Brenier, De Lellis, and Sz\'{e}kelyhidi~\cite{brenier2011weak} by the relative energy method of Dafermos~\cite{Dafermos1979} and DiPerna~\cite{DiPerna1979}. 
For dissipative statistical solutions of the incompressible Euler equations, uniqueness for initial measures concentrated near classical data, including short-time uniqueness for $H^{m}$-concentrated initial measures and weak-strong uniqueness in two dimensions, is established in~\cite[Theorem 3.1, Corollaries 3.1 and 3.2]{lanthaler2021statistical}. 

Since the first equation of the FMW hierarchy for the one-point marginal $\nu^{1}$ coincides with the measure-valued formulation of DiPerna and Majda without concentration (see~\cite[Remark 3.1]{lanthaler2021statistical}), the result of~\cite{brenier2011weak} applies to the limit objects of the preceding sections, provided that they satisfy the energy admissibility \eqref{eq:admissibility}. 
On the iterated path, the admissibility \eqref{eq:admissibility} of the limit follows from \cite[Remark 4.9, Eq.~(4.41)]{fjordholm2021vanishing} under the hypotheses of Corollary~\ref{thm:eulerConsistency}. 
On the diagonal path it is not verified, and we impose it as a hypothesis.
We state the result for later reference and sketch the argument, which is also used for the stability estimate in Remark~\ref{rem:statStability}. 
Note that it omits the multi-point structure of the statistical solution.
\begin{proposition}[Weak-strong uniqueness, {\cite{brenier2011weak,lanthaler2021statistical}}]\label{thm:weakStrong}
    Let $\bar{\bm{u}} \in C([0, T^*]; W^{1,\infty}(\Omega))$ be a strong solution to the 3D incompressible Euler equations (so that $\partial_{t} \bar{\bm{u}} \in C([0,T^{*}]; L^{2}(\Omega))$ by the equation itself). 
    Let $\nu^{1} = (\nu^{1}_{t, \bm{x}})$ be a one-point statistical solution satisfying the first FMW hierarchy equation together with the energy admissibility
    \begin{align}\label{eq:admissibility}
        \int_{\Omega} \int_{\mathbb{R}^{3}} \frac{1}{2} \vert \bm{\xi} \vert^{2} \,\mathrm{d}\nu^{1}_{t, \bm{x}}(\bm{\xi}) \,\mathrm{d}\bm{x} \leq \int_{\Omega} \int_{\mathbb{R}^{3}} \frac{1}{2} \vert \bm{\xi} \vert^{2} \,\mathrm{d}\nu^{1}_{0, \bm{x}}(\bm{\xi}) \,\mathrm{d}\bm{x} \quad \text{for a.e.\ } t \in [0, T^{*}] .
    \end{align}
    If $\nu^{1}_{0, \bm{x}} = \delta_{\bar{\bm{u}}_{0}(\bm{x})}$ for a.e.\ $\bm{x} \in \Omega$, then $\nu_{t,\bm{x}}^1 = \delta_{\bar{\bm{u}}(t,\bm{x})}$ for a.e.\ $(t, \bm{x}) \in [0, T^{*}] \times \Omega$.
\end{proposition}
\begin{proof}[Sketch of the proof]
    Following \cite{brenier2011weak}, we define the relative energy functional
    \begin{equation}
        \mathcal{E}_{\mathrm{rel}}(t) \coloneqq \int_{\Omega} \int_{\mathbb{R}^3} \frac{1}{2} |\bm{\xi} - \bar{\bm{u}}(t,\bm{x})|^2 \, \mathrm{d}\nu_{t,\bm{x}}^1(\bm{\xi}) \, \mathrm{d}\bm{x}
    \end{equation}
    to quantify the distance between the statistical measure and the strong deterministic solution. 
    Expanding the quadratic integrand yields three components: the statistical kinetic energy $\frac{1}{2}|\bm{\xi}|^2$, the strong classical energy $\frac{1}{2}|\bar{\bm{u}}|^2$, and the cross-term $-\bm{\xi} \cdot \bar{\bm{u}}$. 
    We compare each component at time $t$ with its initial value. 
    Since $\mathcal{E}_{\mathrm{rel}}$ is a priori not differentiable in time, all balances are used in time-integrated form.
    
    \textit{1. Statistical energy (admissibility):} 
    By \eqref{eq:admissibility},
    \begin{align}
        \int_\Omega \int_{\mathbb{R}^{3}} \frac{1}{2} |\bm{\xi}|^2 \, \mathrm{d}\nu_{t,\bm{x}}^1 \, \mathrm{d}\bm{x} - \int_\Omega \int_{\mathbb{R}^{3}} \frac{1}{2} |\bm{\xi}|^2 \, \mathrm{d}\nu_{0,\bm{x}}^1 \, \mathrm{d}\bm{x} \leq 0 \quad \text{for a.e.\ } t . 
    \end{align}
    
    \textit{2. Strong energy:} 
    Smooth classical Euler solutions on periodic domains conserve kinetic energy,
    \begin{align}
        \int_{\Omega} \frac{1}{2} |\bar{\bm{u}}(t)|^2 \, \mathrm{d}\bm{x} - \int_{\Omega} \frac{1}{2} |\bar{\bm{u}}_{0}|^2 \, \mathrm{d}\bm{x} = 0  .
    \end{align}
    
    \textit{3. Cross-term:} 
    We use the strong solution $\bar{\bm{u}}$ as a (Lipschitz-in-time, admissible by density) test function in the first hierarchy equation for $\nu^{1}$, and pair the strong Euler equations with the mean velocity field $\int_{\mathbb{R}^{3}} \bm{\xi} \,\mathrm{d}\nu_{t,\bm{x}}^{1}(\bm{\xi})$. 
    Adding the two time-integrated identities, the pressure gradients drop by the divergence-free constraints, and the remaining terms combine into the time-integrated evolution of the cross-term. 
    
    \textit{Conclusion:}
    Summing the three balances, all linear parts cancel, and the remaining nonlinear convective tensor products complete the square, yielding the time-integrated Reynolds stress contraction
    \begin{equation} \label{eq:relative_energy_bound}
        \mathcal{E}_{\mathrm{rel}}(t) \leq \mathcal{E}_{\mathrm{rel}}(0) - \int_{0}^{t} \int_{\Omega} \int_{\mathbb{R}^3} \Big( (\bm{\xi} - \bar{\bm{u}}) \otimes (\bm{\xi} - \bar{\bm{u}}) \Big) : \bm{\nabla} \bar{\bm{u}}(t^{\prime},\bm{x}) \, \mathrm{d}\nu_{t^{\prime},\bm{x}}^1(\bm{\xi}) \, \mathrm{d}\bm{x} \,\mathrm{d}t^{\prime}
    \end{equation}
    for a.e.\ $t \in [0, T^{*}]$. 
    Because $\bar{\bm{u}} \in W^{1,\infty}(\Omega)$ up to time $T^*$, its spatial gradient is uniformly bounded by the constant $L_{\bar{\bm{u}}} = \|\bm{\nabla} \bar{\bm{u}}\|_{L^\infty([0,T^*] \times \Omega)}$. 
    Bounding the integrand pointwise via 
    \begin{align}
    |(\bm{\xi} - \bar{\bm{u}}) \otimes (\bm{\xi} - \bar{\bm{u}}) : \bm{\nabla} \bar{\bm{u}}| \leq L_{\bar{\bm{u}}} |\bm{\xi} - \bar{\bm{u}}|^{2}    
    \end{align}
    yields
    \begin{equation}
        \mathcal{E}_{\mathrm{rel}}(t) \leq \mathcal{E}_{\mathrm{rel}}(0) + 2 L_{\bar{\bm{u}}} \int_{0}^{t} \mathcal{E}_{\mathrm{rel}}(t^{\prime}) \,\mathrm{d}t^{\prime} .
    \end{equation}
    By hypothesis, $\mathcal{E}_{\mathrm{rel}}(0) = 0$. 
    The integral form of Gr\"onwall's inequality then gives $\mathcal{E}_{\mathrm{rel}}(t) = 0$ for a.e.\ $t \in [0, T^*]$, which implies $\nu_{t,\bm{x}}^1 = \delta_{\bar{\bm{u}}(t,\bm{x})}$ for a.e.\ $(t, \bm{x}) \in [0,T^{*}] \times \Omega$.
\end{proof}
\begin{remark}[Statistical stability]\label{rem:statStability}
    For non-atomic initial measures with fixed perturbation width $\mathfrak{w} > 0$, as employed in our RTGV computations, the hypothesis $\mathcal{E}_{\mathrm{rel}}(0) = 0$ does not hold. 
    In this case, the same Gr\"onwall argument yields the stability estimate $\mathcal{E}_{\mathrm{rel}}(t) \leq e^{2 L_{\bar{\bm{u}}} t} \mathcal{E}_{\mathrm{rel}}(0)$ for a.e.\ $t \in [0, T^{*}]$: as long as a strong Euler solution exists, the statistical solution depends continuously on the initial measure, and the collapse onto the Dirac mass is recovered in the limit of vanishing initial uncertainty. 
    This is the estimate underlying the construction of the generic set of initial data in~\cite[Section 3.1, Eq.~(3.3)]{lanthaler2021statistical}, where the factor $e^{-C(\bar{\bm{v}})T}$ enters in the same role. 
    On the iterated path, the energy admissibility \eqref{eq:admissibility} of the limits holds under the hypotheses of Corollary~\ref{thm:eulerConsistency} by \cite[Remark 4.9]{fjordholm2021vanishing}. 
    For the diagonal path, it is expected from the entropy balance \eqref{eq:entropyTelescope} under Assumption~\ref{ass:coercivity}, but this inheritance remains formal (cf.\ the remark after Proposition~\ref{thm:diagonal_limit}).
\end{remark}

\subsection{Formal derivation of Wasserstein convergence rates}\label{sec:formalWrate}

The relative energy argument of Proposition~\ref{thm:weakStrong} is confined to an interval on which a strong Euler solution exists. 
Beyond such an interval, whether it terminates at a finite critical time $T^{*}$ or whether its extension is unknown, the Lipschitz bound $L_{\bar{\bm{u}}}$ is unavailable, and the Gr\"onwall constant in the relative energy estimate diverges, so that the argument yields no control. 
Independently of the question of singularity formation, the fully developed turbulent regime exhibits exponential separation of neighboring trajectories, so that comparing individual samples pathwise ceases to be informative (cf.\ Section~\ref{sec:sampleDivergence}).

To quantify the convergence of the statistical ensemble, we use the 1-Wasserstein metric $W_1$ and follow the structure of Kuznetsov's argument for conservation laws \cite{Kuznetsov1976}: the error is split by inserting a spatially mollified measure $\mu_{t,\lambda}$ that is smoothed at a finite length scale $\lambda > 0$. 
The two scaling inputs are formulated as explicit assumptions, and the resulting rate is then elementary.
\begin{assumption}[Kuznetsov scaling]\label{ass:kuznetsov}
    There exist constants $C_{1}, C_{2} > 0$, a regularity index $s \in (0, 1]$, and an amplification exponent $\zeta > 0$ such that, for all sufficiently small $\eps$ and $\lambda$,
    \begin{enumerate}
        \item[(K1)] \textit{Spatial interpolation:} $W_1(\mu_{t,\lambda}, \mu_t) \leq C_{1} \lambda^{s}$, and
        \item[(K2)] \textit{Numerical truncation:} $W_1(\mu^\eps_t, \mu_{t,\lambda}) \leq C_{2} \eps \lambda^{-\zeta}$.
    \end{enumerate}
\end{assumption}
Both parts are motivated as follows. 
\begin{itemize}
    \item Assumption (K1) is a consequence of the structure function scaling in Assumption~\ref{ass:scaling}: By the mollification estimate in the proof of \cite[Theorem 2.1, Eq.~(A.1)]{lanthaler2021statistical}, $W_1(\mu_{t,\lambda}, \mu_t)$ is bounded by a constant times the (square-rooted) second-order structure function at separation $\lambda$, which scales as $\lambda^{s}$ with $s = 1/3$ under Kolmogorov's K41 scaling \cite{kolmogoroff1941local}. 
    This corresponds to velocity increments of fractional Besov regularity $B^s_{3,\infty}(\Omega)$. 
    \item (K2) is a modeling assumption: Under the diagonal inviscid scaling $R\!e \sim \eps^{-1}$, the macroscopic viscosity $\vis_\eps = c\eps$ acts as a first-order dissipative perturbation of the mollified field, and its effect is expected to grow as the filter scale $\lambda$ decreases. 
    The exponent $\zeta$ quantifies this amplification and is not derived here.
\end{itemize}
Based on Assumption~\ref{ass:kuznetsov}, we can provide the following statement. 
\begin{lemma}[Optimal mollification rate]\label{lem:kuznetsovRate}
    Under Assumption~\ref{ass:kuznetsov}, 
    \begin{align}
        W_1(\mu^\eps_t, \mu_t) \leq (C_1 + C_2) \, \eps^{\frac{s}{s+\zeta}} .
    \end{align}
\end{lemma}
\begin{proof}
    By the triangle inequality, the total Wasserstein error splits into the two components of Assumption~\ref{ass:kuznetsov}, i.e., 
    \begin{equation} \label{eq:kuznetsov_split}
        W_1(\mu^\eps_t, \mu_t) \leq \underbrace{W_1(\mu^\eps_t, \mu_{t,\lambda})}_{\text{Numerical truncation}} + \underbrace{W_1(\mu_{t,\lambda}, \mu_t)}_{\text{Spatial interpolation}} \leq E(\eps, \lambda),
    \end{equation}
    where
    \begin{equation} \label{eq:error_functional}
        E(\eps, \lambda) = C_1 \lambda^s + C_2 \eps \lambda^{-\zeta} .
    \end{equation}
    We minimize this error functional with respect to the filter scale $\lambda$ by setting its partial derivative to zero, i.e., 
    \begin{align}
        \frac{\partial E}{\partial \lambda} = s C_1 \lambda^{s-1} - \zeta C_2 \eps \lambda^{-\zeta-1} = 0 
        \quad \implies \quad \lambda^{s+\zeta} &= \left( \frac{\zeta C_2}{s C_1} \right) \eps,
    \end{align}
    which is the unique global minimizer since $E(\eps, \cdot)$ is strictly convex on $(0, \infty)$ with $E \to \infty$ at both ends. 
    Solving for the optimal smoothing scale $\lambda_\mathrm{opt}$ yields an algebraic coupling between the numerical grid and the turbulent filter width:
    \begin{equation} \label{eq:optimal_lambda}
        \lambda_\mathrm{opt} \sim \mathcal{O}\left( \eps^{\frac{1}{s+\zeta}} \right).
    \end{equation}
    Inserting $\lambda = \eps^{1/(s+\zeta)}$ into \eqref{eq:error_functional} gives
    \begin{align}
        W_1(\mu^\eps_t, \mu_t) 
        \leq 
        C_1 \left( \eps^{\frac{1}{s+\zeta}} \right)^s + C_2 \eps \left( \eps^{\frac{1}{s+\zeta}} \right)^{-\zeta} 
        = 
        (C_1 + C_2) \eps^{\frac{s}{s+\zeta}}. \label{eq:final_derivation}
    \end{align}
\end{proof}
\begin{hypothesis}[Wasserstein convergence rate]\label{hyp:WassersteinRate}
    Assumption~\ref{ass:kuznetsov} holds for the MC KBC LBM statistical solutions with the K41 regularity index $s \approx 1/3$ and an amplification exponent $\zeta \approx s$, so that, by Lemma~\ref{lem:kuznetsovRate},
    \begin{equation} \label{eq:alpha_prediction}
        W_1(\mu^\eps_t, \mu_t) \sim \mathcal{O}(\eps^\alpha), \quad \text{where } \alpha = \frac{s}{s+\zeta} \approx 0.5 .
    \end{equation}
\end{hypothesis}
We formulate the explicit rate in \eqref{eq:alpha_prediction} as a hypothesis since the value $\zeta \approx s$ is not derived. 
It is the value for which Lemma~\ref{lem:kuznetsovRate} reproduces the rates measured in Section~\ref{subsec:numWasserstein}. 
For comparison, if the numerical dissipation acted on the mollified field like a single derivative, i.e., $\zeta = 1$, the resulting rate would be $\alpha = s/(1+s) = 1/4$. 
Hypothesis~\ref{hyp:WassersteinRate} is therefore rather an empirical fit within the derived Kuznetsov-type structure than a prediction. 
As such, the computations in Section~\ref{sec:numerics} do not confirm it independently but motivate the validity of our Hypothesis~\ref{hyp:WassersteinRate} and suggest that \(\zeta \approx s\) is plausible in turbulent regimes.

\subsection{Numerical validation of convergence rates}

To measure the rates entering Hypothesis~\ref{hyp:WassersteinRate}, we compute the 1-Wasserstein distance with respect to the $L^1(\Omega)$ ground metric. 
We approximate the true continuous measures using discrete MC empirical measures $\mu^\eps_t = \frac{1}{M} \sum_{m=1}^M \delta_{\bm{u}^{\eps}_{m}(t)}$ generated by evaluating $M$ independent realizations of the KBC LBM ensemble.
To numerically realize the diagonal inviscid limit ($\vis_\eps \searrow 0$), we do not target a fixed physical Reynolds number. 
Instead, the macroscopic kinematic viscosity of the KBC LBM scheme is coupled directly to the spatial lattice resolution $N$ (where the grid spacing is $\eps = 1/N$), and \(N\) denotes the number of grid points in one coordinate direction of a periodic box domain. 
To maintain a constant grid Reynolds number, we enforce the linear scaling $R\!e_{\mathrm{eff}} \sim \eps^{-1}$. 
For example, in the computational setup for computing the Wasserstein convergence rate, we scale the effective Reynolds number as $R\!e_{N} = 40 N$. 
Consequently, as the grid is refined ($N \to \infty$, $\eps \to 0$), the physical dissipation vanishes.
Because the empirical marginal measures evaluated at any spatial location $\bm{x}$ are composed of an equal number of $M$ uniform Dirac masses, the general Kantorovich optimal transport relaxation locally simplifies to a classical Monge assignment problem.
By the Birkhoff--von Neumann theorem (see, e.g., \cite{Peyre2019, Villani2003}), the polytope of doubly stochastic transport matrices has permutation matrices as its extreme points. 
Since the cost functional is linear, the minimum is achieved at an extreme point, meaning that the localized optimal transport plan is defined by a pure bijection. 
Thus, the exact $1$-Wasserstein distance between the time-local 1-point marginals at a specific location $\bm{x}$ is obtained by minimizing the distance cost over the symmetric group of all possible permutations $\mathcal{S}_M$ independently at that evaluation point, i.e., 
\begin{equation}\label{eq:wassersteinPermutations}
    W_1(\nu_{t,\bm{x}}^{\eps_1}, \nu_{t,\bm{x}}^{\eps_2}) = \min_{\pi_{\bm{x}} \in \mathcal{S}_M} \frac{1}{M} \sum_{m=1}^M \left| \bm{u}_{m}^{\eps_1}(t, \bm{x}) - \bm{u}_{\pi_{\bm{x}}(m)}^{\eps_2}(t, \bm{x}) \right|.
\end{equation}
The approximation of the integrated marginal distance $\|W_1(\nu_{t,\bm{x}}^{\eps_1}, \nu_{t,\bm{x}}^{\eps_2})\|_{L^1(\Omega)}$ is then obtained by solving these independent assignment problems at each point and averaging the minimized costs over the discrete spatial domain.
Note that \(\pi_{\bm{x}}\) denotes a specific permutation function from the symmetric group $\mathcal{S}_M$, which acts as a mapping that shuffles the order of the points to find the optimal minimum-distance matching between the two distributions at location \(\bm{x}\).

\subsection{Error metrics between statistical solutions}

While the previous section derived, under the scaling assumptions of Section~\ref{sec:formalWrate}, a convergence bound for the full global measure $\mu_t$, practical computation requires evaluating localized objects. 
Since statistical solutions are infinite-dimensional, we constrain ourselves to tracking distances between their finite-dimensional correlation marginals. 
This reduces the problem to computing the Wasserstein distance between probability measures on $\mathbb{R}^{3k}$.
\begin{remark}[Transfer of the convergence rate to marginals]
    While the formal Kuznetsov-type argument yields a fractional convergence rate $\mathcal{O}(\eps^\alpha)$ for the global measure $\mu_t$ over the infinite-dimensional phase space, evaluating this globally is computationally intractable. 
    This bound transfers to our localized marginal computations. 
    For the $1$-Wasserstein distance evaluated with an $L^1(\Omega)$ spatial ground metric, Fubini's theorem and the sub-additivity of the infimum guarantee that the integrated distance between the time-local 1-point marginals is bounded by the global distance between the corresponding time-local measures, i.e., for every $t \in I$, 
    \begin{equation} \label{eq:marginal_bound}
        \|W_1(\nu_{t,\bm{x}}^{1,1}, \nu_{t,\bm{x}}^{1,2})\|_{L^1(\Omega)} \leq W_1(\mu_{t}^{1}, \mu_{t}^{2}) \sim \mathcal{O}(\eps^\alpha),
    \end{equation}
    where $\nu^{1,i}_{t,\bm{x}}$ denotes the one-point marginal of $\mu_{t}^{i}$, $i \in \{1, 2\}$.
    We emphasize that this estimate \eqref{eq:marginal_bound} is stated already for general $k$-point marginals in \cite[Eq.~(5.8)]{lanthaler2021statistical}. 
    Because the local pointwise transport plans are less constrained than a global field-to-field transport plan, the computationally feasible marginal metric inherits the fractional convergence bound of the global statistical Euler limit. 
    Note that the rate on the right-hand side of \eqref{eq:marginal_bound} is the one stated in Hypothesis~\ref{hyp:WassersteinRate} rather than a proven bound, and that the estimate is one-sided. 
    Observing a rate in the localized measurements is therefore consistent with, but does not establish, the same rate for the global measure.
\end{remark}
Generally, we are interested in the value of $\|W_p(\nu^{k,1},\nu^{k,2})\|_{L^p(\Omega^k)}$, where $\nu^{k,1}$ and $\nu^{k,2}$ are the $k$-point correlation marginals of the empirical measures induced by $M$ MC samples $\{\bm{u}_m^i\}_{m=1}^M$ of two different statistical solutions $i \in \{1, 2\}$ (below identified with the resolutions $N$ and $N_{\mathrm{ref}}$), given by
\begin{equation}
    \nu^{k,i}_{\bm{x}_1,\dots,\bm{x}_k} = \frac{1}{M}\sum_{m=1}^M \delta_{\bm{u}_m^i(\bm{x}_1),\dots,\bm{u}_m^i(\bm{x}_k)}.
\end{equation}
By choosing a separable ground metric in phase space, we can compute the distance component-wise. We approximate the $L^p$ norm in physical space by a Riemann sum to obtain
\begin{align}
    & \left\|W_p(\nu^{k,1},\nu^{k,2})\right\|_{L^p(\Omega^k)}^p &\nonumber \\
    & = \int_{\Omega^k}\! \left|W_p(\nu^{k,1}_{\bm{x}_1,\dots,\bm{x}_k},\nu^{k,2}_{\bm{x}_1,\dots,\bm{x}_k})\right|^p \,\mathrm{d}\bm{x}_1\dots\mathrm{d}\bm{x}_k \label{eq:W1}\\
    & \approx \frac{1}{Q^{dk}}\sum_{(\bm{x}_1,\dots,\bm{x}_k) \in I_Q^k} \left|W_p(\nu^{k,1}_{\bm{x}_1,\dots,\bm{x}_k},\nu^{k,2}_{\bm{x}_1,\dots,\bm{x}_k})\right|^p \nonumber \\
    & = \frac{1}{Q^{dk}}\sum_{(\bm{x}_1,\dots,\bm{x}_k) \in I_Q^k} \left|\begin{pmatrix}
        W_p\left(\frac{1}{M}\sum_{m=1}^M \delta_{u_{1,m}^{1}(\bm{x}_1),\dots,u_{1,m}^{1}(\bm{x}_k)},\frac{1}{M}\sum_{m=1}^M \delta_{u_{1,m}^{2}(\bm{x}_1),\dots,u_{1,m}^{2}(\bm{x}_k)}\right) \\
        W_p\left(\frac{1}{M}\sum_{m=1}^M \delta_{u_{2,m}^{1}(\bm{x}_1),\dots,u_{2,m}^{1}(\bm{x}_k)},\frac{1}{M}\sum_{m=1}^M \delta_{u_{2,m}^{2}(\bm{x}_1),\dots,u_{2,m}^{2}(\bm{x}_k)}\right) \\
        W_p\left(\frac{1}{M}\sum_{m=1}^M \delta_{u_{3,m}^{1}(\bm{x}_1),\dots,u_{3,m}^{1}(\bm{x}_k)},\frac{1}{M}\sum_{m=1}^M \delta_{u_{3,m}^{2}(\bm{x}_1),\dots,u_{3,m}^{2}(\bm{x}_k)}\right)
    \end{pmatrix}\right|^p, \label{eq:W1approx}
\end{align}
where the index set $I_Q$ is given by $I_Q = \left\{ \bm{x} \in \mathbb{R}^d : 0 < \bm{x} \leq 1, Q\bm{x} \in \mathbb{N}^d \right\}$ (in coordinates normalized by the domain length $2\pi$), $d=3$ is the spatial dimensionality of the domain, $Q^{dk}$ is the total number of evaluation points, the outer $\vert \cdot \vert$ denotes the $\ell^{p}$-norm over the velocity components, and $u_{c,m}^i$ ($c = 1, 2, 3$) are the three scalar velocity components of the flow field $\bm{u}_m^i$, matching the component notation of the algorithms in Appendix~\ref{appsec:compWasserstein}. 
The prefactor $Q^{-dk}$ is the uniform quadrature weight of the Riemann sum (each evaluation point represents a cell of volume $Q^{-dk}$ in the normalized domain). 
Since all resolutions $N$ are evaluated on the same downsampled grid with $Q = N_{\mathrm{ds}} = 8$ points per direction (see Appendix~\ref{appsec:compWasserstein}), this weight is a fixed constant that only scales the absolute magnitude of the error and cancels in the experimental orders of convergence.

Alternatively, for the special case of the 2-point correlation marginals ($k=2$), we can avoid separating the components and instead evaluate the distance directly on the full vector-valued states. 
In this approach, the ground metric is defined on the $\mathbb{R}^6$ phase space of the concatenated velocity vectors, yielding the scalar distance
\begin{align}
     \left\|W_p(\nu^{2,1},\nu^{2,2})\right\|_{L^p(\Omega^2)}^p \approx \frac{1}{Q^{2d}}\sum_{(\bm{x}_1,\bm{x}_2) \in I_Q^2} W_p\left( \frac{1}{M}\sum_{m=1}^M \delta_{\bm{u}_m^1(\bm{x}_1),\bm{u}_m^1(\bm{x}_2)}, \frac{1}{M}\sum_{m=1}^M \delta_{\bm{u}_m^2(\bm{x}_1),\bm{u}_m^2(\bm{x}_2)} \right)^p, \label{eq:W1_vector}
\end{align}
where the inner Wasserstein distance $W_p$ is computed using the Euclidean distance in $\mathbb{R}^6$ between the evaluated state vectors.

In practice, the two cases are evaluated differently. 
For the 1-point correlation marginals ($k=1$), the component-wise distances in \eqref{eq:W1approx} are one-dimensional optimal transport problems between empirical measures with equal uniform weights. 
Their optimal coupling is the monotone (sorted) rearrangement, so the local $W_1$ in \eqref{eq:wassersteinPermutations} coincides with the $L^1$ distance between the empirical cumulative distribution functions of the two samples, which is compared for each velocity component at each point in space and time. 
For the 2-point correlation marginals ($k=2$), both component-wise \eqref{eq:W1approx} and vector-valued \eqref{eq:W1_vector}, the joint empirical distributions at the point pair $(\bm{x}_1, \bm{x}_2)$ are compared by assembling the $M \times M$ cost matrix of pairwise Euclidean distances in $\mathbb{R}^{2}$ and $\mathbb{R}^{6}$, respectively, and solving the resulting discrete optimal transport problem. 
The corresponding implementations are given in Algorithms~\ref{alg:WassersteinDistance1}, \ref{alg:WassersteinDistance2}, and \ref{alg:WassersteinDistance2_vector} of Appendix~\ref{appsec:compWasserstein}, which realize the quantities $W_{1,1}$, $W_{1,2}$, and $W_{1,2}^{\mathrm{v}}$ reported in Section~\ref{subsec:numWasserstein}.

Therefore, computing the Wasserstein distance between two statistical solutions based on MC samples requires solving an optimal transport problem between the samples for a given set of points in the domain of the correlation marginals of interest. 
Then, the $L^p$-norm is taken over all the obtained pointwise Wasserstein distances. 
The computational complexity of the two cases differs substantially. 
For the 1-point marginals, the sorted-rearrangement evaluation of the one-dimensional distances requires only $\mathcal{O}(M \log M)$ operations per evaluation point and component, yielding a total complexity of $\mathcal{O}(Q^{d} M \log M)$. 
For the $k$-point marginals with $k \geq 2$, the joint distributions are no longer one-dimensional, and the discrete optimal transport (minimum-cost flow) problem must be solved from the $M \times M$ cost matrix, which requires $\mathcal{O}(M^3)$ operations per evaluation tuple if both solutions are approximated by empirical measures based on the same number of samples, and $\mathcal{O}(M^3 \log M)$ otherwise. 
This leads to a total complexity of $\mathcal{O}(Q^{dk}M^3)$ and $\mathcal{O}(Q^{dk}M^3\log M)$, respectively. 
These computational complexities make it clear that computing the Wasserstein distance quickly becomes infeasible when the order of the correlation marginals is increased. 
We will therefore restrict the computation to the 1-point and 2-point marginals, where the 2-point computations already necessitate approximating the Riemann sum with a lower resolution than the underlying grid utilized for the simulations.

\section{Numerical experiments}\label{sec:numerics}

\subsection{Randomized Taylor--Green vortex flow}

Let the domain be a periodic box \(\Omega = \left[ 0, 2\pi l_{\mathrm{c}}\right]^3\).  
We define a reference length \(l_{\mathrm{c}}\), a reference velocity \(U_{\mathrm{c}}\), and a reference density \(\rho_{\mathrm{c}}\), where \(l_{\mathrm{c}} = 1 \mathrm{m}\) and \(U_{\mathrm{c}}= 1 \mathrm{m}/\mathrm{s}\) simplify the Reynolds number to \(R\!e = (1\mathrm{m}^{2}/\mathrm{s})/ \vis\).  
The deterministic part of the initial velocity is the TGV field (see for example~\cite{brachet1991direct} and references therein)
\begin{align} \label{eq:tgv_init_velo}
    \bm{u}_{0}^{\mathrm{det}} \left(\bm{x} \right) = 
    \begin{pmatrix} U_{\mathrm{c}} \mathrm{sin}\left(\frac{x}{l_{\mathrm{c}}}\right) \mathrm{cos}\left(\frac{y}{l_{\mathrm{c}}}\right) \mathrm{cos}\left(\frac{z}{l_{\mathrm{c}}}\right) \\
    -U_{\mathrm{c}}\mathrm{cos}\left(\frac{x}{l_{\mathrm{c}}}\right) \mathrm{sin}\left(\frac{y}{l_{\mathrm{c}}}\right) \mathrm{cos}\left(\frac{z}{l_{\mathrm{c}}}\right) \\
    0 \end{pmatrix}.
\end{align}

As motivated by~\cite{rohner2024efficient,simonis2024computing}, we extend the classical TGV flow benchmark with a probabilistic initial velocity that retains its symmetry. 
To this end, the TGV flow initial condition \eqref{eq:tgv_init_velo} is perturbed with \(8d = 24\) IID random variables $X_{\alpha,i,j,k} \sim \mathcal{U}_{[-\mathfrak{w},\mathfrak{w}]}$ to obtain the RTGV initial velocity field 
\begin{align}\label{eq:initialRTGV}
    \bm{u}_{0} = \bm{u}_{0}^{\mathrm{det}} + \bm{\mathfrak{s}} \in \mathbb{R}^{d}, 
\end{align}
where the \(\alpha\)th perturbation in \(\bm{\mathfrak{s}} = (\mathfrak{s}_{\alpha})_{1\leq\alpha\leq d}\) is chosen as in \eqref{eq:alpha_perturbation}. 
In this randomized setting, the computational domain is still \(\Omega = [0, 2\pi]^{3}\), and the support parameter of the uniform distribution is set to \(\mathfrak{w}=0.025\). 

In our numerical computations, in contrast to the construction in Section~\ref{sec:approximation}, we omit the explicit Leray projection $\mathbf{P}$ both in the initialization \eqref{eq:initialRTGV} and when evaluating the statistical metrics. 
The perturbation $\bm{\mathfrak{s}}$ is then not exactly solenoidal, with $\bm{\nabla} \cdot \bm{\mathfrak{s}} = \mathcal{O}(\mathfrak{w})$ independent of the resolution, so that the equilibrium initialization \eqref{eq:initialPopulations} emits a weak initial acoustic transient. 
The initial measures of Section~\ref{sec:approximation} and the present section coincide in the limit $\mathfrak{w} \to 0$. 
For the evolved fields, the solver enforces a discrete divergence-free condition, so that the velocity field for the \(m\)th sample satisfies $\bm{\nabla} \cdot \bm{u}_{m} = \mathcal{O}(\triangle x^{2})$, which is the spatial order of accuracy of the KBC LBM in diffusive scaling. 
Since the discrepancy between the discretely divergence-free fields and their exact projection onto the divergence-free manifold is bounded by the spatial truncation error, omitting the projection when evaluating the local metrics is not expected to affect the measured convergence orders. 
We therefore compute the local metrics directly on the discrete fields. 
All initial samples are thus initialized with \eqref{eq:initialRTGV} according to \eqref{eq:initialPopulations} and evolved in time with the deterministic KBC LBM until \(t=T\) to obtain 
\begin{align}
    \bigl\{ \bigl. \bm{u}^{\vis}_{m} (t) ~ \bigr\vert ~ m=1, 2, \ldots, M \bigr\} \sim \mu_{t}^{\vis,N,M} .
\end{align}

\begin{figure}[ht!]%trim = {left low right up}
	\centerline{
       \subfloat[\(t=0\)]{\includegraphics[width=0.24\linewidth ,trim={0cm 12cm 0cm 0cm},clip]{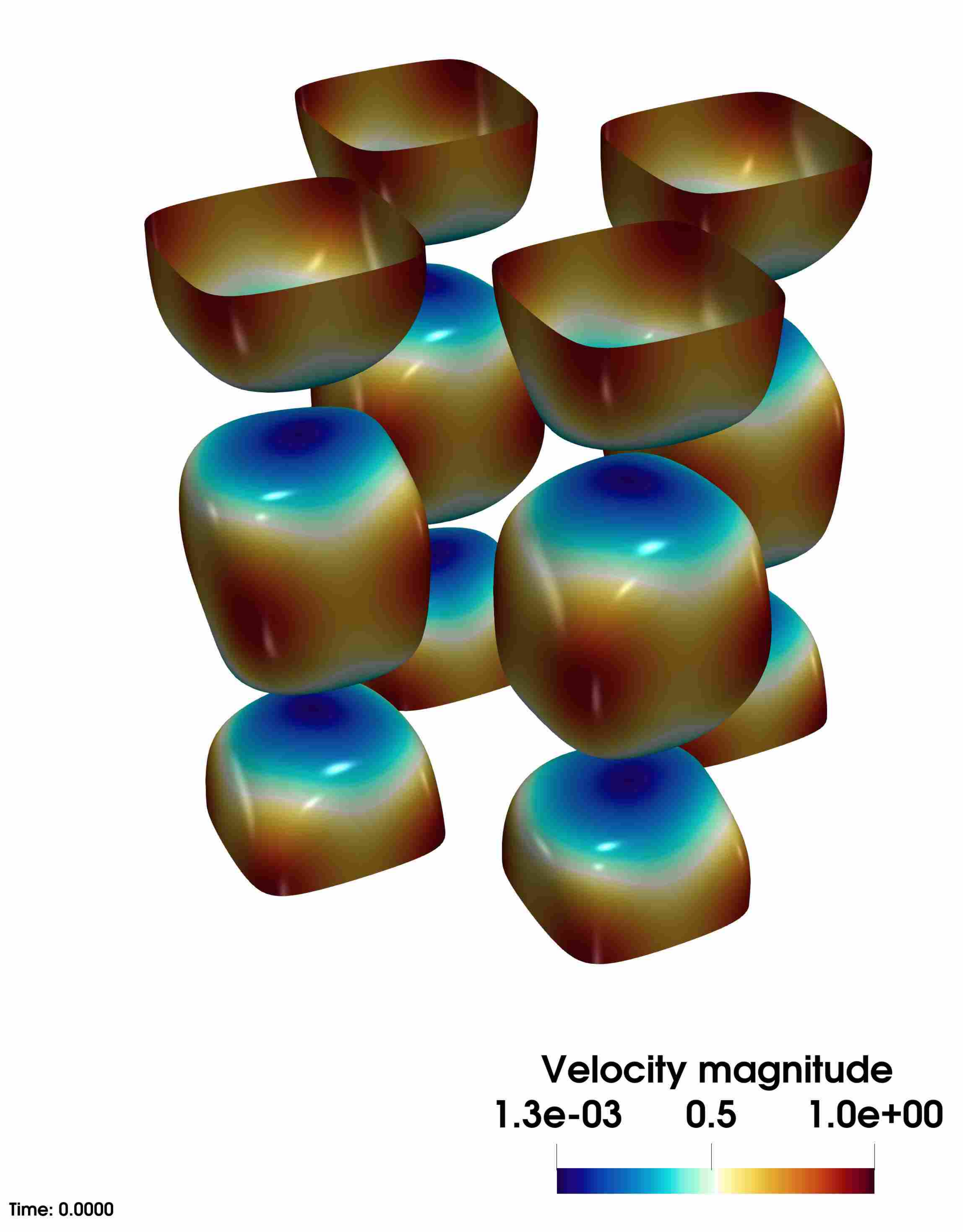}}
       \subfloat[\(t\approx 2\)]{\includegraphics[width=0.24\linewidth ,trim={0cm 12cm 0cm 0cm},clip]{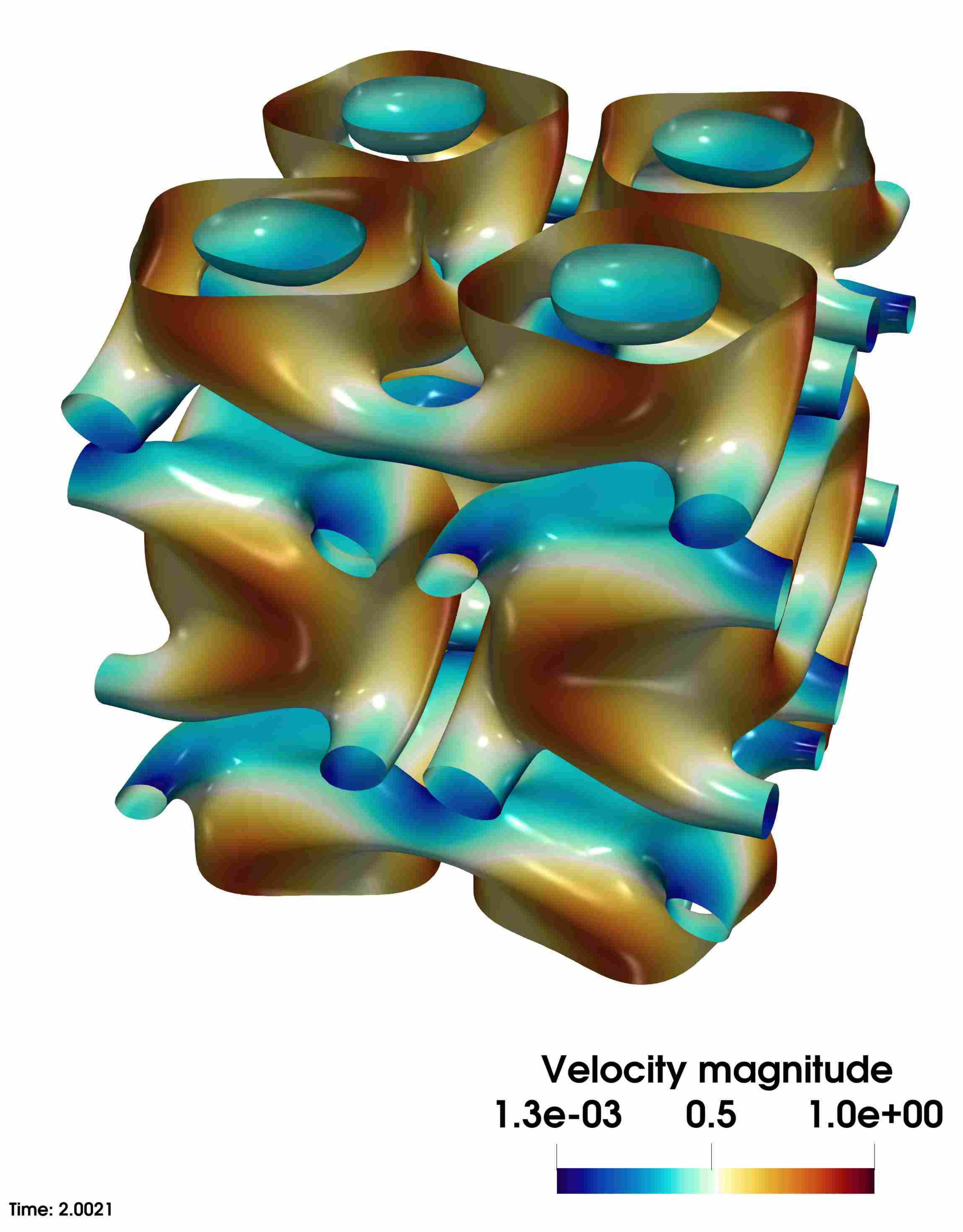}}
       \subfloat[\(t\approx 4\)]{\includegraphics[width=0.24\linewidth ,trim={0cm 12cm 0cm 0cm},clip]{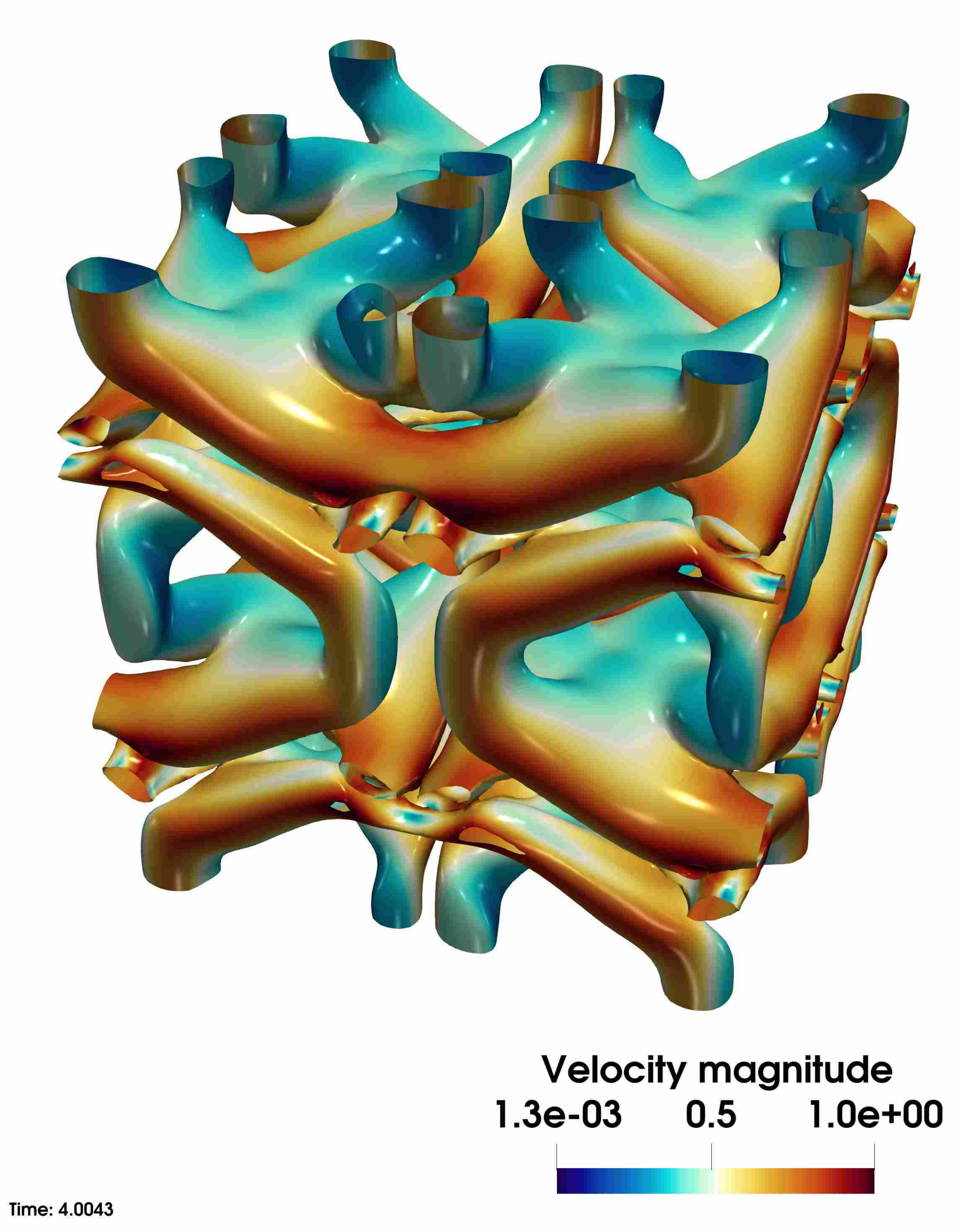}}
       \subfloat[\(t\approx 6\)]{\includegraphics[width=0.24\linewidth ,trim={0cm 12cm 0cm 0cm},clip]{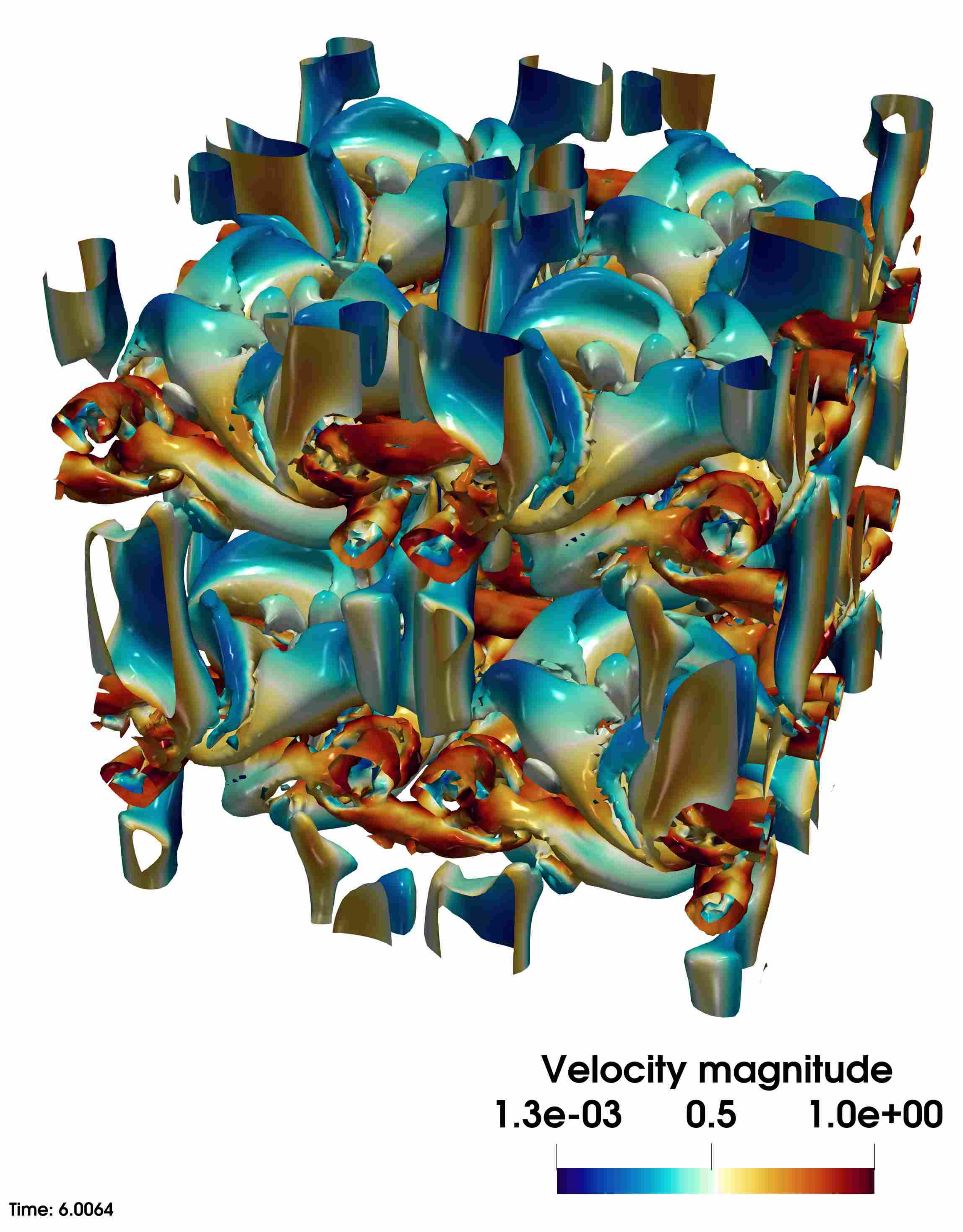}}
       
	}
	\centerline{
       \subfloat[\(t\approx 8\)]{\includegraphics[width=0.24\linewidth ,trim={0cm 12cm 0cm 0cm},clip]{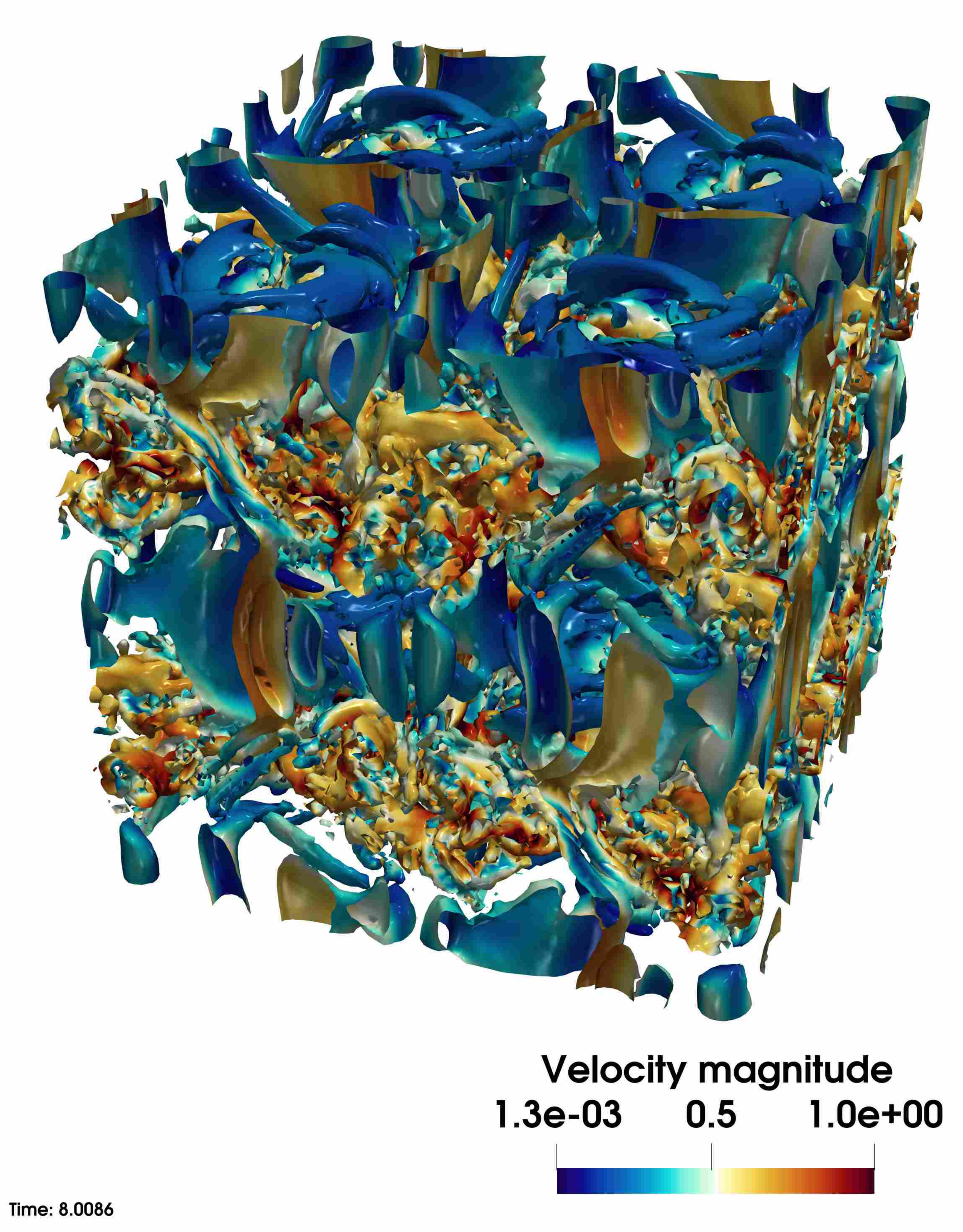}}
       \subfloat[\(t\approx 10\)]{\includegraphics[width=0.24\linewidth ,trim={0cm 12cm 0cm 0cm},clip]{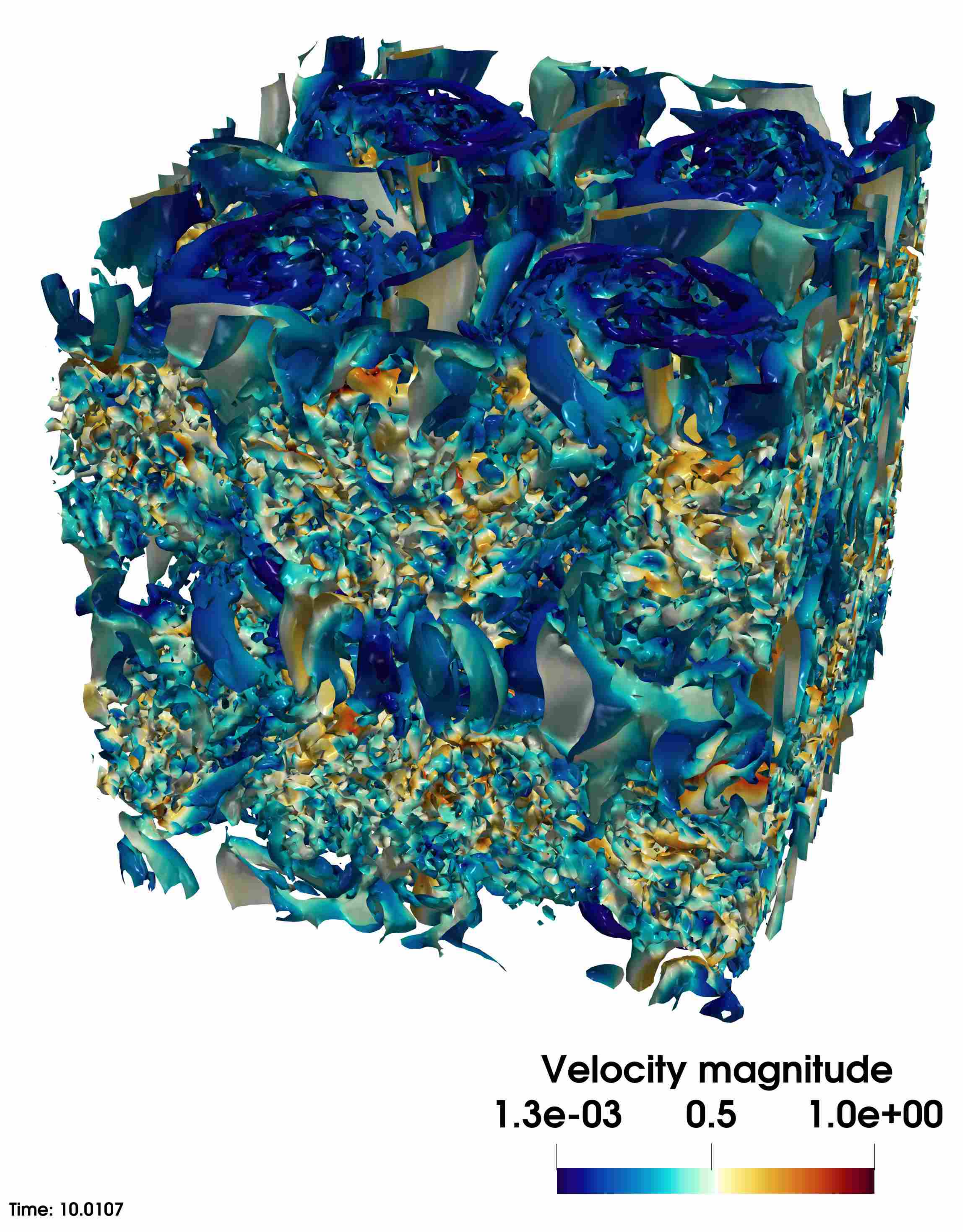}}
       \subfloat[\(t\approx 12\)]{\includegraphics[width=0.24\linewidth ,trim={0cm 12cm 0cm 0cm},clip]{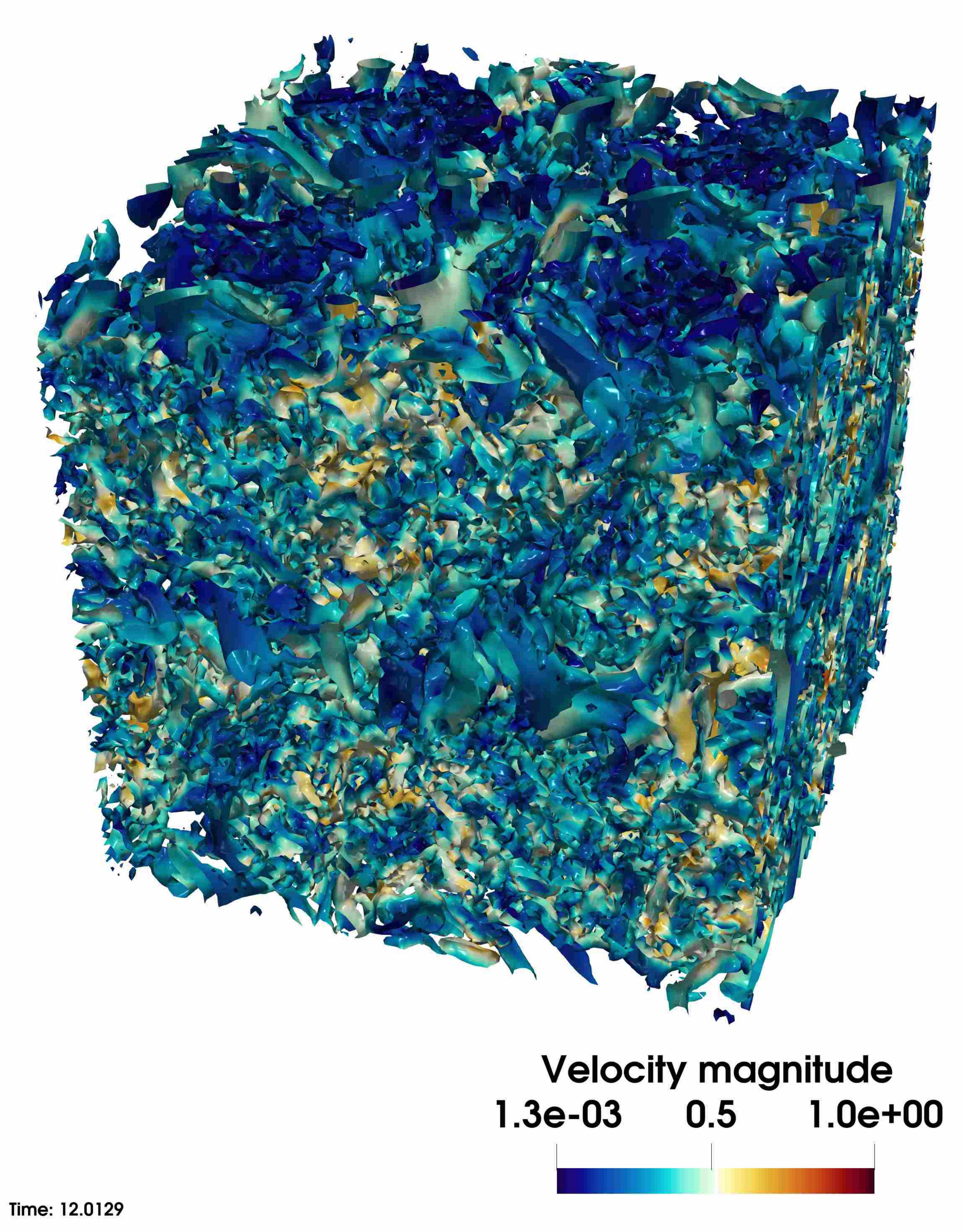}}
       \subfloat[\(t\approx 14\)]{\includegraphics[width=0.24\linewidth ,trim={0cm 12cm 0cm 0cm},clip]{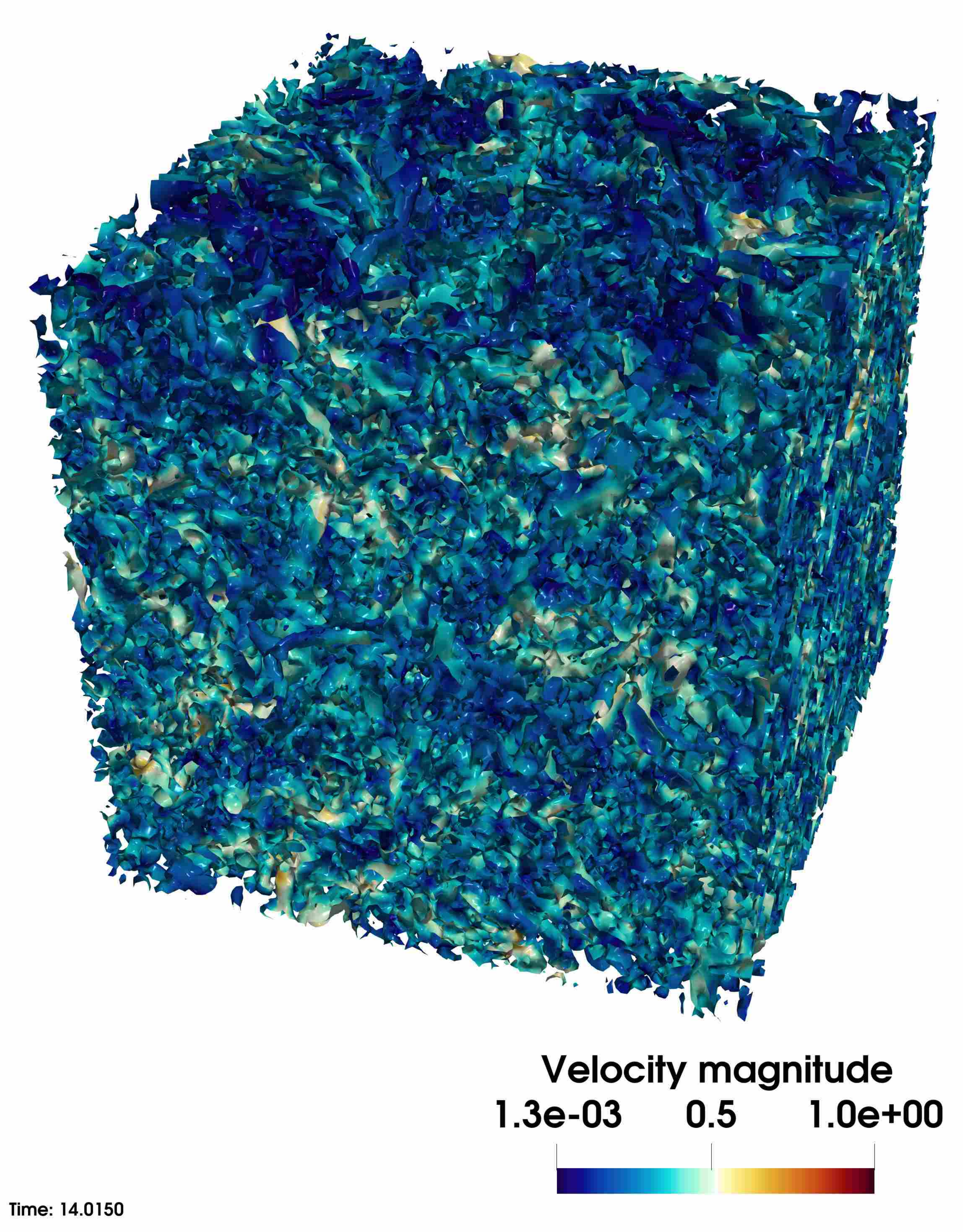}}    
    }
    \centerline{
       \subfloat[\(t\approx 16\)]{\includegraphics[width=0.24\linewidth ,trim={0cm 12cm 0cm 0cm},clip]{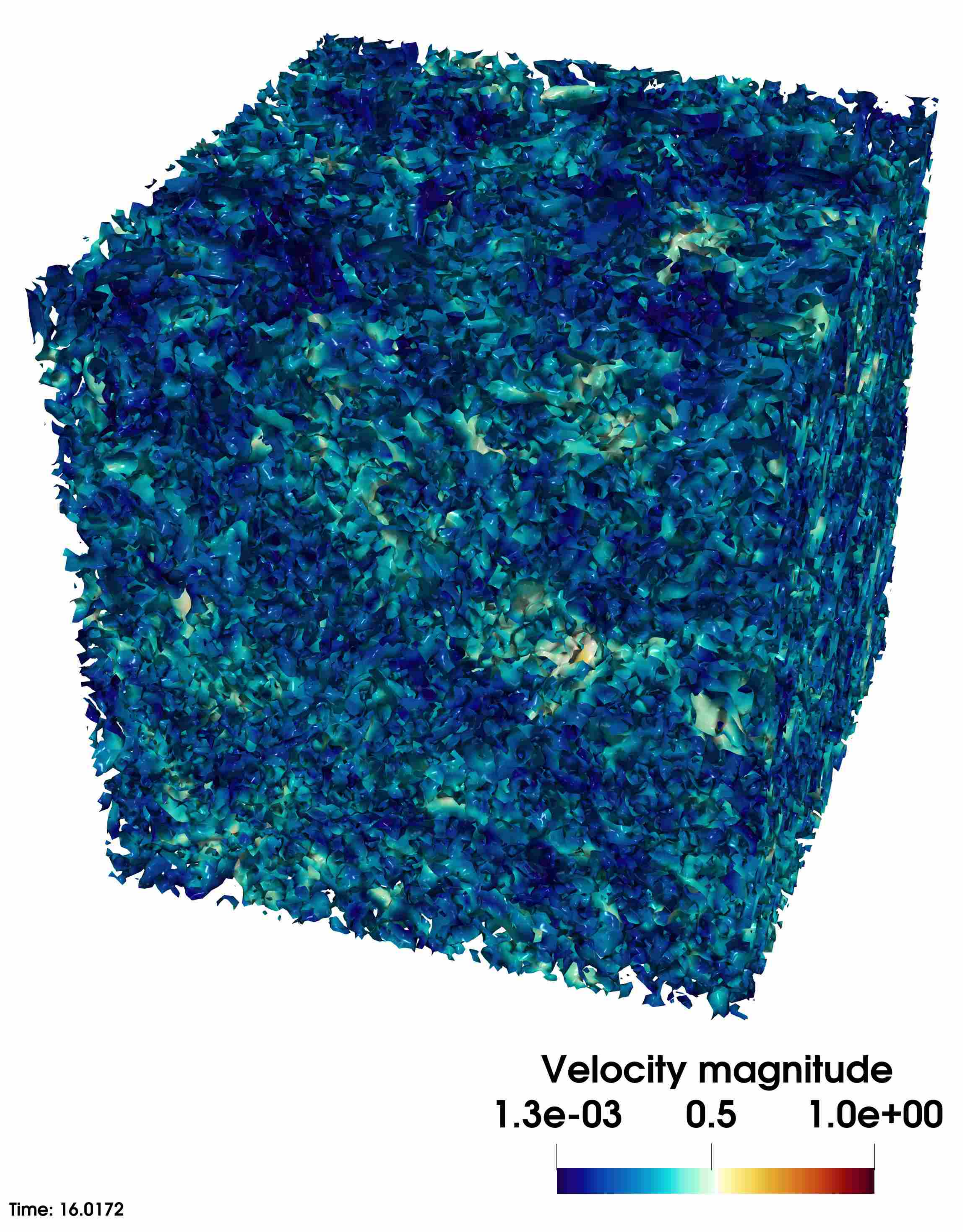}}
       \subfloat[\(t\approx 20\)]{\includegraphics[width=0.24\linewidth ,trim={0cm 12cm 0cm 0cm},clip]{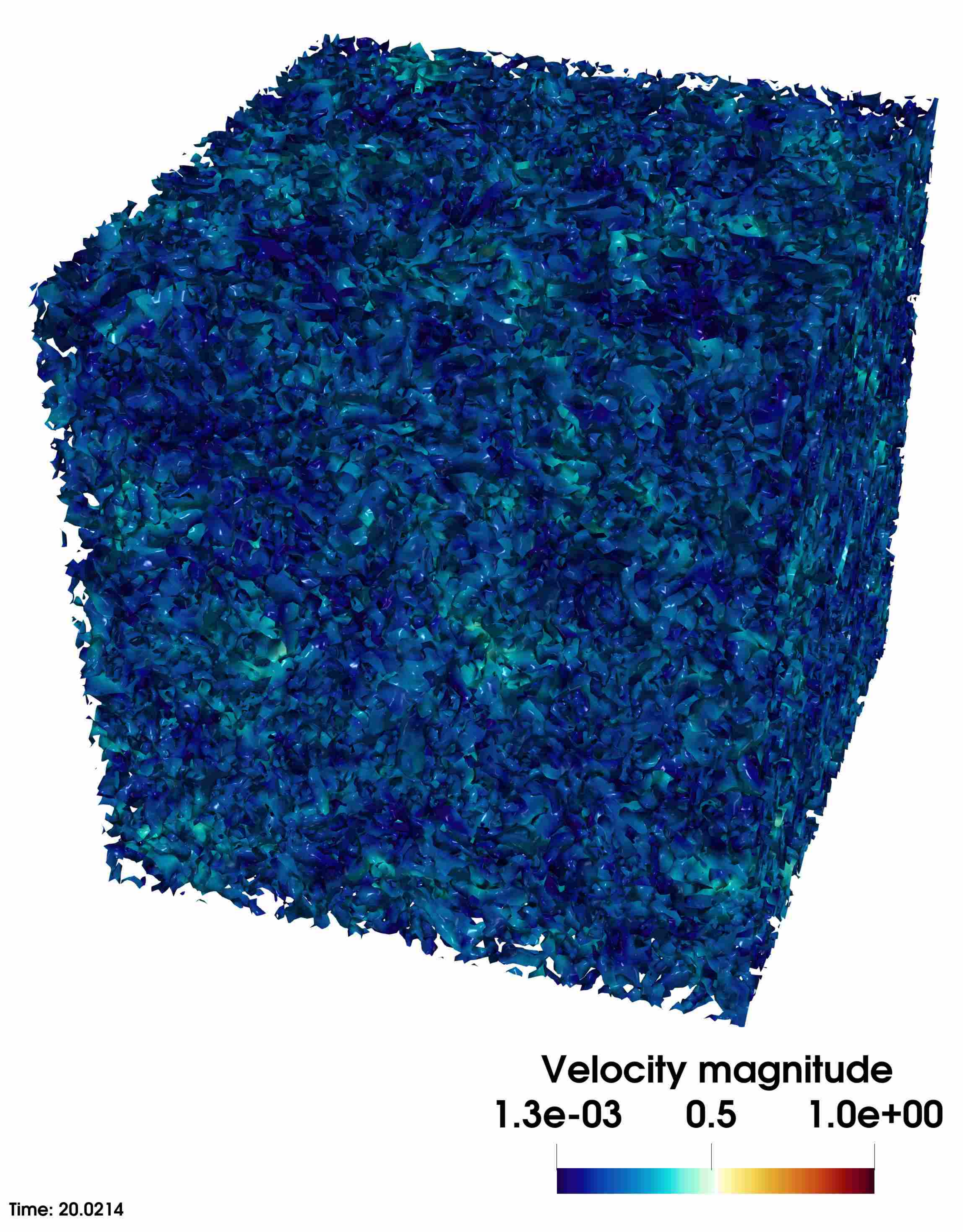}}
       \subfloat[\(t\approx 24\)]{\includegraphics[width=0.24\linewidth ,trim={0cm 12cm 0cm 0cm},clip]{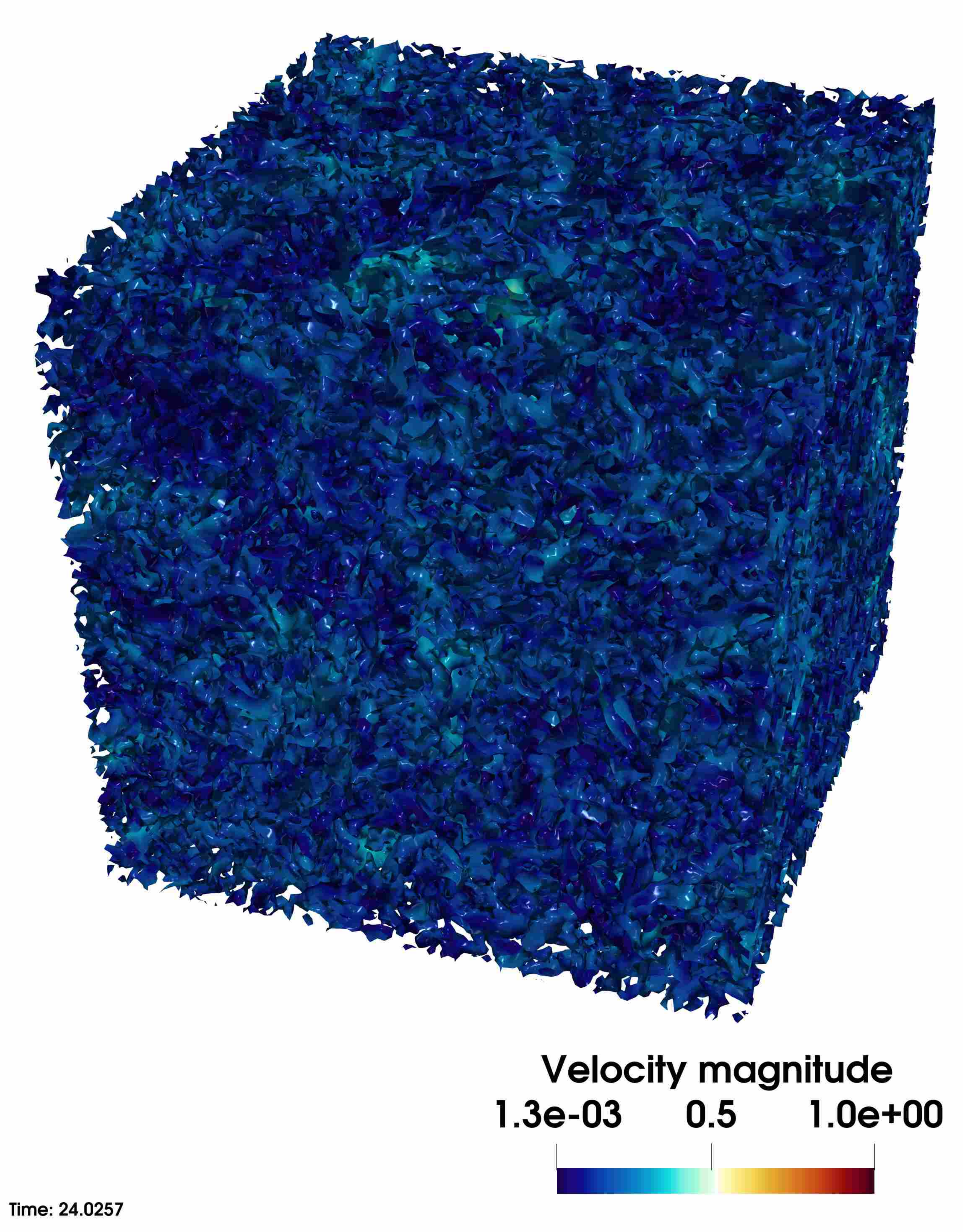}}
       \subfloat[\(t\approx 30\)]{\includegraphics[width=0.24\linewidth ,trim={0cm 12cm 0cm 0cm},clip]{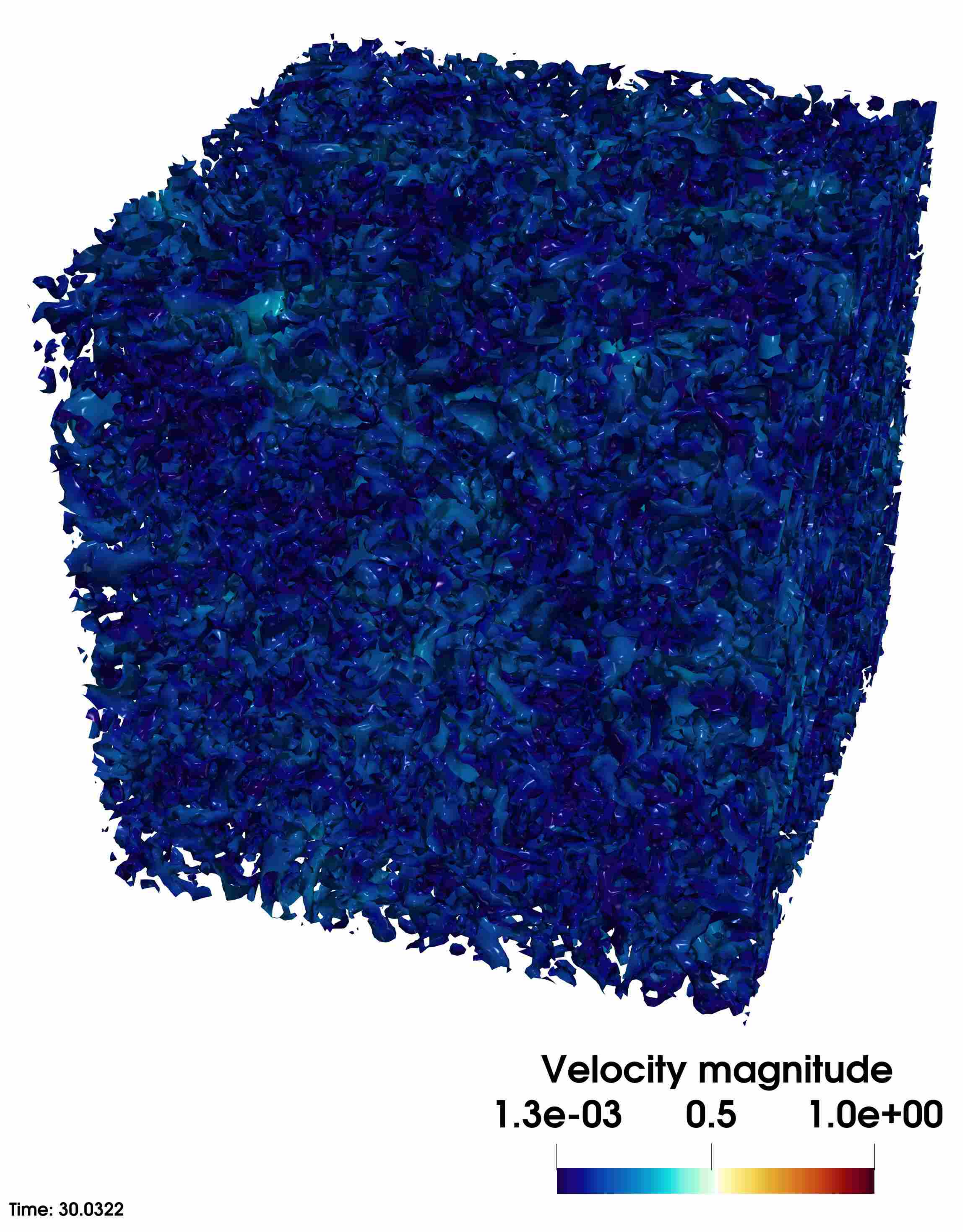}
	   }    
    }\vspace{.5em} 
    \centerline{
        \includegraphics[scale=1]{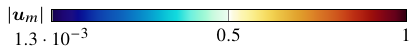}        
	}\vspace{-.5em} 
    \caption{Time evolution of the \(Q\)-criterion (\(Q_{m}=0.1\)) colored by \(|\bm{u}_{m}|\) of a single RTGV flow sample computed with KBC LBM for \(R\!e=10240\), resolution \(N=256\), and \(M\!a=0.0125\) in a separate run not belonging to Set~1 or Set~2 (Figure from \cite{simonis2023lattice} with permission of the author).}
    \label{fig:sampleQ}
\end{figure}
To illustrate the probabilistic extension of the classical TGV flow, Figure~\ref{fig:sampleQ} highlights the random perturbation in the \(Q\)-criterion~\cite{hunt1988eddies} of a single sample RTGV flow field, where
\begin{align}
    Q_{m} = \frac{1}{2} ( \| \mathbf{N}_{m} \|^{2}_{2} - \| \mathbf{D}_{m} \|^{2}_{2} )
\end{align}
is computed using the spectral norm \(\|\cdot \|_{2}\) of the rate of strain \(\mathbf{D}_{m}\) and its antisymmetric counterpart $\mathbf{N}_{m} = \frac{1}{2} [ \Grad_{\bm{x}} \bm{u}_{m} - ( \Grad_{\bm{x}} \bm{u}_{m} )^{\mathrm{T}} ]$, respectively.

\subsection{Visualizations of statistical flow fields} 

We compute the flow field induced by the uniformly RTGV initial condition. 
The single samples and mean fields are contrasted in Figure~\ref{fig:statSolRTGV}. 
\begin{figure}[ht!]
    \subfloat[Sample, \(R\!e = 1280\)]{\includegraphics[width=0.27\textwidth,trim={0cm 1.5cm 0cm 2cm},clip]{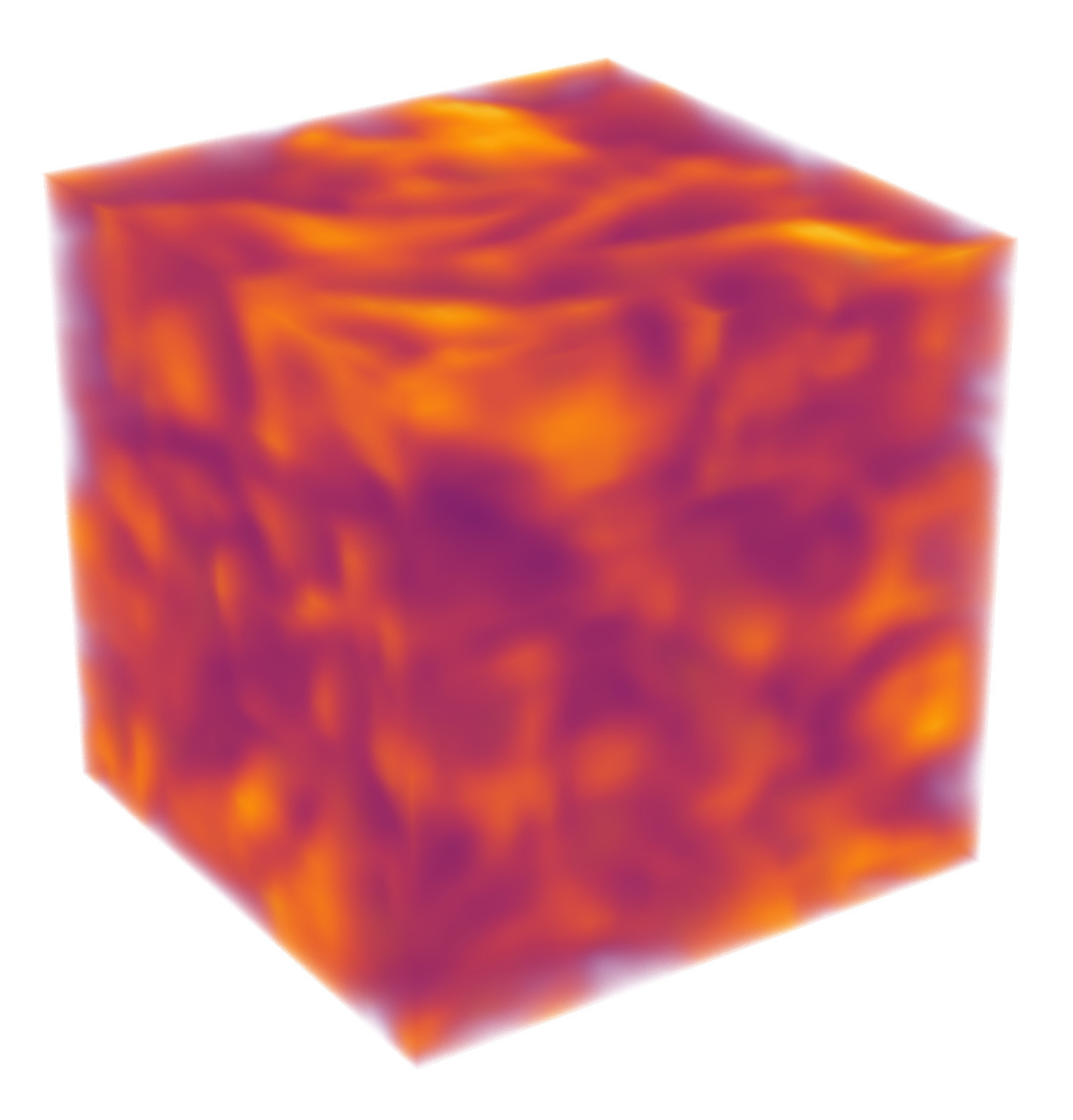}
          \label{subfig:statSolRTGV-sample-1280}}
    \subfloat[Sample, \(R\!e = 2560\)]{\includegraphics[width=0.27\textwidth,trim={0cm 1.5cm 0cm 2cm},clip]{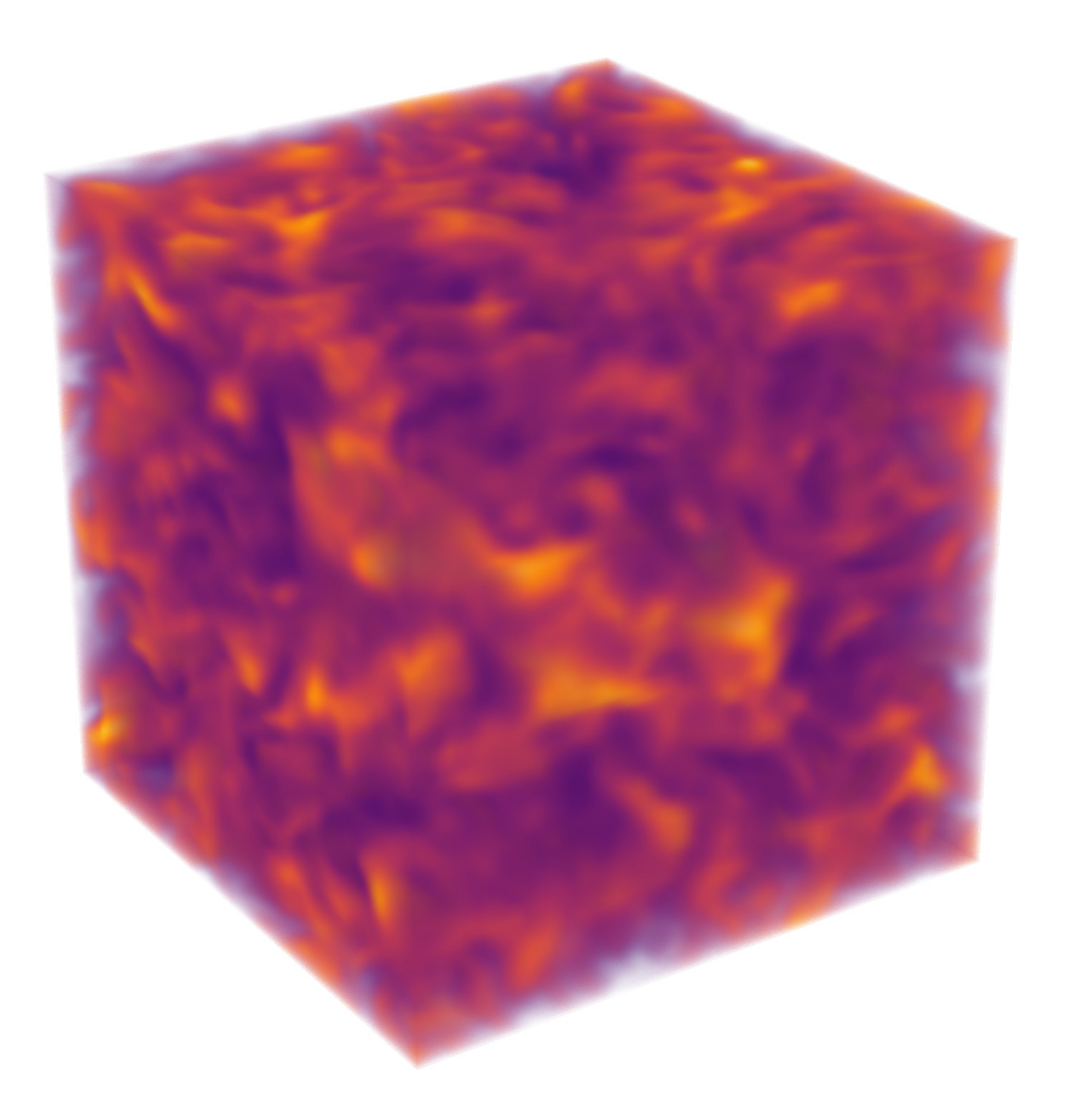}
          \label{subfig:statSolRTGV-sample-2560}}
   	\subfloat[Sample, \(R\!e = 5120\)]{\includegraphics[width=0.27\textwidth,trim={0cm 1.5cm 0cm 2cm},clip]{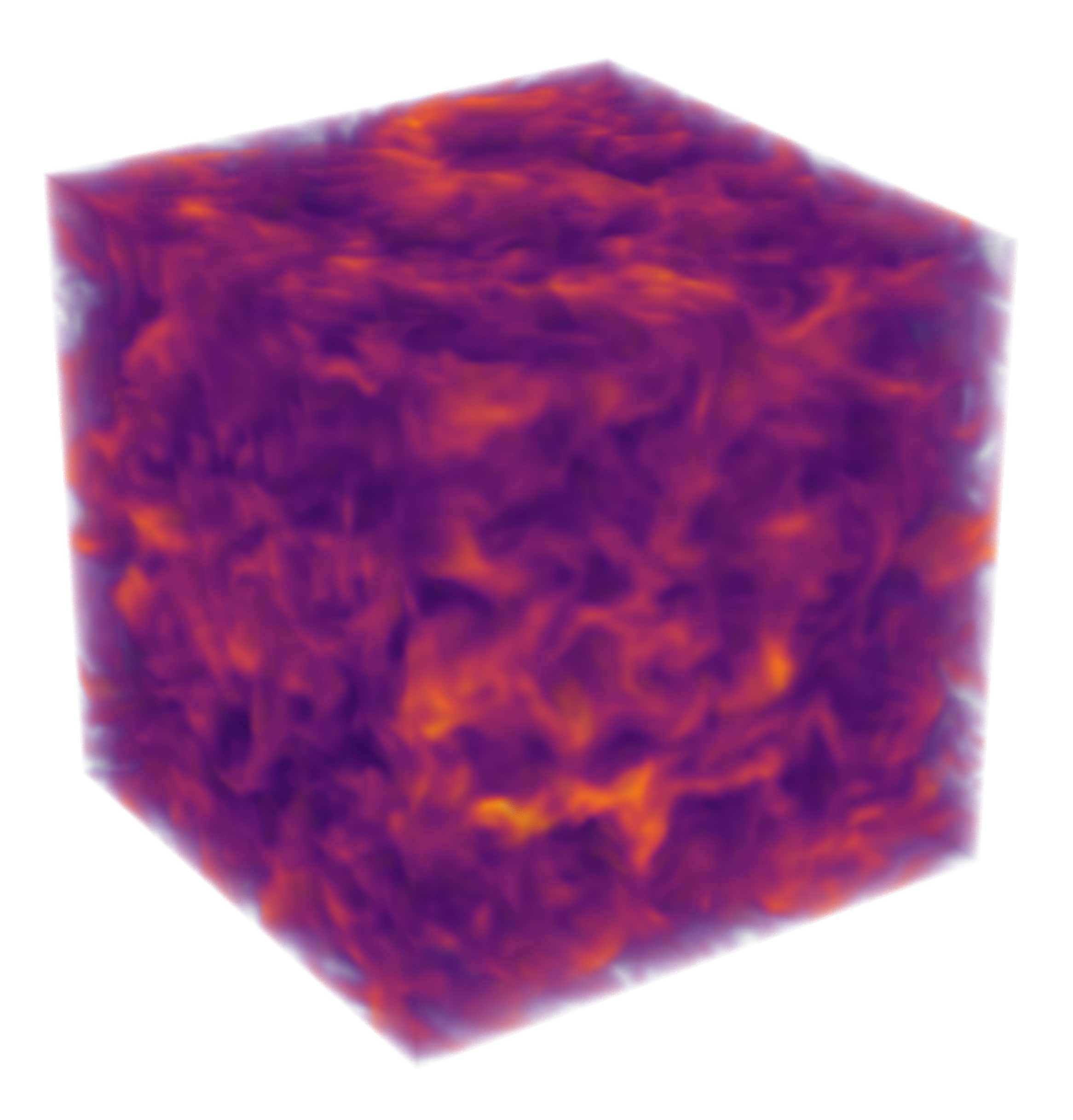}
          \label{subfig:statSolRTGV-sample-5120}} 
    \hspace{.1em}
    \includegraphics[scale=1]{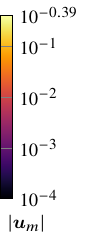}
    \\
    \subfloat[Mean, \(R\!e = 1280\)]{\includegraphics[width=0.27\textwidth,trim={0cm 1.5cm 0cm 2cm},clip]{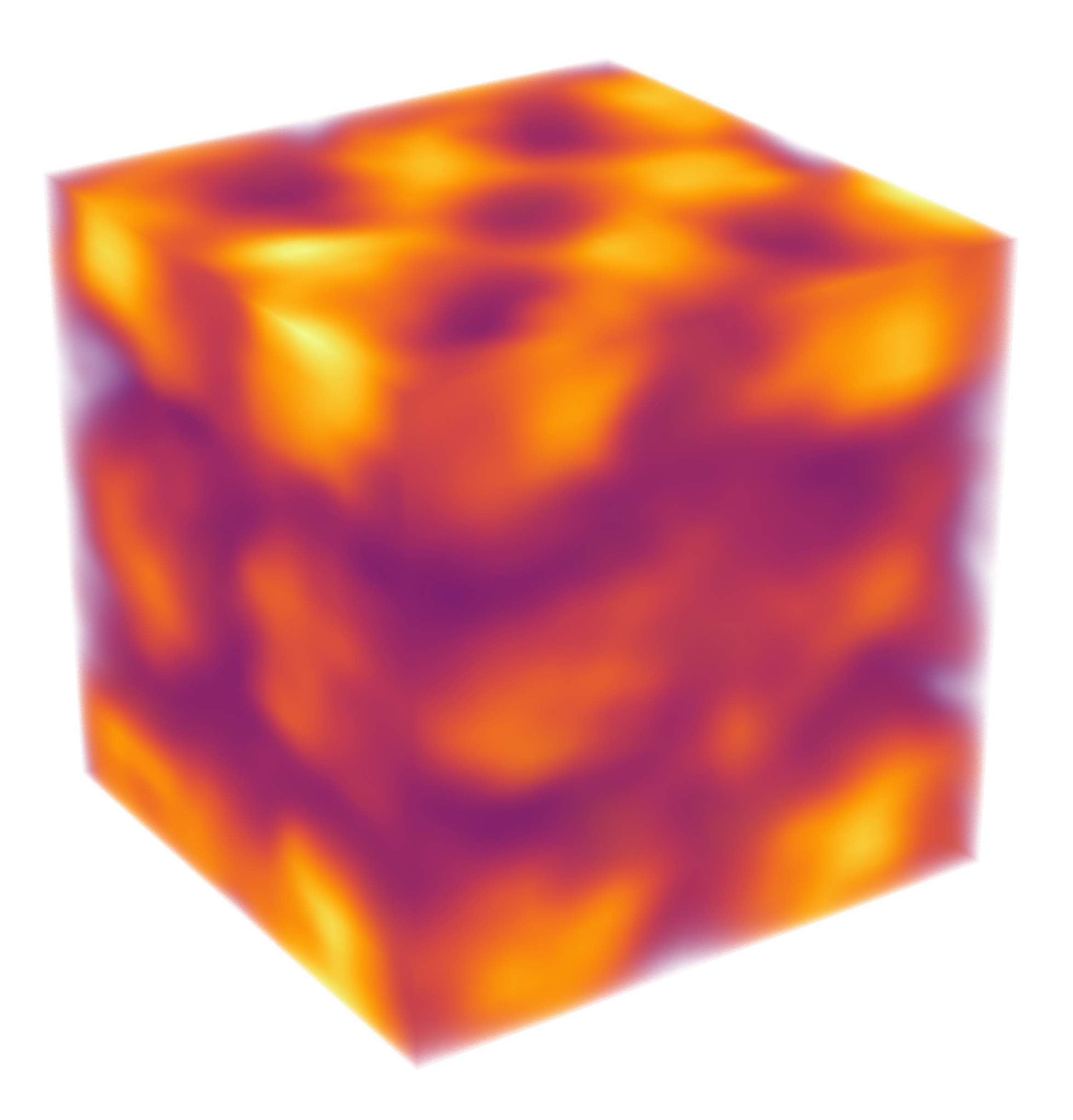}
          \label{subfig:statSolRTGV-mean-1280}}
	\subfloat[Mean, \(R\!e = 2560\)]{\includegraphics[width=0.27\textwidth,trim={0cm 1.5cm 0cm 2cm},clip]{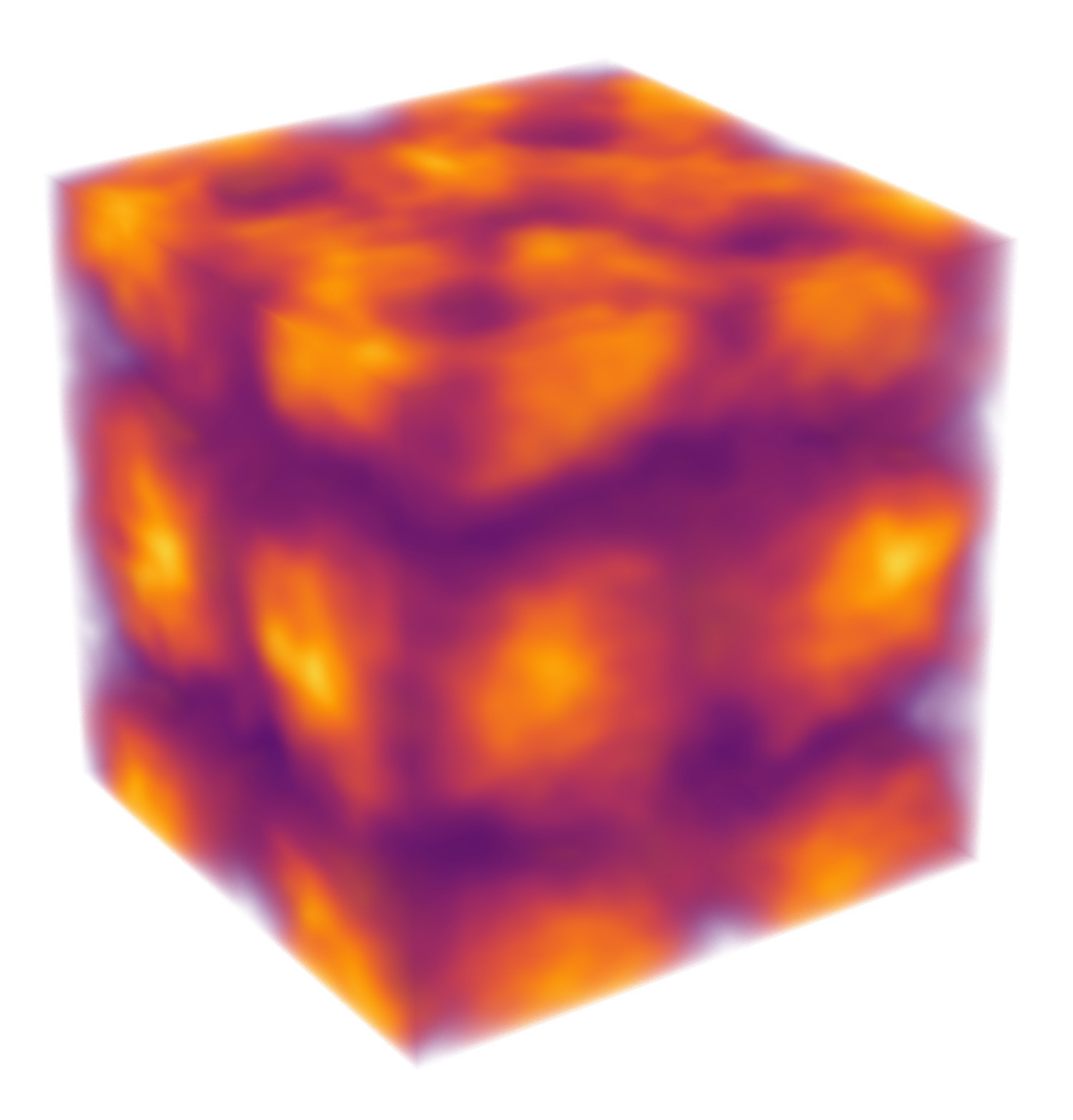}
          \label{subfig:statSolRTGV-mean-2560}}
    \subfloat[Mean, \(R\!e = 5120\)]{\includegraphics[width=0.27\textwidth,trim={0cm 1.5cm 0cm 2cm},clip]{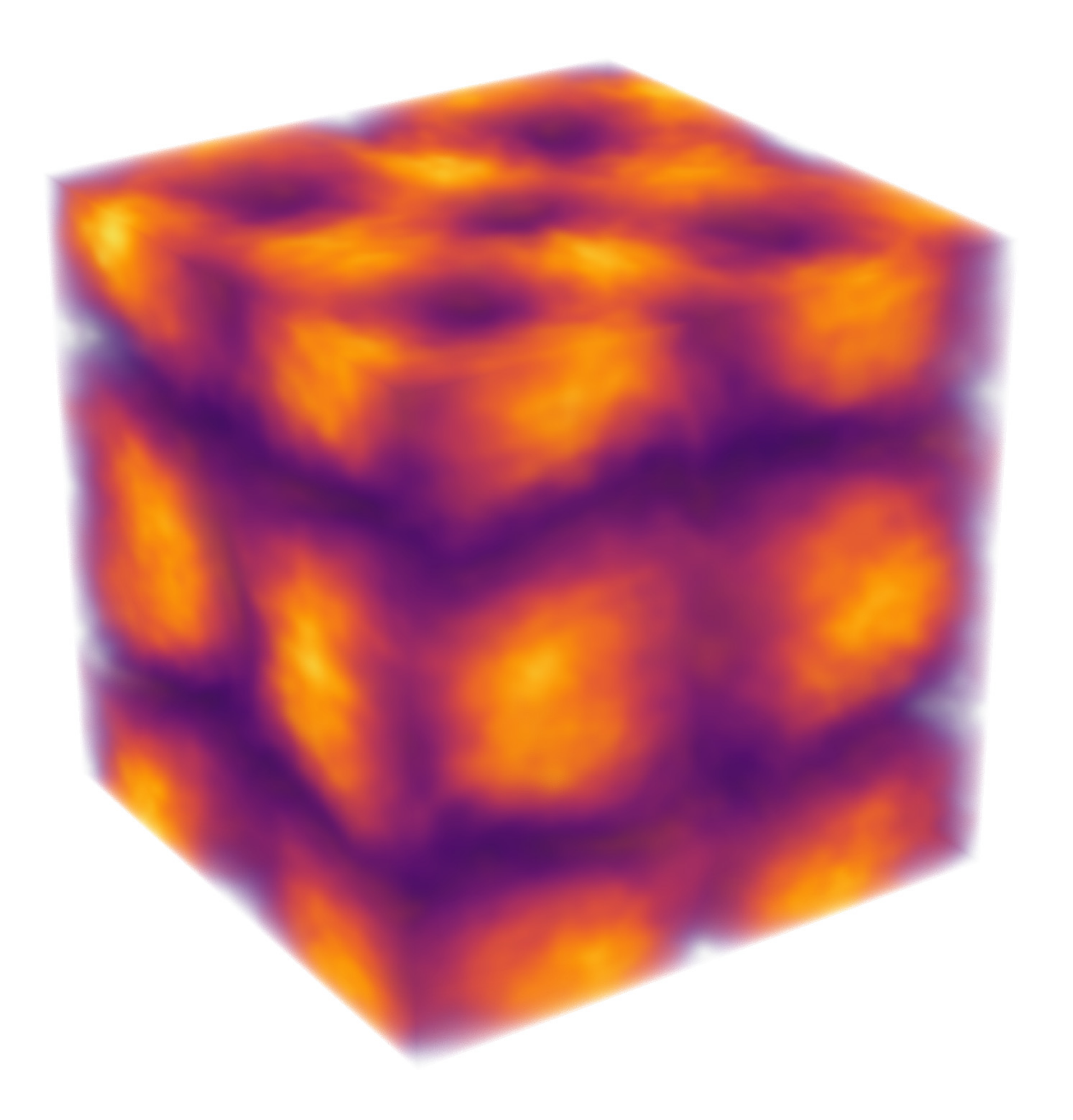}
          \label{subfig:statSolRTGV-mean-5120}}
    \hspace{.1em}
    \includegraphics[scale=1]{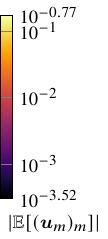}
    \\
    \subfloat[Std, \(R\!e = 1280\)]{\includegraphics[width=0.27\textwidth,trim={0cm 1.5cm 0cm 2cm},clip]{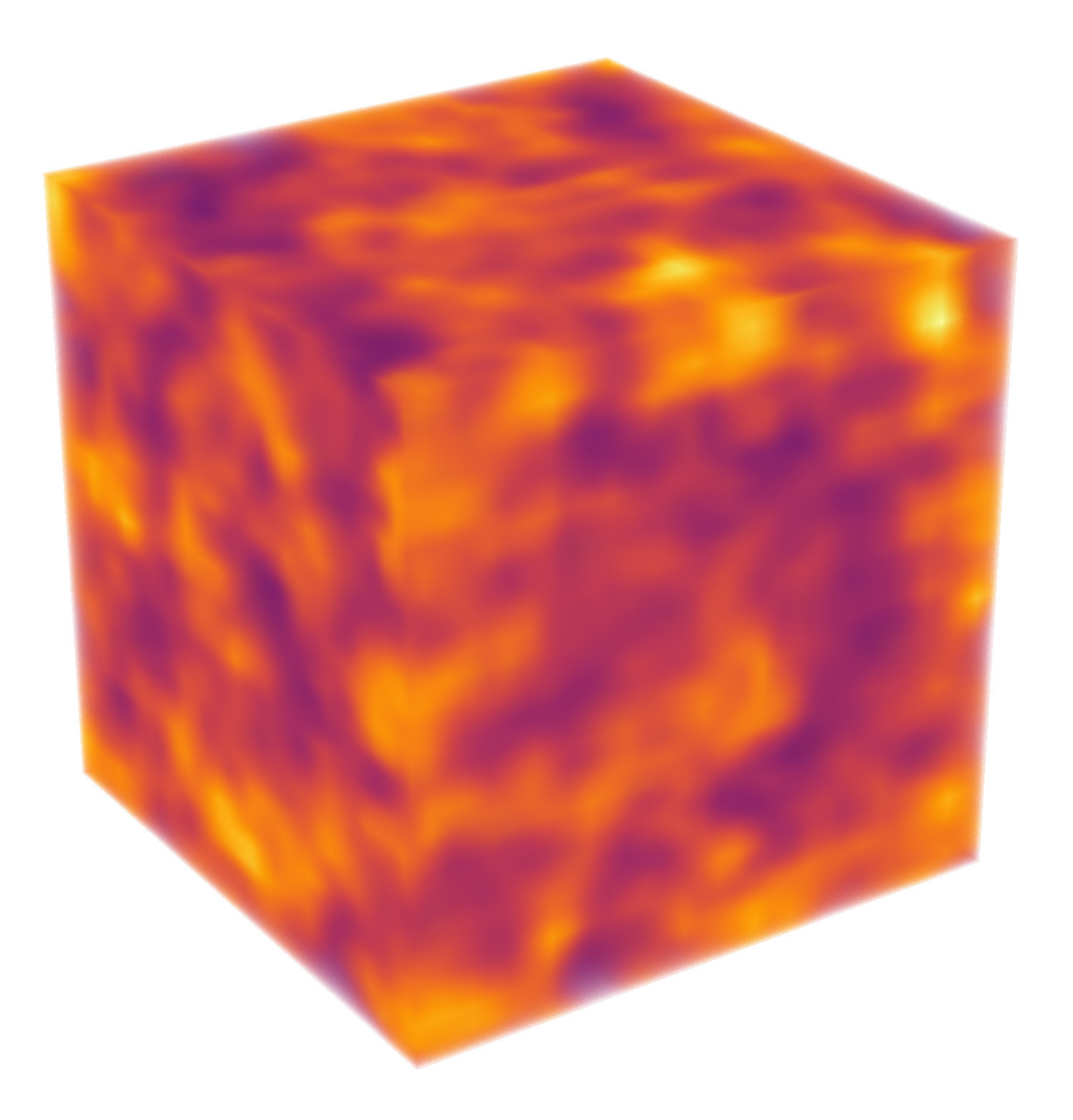}
          \label{subfig:statSolRTGV-std-1280}}
	\subfloat[Std, \(R\!e = 2560\)]{\includegraphics[width=0.27\textwidth,trim={0cm 1.5cm 0cm 2cm},clip]{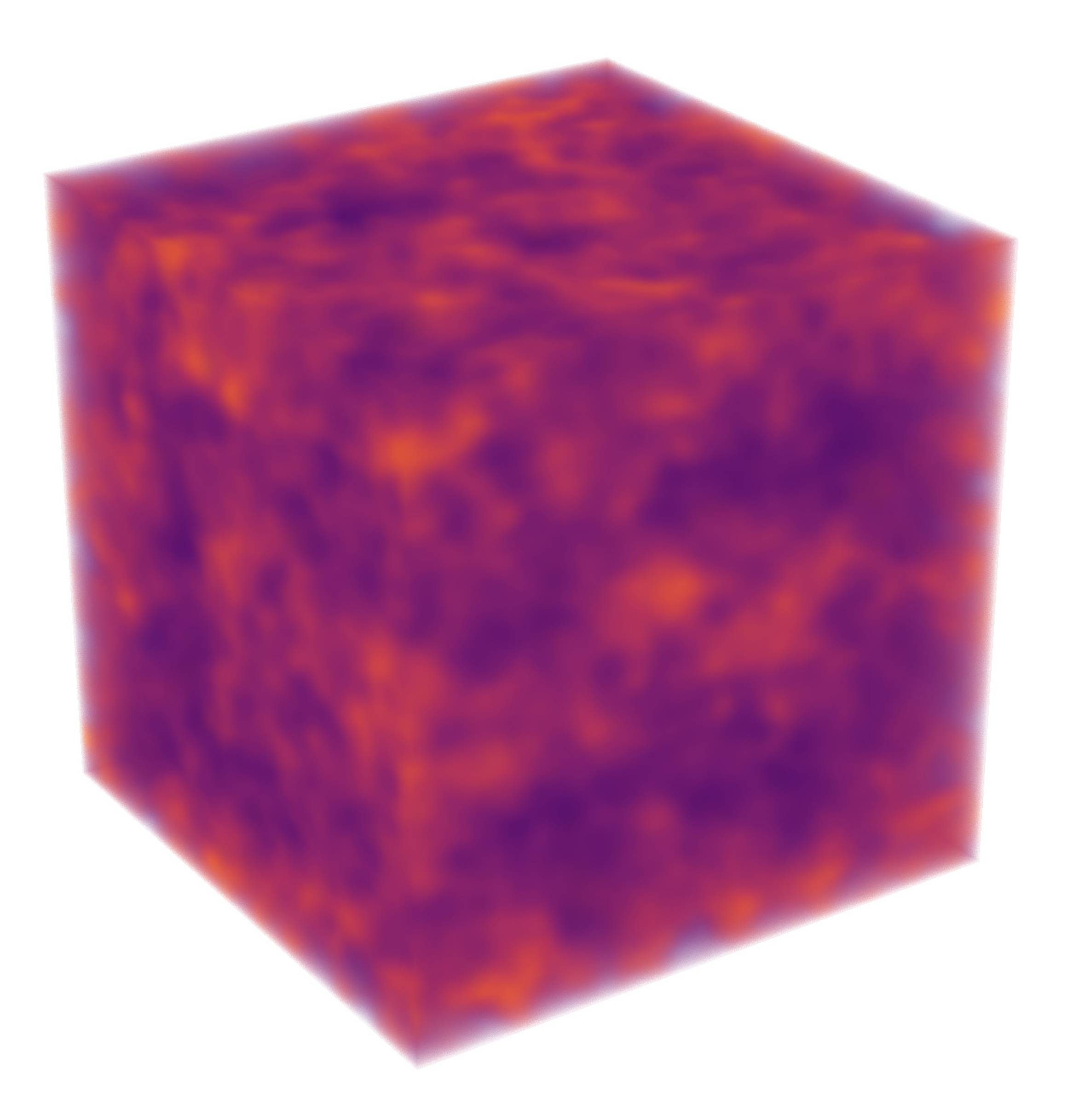}
          \label{subfig:statSolRTGV-std-2560}}
    \subfloat[Std, \(R\!e = 5120\)]{\includegraphics[width=0.27\textwidth,trim={0cm 1.5cm 0cm 2cm},clip]{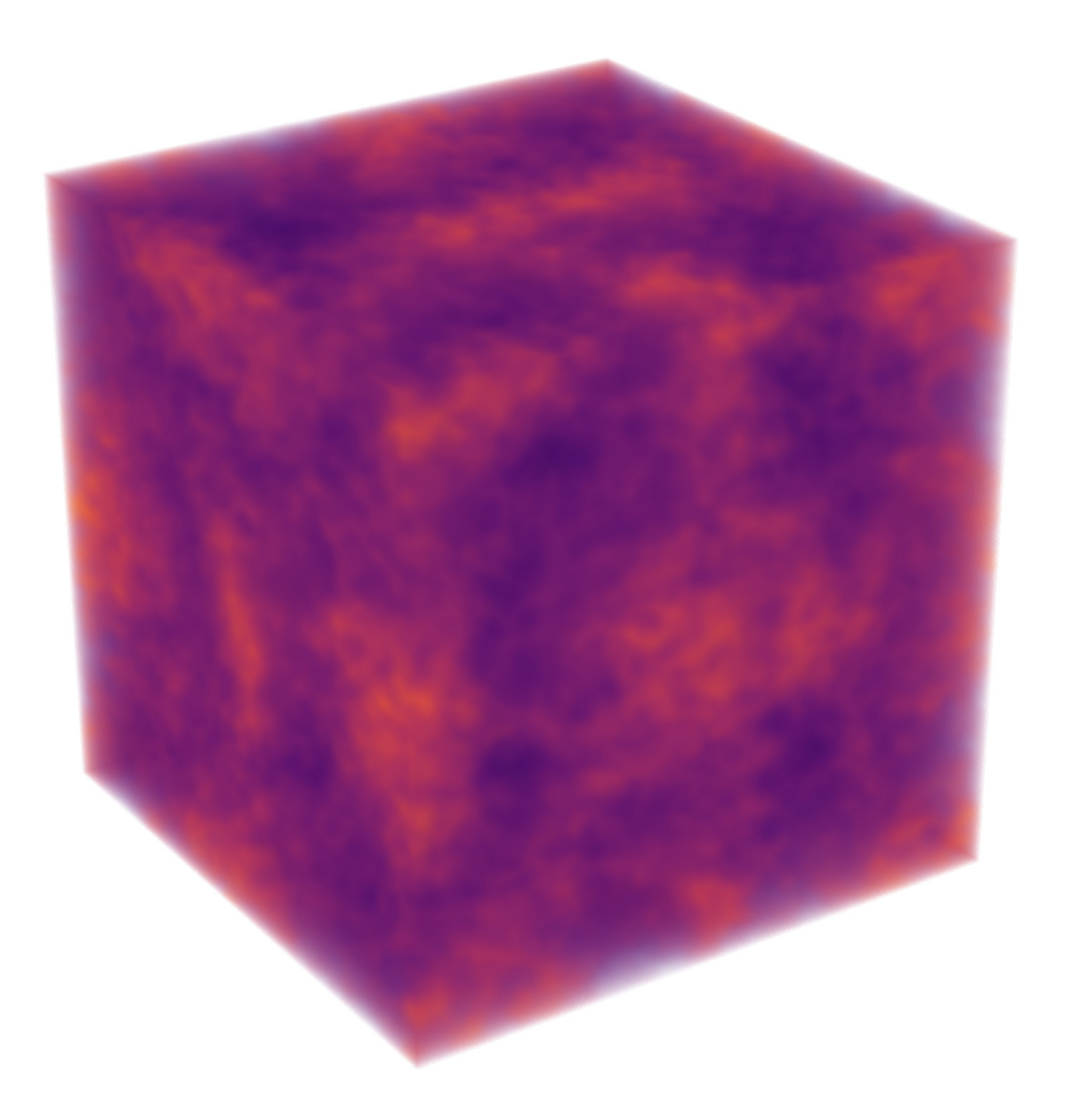}
          \label{subfig:statSolRTGV-std-5120}}
    \hspace{.1em}
    \includegraphics[scale=1]{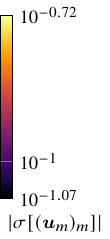}
    \\ 
	\caption{Iso-volume rendering of the velocity magnitude of the RTGV flow at \(t\approx 30\mathrm{s}\).  
    Top row: single samples; middle row: mean; bottom row: standard deviation (std); columns from left to right: \(R\!e = 1280, 2560, 5120\). 
    In the renderings, the color map is combined with an opacity transfer function that is linear in the data value, ranging from opacity zero at the minimum to one at the maximum of each colorbar. 
    The colorbars themselves omit this opacity. 
    Discretization parameters from Set~2 (Table~\ref{tab:params2}) are used. 
    Figure partly reproduced from~\cite{simonis2024computing} with permission of the authors.
    Further time steps are visualized in Appendix~\ref{appsec:flowfields}.}
    \label{fig:statSolRTGV}
\end{figure}
The sample solutions (Figures~\ref{subfig:statSolRTGV-sample-1280},~\ref{subfig:statSolRTGV-sample-2560},~\ref{subfig:statSolRTGV-sample-5120}) show no sign of convergence under grid refinement, whereas the respective mean and standard deviation approximations (Figures~\ref{subfig:statSolRTGV-mean-1280},~\ref{subfig:statSolRTGV-mean-2560},~\ref{subfig:statSolRTGV-mean-5120} and Figures~\ref{subfig:statSolRTGV-std-1280},~\ref{subfig:statSolRTGV-std-2560},~\ref{subfig:statSolRTGV-std-5120}) representing the statistical solutions behave as expected for \(\vis \searrow 0\). 
Approximations for the flow fields (single sample strain, mean, and standard deviation) at additional time steps are provided in Appendix~\ref{appsec:flowfields}. 
We note that visually, the convergence behavior of single samples and stochastic moments seems to split apart roughly at \(t\approx 10\), where the turbulent structures have started to develop in the flow field.

\subsection{Computed scaling assumption}

Based on~\cite{fjordholm2021vanishing}, the evaluation of energy spectra in the sense of K41 theory~\cite{kolmogoroff1941local,kolmogorov1991local} (testing for an asymptotic power law decay \eqref{eq:energyScaling}) is used to indicate the approximation of a statistical solution to the incompressible EE for \(\vis \searrow 0\). 
\begin{table}[ht!]
    \caption{Set~1 of MC LBM discretization parameters for RTGV flow simulations.}
    \label{tab:params1}
    \centering
    \small
    \begin{tabular}{r r l l r l}
    \hline
        \makecell[c]{$M$} & \makecell[c]{$N$} & \makecell[c]{$\triangle t$} & \makecell[c]{$M\!a$} & \makecell[c]{$R\!e$} & \makecell[c]{$\tau_{\mathrm{rel}}$} \\
        \hline
         $32$ &  $32$ & $1.17 \times 10^{-2}$ & $0.1$    &  $640$ & $0.501335$ \\
         $64$ &  $64$ & $2.87 \times 10^{-3}$ & $0.05$   & $1280$ & $0.500678$ \\
        $128$ & $128$ & $7.14 \times 10^{-4}$ & $0.025$  & $2560$ & $0.500342$ \\
        $256$ & $256$ & $1.78 \times 10^{-4}$ & $0.0125$ & $5120$ & $0.500172$ \\
    \hline
    \end{tabular}
\end{table}
To render the spectral quantities dimensionless, we introduce the sample dissipation rate per unit volume
\begin{align}
    \mathfrak{d}_{m}(t) = \frac{\vis}{\vert \Omega \vert} \|\bm{\nabla} \bm{u}_{m}(t)\|_{L^2(\Omega)}^2 ,
\end{align}
which is numerically approximated by the temporal decay of the kinetic energy per unit volume obtained from summing the energy spectrum, that is \(\mathfrak{d}_{m}(t) \approx - \frac{\mathrm{d}}{\mathrm{d} t} \sum_{\kappa} E_{m}(\kappa, t)\).
The corresponding Kolmogorov length scale 
\begin{equation} \label{eq:kolmogorov_microscale}
    \eta(t) = \left( \frac{\vis^3}{\E[\mathfrak{d}(t)]} \right)^{1/4}
\end{equation}
is computed from the ensemble-mean dissipation rate \(\E[\mathfrak{d}(t)]\) and defines the inner-scaled wavenumber \(\kappa \eta(t)\). 

Figure~\ref{fig:energySpectra9} shows the mean \(\E[\cdot]\) and standard deviation \(\sigma(\cdot)\) of the time-dependent, inner-scaled compensated energy spectrum 
\begin{align}\label{eq:compensatedE}
    E^{\mathrm{c}}_{m}(\kappa\eta(t), t) = C_{\mathrm{Kol}}^{-1} \mathfrak{d}_{m}^{\frac{-2}{3}}(t) \kappa^{\frac{5}{3}} E_{m}(\kappa, t)
\end{align}
for four consecutive Reynolds numbers \(R\!e = 20N\) in diffusive scaling, where \(N=32, 64, 128, 256\). 
The other discretization parameters for this set of simulations (Set~1) are summarized in Table~\ref{tab:params1}. 
Here, the number of samples is chosen as \(M=N\). 
By connecting the Reynolds number and the grid resolution per spatial dimension, an inviscid limit is superimposed for \(N\nearrow \infty\). 

Note that the compensation in \eqref{eq:compensatedE} is performed sample-wise with \(\mathfrak{d}_{m}\) prior to computing the mean and standard deviation, whereas the inner scaling of the wavenumber argument uses the ensemble-mean dissipation rate through \eqref{eq:kolmogorov_microscale}. 
Owing to the division by \(C_{\mathrm{Kol}}\), any wavenumber range that strictly obeys the K41 scaling \eqref{eq:energyScaling} manifests as a horizontal plateau at unity, while the inner scaling collapses the dissipative ranges of all Reynolds numbers near \(\kappa\eta = \mathcal{O}(1)\). 
The approximation of the energy spectrum is described in Appendix~\ref{sec:appendix-energySpec}. 
The K41 constant~\cite{kolmogoroff1941local,kolmogorov1991local} is \(C_{\mathrm{Kol}} = 1.5 \). 

Figure~\ref{fig:energySpectra9} shows that the mean spectrum approaches the K41 inertial subrange, which extends toward smaller \(\kappa\eta\) with increasing Reynolds number. 
\begin{figure}[ht!]
		\centering
		\includegraphics[scale=1]{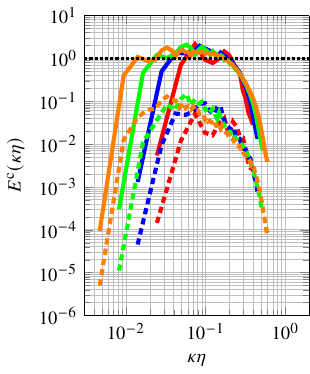}\quad\quad
	    \includegraphics[scale=1]{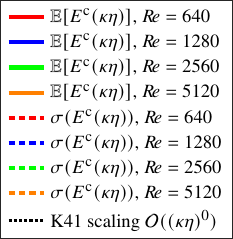}
        \caption{Inner-scaled compensated energy spectrum \eqref{eq:compensatedE} approximations of statistical solutions (mean \(\E\) and standard deviation (std) \(\sigma\)) at \(t\approx 9\mathrm{s}\) with increasing Reynolds numbers computed with MC LBM for discretization parameters from Set~1 (see Table~\ref{tab:params1}).
        The wavenumber \(\kappa\) is inner-scaled by the Kolmogorov length scale \(\eta(t)\) from \eqref{eq:kolmogorov_microscale}. 
  }
        \label{fig:energySpectra9}
\end{figure}
The time evolution of the inner-scaled compensated energy spectra is visualized with waterfall diagrams in Figure~\ref{fig:spectraWaterfalls}. 
\begin{figure}[ht!]
    \centering
    \subfloat[Mean, \(R\!e=1280\)]{
    \includegraphics[scale=1]{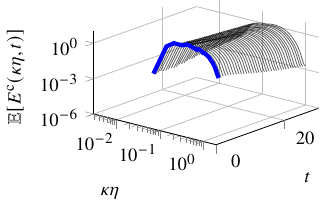}}
	\quad
    \subfloat[Std, \(R\!e=1280\)]{
    \includegraphics[scale=1]{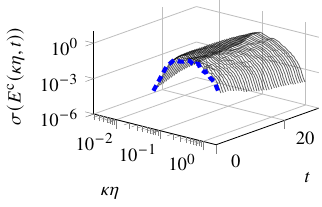}}
	\\ 
    \subfloat[Mean, \(R\!e=2560\)]{
    \includegraphics[scale=1]{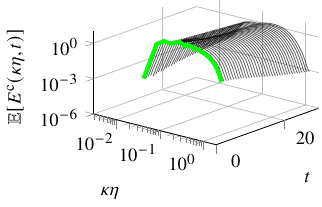}}
    \quad
    \subfloat[Std, \(R\!e=2560\)]{
    \includegraphics[scale=1]{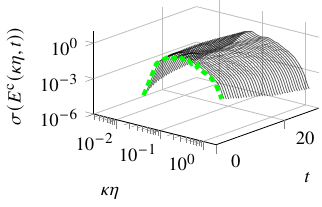}}
	\\
    \subfloat[Mean, \(R\!e=5120\)]{
    \includegraphics[scale=1]{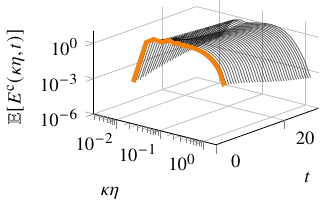}}
    \quad
    \subfloat[Std, \(R\!e=5120\)]{
    \includegraphics[scale=1]{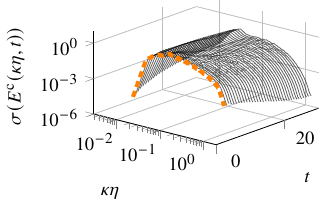}}
    \caption{Inner-scaled compensated energy spectrum \eqref{eq:compensatedE} approximations of statistical solutions (mean \(\E\) and standard deviation (std) \(\sigma\)) over time \(t\in [9,30]\) with increasing Reynolds numbers computed with MC LBM for discretization parameters from Set~1 (see Table~\ref{tab:params1}).
    Colors refer to the respective data lines in Figure~\ref{fig:energySpectra9}.
  }
    \label{fig:spectraWaterfalls}
\end{figure}

To numerically examine the scaling assumption \eqref{eq:energyScaling}, we investigate the scaling of second-order time-local structure functions \eqref{eq:structuresLocal} with increasing Reynolds numbers. 
The time-local structure functions in \eqref{eq:structuresLocal}, excluding the outer \(1/p\)-root in \eqref{eq:structures}, are approximated as described in Appendix~\ref{sec:appendix-struct} and computed from a second set of simulation runs (Set~2, summarized in Table~\ref{tab:params2}, but with enforcing \(M=N\)). 
\begin{table}[ht!]
    \caption{Set~2 of MC LBM discretization parameters for RTGV flow simulations.}
    \label{tab:params2}
    \centering
    \small
    \begin{tabular}{r r l l r l}
    \hline
        \makecell[c]{$M$} & \makecell[c]{$N$} & \makecell[c]{$\triangle t$} & \makecell[c]{$M\!a$} & \makecell[c]{$R\!e$} & \makecell[c]{$\tau_{\mathrm{rel}}$} \\
        \hline
         $1000$ &   $8$ & $1.03 \times 10^{-1}$ & $0.2$    &  $320$ & $0.501206$\\
         $1000$ &  $16$ & $2.41 \times 10^{-2}$ & $0.1$    &  $640$ & $0.500646$ \\
         $1000$ &  $32$ & $5.67 \times 10^{-3}$ & $0.05$   & $1280$ & $0.500345$ \\
         $1000$ &  $64$ & $1.42 \times 10^{-3}$ & $0.025$  & $2560$ & $0.500172$ \\
         $1000$ & $128$ & $3.54 \times 10^{-4}$ & $0.0125$ & $5120$ & $0.500086$ \\
         $1000$ & $256$ & $8.86 \times 10^{-5}$ & $0.00625$ & $10240$ & $0.500043$ \\
         $1000$ & $512$ & $2.21 \times 10^{-5}$ & $0.003125$ & $20480$ & $0.500022$ \\
      \hline
    \end{tabular}
\end{table}

To assess the development of the inertial cascade, the second-order structure functions are likewise transformed into dimensionless coordinates. 
Note that, following Appendix~\ref{sec:appendix-struct}, we evaluate the trace variant of \eqref{eq:structuresLocal} based on the full velocity increment. 
The separation distance $r$ is normalized by the Kolmogorov length scale \eqref{eq:kolmogorov_microscale}, which yields the inner-scaled separation distance $r/\eta(t)$. 
Furthermore, to visualize the K41 inertial-range scaling, the structure function is compensated sample-wise prior to computing the mean and standard deviation, in analogy to the energy spectra. 
In the inertial subrange ($\eta(t) \ll r \ll L$), K41 theory predicts that the trace of the second-order structure function scales as $S^2(r,t) \approx C_{\mathrm{trace}} (\E[\mathfrak{d}(t)] r)^{2/3}$, which corresponds to the scaling \eqref{eq:energyScaling}. 
Dividing the raw structure function of the \(m\)th sample by this theoretical scaling factor yields the compensated structure function
\begin{equation}\label{eq:compensatedSF}
    S^{2,\mathrm{c}}_{m}(r/\eta(t), t) = \frac{S^2_{m}(r,t)}{(\mathfrak{d}_{m}(t) r)^{2/3}} .
\end{equation}
Under this compensation, any spatial domain that strictly obeys the K41 $r^{2/3}$ scaling will manifest as a horizontal plateau. 
For the full 3D trace of the velocity correlation tensor, this universal plateau corresponds to the constant $C_{\mathrm{trace}} \approx 7.7$.
In addition to the ensemble mean $\E[S^{2,\mathrm{c}}(r/\eta(t), t)]$, the standard deviation of the compensated structure function, $\sigma[S^{2,\mathrm{c}}(r/\eta(t), t)]$, is computed across the statistical samples. 
Both quantities are visualized in Figure~\ref{fig:SF_2D}. 
Note that the evaluation domain is strictly truncated at a physical distance of $r = \pi$. 

As the Reynolds number increases and the inertial plateau begins to form, the standard deviation broadens, indicating that a wider range of turbulent scales is resolved below the truncation at $r = \pi$. 
Hence, Figure~\ref{fig:SF_2D} shows that the local structure functions of second order with \(S_{r,t}^{2} (\mu_{t}^{\vis,N,M})\) approximately obey a \(2/3\)-power law, which supports the scaling assumption \eqref{eq:energyScaling} underlying Assumption~\ref{ass:scaling}.
\begin{figure}[ht!]
	\centering
    \includegraphics[scale=1]{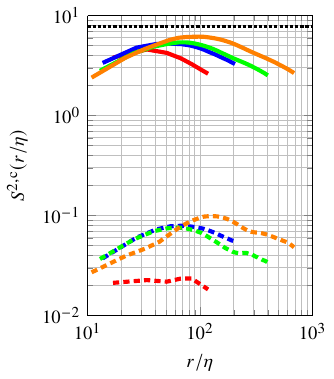}
	\quad
    \includegraphics[scale=1]{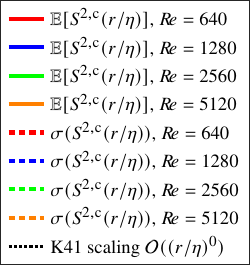}
   	\caption{Compensated structure function approximations (Algorithm~\ref{alg:SF_spectral}) of statistical solutions (mean and standard deviation) at time \(t\approx 9\mathrm{s}\) with increasing Reynolds numbers computed with MC LBM for discretization parameters from Set~2 (see Table~\ref{tab:params2}). 
    }
    \label{fig:SF_2D}
\end{figure}
Figure~\ref{fig:structureWaterfalls} visualizes the evolution of the compensated structure functions over time. 
\begin{figure}[ht!]
    \centering 
    \subfloat[Mean, \(R\!e=1280\)]{
    \includegraphics[scale=1]{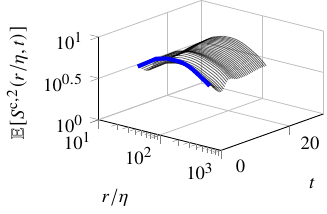}}
    \quad
    \subfloat[Std, \(R\!e=1280\)]{
    \includegraphics[scale=1]{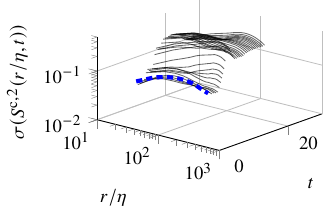}}
    \\ 
    \subfloat[Mean, \(R\!e=2560\)]{
    \includegraphics[scale=1]{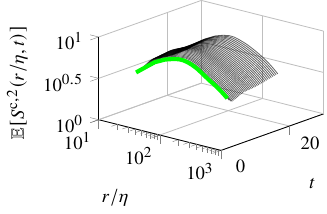}}
    \quad
    \subfloat[Std, \(R\!e=2560\)]{
    \includegraphics[scale=1]{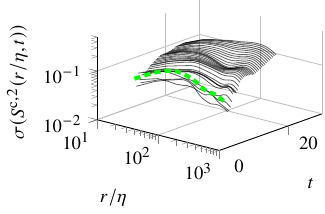}}
    \\
    \subfloat[Mean, \(R\!e=5120\)]{
    \includegraphics[scale=1]{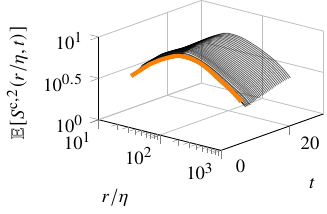}}
    \quad
    \subfloat[Std, \(R\!e=5120\)]{
    \includegraphics[scale=1]{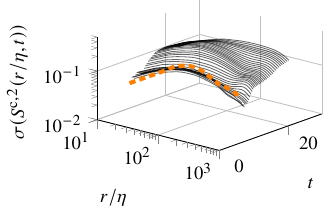}}
    \caption{Compensated structure function \eqref{eq:compensatedSF} approximations of statistical solutions (mean \(\E\) and standard deviation (std) \(\sigma\)) over time \(t\in [9,30]\) with increasing Reynolds numbers computed with MC LBM for discretization parameters from Set~2 (see Table~\ref{tab:params2}). 
    Colors refer to the respective data lines in Figure~\ref{fig:SF_2D}.
  }
    \label{fig:structureWaterfalls}
\end{figure}

\subsection{Computed sample convergence}\label{sec:sampleDivergence}

To quantify the failure of classical strong convergence in the chaotic regime, we evaluate the ensemble-averaged relative $L^{1}$ error between individual MC samples
\begin{align}
\label{eq:sample_error}
    \bar{\mathrm{e}}_{\mathrm{strong}}^{L^{1}, M} (\bm{u}^{N}(t),\bm{u}^{N_{\mathrm{ref}}}(t)) = \frac{1}{M} \sum_{m=1}^{M} \frac{ \sum_{k=1}^{G_{\mathrm{ds}}} \left\vert \bm{u}^{N}_{m}(\bm{x}_{\mathrm{ds},k},t) - \bm{u}^{N_{\mathrm{ref}}}_{m}(\bm{x}_{\mathrm{ds},k}, t) \right\vert }{ \sum_{k=1}^{G_{\mathrm{ds}}} \left\vert \bm{u}^{N_{\mathrm{ref}}}_{m}(\bm{x}_{\mathrm{ds},k},t) \right\vert } 
\end{align}
at varying resolutions $N \in \{32, 64, 128, 256\}$ against high-resolution references at $N_{\mathrm{ref}} \in \{256, 512\}$ for sample sizes $M \in \{100, 1000\}$. 
As defined in \eqref{eq:sample_error}, this sample divergence check compares the macroscopic velocity fields pointwise on a common downsampled grid $G_{\mathrm{ds}} = \vert \widehat{\Omega}_{\mathrm{ds}} \vert = N_{\mathrm{ds}}^{d} = 8^{3}$. 
As depicted in Figure \ref{fig:eosc_convergence_plots}, the temporal evolution of the experimental order of sample convergence (EOSC) mirrors the physical lifecycle of the unforced RTGV under the diagonal inviscid scaling ($\vis \sim \eps$). 
At the initial time $t=0$, the macroscopic flow field is smooth, laminar, and fully deterministic. 
The KBC LBM is formally up to second-order accurate in space, $\mathcal{O}(\triangle x^2)$, and resolves the initial macroscopic scales. 
The nonlinearity of the flow then amplifies grid-dependent truncation errors over time. 
\begin{figure}[ht!]
    \centering
    % --- Top-Left Plot: M=100, N_ref=256 ---
    \subfloat[$M=100$, $N_{\text{ref}}=256$ \label{fig:eosc_m100_256}]{
        \includegraphics[scale=1]{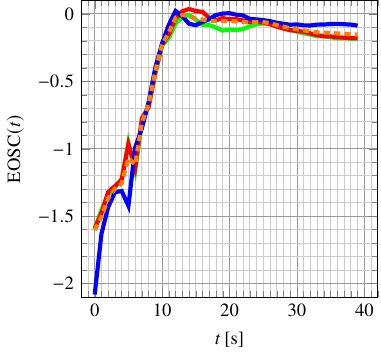}
    }
    \quad
    % --- Top-Right Plot: M=1000, N_ref=256 ---
    \subfloat[$M=1000$, $N_{\text{ref}}=256$ \label{fig:eosc_m1000_256}]{
        \includegraphics[scale=1]{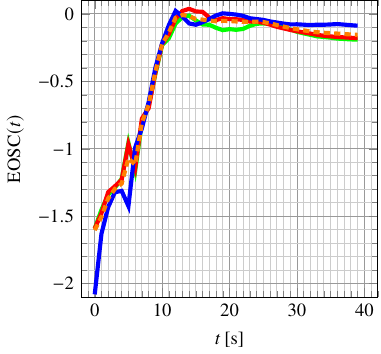}
    }\\
    \centering    
    % --- Bottom-Left Plot: M=100, N_ref=512 ---
    \subfloat[$M=100$, $N_{\text{ref}}=512$ \label{fig:eosc_m100_512}]{
		\includegraphics[scale=1]{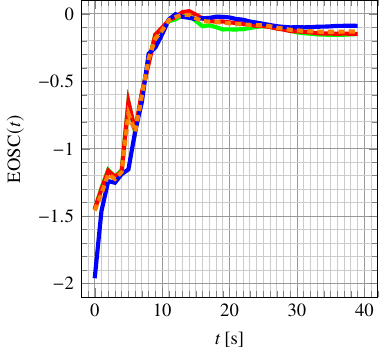}
    }
    \quad
    % --- Bottom-Right Plot: M=1000, N_ref=512 ---
    \subfloat[$M=1000$, $N_{\text{ref}}=512$ \label{fig:eosc_m1000_512}]{
		\includegraphics[scale=1]{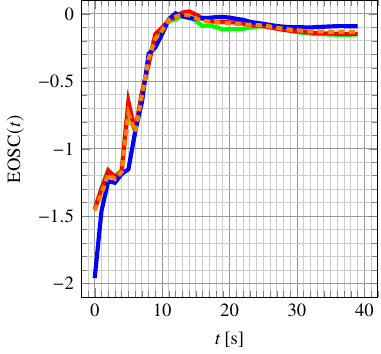}
    }\\
    \vspace{0.5em}
    \centering
	\includegraphics[scale=1]{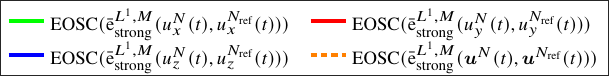}
    \caption{Experimental order of sample convergence (EOSC) for sample sizes $M\in\{ 100, 1000\}$ and for resolutions \(N\geq 32\) from Table~\ref{tab:params2} with respect to $N_{\mathrm{ref}} = 256$ (top row) and $N_{\mathrm{ref}} = 512$ (bottom row). The error $\bar{\mathrm{e}}_{\mathrm{strong}}^{L^{1}, M}$ (see \eqref{eq:sample_error}) is plotted over time for each velocity component and the velocity vector. 
    }  
    \label{fig:eosc_convergence_plots}
\end{figure}
As the flow generates small-scale structures approaching the peak turbulent dissipation ($t \gtrsim 10$), the deterministic trajectories of the individual samples separate. 
Individual realizations lose correlation with one another, pathwise convergence fails, and the relative error \eqref{eq:sample_error} saturates at $\mathcal{O}(1)$ as the turbulence develops, i.e., the EOSC approaches zero. 
This is the discrete counterpart of the loss of strong spatial compactness in the inviscid limit discussed in Section~\ref{subsec:diagonalLimit}. 
At later times ($t > 15$), as the turbulent kinetic energy depletes and the flow enters a viscous decay phase, we observe a slight, artificial recovery of a fractional convergence rate of up to around \(\mathcal{O}(N^{-0.05})\) (drifting toward $\mathcal{O}(N^{-0.15})$ from \(t>25\) onward). 
This late-stage fractional rate is an artifact of using a finite-resolution reference solution (here, $N_{\mathrm{ref}} \in \{256, 512\}$) to compute the strong errors. 
Because the reference solution retains a finite numerical viscosity, it acts as a spatial mollifier during the laminarizing decay phase. 
Thus, the error metric no longer measures the non-convergence of the samples, but the fractional spatial regularity of the smoothed reference field.

\subsection{Computed Wasserstein convergence}\label{subsec:numWasserstein}

The observed loss of strong convergence at the turbulent peak (cf.\ Section~\ref{sec:sampleDivergence}) motivates the transition to statistical solutions.
To complement Proposition~\ref{prop:weak-strong_discrete} and to determine the rates entering Hypothesis~\ref{hyp:WassersteinRate}, we measure the experimental order of Wasserstein convergence (EOWC) of the computed statistical solutions. 
We evaluate the Wasserstein distances between high-resolution reference ensembles (\(N_{\mathrm{ref}}\in\{256,512\}\)) and coarser approximations (\(N\in\{32,64,128,256\}\)). 
Details are provided in Appendix~\ref{appsec:compWasserstein}. 
To isolate the spatial truncation error from the statistical sampling error, we hold the sample size constant across all resolutions, using \(M\in\{100,1000\}\), and project the fields onto a common downsampled evaluation grid \(I_{Q}\) corresponding to the coarsest resolution in Table~\ref{tab:params2}, i.e., \(Q=8\).
We evaluate the time-local \(1\)-Wasserstein distance approximation for several time steps in increments of \(0.1\mathrm{s}\). 
In Figure~\ref{fig:eowc_256}, the component-wise computation for 1-point (\(W_{1,1}\)) and 2-point correlations (\(W_{1,2}\)), both \eqref{eq:W1approx}, is compared to the vector-valued computation \eqref{eq:W1_vector} for 2-point correlations (\(W^{\mathrm{v}}_{1,2}\)).
\begin{figure}[ht!]
    \centering
    % --- Top Left Plot: M=100, N_ref=256 ---
    \subfloat[$M=100$, $\vis_{\mathrm{ref}}=0.0001$, $N_{\mathrm{ref}}=256$ \label{fig:eowc_m100_256}]{
    	\includegraphics[scale=1]{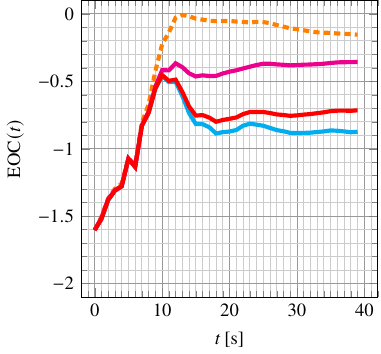}
    }
    \quad
    % --- Top Right Plot: M=1000, N_ref=256 ---
    \subfloat[$M=1000$, $\vis_{\mathrm{ref}}=0.0001$, $N_{\mathrm{ref}}=256$\label{fig:eowc_m1000_256}]{
     	\includegraphics[scale=1]{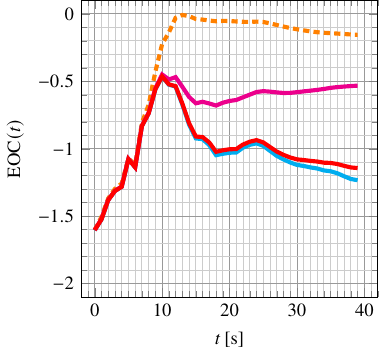}
    }\\    
 % --- Bottom Left Plot: M=100, Nref=512---
    \subfloat[$M=100$, $\vis_{\mathrm{ref}}=0.00005$, $N_{\mathrm{ref}}=512$ \label{fig:eowc_m100_512}]{
       	\includegraphics[scale=1]{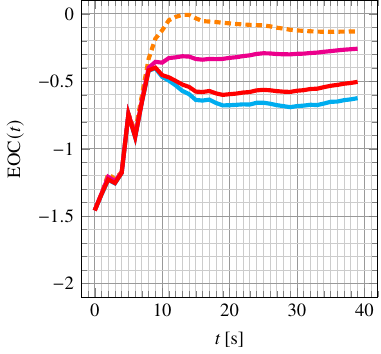}
    }
    \quad
    % --- Bottom Right Plot: M=1000, Nref=512 ---
    \subfloat[$M=1000$, $\vis_{\mathrm{ref}}=0.00005$, $N_{\mathrm{ref}}=512$ \label{fig:eowc_m1000_512}]{
        \includegraphics[scale=1]{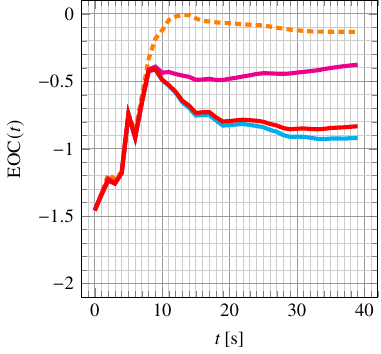}
    }\\  
    \vspace{0.5em}
    \centering
    \includegraphics[scale=1]{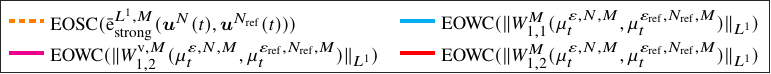}
    \caption{
        Experimental order of Wasserstein convergence (EOWC) plotted over time for \(W_{1,1}\), \(W_{1,2}\) (both \eqref{eq:W1approx}), and \(W_{1,2}^{\mathrm{v}}\) \eqref{eq:W1_vector}, respectively for $M\in\{100, 1000 \}$ samples (column-wise) and for resolutions \(N\geq 32\) from Table~\ref{tab:params2} with respect to $N_{\mathrm{ref}}=256$ (top row) and $N_{\mathrm{ref}}=512$ (bottom row). 
        EOSC of the complete velocity vector from Figures~\ref{fig:eosc_m100_256},~\ref{fig:eosc_m1000_256},~\ref{fig:eosc_m100_512}, and~\ref{fig:eosc_m1000_512}, respectively (top left to bottom right) are included for comparison.
    }
    \label{fig:eowc_256}
\end{figure}

Since the optimal coupling $\pi_{\bm{x}}$ (see \eqref{eq:wassersteinPermutations}) matches realizations across the two ensembles, $W_1$ measures the convergence of the distribution rather than of individual samples. 
As shown in Figure \ref{fig:eowc_256}, under the diagonal inviscid scaling, we observe rates $W_1 \leq C \eps^\alpha$ with $\alpha \approx 0.4$--$0.5$. 
In particular, for all the tested \(W_{1,1}\), \(W_{1,2}\), and \(W_{1,2}^{\mathrm{v}}\), the magnitude of the EOWC is larger than that of the EOSC over the computed time horizon, and the difference increases with the number of samples. 
These are the rates that motivate the choice $\zeta \approx s$ in Hypothesis~\ref{hyp:WassersteinRate}.

Besides, we observe a persistent difference between \(W_{1,2}^{\mathrm{v}}\) (vector-valued) and \(W_{1,2}\) (component-wise). 
After the dissipation peak region at \(t\approx10\), the component-wise EOWC increases back to a linear convergence rate, whereas the vector-valued EOWC stays around \(-0.5\), i.e., \(\alpha \approx 0.5\).
The difference between these two notions of Wasserstein distances is abstractly quantified in~\cite{catalano2025measuresdependencebasedwasserstein} and does not vanish in general. 
As we approximate a component-wise coupled probabilistic vector field (cf.\ \eqref{eq:incNSE}), the choice of a vector-valued Wasserstein metric seems more appropriate, since it retains the dependence between the velocity components. 

Moreover, it is evident that while the 1-point and component-wise metrics maintain a relatively stable fractional convergence, the EOWC of the $\mathbb{R}^6$-based vector metric noticeably degrades after \(t\approx 17\). 
This is a direct manifestation of the curse of dimensionality in empirical optimal transport. 
Because the statistical sampling error of MC empirical measures in the $1$-Wasserstein distance scales as $\mathcal{O}(M^{-1/d})$ for phase-space dimension $d > 2$, the six-dimensional $W_{1,2}^{\mathrm{v}}$ metric possesses a substantially higher statistical noise floor than its one- and two-dimensional counterparts $W_{1,1}$ and $W_{1,2}$.
As the macroscopic flow smoothens during the decay phase and the spatial discretization errors shrink, this dimensional statistical noise begins to dominate the total error. 
Refining the grid then reduces the empirical distance less and less, and the measured rate flattens.

In general, the EOWC, denoted as $\alpha(t)$, seems to be correlated in time with the ensemble-averaged kinetic energy dissipation rate $\E[\mathfrak{d}(t)]$. 
In fact, the EOWC mirrors the dissipation curve inversely: the convergence rate drops to its minimum during the peak of the turbulent cascade and recovers during the viscous decay phase.
We interpret our results as follows. 
In the Kuznetsov-type argument of Section~\ref{sec:formalWrate}, the convergence exponent $\alpha = \frac{s}{s+\zeta}$ depends on the regularity index $s$ of the velocity field and on the exponent $\zeta$. 
According to K41 theory, the flow exhibits a regularity of $s = 1/3$ strictly within the inertial range. 
In the diagonal EOC test, where the viscosity is scaled with the grid resolution ($\vis \sim \eps$), the numerical scheme approximates the inviscid limit. 
However, the effective physical regularity $s_{\mathrm{eff}}(t)$ captured by the Wasserstein metric remains a dynamic, time-dependent variable governed by the ratio of the grid spacing $\eps$ to the Kolmogorov microscale \eqref{eq:kolmogorov_microscale}.
The local scaling of the second-order structure function determines the effective regularity $s_{\mathrm{eff}}(t)$ observed at the grid scale $\eps$:
\begin{align} \label{eq:structure_function_regimes}
    S^{2}(\eps, t) \sim 
    \begin{cases} 
        C_1 \, \E[\mathfrak{d}(t)]^{2/3} \eps^{2/3} & \text{if } \eps \gg \eta(t) \quad (\text{inertial range, } s_{\mathrm{eff}} \to 1/3), \\
        C_2 \, \E[\mathfrak{d}(t)] \vis^{-1} \eps^2 & \text{if } \eps \ll \eta(t) \quad (\text{viscous range, } s_{\mathrm{eff}} \to 1).
    \end{cases}
\end{align}
During the peak of the turbulent cascade, $\E[\mathfrak{d}(t)]$ reaches its global maximum, causing $\eta(t)$ to shrink to its minimum. 
Although $\vis$ decreases along the diagonal, the high dissipation keeps the grid spacing large compared to the microscale ($\eps \gg \eta(t)$), so the grid samples the $2/3$-power law of the inertial range. 
This forces $s_{\mathrm{eff}}(t) \approx 1/3$ and reduces $\alpha$ to its fractional minimum.
Conversely, as the flow enters the decay phase ($t > 15$), $\E[\mathfrak{d}(t)] \to 0$ and $\eta(t)$ expands rapidly. 
Once the expanding microscale outpaces the fixed grid spacing ($\eps \ll \eta(t)$), the grid no longer resolves the rough increments but instead evaluates the smooth, relaminarizing viscous subrange. 
The effective regularity shifts toward $s_{\mathrm{eff}}(t) \approx 1$, and the EOWC increases toward integer-order rates.

\subsection{Computed robustness with respect to the numerical scheme}\label{sec:numWassersteinRef}

In addition to the consistency result in Section~\ref{subsec:numWasserstein}, we compute the Wasserstein convergence toward a reference solution obtained with an entirely different solver. 
The spectral hyperviscosity solver proposed by Rohner and Mishra~\cite{rohner2024efficient} (azeban) is used to produce \(M=992\) samples with an identical set of IID initial perturbations (see Appendix~\ref{appsec:refSol}). 
We compare against two ensembles, consisting of the first \(M=100\) samples and of all \( M=992\) samples. 
In contrast to the MC KBC LBM data, azeban includes a Leray projection and removes molecular viscosity to directly approximate a statistical Euler solution. 
Based on this difference in solution approximations, we investigate both the accuracy of the MC KBC LBM with respect to a reference statistical Euler solution and whether the computed limit depends on the regularization mechanism. 
A corresponding comparison between two solvers in two dimensions was reported in \cite[Section 5.5]{lanthaler2021statistical}. 
Details on the hyperviscosity method are provided in Appendix~\ref{appsec:refSol}. 
The sample convergence slopes and the Wasserstein convergence results with respect to this reference are summarized in Figure~\ref{fig:eowc_m128_eulerRef}. 
\begin{figure}[ht!]
    \centering
    % --- Top Left Plot: M=100, N_max=128 ---
    \subfloat[$M=100$, $N_{\mathrm{max}}=128$ \label{fig:eowc_m100_128_euler}]{
    	\includegraphics[scale=1]{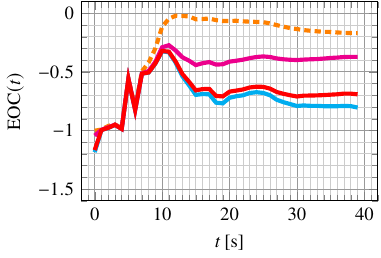}
    }
    \quad
    % --- Top Right Plot: M=992, N_max=128 ---
    \subfloat[$M=992$, $N_{\mathrm{max}}=128$\label{fig:eowc_m992_128_euler}]{
        \includegraphics[scale=1]{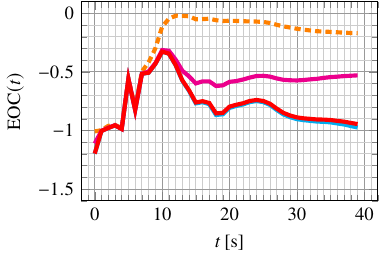}
    }\\    
 % --- Bottom Left Plot: M=100, Nmax=256---
    \subfloat[$M=100$, $N_{\mathrm{max}}=256$ \label{fig:eowc_m100_256_euler}]{
        \includegraphics[scale=1]{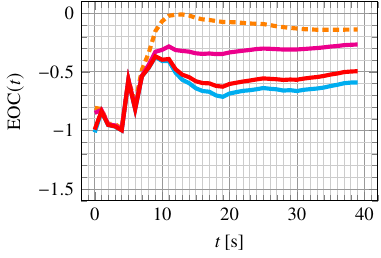}
    }
    \quad
    % --- Bottom Right Plot: M=992, Nmax=256 ---
    \subfloat[$M=992$, $N_{\mathrm{max}}=256$ \label{fig:eowc_m992_256_euler}]{
        \includegraphics[scale=1]{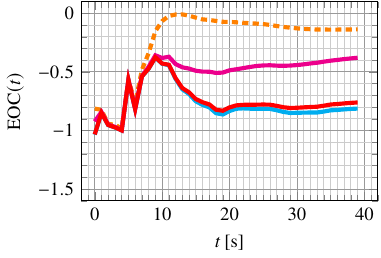}
    }\\  
    \vspace{0.5em}
    \centering
    \includegraphics[scale=1]{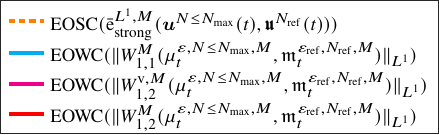}
 	\caption{
        Experimental order of Wasserstein convergence (EOWC) plotted over time for \(W_{1,1}\), \(W_{1,2}\) (both \eqref{eq:W1approx}), and \(W_{1,2}^{\mathrm{v}}\) \eqref{eq:W1_vector}, respectively for $M\in\{100, 992 \}$ samples (column-wise) and for resolutions $32\leq N\leq N_{\mathrm{max}}=128$ (top row) and $32 \leq N \leq N_{\mathrm{max}}=256$ (bottom row) from Table~\ref{tab:params2}. 
        EOSC slopes of the complete velocity vector are included for comparison.
        All error slopes are computed with respect to the spectral hyperviscosity results (see Appendix~\ref{appsec:refSol}), where \(\bm{\mathfrak{u}}^{N_{\mathrm{ref}}}\) denotes the velocity vector, and \(\mathfrak{m}_{t}^{\vis_{\mathrm{ref}},N_{\mathrm{ref}},M}\) is the computed statistical solution with \(N_{\mathrm{ref}}=256\) and \(\vis_{\mathrm{ref}} = 0.01/N_{\mathrm{ref}}^{2}\).
    }
    \label{fig:eowc_m128_eulerRef}
\end{figure}
Analogously to Section~\ref{sec:sampleDivergence}, sample convergence also fails with respect to the spectral hyperviscosity Euler approximation, indicating that the pathwise divergence is not an artifact of a particular discretization. 

Further, the shape of the EOWC rates over time in Figure~\ref{fig:eowc_m128_eulerRef} (with respect to the spectral hyperviscosity Euler approximation) and the consistency-based EOWC rates in Figure~\ref{fig:eowc_256} appear to be similar in terms of bends and overall magnitude, although the magnitude is slightly lower in the cross-solver test.

Besides, the convergence rates at early times are affected by the missing initial Leray projection in the MC KBC LBM. 
The omitted projection leaves a remainder $(\mathbf{I}_{d}-\mathbf{P})\bm{\mathfrak{s}}$ of size $\mathcal{O}(\mathfrak{w})$, independent of $N$, which is not covered by the well-preparedness of Section~\ref{sec:approximation}: Lemma~\ref{lem:wellprepared} assumes projected data, and the acoustic transient emitted by the equilibrium initialization \eqref{eq:initialPopulations} produces density fluctuations of order $\mathcal{O}(M\!a_{N}\mathfrak{w})$ rather than the $\mathcal{O}(M\!a_{N}^{2})$ required by Assumption~\ref{ass:lowMach}(ii) at early times. 
We expect its influence on the solenoidal dynamics to be small since the dilatational component carries only an $\mathcal{O}(\mathfrak{w})$ fraction of the velocity and is damped by the scheme. 
The visible imprint of the omitted projection is the initial plateau in Figure~\ref{fig:eowc_m128_eulerRef} at $t=0$. 
At this timestep, the distance to the Leray-projected reference equals the remainder itself, i.e., an $N$-independent random offset of magnitude $\mathcal{O}(\mathfrak{w})$, which explains the reduced convergence rate at initial time. 
For $t>0$, this offset no longer dominates the measured distances, and the rates approach those of Figure~\ref{fig:eowc_256}. 

We recall that a proof of uniqueness of statistical solutions to the 3D NSE and EE is not available. 
However, our comparison with the reference solution provides evidence on whether the computed limit depends on the regularization. 
Comparable convergence behavior in the Wasserstein metric between the MC KBC LBM and the spectral hyperviscosity method indicates that the two computed sequences approach a common limit in \(\mathcal{P}_{1}\). 
Since the two schemes employ different regularization mechanisms (mesoscopic entropic moment relaxation on the one hand and explicit macroscopic higher-order differential operators on the other), the agreement of the measured rates suggests that, within the range of resolutions considered here, the computed measure $\mu_t$ does not depend on the regularization. 
This is consistent with the two-dimensional observations in \cite[Section 5.5]{lanthaler2021statistical}.
\section{Conclusion}

\label{sec:conclusion}

We have combined a single level Monte Carlo method with an entropic lattice Boltzmann solver to approximate statistical solutions of the incompressible Navier--Stokes equations in three dimensions. 
The entropy-controlled, space-time dependent relaxation of the kinetic moments keeps the computations stable as the viscosity decreases. 
The randomness is carried by the initial data, and each sample is evolved with the KBC LBM. 
The implementation uses OpenLB-UQ and was run on the heterogeneous systems listed in Appendix~\ref{sec:appendix-resources}. 

Beyond the computations, we provide a conditional convergence analysis of the proposed scheme. 
Since the stability and consistency properties of the fully discrete entropic dynamics are not available as theorems in three dimensions, they are isolated in four standing assumptions, namely a uniform low Mach regime (Assumption~\ref{ass:lowMach}), coercivity of the discrete entropy production on the shear moments (Assumption~\ref{ass:coercivity}), weak consistency of the exact discrete moment balances (Assumption~\ref{ass:consistency}), and a uniform bound on the second-order structure functions (Assumption~\ref{ass:scaling}). 
Under the first three, the laws of the KBC LBM ensemble converge at fixed viscosity along a subsequence to a limit satisfying the Foias--Temam Liouville formulation of the NSE (Theorem~\ref{thm:nse_limit}), and adding the fourth, the vanishing viscosity theorem of \cite{fjordholm2021vanishing} applies to the limit families and yields the FMW multi-point statistical Euler hierarchy (Corollary~\ref{thm:eulerConsistency}). 
The diagonal scaling $\vis_{\eps} = c \eps$ underlying our computations yields the first equation of the same hierarchy without the coercivity assumption (Proposition~\ref{thm:diagonal_limit}), because no control of the discrete velocity gradients is required along that path.
Independently of these assumptions, the equilibrium initialization is shown to be well-prepared in the entropic sense (Lemma~\ref{lem:wellprepared}), the empirical measures converge for $M \to \infty$ (Lemma~\ref{lem:empirical}), and, provided they are admissible, the limit measures inherit the weak-strong uniqueness principle of \cite{brenier2011weak,lanthaler2021statistical} together with a stability estimate for as long as a strong Euler solution exists (Proposition~\ref{thm:weakStrong}). 
The structure-function scaling postulated in Assumption~\ref{ass:scaling} is precisely the quantity examined numerically in Section~\ref{sec:numerics}, so that the analytical and computational parts of this work carry complementary halves of the same statement.

The computations follow the inviscid limit of statistical Navier--Stokes solutions in three dimensions toward the statistical Euler limit for a Cauchy problem with randomized Taylor--Green vortex data. 
Preliminary runs with the same algorithm were reported in~\cite{simonis2023lattice,simonis2024computing}, and the present work adds extensive convergence measurements in the Wasserstein metric, the structure functions, and the cross-solver comparison. 
Whenever a strong Euler solution exists, the limit is singled out by Proposition~\ref{thm:weakStrong}. 
The computed energy spectra and structure functions show the K41 scaling over a range of scales that widens with the Reynolds number. 

Our numerical experiments show that deterministic pointwise tracking is not attainable in the chaotic regime. 
The strong $L^1$ error between individual sample trajectories saturates at $\mathcal{O}(1)$, as the nonlinearities of the turbulent flow amplify the underlying spatial truncation errors. 
This divergence motivates our transition to a statistical viewpoint. 
By evaluating the localized $1$-Wasserstein distance between finite-dimensional correlation marginals, we measure fractional convergence rates of $\mathcal{O}(\eps^\alpha)$ with $\alpha \approx 0.4$--$0.5$. 
Within the Kuznetsov-type argument of Section~\ref{sec:formalWrate}, these rates correspond to an amplification exponent $\zeta \approx s$ (Hypothesis~\ref{hyp:WassersteinRate}).
In addition, we compare the computed measures with a reference solution produced with the MC spectral hyperviscosity method of Rohner and Mishra~\cite{rohner2024efficient}. 
Although the two solvers rely on different regularization mechanisms, namely mesoscopic entropic moment relaxation on one hand and explicit macroscopic higher-order differential operators combined with a Leray projection on the other, the cross-solver comparison reproduces the behavior observed within the MC KBC LBM hierarchy itself. 
Sample convergence also fails with respect to the spectral hyperviscosity reference, indicating that the pathwise divergence is a property of the flow rather than an artifact of a particular discretization. 
Conversely, the experimental orders of Wasserstein convergence toward the spectral hyperviscosity reference are similar in shape and overall magnitude to the consistency-based rates measured against the high-resolution MC KBC LBM reference. 
This agreement suggests that, within the range of resolutions considered, the computed measure does not depend on the regularization mechanism.

In summary, the paper provides three-dimensional data on how statistical Navier--Stokes solutions approach a statistical Euler solution in the Wasserstein metric. 
As long as a strong Euler solution exists, that limit is unique. 
Promising future research includes establishing the standing assumptions of Section~\ref{sec:approximation} for the fully discrete entropic scheme, as well as the exploratory computation of statistical solutions to initial boundary value problems describing wall-bounded turbulent fluid flows with vanishing viscosity. 
The present results indicate that the combination of sampling techniques with efficient LBM implementations is a viable route to computing statistical solutions of turbulent flows on current HPC machinery.

\section*{Funding}

S.S.\ is supported by the PRIME programme of the German Academic Exchange Service (DAAD) with funding from the Federal Ministry of Research, Technology and Space (BMFTR). 
S.S.\ acknowledges financial support from a Networking grant and a ConYS grant by KHYS at KIT.
S.S.\ acknowledges support from the state of Baden-Württemberg through bwHPC. 
The authors gratefully acknowledge the computing time provided on the high-performance computer HoreKa by the National High-Performance Computing Center at KIT (NHR@KIT). 
This center is jointly supported with funding from the Federal Ministry of Research, Technology and Space and the Ministry of Science, Research and the Arts of Baden-Württemberg, as part of the National High-Performance Computing (NHR) joint funding program. 
HoreKa is partly funded by the German Research Foundation (DFG).
This work was supported by a grant from the Swiss National Supercomputing Centre (CSCS) under project ID 1217.

\section*{Acknowledgements}

The authors thank: M.J.\ Krause, M.\ Frank and S.\ Mishra for providing the excellent research infrastructure at KIT and ETH Zurich; S.\ Ito and A.\ Kummerländer for their specific contributions to OpenLB, upon which the numerical implementations were built; T.\ Braun for computing preliminary RTGV results; and M.\ Zhong for constructing the initial OpenLB-UQ module. 

\section*{Author contribution statement}

\textbf{J.L.G.:} 
    Methodology, 
    Software, 
    Validation, 
    Investigation, 
    Data Curation, 
    Writing - Original Draft, 
    Visualization; 
\textbf{T.R.:} 
    Conceptualization, 
    Methodology, 
    Software, 
    Validation, 
    Writing - Original Draft; 
\textbf{S.S.:} 
    Conceptualization, 
    Methodology, 
    Software, 
    Validation, 
    Formal Analysis, 
    Investigation, 
    Resources, 
    Data Curation, 
    Writing - Original Draft, 
    Writing - Review {\&} Editing, 
    Visualization, 
    Supervision, 
    Project administration, 
    Funding Acquisition. 
All authors read and approved the final version of this paper.

\section*{Data availability statement}

All MC LBM computations have been conducted using the OpenLB-UQ module~\cite{zhong2025openlbuquncertaintyquantificationframework}, which is released open source as part of OpenLB~\cite{krause2020openlb} under the GNU General Public License, version 2. 
Code contributions (also unreleased) are based on commit \href{https://gitlab.com/openlb/olb/-/commit/58a0c5d722648bfc0d0bfebb49656a7fab57c348}{58a0c5d7}. 
The statistical reference solution has been produced with the GPU-based spectral hyperviscosity solver azeban~\cite{rohner2024efficient}. 
Its source code is available under \href{https://github.com/TobiasRohner/azeban}{https://github.com/TobiasRohner/azeban}.
Visualizations of the flow fields have been performed using the open-source software ParaView~\cite{ParaView}. 
The statistical postprocessing was done using the open-source python modules NumPy~\cite{harris2020array}, the POT module \cite{flamary2021pot,flamary2024pot}, and SciPy \cite{virtanenSciPy10Fundamental2020}. 
Simulation data are available upon request.

\section*{Declaration of generative AI in the manuscript preparation process}

During the preparation of this work, the authors used Anthropic Claude and Google Gemini to assist with code development, data analysis, text drafting, and formatting. 
After using these tools, the authors reviewed and edited the content as needed and take full responsibility for the content of the publication.

\section*{Competing interests statement}

The authors declare that they have no known competing interests that could have appeared to influence the work reported in this paper.

{\printbibliography[resetnumbers]}

@book{babovsky1998boltzmann,
  title={{Die Boltzmann-Gleichung: Modellbildung-Numerik-Anwendungen}},
  author={Babovsky, Hans},
  year={1998},
  doi={10.1007/978-3-663-12034-6},
  place={Wiesbaden}, 
  publisher={Springer{\slash}Vieweg+Teubner}
}

@article{simon1986, 
	author={Simon, Jacques}, 
	year={1986}, 
	title={Compact sets in the space $L^p(0,T; B)$}, 
	journal={Annali di Matematica Pura ed Applicata},
	pages={65--96}, 
	volume={146}, 
	issue={1}, 
	doi={10.1007/BF01762360}, 
}

@book{Billingsley1999,
	author={Billingsley, Patrick},
	year={1999},
	title={Convergence of Probability Measures},
	edition={2nd},
	publisher={John Wiley \& Sons},
	doi={10.1002/9780470316962}
}

@book{Billingsley1995,
  title     = {Probability and Measure},
  author    = {Billingsley, Patrick},
  edition   = {3rd},
  year      = {1995},
  publisher = {John Wiley \& Sons},
  address   = {New York}
}

@book{Villani2003,
  title     = {{Topics in Optimal Transportation}},
  author    = {Villani, C{\'e}dric},
  year      = {2003},
  doi={10.1090/gsm/058},
  publisher = {American Mathematical Society},
  address   = {Providence, Rhode Island}
}

@article{Peyre2019,
  title     = {{Computational Optimal Transport: With Applications to Data Science}},
  author    = {Peyr{\'e}, Gabriel and Cuturi, Marco},
  journal   = {Foundations and Trends{\textregistered} in Machine Learning},
  volume    = {11},
  number    = {5-6},
  pages     = {355--607},
  year      = {2019},
  doi={10.1561/2200000073}
}

@article{Dafermos1979,
	author={Dafermos, Constantine M.},
	year={1979},
	title={The second law of thermodynamics and stability},
	journal={Archive for Rational Mechanics and Analysis},
	volume={70},
	number={2},
	pages={167--179},
	doi={10.1007/BF00250353}
}

@article{DiPerna1979,
	author={DiPerna, Ronald J.},
	year={1979},
	title={Uniqueness of solutions to hyperbolic conservation laws},
	journal={Indiana University Mathematics Journal},
	volume={28},
	number={1},
	pages={137--188},
	doi={10.1512/iumj.1979.28.28011}
}

@article{Kuznetsov1976,
	author={Kuznetsov, N. N.},
	year={1976},
	title={Accuracy of some approximate methods for computation of the weak solutions of a first-order quasi-linear equation},
	journal={USSR Computational Mathematics and Mathematical Physics},
	volume={16},
	number={6},
	pages={105--119},
	doi={10.1016/0041-5553(76)90159-4}
}

@article{zhong2025uncertaindataassimilationurban,
      title={{Uncertain data assimilation for urban wind flow simulations with OpenLB-UQ}}, 
      author={Zhong, M. and Teutscher, D. and Kummerländer, A. and Krause, M. J. and Frank, M. and Simonis, S.},
      year={2025},
      journal={arXiv preprint},
      doi={10.48550/arXiv.2508.18202}   
}

@article{zhong2025openlbuquncertaintyquantificationframework,
      title={{OpenLB-UQ: An Uncertainty Quantification Framework for Incompressible Fluid Flow Simulations}}, 
      author={Zhong, M. and Kummerländer, A. and Ito, S. and Krause, M. J. and Frank, M. and Simonis, S.},
      year={2025},
      journal={arXiv preprint}, 
      doi={10.48550/arXiv.2508.13867}  
}

@inproceedings{simonis2024computing,
    author = {Simonis, S. and Mishra, S.},
    title = {{Computing statistical Navier--Stokes solutions}},
    booktitle = {{Hyperbolic Balance Laws: Interplay between Scales and Randomness}},
    series = {Oberwolfach Report 21}, 
    year = {2024}, 
    number = {1}, 
    pages = {567-656}, 
    publisher = {EMS Press}, 
    doi={10.4171/OWR/2024/10}, 
    editor={R\'{e}mi Abgrall and Mauro Garavello and M\'{a}ria Luk\'{a}\v{c}ov\'{a}-Medvid{'}ov\'{a} and Konstantina Trivisa} 
}

@article{simonis2024spectral,
    author = {Simonis, Stephan and Dorschner, Benedikt and Karlin, Ilya V. and Krause, Mathias J.},
    title = {{Spectral effects of entropic multi-relaxation in lattice Boltzmann methods}},
    journal = {Preprint available at SSRN},
    doi = {10.2139/ssrn.4972419}, 
    year = {2024}
}

@article{leonardi2016numerical,
  title={{Numerical approximation of statistical solutions of planar, incompressible flows}},
  author={Leonardi, Filippo and Mishra, Siddhartha and Schwab, Christoph},
  journal={Mathematical Models and Methods in Applied Sciences},
  volume={26},
  number={13},
  doi={10.1142/S0218202516500597},
  pages={2471--2523},
  year={2016},
  publisher={World Scientific}
}

@article{bansal2021numerical,
  doi = {10.48550/arXiv.2107.06073},
  author = {Bansal, Pratyuksh},
  journal = {arXiv}, 
  volume = {preprint},
  title = {{Numerical approximation of statistical solutions of the incompressible Navier-Stokes Equations}},
  year = {2021}
}

@article{hunt1988eddies,
  title={Eddies, streams, and convergence zones in turbulent flows},
  author={Hunt, Julian CR and Wray, Alan A and Moin, Parviz},
  journal={Studying Turbulence Using Numerical Simulation Databases, 2. Proceedings of the 1988 Summer Program},
  year={1988}, 
  volume={N89},
  number={24555}, 
  pages={193--208},
  url={https://ntrs.nasa.gov/citations/19890015184}
}

@article{zhao2021lattice,
author = {Zhao, Weifeng and Huang, Juntao and Yong, Wen-An},
title = {{Lattice Boltzmann Method for Stochastic Convection-Diffusion Equations}},
journal = {SIAM{\slash}ASA Journal on Uncertainty Quantification},
volume = {9},
number = {2},
pages = {536--563},
year = {2021},
doi = {10.1137/19M1270665}
}

@article{zhao2019stochastic,
  title={{Stochastic modeling of the permeability of randomly generated porous media via the lattice Boltzmann method and probabilistic collocation method}},
  author={Zhao, Lei and Li, Heng},
  journal={Transport in Porous Media},
  volume={128},
  number={2},
  pages={613--631},
  year={2019},
  doi={10.1007/s11242-019-01261-7},
  publisher={Springer}
}

@article{vandenbos2017nonIntrusive,
title = {{Non-intrusive uncertainty quantification using reduced cubature rules}},
journal = {Journal of Computational Physics},
volume = {332},
pages = {418--445},
year = {2017},
doi = {10.1016/j.jcp.2016.12.011},
author = {L.M.M. {van den Bos} and B. Koren and R.P. Dwight}
}

@article{brachet1991direct,
  title = {Direct simulation of three-dimensional turbulence in the {Taylor}--{Green} vortex},
  author = {M. E. Brachet},
  journal = {Fluid Dynamics Research},
  volume = {8},
  number = {1--4},
  pages = {1--8},
  year = {1991},
  doi = {10.1016/0169-5983(91)90026-F},
  publisher={IOP Publishing}
}

@article{karlin1999perfect,
doi = {10.1209/epl/i1999-00370-1},
year = {1999},
volume = {47},
number = {2},
pages = {182},
author = {I. V. Karlin and  A. Ferrante and  H. C. Öttinger},
title = {{Perfect entropy functions of the Lattice Boltzmann method}},
journal = {Europhysics Letters}
}

@article{perthame1993weighted,
  title={{Weighted \(L^{\infty}\) bounds and uniqueness for the Boltzmann BGK model}},
  author={Perthame, B and Pulvirenti, M},
  journal={Archive for Rational Mechanics and Analysis},
  volume={125},
  pages={289--295},
  doi={10.1007/BF00383223},
  year={1993},
  publisher={Springer}
}

@article{mischler1996uniqueness,
  title={{Uniqueness for the BGK-equation in \(\mathbb{R}^{N}\) and rate of convergence for a semi-discrete scheme}},
  author={Mischler, St\'{e}phane},
  journal={Differential and Integral Equations},
  volume={9}, 
  number={5}, 
  doi={10.57262/die/1367871533},
  pages={1119--1138},
  year={1996}
}

@book{eden1995exponential,
author = {Eden, Alp and Foias, C. and Nicolaenko, B. and Temam, Roger},
title = {Exponential Attractors for Dissipative Evolution Equations},
series="Research in Applied Mathematics",
publisher="John Wiley {\&} Sons",
year="1994",
URL="https://cir.nii.ac.jp/crid/1570291224960836224"
}

@article{lallemand2021lattice,
  title={{The Lattice Boltzmann Method for Nearly Incompressible Flows}},
  author={Lallemand, Pierre and Luo, Li-Shi and Krafczyk, Manfred and Yong, Wen-An},
  journal={Journal of Computational Physics},
  volume={(In press)},
  pages={109713},
  year={2020},
  doi={10.1016/j.jcp.2020.109713},
  publisher={Elsevier}
}

@article{gorban2018hilbert,
author = {Gorban, A. N. },
title = {{Hilbert{\textquotesingle}s sixth problem: the endless road to rigour}},
journal = {Philosophical Transactions of the Royal Society A},
volume = {376},
number = {2118},
pages = {20170238},
year = {2018},
doi = {10.1098/rsta.2017.0238}
}

@article{perthame1989global,
  title={{Global existence to the BGK model of Boltzmann equation}},
  author={Perthame, B.},
  journal={Journal of Differential Equations},
  volume={82},
  number={1},
  pages={191--205},
  year={1989},
  publisher={Elsevier}, 
  doi={10.1016/0022-0396(89)90173-3}
}

@article{eden1994exponential,
author = {Eden, Alp and Foias, C. and Nicolaenko, B.},
year = {1994},
pages = {301-323},
title = {{Exponential attractors of optimal Lyapunov dimension for Navier-Stokes equations}},
volume = {6},
journal = {Journal of Dynamics and Differential Equations},
doi = {10.1007/BF02218532}
}

@article{eden1993exponential,
author = {Eden, Alp and Foias, C. and Nicolaenko, B. and She, Zhensu},
year = {1993},
pages = {350-360},
title = {Exponential attractors and their relevance to fluid dynamics systems},
volume = {63},
journal = {Physica D: Nonlinear Phenomena},
doi = {10.1016/0167-2789(93)90116-I}
}

@article{boesch2015entropic,
  title = {{Entropic multirelaxation lattice Boltzmann models for turbulent flows}},
  author = {B{\"o}sch, Fabian and Chikatamarla, Shyam S. and Karlin, Ilya V.},
  journal = {Physical Review E},
  volume = {92},
  issue = {4},
  pages = {043309},
  numpages = {15},
  year = {2015},
  publisher = {American Physical Society},
  doi = {10.1103/PhysRevE.92.043309}
}

@article{karlin2014gibbs,
  title = {{Gibbs' principle for the lattice-kinetic theory of fluid dynamics}},
  author = {Karlin, I. V. and B{\"o}sch, F. and Chikatamarla, S. S.},
  journal = {Physical Review E},
  volume = {90},
  issue = {3},
  pages = {031302},
  numpages = {5},
  year = {2014},
  publisher = {American Physical Society},
  doi = {10.1103/PhysRevE.90.031302}
}

@article{hopf1950anfangswertaufgabe,
  title={{{\"U}ber die Anfangswertaufgabe f{\"u}r die hydrodynamischen Grundgleichungen}},
  author={Hopf, Eberhard},
  journal={Mathematische Nachrichten},
  volume={4},
  number={1-6},
  pages={213--231},
  year={1950},
  doi={10.1002/mana.3210040121},
  publisher={Wiley Online Library}
}

@article{fefferman2000existence,
  title={{Existence and smoothness of the Navier--Stokes equation}},
  author={Fefferman, Charles L},
  journal={The millennium prize problems},
  volume={57},
  pages={67},
  url={https://www.claymath.org/wp-content/uploads/2022/06/navierstokes.pdf},
  year={2000}
}

@article{fjordholm2021vanishing,
    author = {Fjordholm, Ulrik Skre and Mishra, Siddhartha and Weber, Franziska},
    title = {{On the Vanishing Viscosity Limit of Statistical Solutions of the Incompressible Navier--Stokes Equations}},
    journal = {SIAM Journal on Mathematical Analysis},
    volume = {56},
    number = {4},
    pages = {5099--5143},
    year = {2024},
    doi = {10.1137/23M1566261}
}

@article{gu2024incompressible,
    doi = {10.1088/1361-6544/adca81},
    year = {2025},
    volume = {38},
    number = {5},
    pages = {055014},
    author = {Gu, Zhongyang and Hu, Xin and Matharu, Pritpal and Protas, Bartosz and Sasada, Makiko and Yoneda, Tsuyoshi},
    title = {{The incompressible Navier–Stokes limit from the discrete-velocity BGK Boltzmann equation}},
    journal = {Nonlinearity}
}

@incollection{eymard2000finite,
    title = {{Finite volume methods}},
    series = {{Handbook of Numerical Analysis}},
    publisher = {Elsevier},
    volume = {7},
    pages = {713-1018},
    year = {2000},
    booktitle = {{Solution of Equation in \(\mathbb{R}^{n}\) (Part 3), Techniques of Scientific Computing (Part 3)}},
    issn = {1570-8659},
    doi = {10.1016/S1570-8659(00)07005-8},
    author = {Robert Eymard and Thierry Gallouët and Raphaèle Herbin}
}

@inproceedings{rohner2024efficient,
    author = {Rohner, Tobias and Mishra, Siddhartha},
    title = {{Efficient Computation of Large-Scale Statistical Solutions to Incompressible Fluid Flows}},
    year = {2024},
    isbn = {9798400706394},
    publisher = {Association for Computing Machinery},
    address = {New York, NY, USA},
    doi = {10.1145/3659914.3659922},
    booktitle = {Proceedings of the Platform for Advanced Scientific Computing Conference},
    articleno = {8},
    pages = {1--11},
    location = {Zurich, Switzerland},
    series = {PASC '24}
}

@article{leray1934mouvement,
  title={Sur le mouvement d'un liquide visqueux emplissant l'espace},
  author={Leray, Jean},
  journal={Acta Mathematica},
  volume={63},
  number={1},
  pages={193--248},
  year={1934},
  doi={10.1007/BF02547354},
  publisher={Springer}
}

@article{simonis2022constructing,
title = {{Constructing relaxation systems for lattice Boltzmann methods}},
journal = {Applied Mathematics Letters},
volume = {137},
pages = {108484},
year = {2023},
doi = {10.1016/j.aml.2022.108484},
author = {Simonis, S.  and Frank, Martin and Krause, Mathias J.}  
}

@article{saint-raymond2003bgk,
     author = {Saint-Raymond, L.},
     title = {{From the {BGK} model to the {Navier-Stokes} equations}},
     journal = {Annales Scientifiques de l'\'{E}cole Normale Sup\'{e}rieure},
     pages = {271--317},
     publisher = {Elsevier},
     volume = {Ser. 4, 36},
     number = {2},
     year = {2003},
     doi = {10.1016/S0012-9593(03)00010-7},
}

@article{lohner2019towards,
  author = {L{\"{o}}hner, Rainald},
  doi = {10.1080/10618562.2019.1612052},
  journal = {International Journal of Computational Fluid Dynamics},
  number = {3},
  pages = {87--97},
  title = {{Towards overcoming the LES crisis}},
  volume = {33},
  year = {2019}
}

@article{dorschner2018particles,
  title = {{Particles on Demand for Kinetic Theory}},
  author = {Dorschner, B. and B\"osch, F. and Karlin, I. V.},
  journal = {Physical Review Letters},
  volume = {121},
  issue = {13},
  pages = {130602},
  numpages = {5},
  year = {2018},
  publisher = {American Physical Society},
  doi = {10.1103/PhysRevLett.121.130602},
}

@article{wilde2021high,
  title = {{High-order semi-Lagrangian kinetic scheme for compressible turbulence}},
  author = {Wilde, Dominik and Kr\"amer, Andreas and Reith, Dirk and Foysi, Holger},
  journal = {Physical Review E},
  volume = {104},
  issue = {2},
  pages = {025301},
  numpages = {15},
  year = {2021},
  publisher = {American Physical Society},
  doi = {10.1103/PhysRevE.104.025301}
}

@article{kallikounis2022particles,
  title = {{Particles on demand for flows with strong discontinuities}},
  author = {Kallikounis, N. G. and Dorschner, B. and Karlin, I. V.},
  journal = {Physical Review E},
  volume = {106},
  issue = {1},
  pages = {015301},
  numpages = {18},
  year = {2022},
  doi = {10.1103/PhysRevE.106.015301}
}

@article{simonis2020relaxation,
  title={{On relaxation systems and their relation to discrete velocity Boltzmann models for scalar advection--diffusion equations}},
  author={Simonis, S. and Frank, Martin and Krause, Mathias J},
  journal={Philosophical Transactions of the Royal Society A},
  volume={378},
  pages={20190400},
  doi={10.1098/rsta.2019.0400},
  year={2020}  
}

@article{simonis2021linear, 
  title={{Linear and brute force stability of orthogonal moment multiple-relaxation-time lattice Boltzmann methods applied to homogeneous isotropic turbulence}},
  author={Simonis, S. and Haussmann, Marc and Kronberg, Louis and D{\"o}rfler, Willy and Krause, Mathias J},
  journal={Philosophical Transactions of the Royal Society A},
  volume={379},
  pages={20200405},
  doi={10.1098/rsta.2020.0405},
  year={2021}  
}

@article{Skorokhod1956,
  title   = {{Limit theorems for stochastic processes}},
  author  = {Skorokhod, Anatoliy V.},
  journal = {Theory of Probability \& Its Applications},
  volume  = {1},
  number  = {3},
  pages   = {261--290},
  year    = {1956},
  doi = {10.1137/1101022},
  publisher={SIAM}
}

@article{simonis2022limit, 
title = {{Limit Consistency of Lattice Boltzmann Equations}}, 
author = {S. Simonis and M. J. Krause}, 
DOI= "10.1051/m2an/2025026", 
journal = {ESAIM: M2AN},
	year = 2025,
	volume = 59,
	number = 3,
	pages = "1271-1299",
}

@article{catalano2025measuresdependencebasedwasserstein,
      title={{Measures of Dependence based on Wasserstein distances}}, 
      author={Marta Catalano and Hugo Lavenant},
      year={2025},
      journal={arXiv}, 
      volume={preprint},
      doi={10.48550/arXiv.2510.06034}
}

@Article{harris2020array,
 title         = {Array programming with {NumPy}},
 author        = {Charles R. Harris and K. Jarrod Millman and St{\'{e}}fan J.
                 van der Walt and Ralf Gommers and Pauli Virtanen and David
                 Cournapeau and Eric Wieser and Julian Taylor and Sebastian
                 Berg and Nathaniel J. Smith and Robert Kern and Matti Picus
                 and Stephan Hoyer and Marten H. van Kerkwijk and Matthew
                 Brett and Allan Haldane and Jaime Fern{\'{a}}ndez del
                 R{\'{i}}o and Mark Wiebe and Pearu Peterson and Pierre
                 G{\'{e}}rard-Marchant and Kevin Sheppard and Tyler Reddy and
                 Warren Weckesser and Hameer Abbasi and Christoph Gohlke and
                 Travis E. Oliphant},
 year          = {2020},
 journal       = {Nature},
 volume        = {585},
 number        = {7825},
 pages         = {357--362},
 doi           = {10.1038/s41586-020-2649-2},
}

@article{zhong2024stochastic,
    author = {Zhong, Mingliang and Xiao, Tianbai and Krause, Mathias J. and Frank, Martin and Simonis, Stephan},
    title = {{A stochastic Galerkin lattice Boltzmann method for incompressible fluid flows with uncertainties}},
    journal = {Journal of Computational Physics},
    volume = {517}, 
    pages = {113344}, 
    doi = {10.1016/j.jcp.2024.113344},
    year = {2024}
}

@article{flamary2021pot,
  author  = {R{\'e}mi Flamary and Nicolas Courty and Alexandre Gramfort and Mokhtar Z. Alaya and Aur{\'e}lie Boisbunon and Stanislas Chambon and Laetitia Chapel and Adrien Corenflos and Kilian Fatras and Nemo Fournier and L{\'e}o Gautheron and Nathalie T.H. Gayraud and Hicham Janati and Alain Rakotomamonjy and Ievgen Redko and Antoine Rolet and Antony Schutz and Vivien Seguy and Danica J. Sutherland and Romain Tavenard and Alexander Tong and Titouan Vayer},
  title   = {{POT: Python Optimal Transport}},
  journal = {Journal of Machine Learning Research},
  year    = {2021},
  volume  = {22},
  number  = {78},
  pages   = {1-8},
  url     = {http://jmlr.org/papers/v22/20-451.html}
}

@misc{flamary2024pot,
  author = {Flamary, R{\'e}mi and Vincent-Cuaz, C{\'e}dric and Courty, Nicolas and Gramfort, Alexandre and Kachaiev, Oleksii and Quang Tran, Huy and David, Laurène and Bonet, Cl{\'e}ment and Cassereau, Nathan and Gnassounou, Th{\'e}o and Tanguy, Eloi and Delon, Julie and Collas, Antoine and Mazelet, Sonia and Chapel, Laetitia and Kerdoncuff, Tanguy and Yu, Xizheng and Feickert, Matthew and Krzakala, Paul and Liu, Tianlin and Fernandes Montesuma, Eduardo},
  title = {{POT Python Optimal Transport (version 0.9.5)}},
  url = {https://github.com/PythonOT/POT},
  year = {2024}
}

@phdthesis{simonis2023lattice,
    author = {Stephan Simonis},
    title = {{Lattice Boltzmann Methods for Partial Differential Equations}},
    school = {Karlsruhe Institute of Technology (KIT)},
    doi          = {10.5445/IR/1000161726},
    year = {2023}
}

@article{siodlaczek2021numerical,
  title={{Numerical evaluation of thermal comfort using a large eddy lattice Boltzmann method}},
  author={Siodlaczek, Marc and Gaedtke, Maximilian and Simonis, S. and Schweiker, Marcel and Homma, Naohiko and Krause, Mathias J},
  journal={Building and Environment},
  volume={192},
  pages={107618},
  year={2021},
  doi={10.1016/j.buildenv.2021.107618},
  publisher={Elsevier}  
}

@article{coreixas2020compressible,
author = {Coreixas,Christophe  and Latt,Jonas },
title = {{Compressible lattice Boltzmann methods with adaptive velocity stencils: An interpolation-free formulation}},
journal = {Physics of Fluids},
volume = {32},
number = {11},
pages = {116102},
year = {2020},
doi = {10.1063/5.0027986}
}

@InCollection{ParaView,
  author    = {Ahrens, James and Geveci, Berk and Law, Charles},
  title   = {{ParaView: An End-User Tool for Large Data Visualization}},
  editor    = {Hansen, Charles D. and Johnson, Christopher R.},
  pages     = {717--731},
  publisher = {Elsevier Inc.},
  booktitle     = {Visualization Handbook},
  year      = {2005},
  address   = {Burlington, MA, USA},
  url       = {https://www.sciencedirect.com/book/9780123875822/visualization-handbook},
}

@article{coreixas2019comprehensive,
  title = {Comprehensive comparison of collision models in the lattice {Boltzmann} framework: {Theoretical} investigations},
  author = {Coreixas, Christophe and Chopard, Bastien and Latt, Jonas},
  journal = {Physical Review E},
  volume = {100},
  issue = {3},
  pages = {033305},
  numpages = {46},
  year = {2019},
  publisher = {American Physical Society},
  doi = {10.1103/PhysRevE.100.033305}
}

@book{foias2001navier,
  title={{Navier-Stokes equations and turbulence}},
  author={Foias, Ciprian and Manley, Oscar and Rosa, Ricardo and Temam, Roger},
  volume={83},
  year={2001},
  doi={10.1017/CBO9780511546754},
  publisher={Cambridge University Press}
}

@article{foias1976solutions,
  title={{Sur les solutions statistiques des {\'e}quations de Navier-Stokes}},
  author={Foias, Ciprian and Prodi, G},
  journal={Annali di Matematica Pura ed Applicata},
  volume={111},
  number={1},
  pages={307--330},
  year={1976},
  doi={10.1007/BF02411822},
  publisher={Springer}
}

@article{kolmogoroff1941local,
  title={{The Local Structure of Turbulence in Incompressible Viscous Fluid for Very Large Reynolds Numbers}},
  author={Kolmogorov, A. N.},
  journal={Proceedings of the USSR Academy of Sciences},
  volume={30},
  pages={301--305},
  year={1941}
}

@article{kolmogorov1991local,
author = {Kolmogorov, Andrei Nikolaevich and Levin, V. and Hunt, Julian Charles Roland  and Phillips, Owen Martin  and Williams, David},
title = {{The local structure of turbulence in incompressible viscous fluid for very large Reynolds numbers}},
journal = {Proceedings of the Royal Society A},
volume = {434},
number = {1890},
pages = {9--13},
year = {1991},
doi = {10.1098/rspa.1991.0075}
}

@article{lanthaler2021statistical,
author = {Lanthaler, S. and Mishra, Siddhartha and Par\'{e}s-Pulido, C.},
title = {{Statistical solutions of the incompressible Euler equations}},
journal = {Mathematical Models and Methods in Applied Sciences},
volume = {31},
number = {02},
pages = {223--292},
year = {2021},
doi = {10.1142/S0218202521500068}
}

@article{lanthaler2021conservation,
doi = {10.1088/1361-6544/abb452},
year = {2021},
publisher = {IOP Publishing},
volume = {34},
number = {2},
pages = {1084},
author = {S Lanthaler and S Mishra and C Par{\'{e}}s-Pulido},
title = {{On the conservation of energy in two-dimensional incompressible flows}},
journal = {Nonlinearity}
}

@article{chae1991vanishing,
title = {{The vanishing viscosity limit of statistical solutions of the Navier-Stokes equations. I. 2-D periodic case}},
journal = {Journal of Mathematical Analysis and Applications},
volume = {155},
number = {2},
pages = {437--459},
year = {1991},
doi = {10.1016/0022-247X(91)90012-O},
author = {Dongho Chae},
}

@article{glimm2023smooth,
      title={{Smooth vs. Physical Solutions of the Navier-Stokes Equation}}, 
      author={J. Glimm and J. Petrillo and M. C. Lee},
      year={2023},
      journal={arXiv}, 
      volume ={preprint},  
      doi={10.48550/arXiv.2212.14734}, 
      url={https://arxiv.org/abs/2212.14734v3}
}

@article{he1997theory,
  title = {Theory of the lattice {Boltzmann} method: {From the Boltzmann equation to the lattice Boltzmann equation}},
  author = {He, Xiaoyi and Luo, Li-Shi},
  journal = {Physical Review E},
  volume = {56},
  issue = {6},
  pages = {6811--6817},
  numpages = {0},
  year = {1997},
  publisher = {American Physical Society},
  doi = {10.1103/PhysRevE.56.6811}
}

@article{junk2003rigorous,
  title={Rigorous {Navier--Stokes} limit of the lattice {Boltzmann} equation},
  author={Junk, Michael and Yong, Wen-An},
  journal={Asymptotic Analysis},
  volume={35},
  doi={10.3233/ASY-2003-563},
  number={2},
  pages={165--185},
  year={2003},
  publisher={IOS Press}
}

@article{junk2005asymptotic,
title = {{Asymptotic analysis of the lattice Boltzmann equation}},
journal = {Journal of Computational Physics},
volume = {210},
number = {2},
pages = {676--704},
year = {2005},
doi = {10.1016/j.jcp.2005.05.003},
author = {M. Junk and A. Klar and L.-S. Luo}
}

@article{hou2024nearly,
      title={{Nearly self-similar blowup of generalized axisymmetric Navier-Stokes and Boussinesq equations}}, 
      author={Thomas Y. Hou},
      year={2024},
      journal={arXiv}, 
      volume={preprint},
      doi={10.48550/arXiv.2405.10916}
}

@article{wang2023discovering,
  title = {{Asymptotic Self-Similar Blow-Up Profile for Three-Dimensional Axisymmetric Euler Equations Using Neural Networks}},
  author = {Wang, Y. and Lai, C.-Y. and G\'omez-Serrano, J. and Buckmaster, T.},
  journal = {Physical Review Letters},
  volume = {130},
  issue = {24},
  pages = {244002},
  numpages = {6},
  year = {2023},
  doi = {10.1103/PhysRevLett.130.244002}
}

@article{jia2014local,
  title={{Local-in-space estimates near initial time for weak solutions of the Navier-Stokes equations and forward self-similar solutions}},
  author={Jia, Hao and {\v{S}}ver{\'a}k, Vladim{\'\i}r},
  journal={Inventiones mathematicae},
  volume={196},
  issue={1},
  doi={10.1007/s00222-013-0468-x},
  pages={233--265},
  year={2014},
  publisher={Springer}
}

@article{guillod2023numerical,
  title={{Numerical investigations of non-uniqueness for the Navier--Stokes initial value problem in borderline spaces}},
  author={Guillod, Julien and {\v{S}}ver{\'a}k, Vladim{\'\i}r},
  journal={Journal of Mathematical Fluid Mechanics},
  volume={25},
  number={3},
  doi={10.1007/s00021-023-00789-5},
  pages={46},
  year={2023},
  publisher={Springer}
}

@article{buckmaster2019nonuniqueness,
author = {Tristan Buckmaster and Vlad Vicol},
title = {{Nonuniqueness of weak solutions to the Navier-Stokes equation}},
volume = {189},
journal = {Annals of Mathematics},
number = {1},
publisher = {Department of Mathematics of Princeton University},
pages = {101--144},
year = {2019},
doi = {10.4007/annals.2019.189.1.3}
}

@article{malaspinas2012consistent,
  title={Consistent subgrid scale modelling for lattice {Boltzmann} methods},
  author={Malaspinas, Orestis and Sagaut, Pierre},
  journal={Journal of Fluid Mechanics},
  volume={700},
  pages={514--542},
  year={2012},
  publisher={Cambridge University Press}, 
  doi={10.1017/jfm.2012.155}
}

@Article{aidun2010lattice,
author = {Aidun, Cyrus K. and Clausen, Jonathan R.},
title = {{Lattice-Boltzmann Method for Complex Flows}},
journal = {Annual Review of Fluid Mechanics},
volume = {42},
number = {1},
pages = {439--472},
year = {2010},
doi = {10.1146/annurev-fluid-121108-145519},
}

@article{jahanshaloo2013review,
  author = {Jahanshaloo, L. and Pouryazdanpanah, E. and {Che Sidik}, N. A.},
  doi = {10.1080/10407782.2013.807690},
  journal = {Numerical Heat Transfer; Part A: Applications},
  number = {11},
  pages = {938--953},
  title = {{A review on the application of the lattice Boltzmann method for turbulent flow simulation}},
  volume = {64},
  year = {2013}
}

@Article{bhatnagar1954model,
  title = {{A Model for Collision Processes in Gases. I. Small Amplitude Processes in Charged and Neutral One-Component Systems}},
  author = {Bhatnagar, P. L. and Gross, E. P. and Krook, M.},
  journal = {Physical Review},
  volume = {94},
  issue = {3},
  pages = {511--525},
  numpages = {0},
  year = {1954},
  publisher = {American Physical Society},
  doi = {10.1103/PhysRev.94.511}
}

@article{krause2020openlb,
title = {{OpenLB{\textemdash}Open source lattice Boltzmann code}},
journal = {Computers {\&} Mathematics with Applications},
volume = {81},
pages = {258--288},
year = {2021},
doi = {10.1016/j.camwa.2020.04.033},
author = {Mathias J. Krause and Adrian Kummerländer and Samuel J. Avis and Halim Kusumaatmaja and Davide Dapelo and Fabian Klemens and Maximilian Gaedtke and Nicolas Hafen and Albert Mink and Robin Trunk and Jan E. Marquardt and Marie-Luise Maier and Marc Haussmann and S. Simonis}  
}

@article{dreher2012compact,
title = {{Compact families of piecewise constant functions in $L^p(0,T;B)$}},
journal = {Nonlinear Analysis: Theory, Methods \& Applications},
volume = {75},
number = {6},
pages = {3072-3077},
year = {2012},
doi = {10.1016/j.na.2011.12.004},
author = {Michael Dreher and Ansgar Jüngel}
}

@Article{brenier2011weak,
  author = {Brenier, Yann and De Lellis, Camillo and Székelyhidi, László Jr.},
  title = {{Weak-strong uniqueness for measure-valued Solutions}},
  journal = {Communications in Mathematical Physics},
  year = {2011},
  volume = {305},
  number = {2},
  pages = {351--361},
  doi={10.1007/s00220-011-1267-0}
}

@article{varadarajan1958convergence,
  author={Varadarajan, V. S.},
  title={On the Convergence of Sample Probability Distributions},
  journal={Sankhyā: The Indian Journal of Statistics (1933-1960)},
  volume={19},
  number={1/2},
  year={1958},
  pages={23-–26},
  url={http://www.jstor.org/stable/25048365}
}

@Article{haussmann2020evaluation,
AUTHOR = {Haussmann, Marc and Ries, Florian and Jeppener-Haltenhoff, Jonathan B. and Li, Yongxiang and Schmidt, Marius and Welch, Cooper and Illmann, Lars and Böhm, Benjamin and Nirschl, Hermann and Krause, Mathias J. and Sadiki, Amsini},
title = {{Evaluation of a Near-Wall-Modeled Large Eddy Lattice Boltzmann Method for the Analysis of Complex Flows Relevant to IC Engines}},
JOURNAL = {Computation},
VOLUME = {8},
YEAR = {2020},
NUMBER = {2},
ARTICLE-NUMBER = {43},
DOI = {10.3390/computation8020043}
}

@article{simonis2022temporal,
title = {{Temporal large eddy simulation with lattice Boltzmann methods}},
journal = {Journal of Computational Physics},
volume = {454},
pages = {110991},
year = {2022},
doi = {10.1016/j.jcp.2022.110991},
author = {S. Simonis and D. Oberle and M. Gaedtke and P. Jenny and M. J. Krause}  
}

@inproceedings{kajzer2014large,
  title={{Large-eddy simulations of 3D Taylor-Green vortex: Comparison of smoothed particle hydrodynamics, lattice Boltzmann and finite volume methods}},
  author={Kajzer, A and Pozorski, J and Szewc, K},
  booktitle={Journal of Physics: Conference Series},
  doi = {10.1088/1742-6596/530/1/012019},
  volume={530},
  pages={012019},
  year={2014}
}

@article{lallemand2000theory,
  author={Lallemand, Pierre and Luo, Li-Shi},
  title={Theory of the lattice {Boltzmann} method: {Dispersion}, dissipation, isotropy, {Galilean} invariance, and stability},
  journal={Physical Review E},
  volume={61},
  number={6},
  pages={6546},
  year={2000},
  doi={10.1103/PhysRevE.61.6546},
  publisher={APS}
}

@article{Tadmor1989,
    author = {Tadmor, Eitan},
    title = {{Convergence of Spectral Methods for Nonlinear Conservation Laws}},
    journal = {SIAM Journal on Numerical Analysis},
    volume = {26},
    number = {1},
    pages = {30--44},
    year = {1989},
    doi = {10.1137/0726003}
}

@article{Tadmor2004,
    author = {Eitan Tadmor},
    title = {{Burgers' Equation with Vanishing Hyper-Viscosity}},
    volume = {2},
    journal = {Communications in Mathematical Sciences},
    number = {2},
    doi={10.4310/CMS.2004.v2.n2.a9},
    publisher = {International Press of Boston},
    pages = {317--324},
    year = {2004},
}

@article{virtanenSciPy10Fundamental2020,
  title = {{{SciPy}} 1.0: Fundamental Algorithms for Scientific Computing in {{Python}}},
  author = {Virtanen, Pauli and Gommers, Ralf and Oliphant, Travis E. and Haberland, Matt and Reddy, Tyler and Cournapeau, David and Burovski, Evgeni and Peterson, Pearu and Weckesser, Warren and Bright, Jonathan and Van Der Walt, Stéfan J. and Brett, Matthew and Wilson, Joshua and Millman, K. Jarrod and Mayorov, Nikolay and Nelson, Andrew R. J. and Jones, Eric and Kern, Robert and Larson, Eric and Carey, C. J. and Polat, Ilhan and Feng, Yu and Moore, Eric W. and VanderPlas, Jake and Laxalde, Denis and Perktold, Josef and Cimrman, Robert and Henriksen, Ian and Quintero, E. A. and Harris, Charles R. and Archibald, Anne M. and Ribeiro, Antônio H. and Pedregosa, Fabian and Van Mulbregt, Paul and {SciPy 1.0 Contributors} and Vijaykumar, Aditya and Bardelli, Alessandro Pietro and Rothberg, Alex and Hilboll, Andreas and Kloeckner, Andreas and Scopatz, Anthony and Lee, Antony and Rokem, Ariel and Woods, C. Nathan and Fulton, Chad and Masson, Charles and Häggström, Christian and Fitzgerald, Clark and Nicholson, David A. and Hagen, David R. and Pasechnik, Dmitrii V. and Olivetti, Emanuele and Martin, Eric and Wieser, Eric and Silva, Fabrice and Lenders, Felix and Wilhelm, Florian and Young, G. and Price, Gavin A. and Ingold, Gert-Ludwig and Allen, Gregory E. and Lee, Gregory R. and Audren, Hervé and Probst, Irvin and Dietrich, Jörg P. and Silterra, Jacob and Webber, James T and Slavič, Janko and Nothman, Joel and Buchner, Johannes and Kulick, Johannes and Schönberger, Johannes L. and De Miranda Cardoso, José Vinícius and Reimer, Joscha and Harrington, Joseph and Rodríguez, Juan Luis Cano and Nunez-Iglesias, Juan and Kuczynski, Justin and Tritz, Kevin and Thoma, Martin and Newville, Matthew and Kümmerer, Matthias and Bolingbroke, Maximilian and Tartre, Michael and Pak, Mikhail and Smith, Nathaniel J. and Nowaczyk, Nikolai and Shebanov, Nikolay and Pavlyk, Oleksandr and Brodtkorb, Per A. and Lee, Perry and McGibbon, Robert T. and Feldbauer, Roman and Lewis, Sam and Tygier, Sam and Sievert, Scott and Vigna, Sebastiano and Peterson, Stefan and More, Surhud and Pudlik, Tadeusz and Oshima, Takuya and Pingel, Thomas J. and Robitaille, Thomas P. and Spura, Thomas and Jones, Thouis R. and Cera, Tim and Leslie, Tim and Zito, Tiziano and Krauss, Tom and Upadhyay, Utkarsh and Halchenko, Yaroslav O. and Vázquez-Baeza, Yoshiki},
  year = {2020},
  journal = {Nature Methods},
  volume = {17},
  number = {3},
  pages = {261--272},
  doi = {10.1038/s41592-019-0686-2}
}

\appendix

\section{Comparative summary of contributions}\label{sec:summary}

A summary of the contributions in this work is provided in Table~\ref{tab:overview}. 
\begin{table}[ht!]
    \caption{Overview of the analytical statements of Section~\ref{sec:approximation} and of the quantities computed in Section~\ref{sec:numerics}. 
    (A1)--(A4) abbreviate Assumptions~\ref{ass:lowMach}--\ref{ass:scaling} on the fully discrete dynamics.}
    \label{tab:overview}
    \centerline{
    \footnotesize
    \begin{tabular}{l l p{0.42\textwidth}}
    \hline\hline
    \textbf{Object} & \textbf{Requirements} & \textbf{Result} \\
    \hline\hline
    \multicolumn{3}{l}{\textit{Statements proven without assumptions on the scheme}} \\
    \hline
    Lemma~\ref{lem:wellprepared} & --- & $\mathcal{H}_{\mathrm{rel}}(0) \leq C_{\mathrm{wp}} \eps^{2}$ almost surely \\
    Lemma~\ref{lem:empirical} & --- & $\mu^{\vis,N,M}_{t} \rightharpoonup \mu^{\vis,N}_{t}$ for $M \to \infty$ \\
    Lemma~\ref{lem:kuznetsovRate} & Assumption~\ref{ass:kuznetsov} & $W_{1}(\mu^{\eps}_{t},\mu_{t}) \leq (C_{1}+C_{2})\eps^{s/(s+\zeta)}$ \\
    \hline\hline
    \multicolumn{3}{l}{\textit{Known results applied to the limit objects}} \\
    \hline
    Proposition~\ref{thm:weakStrong}, \cite{brenier2011weak,lanthaler2021statistical} & \eqref{eq:admissibility}, strong $\bar{\bm{u}}$, $\nu^{1}_{0,\bm{x}} = \delta_{\bar{\bm{u}}_{0}(\bm{x})}$ & $\nu^{1}_{t,\bm{x}} = \delta_{\bar{\bm{u}}(t,\bm{x})}$, and stability for $\mathfrak{w} > 0$ \\
    Corollary~\ref{thm:eulerConsistency}, \cite[Theorem 4.8]{fjordholm2021vanishing} & (A1)--(A4), \cite[Definition 3.6]{fjordholm2021vanishing} & FMW Euler hierarchy and energy admissibility \eqref{eq:admissibility} in the iterated limit $\vis \searrow 0$ \\
    \hline\hline
    \multicolumn{3}{l}{\textit{Proven conditional on the standing assumptions}} \\
    \hline
    Lemma~\ref{lem:energy} & (A1) & $\sup_{t} \|\tilde{\bm{u}}^{\eps}\|_{L^{\infty}} \leq C_{G}$, $\|\rho^{\eps}-1\|_{L^{\infty}} \leq G\eps^{2}$ \\
    Lemma~\ref{lem:energy} & (A1), (A2) & $\vis \int_{0}^{T} \|\bm{\nabla}^{\eps} \tilde{\bm{u}}^{\eps}\|_{L^{2}}^{2} \,\mathrm{d}t^{\prime} \leq C$ \\
    Lemma~\ref{lem:time_deriv} & (A1), (A3) & $\partial_{t}^{\eps} \bm{m}^{\eps}$ bounded in $L^{2}(0,T;H^{-3}(\Omega))$ \\
    Theorem~\ref{thm:nse_limit} & (A1)--(A3) & subsequential limit solves the Foias--Temam Liouville equation \\
    Proposition~\ref{thm:diagonal_limit} & (A1), (A3), (A4) & first equation of the FMW Euler hierarchy along the diagonal $\vis_{\eps} = c\eps$ \\
    Proposition~\ref{prop:weak-strong_discrete} & (A1)--(A4), \cite[Definition 3.6]{fjordholm2021vanishing} & iterated limits $M \to \infty$, $N \to \infty$, $\vis \searrow 0$ along subsequences \\
    \hline\hline
    \multicolumn{3}{l}{\textit{Conjectured}} \\
    \hline
    Hypothesis~\ref{hyp:WassersteinRate} & --- & (K2) holds with $\zeta \approx s \approx 1/3$ ($\zeta$ plausible with measured rates), hence $\alpha \approx 0.5$ \\
    \hline\hline
    \multicolumn{3}{l}{\textit{Computed}} \\
    \hline
    Spectra $E^{\mathrm{c}}(\kappa\eta)$, Figure~\ref{fig:energySpectra9} & Set~1 & inertial plateau consistent with K41 \\
    Structure functions $S^{2,\mathrm{c}}$, Figure~\ref{fig:SF_2D} & Set~2 & $2/3$ power law, evidence for (A4) \\
    EOSC, Section~\ref{sec:sampleDivergence} & Set~2 & pathwise convergence fails, relative error saturates at $\mathcal{O}(1)$ \\
    EOWC, Section~\ref{subsec:numWasserstein} & Set~2 & $\alpha \approx 0.4$--$0.5$ \\
    Cross-solver EOWC, Section~\ref{sec:numWassersteinRef} & azeban reference & comparable rates, evidence that $\mu_{t}$ is insensitive to the regularization \\
    \hline\hline
    \end{tabular}}
\end{table}

\section{Algorithmic details}\label{sec:appendix-algorithms}

\subsection{Implementation in OpenLB-UQ}

To compute approximate statistical solution candidates for the incompressible NSE~\eqref{eq:incNSE}, we make use of OpenLB-UQ~\cite{zhong2025openlbuquncertaintyquantificationframework}, which enables highly parallel sample production for any application case currently available in OpenLB~\cite{krause2020openlb}. 
OpenLB-UQ includes several submodules for large-scale sampling using established UQ methods, such as MC, quasi-MC (QMC), Stochastic Collocation (SC), and Latin hypercube sampling. 
The module has been validated for incompressible fluid flow benchmarks \cite{zhong2025openlbuquncertaintyquantificationframework}, including statistical solution convergence in the 2D incompressible TGV test case. 
The tests demonstrate efficient scalability over thousands of CPU cores and thousands of samples with both sample-level (probability space) and domain-level (position space) parallelization \cite{zhong2025openlbuquncertaintyquantificationframework}. 

Here, for single sample computation, we use the KBC collision kernel proposed in~\cite{simonis2024spectral}. 
The single level MC wrapper in OpenLB-UQ acts as a CPU-based pre- and postprocessor for multiple executables with random input data and evolves LBM instances on multiple CPUs and/or GPUs in parallel over time. 
The I/O for writing downsampled velocity fields to disk is partially scheduled in the background. 
A CPU-based post-processor then uses the output stack of the samples to compute statistical quantities from the approximated hydrodynamic moments.

\subsection{Computational resources and statistics for the largest campaign} \label{sec:appendix-resources}

The smaller MC LBM campaign (Set~1, Table~\ref{tab:params1}) was computed on uc3 at SCC, KIT. 
The comparative statistical solution with the spectral hyperviscosity method was computed on Alps at CSCS (see Appendix~\ref{appsec:refSol}). 
Since it was the most demanding computational task in this work, we focus the description of the computational resources and statistics on the larger campaign for data production (Set~2, Table~\ref{tab:params2}) computed on HoreKa Green at SCC, KIT. 
Each of the samples in Set~2 (clipped to resolutions \(N=32\) to \(N=512\), \(5000\) samples in total) was computed using a single NVIDIA A100 GPU with 40~GiB of main memory and a single mapped core on an Intel Xeon Platinum 8368 CPU. 
For this clipped Set~2, approximately $2.5 \times 10^4$ GPU hours were consumed in total. 
Batching a varying number of samples per job in an array depending on the cluster's capacity, on average, we ran \(50\) batch jobs in parallel. 
Due to the queuing, the completion of this campaign required $33.4$ days of physical time. 
Table~\ref{tab:runtime_statistics} summarizes the averaged runtime statistics of the mapped hardware per resolution. 
It is noteworthy that the included time for I/O operations was on the order of milliseconds due to the in-situ downsampling, which required a similarly negligible amount of time compared to the data production.
\begin{table}[ht!]
    \caption{Mean and standard deviation over $M=1000$ samples of measured wall clock time and CPU-time on heterogeneous allocation per sample for each resolution including I/O. 
    }
    \label{tab:runtime_statistics}
    \centering
    \small
    \begin{tabular}{r r r}
    \hline
        \makecell[c]{$N$} & 
        \makecell[c]{Wall clock time $[\text{h}]$} & 
        \makecell[c]{CPU-time $[\text{h}]$} \\
        \hline
         $32$ & $1.01 \times 10^{-2} \pm 8.25 \times 10^{-4}$ & $2.78 \times 10^{-4} \pm 5.04 \times 10^{-6}$ \\
         $64$ & $1.16 \times 10^{-2} \pm 1.13 \times 10^{-3}$ & $1.83 \times 10^{-3} \pm 2.47 \times 10^{-5}$ \\
         $128$ & $4.10 \times 10^{-2} \pm 2.53 \times 10^{-3}$ & $3.00 \times 10^{-2} \pm 1.90 \times 10^{-4}$ \\
         $256$ & $8.20 \times 10^{-1} \pm 1.17 \times 10^{-2}$ & $8.02 \times 10^{-1} \pm 1.17 \times 10^{-2}$ \\
         $512$ & $2.43 \times 10^{+1} \pm 3.75 \times 10^{-1}$ & $2.42 \times 10^{+1} \pm 3.70 \times 10^{-1}$ \\
      \hline
    \end{tabular}
\end{table}
The computation of the vector-valued and component-wise Wasserstein distances (see Figure~\ref{fig:eowc_256}) for $M=1000$ samples at reference resolutions $N \in \{256, 512\}$ consumed $1.8 \times 10^4$ CPU hours, taking $2.6 \times 10^2$ wall-clock hours when parallelized across $76$ cores on a dual-socket Intel Xeon Platinum 8368 node with 512 GiB of RAM.

\subsection{Computation of the energy spectrum}\label{sec:appendix-energySpec}

The scaling assumption \eqref{eq:energyScaling} is formulated for the time-integrated expected energy spectrum $E_{T}\left( \bm{\mu}^{\vis}, \kappa\right)$, which is the natural object for statistically stationary flows. 
The RTGV flow considered here is unsteady, so all spectral quantities are evaluated at the time-indexed measures $\mu_{t}^{\vis}$ and the resulting ensemble statistics remain time-dependent. 
We therefore describe the approximation of the time-local expected spectrum $E\left( \mu_{t}^{\vis}, \kappa\right)$ first and recover the time-integrated quantity by an additional integration where it is required. 
For a discrete ensemble of $M$ MC realizations $\{\bm{u}^m\}_{m=1}^M$, the discrete Fourier coefficients of the $m$th sample on the grid nodes \(\bm{x}_{\bm{n}} = 2\pi \bm{n}/N\), \(\bm{n} \in \{0, 1, \ldots, N-1\}^{d}\), are
\begin{align}\label{eq:dftCoeff}
    \tilde{u}_{m,\alpha} \left( \bm{k}, t \right) = \frac{1}{N^{d}} \sum \limits_{\bm{n} \in \{0, 1, \ldots, N-1\}^{d}} u_{m,\alpha} (\bm{x}_{\bm{n}}, t ) \exp \left( - \frac{2\pi\mathsf{i}}{N}  \bm{k}\cdot\bm{n} \right) , 
    \qquad \bm{k} \in \mathcal{K}_{N} \coloneqq \{ -N/2 + 1, \ldots, N/2 \}^{d} ,
\end{align}
for each component \(\alpha = 1, 2, \ldots, d\), so that \(N^{-d} \sum_{\bm{n}} \vert \bm{u}_{m}(\bm{x}_{\bm{n}}, t) \vert^{2} = \sum_{\bm{k} \in \mathcal{K}_{N}} \Phi_{m}(\bm{k}, t)\) with \(\Phi_{m}(\bm{k}, t) = \vert \tilde{\bm{u}}_{m}(\bm{k}, t) \vert^{2}\). 
Compared to the continuous coefficient in \eqref{eq:fourierVelocityEnergy}, this normalization measures energy per unit volume, \(\Phi_{m} = \vert \Omega \vert^{-1} \Vert \tilde{\bm{u}}_{m} \Vert_{2}^{2}\), which is the convention in which the Kolmogorov constant \(C_{\mathrm{Kol}}\) is defined. 
The surface integral in \eqref{eq:energySpec} is approximated by the shell average times the surface area of the sphere of radius \(\kappa\),
\begin{align}\label{eq:energySpecDisc}
    E_m \left(\kappa, t\right)  \approx 
    \frac{4\pi}{\vert S_{\kappa} \vert} \sum\limits_{\bm{k} \in S_{\kappa}} \frac{1}{2} \vert \bm{k} \vert^{2} \Phi_m (\bm{k}, t) , 
    \qquad 
    S_{\kappa} \coloneqq \{ \bm{k} \in \mathcal{K}_{N} : \lfloor \vert \bm{k} \vert \rceil = \kappa \} , 
\end{align}
for \(\kappa = 0, 1, \ldots, N/2\), where \(\lfloor \cdot \rceil\) denotes rounding to the nearest integer and \(\vert S_{\kappa} \vert\) is the number of modes in the shell. 
Since \(\vert \bm{k} \vert^{2} \approx \kappa^{2}\) on \(S_{\kappa}\), \eqref{eq:energySpecDisc} is \(4 \pi \kappa^{2}\) times the mean of \(\frac{1}{2} \Phi_{m}\) over the shell, and \(\sum_{\kappa} E_{m}(\kappa, t)\) approximates the kinetic energy per unit volume. 
Since the input is real, \(\tilde{\bm{u}}_{m}(-\bm{k}, t) = \overline{\tilde{\bm{u}}_{m}(\bm{k}, t)}\), so that \(\Phi_{m}(-\bm{k}, t) = \Phi_{m}(\bm{k}, t)\), and the shell sums are evaluated on the half-space \(k_{1} \geq 0\) returned by the real-to-complex transform, counting the modes with \(0 < k_{1} < N/2\) twice.
The time-local expected spectrum is then approximated by the empirical mean over the ensemble
\begin{equation} \label{eq:spectrum_mean}
    E\left( \mu_{t}^{\vis}, \kappa\right) \approx \E[ E(\mu_{t}^{\vis, N, M}, \kappa)] = \frac{1}{M} \sum_{m=1}^{M} E_m(\kappa, t).
\end{equation}
The statistical fluctuations of the energy distribution across different realizations are quantified by the empirical variance and the corresponding empirical standard deviation
\begin{align} \label{eq:spectrum_var}
    \Var (E(\mu_{t}^{\vis, N, M}, \kappa)) &= \frac{1}{M-1} \sum_{m=1}^{M} \left( E_m(\kappa, t) - \E[ E(\mu_{t}^{\vis, N, M}, \kappa)]\right)^2, \nonumber \\
    \sigma (E(\mu_{t}^{\vis, N, M}, \kappa)) &= \sqrt{ \Var (E(\mu_{t}^{\vis, N, M}, \kappa)) } .
\end{align}
Throughout Section~\ref{sec:numerics}, the standard deviation $\sigma$ is the plotted quantity. 
Showing it alongside the mean spectrum $\E[ E(\mu_{t}^{\vis, N, M}, \kappa)]$ provides direct insight into the sensitivity of the turbulent cascade at specific wavenumbers to the initial random perturbations. 
Whenever the time-integrated spectrum of \eqref{eq:energyScaling} is required, we exploit the linearity of both the time integration and the expectation operator. 
By Fubini's theorem, the integration over time commutes with the integration over the statistical measure, so that the time-integrated spectrum is first computed for each individual sample,
\begin{equation}\label{eq:energySampleSpec}
    E_T^m(\kappa) = \int_{0}^{T} E_m(\kappa, t) \,\mathrm{d}t,
\end{equation}
and the ensemble average is taken afterward, which yields $E_{T}( \bm{\mu}^{\vis}, \kappa) \approx \frac{1}{M} \sum_{m=1}^{M} E_T^m(\kappa)$. 
Evaluating the integrals in this order also provides access to the sample-to-sample variability of the time-integrated spectra.

\subsection{Computation of the structure functions}\label{sec:appendix-struct}

With the sphere $\partial B_r(\bm{x})$ of radius $r\in\{1,2,\ldots,N\}$ and center $\bm{x}\in\Omega$, let 
\begin{align}
    \fint_{\partial B_r(\bm{x})}:= \frac{1}{|\partial B_r(\bm{x})|}\int_{\partial B_r(\bm{x})}
\end{align}
denote the mean over $\partial B_r(\bm{x})$.
\begin{lemma}[Spectral approximation of structure functions]\label{prop:structureFunctionSpectral}
    The local second order structure functions for \(r>0\) and \(t\in (0,T]\) of a statistical solution $\mu_t^{\vis}$ on a homogeneous, periodic domain (the trace variant of \eqref{eq:structuresLocal}) can be approximated with
    \begin{align}\label{eq:SF_spectral}
    S_{r,t}^{2} \left( \mu_{t} \right) 
        & \coloneqq  \int_{L^{2}_{x}} \int_{\Omega} \fint_{\partial B_{r}(\bm{x}_{1})} \left\vert \bm{u}\left(\bm{x}_{2}\right) - \bm{u} \left( \bm{x}_{1} \right) \right\vert^{2} \,\mathrm{d}\bm{x}_{2} \,\mathrm{d}\bm{x}_{1} \,\mathrm{d}\mu_{t}\left(\bm{u}\right)  \\
        & \approx \frac{1}{M} \sum_{m=1}^{M} \frac{1}{\vert C_{r} \vert} \sum_{\bm{h} \in C_{r}} 2 \left( R_{\bm{0}}^{N,m} - R_{\bm{h}}^{N,m} \right), 
    \end{align}
    where $R_{\bm{h}}^{N,m} = N^{-d} \, \mathcal{F}^{-1} \bigl( \sum_{i=1}^d \vert \mathcal{F}(u_i^{N,m}) \vert^2 \bigr)_{\bm{h}}$ is the discrete spatial autocorrelation, $\mathcal{F}$ denotes the unnormalized discrete Fourier transform and $\mathcal{F}^{-1}$ its inverse including the factor $N^{-d}$, $C_{r}$ is the discrete Euclidean shell of radius $r$, and \(\mu_{t}^{\vis, N, M}\) is the numerical statistical solution at the spatial resolution $N$ and a sample size $M$.
\end{lemma}
\begin{proof}
    The approximation is derived in six steps. 
    \begin{enumerate}
        \item Evaluating the defining functional \eqref{eq:SF_spectral} at the empirical measure $\mu_{t}^{\vis,N,M} = \frac{1}{M}\sum_{m=1}^{M} \delta_{\bm{u}^{N,m}(t)}$ from \eqref{eq:statSolDisc}, the outer integral over $L^{2}_{x}$ reduces exactly to the ensemble average
            \begin{align}\label{eq:SF_approx_spec1}
                S_{r,t}^2\left(\mu_{t}^{\vis,N,M}\right) = \frac{1}{M} \sum_{m=1}^M \int_{\Omega} \fint_{\partial B_{r}(\bm{x}_{1})} \left\vert \bm{u}^{N,m}\left(\bm{x}_{2}\right) - \bm{u}^{N,m} \left( \bm{x}_{1} \right) \right\vert^{2} \,\mathrm{d}\bm{x}_{2} \,\mathrm{d}\bm{x}_{1}.
            \end{align}
            Note that this step involves no approximation. The statistical error of replacing $\mu_{t}$ by $\mu_{t}^{\vis,N,M}$ in the lemma is the Monte Carlo sampling error, which is controlled by the law of large numbers as $M \to \infty$ and enters independently of the spatial discretization derived in the remaining steps.
        \item Expanding the squared velocity increment inside the spatial integral yields $\vert \bm{u}(\bm{x}_2) \vert^2 + \vert \bm{u}(\bm{x}_1) \vert^2 - 2\bm{u}(\bm{x}_1)\cdot\bm{u}(\bm{x}_2)$. 
        Let $\bm{h} = \bm{x}_2 - \bm{x}_1$ be the separation vector. 
        By the periodicity of $\Omega$, the spatial integral is translation invariant, i.e., $\int_{\Omega} \vert \bm{u}(\bm{x}_{1}+\bm{h}) \vert^{2} \,\mathrm{d}\bm{x}_{1} = \int_{\Omega} \vert \bm{u}(\bm{x}_{1}) \vert^{2} \,\mathrm{d}\bm{x}_{1}$ for every fixed $\bm{h}$. 
        Hence, substituting $\bm{x}_{2} = \bm{x}_{1} + \bm{h}$ and exchanging the order of integration by Fubini's theorem, we exactly obtain  
            \begin{align}\label{eq:SF_approx_spec2}
                & S_{r,t}^2\left(\mu_{t}^{\vis,N,M}\right) \nonumber \\
                &= 
                \frac{1}{M} \sum_{m=1}^M \int_{\Omega} \fint_{\partial B_{r}(\bm{0})} \left( 2\vert \bm{u}^{N,m}\left(\bm{x}_{1}\right) \vert^{2} - 2\bm{u}^{N,m}\left(\bm{x}_1\right) \cdot \bm{u}^{N,m}\left(\bm{x}_1 + \bm{h}\right) \right) \,\mathrm{d}\bm{h} \,\mathrm{d}\bm{x}_{1}.
            \end{align}
        \item On the discrete space domain $\Omega_{\triangle x}$ with total points $\vert\Omega_{\triangle x}\vert = N^d$, we approximate the spatial integral over $\bm{x}_1$ by the discrete average, which defines the discrete spatial autocorrelation $R^{N,m}$ for a given sample $m$ at a lag vector $\bm{h}$, i.e., 
            \begin{align}\label{eq:SF_approx_spec3}
                R_{\bm{h}}^{N,m} \coloneqq \frac{1}{N^d}\sum_{\bm{x}_1\in \Omega_{\triangle x}} \bm{u}^{N,m}\left(\bm{x}_1\right) \cdot \bm{u}^{N,m}\left(\bm{x}_1 + \bm{h}\right).
            \end{align}
            Consequently, the structure function simplifies to the difference between the zero-lag correlation $R_{\bm{0}}^{N,m}$ (the spatial mean of the squared velocity magnitude, i.e., twice the mean kinetic energy density) and the correlation at lag $\bm{h}$, 
            \begin{align}\label{eq:SF_approx_spec4}
                S_{r,t}^2\left(\mu_{t}^{\vis,N,M}\right)
                \approx 
                \frac{1}{M}\sum_{m=1}^M \fint_{\partial B_r(\bm{0})} 2 \left( R_{\bm{0}}^{N,m} - R_{\bm{h}}^{N,m} \right) \,\mathrm{d}\bm{h}.
            \end{align}
        \item To circumvent the computationally prohibitive $\mathcal{O}(N^{2d})$ operations required to evaluate $R_{\bm{h}}^{N,m}$ in physical space, we apply the discrete Wiener--Khinchin theorem, which holds exactly for periodic grid functions. 
        The spatial autocorrelation is computed via the inverse discrete Fourier transform ($\mathcal{F}^{-1}$) of the power spectral density (PSD), where $\mathcal{F}$ denotes the unnormalized forward transform and $\mathcal{F}^{-1}$ carries the factor $N^{-d}$, i.e.,
            \begin{align}\label{eq:SF_approx_spec5}
                R_{\bm{h}}^{N,m} = \frac{1}{N^{d}} \, \mathcal{F}^{-1} \left( \sum_{i=1}^d \left\vert \mathcal{F} \left( u_i^{N,m} \right) \right\vert^2 \right)_{\bm{h}}.
            \end{align}
        \item The continuous spherical shell $\partial B_r(\bm{0})$ of radius $r$ is replaced by the discrete Euclidean neighborhood $C_r$, which contains all discrete lag vectors $\bm{h} \in \mathbb{Z}^d$ whose $L^2$ norm strictly rounds to $r$, such that $C_r = \{ \bm{h} \in \mathbb{Z}^d : \lfloor \vert \bm{h} \vert_2 \rceil = r \}$. 
        \item The continuous spherical average is approximated by an arithmetic mean over the discrete elements in the Euclidean shell $C_r$, where $\vert C_r \vert$ denotes the cardinality of the discrete shell,
            \begin{align}\label{eq:SF_approx_spec6}
                \fint_{\partial B_r(\bm{0})} \cdot \,\mathrm{d}\bm{h} \approx \frac{1}{\vert C_r \vert} \sum_{\bm{h} \in C_r} \cdot . 
            \end{align}
    \end{enumerate}
    To complete the proof, we insert \eqref{eq:SF_approx_spec6} and \eqref{eq:SF_approx_spec5} into \eqref{eq:SF_approx_spec4}, and combine the result with \eqref{eq:SF_approx_spec1} and \eqref{eq:SF_approx_spec2}.
\end{proof}
We continue with splitting up the spectral approximation in Lemma~\ref{prop:structureFunctionSpectral} to obtain an efficiently computable ensemble implementation. 
Let $\Psi^{m}_{\bm{k}}$ denote the discrete power spectral density at wavenumber index $\bm{k}$. 
We define the single sample autocorrelation mapping
\begin{align}
    \Psi^{m}_{\bm{k}} &:= \sum_{i=1}^d \left\vert \mathcal{F} \left( u_i^{m} \right)_{\bm{k}} \right\vert^2, \nonumber \\
    R_{\bm{h}}^{m} &:= \frac{1}{N^{d}} \, \mathcal{F}^{-1} \left( \Psi^{m} \right)_{\bm{h}}. \label{eq:PSD_def}
\end{align}
The single sample structure function contribution $S_{r,t}^{2,m}$ evaluated at a discrete radius $r$ is then given by the radial average
\begin{align}
    S_{r,t}^{2,m}\left(\mu_{t}^{\vis,N,m}\right) 
    &:= \frac{1}{\vert C_r \vert} \sum_{\bm{h} \in C_r} 2 \left( R_{\bm{0}}^{m} - R_{\bm{h}}^{m} \right), \label{eq:SF_single_spec2}
\end{align}
such that the full statistical ensemble is recovered by
\begin{align}\label{eq:SF_single_spec}
    S_{r,t}^2\left(\mu_{t}^{\vis,N,M}\right) &= \frac{1}{M}\sum_{m=1}^M S_{r,t}^{2,m}\left(\mu^{N,m}_{t}\right). 
\end{align}
Based on the $\mathcal{O}(N^d \log N^d)$ complexity enabled by the discrete Fourier transform, we summarize the global spectral implementation for approximating structure functions in Algorithm~\ref{alg:SF_spectral}.
\begin{algorithm}[ht]
    \caption{Spectral structure function (SF) for all discrete radii $r$ at timestep $t$.}
    \label{alg:SF_spectral}
    \begin{algorithmic}[1]
        \Procedure{computeSFspectral}{$t$} \Comment{\textbf{Input:} $\mu_{t}^{\vis,N,M}$}
            \vspace{0.5em}
            \For{$m \gets 1$ to $M$}
                \State $\Psi^{m}_{\bm{k}} \gets \sum_{i=1}^d \left\vert \mathcal{F} \left( u_i^{N,m} \right)_{\bm{k}} \right\vert^2$ \Comment{Compute PSD via forward FFT} 
                \vspace{0.5em}
                \State $R_{\bm{h}}^{m} \gets \frac{1}{N^{d}} \, \mathcal{F}^{-1} \left( \Psi^{m} \right)_{\bm{h}}$ \Comment{Compute 3D autocorrelation via inverse FFT, cf.\ \eqref{eq:SF_approx_spec5}}
                \vspace{0.5em}
                \State $S_{r,t}^{2,m} \gets \frac{1}{\vert C_r \vert} \sum_{\bm{h} \in C_r} 2 \left( R_{\bm{0}}^{m} - R_{\bm{h}}^{m} \right)$ \Comment{Radially average into 1D Euclidean bins for every $r \leq N/2$ \eqref{eq:SF_single_spec2}}
            \EndFor
            \vspace{0.5em}
            \State $S_{r,t}^2\left(\mu_{t}^{N,M}\right) \gets \frac{1}{M}\sum_{m=1}^M S_{r,t}^{2,m}\left(\mu_{t}^{\vis,N,m}\right)$ \Comment{Ensemble average \eqref{eq:SF_single_spec}}
        \EndProcedure \Comment{\textbf{Output:} $S_{r,t}^2\left(\mu_{t}^{\vis,N,M}\right)$}
    \end{algorithmic}
\end{algorithm}

\subsection{Computation of the Wasserstein distance}\label{appsec:compWasserstein}

This section provides a description of the algorithms used for the computations of the EOWC over time for the Wasserstein distances $W_{1,1}$, $W_{1,2}$, and $W_{1,2}^{\mathrm{v}}$ shown in Figure~\ref{fig:eowc_256} of Section~\ref{subsec:numWasserstein}.
Let $\Omega_N$ denote the discretized spatial domain consisting of $N^d$ nodes for a given resolution $N \in \mathcal{N}$, where $\mathcal{N} = \{ 32,64,128,256,512\}$ represents the set of all evaluated grid resolutions. 
To compute the Wasserstein distances consistently across varying mesh sizes, we project all empirical measures onto a common downsampled domain $\widehat{\Omega}_{\mathrm{ds}}$ consisting of $N_{\mathrm{ds}} = Q = 8$ equidistant nodes per spatial direction, that is, $G_{\mathrm{ds}} = \vert \widehat{\Omega}_{\mathrm{ds}} \vert = N_{\mathrm{ds}}^{d} = 8^{3}$ evaluation points in total, matching the index set $I_{Q}$ of \eqref{eq:W1approx}. 
Since the constant number of evaluation nodes only rescales all distances uniformly and cancels in the experimental orders of convergence, this choice affects the absolute magnitude of the reported distances but not the measured rates. 
The Wasserstein distances are computed separately for the empirical measure for each resolution with respect to the reference solutions $N_{\mathrm{ref}} \in \{256, 512\}$. 
To ensure the projected solutions correspond to physically collocated spatial nodes without requiring interpolation, we restrict the downsampled domain to be a common subset of all discretized grids, such that $\widehat{\Omega}_{\mathrm{ds}} \subseteq \bigcap_{N \in \mathcal{N}} \Omega_N$. 
Consequently, the downsampled Monte Carlo samples are obtained via the exact spatial restriction $\hat{\bm{u}}_m^N = \bm{u}_m^N\big|_{\widehat{\Omega}_{\mathrm{ds}}}$. 
The same node set underlies the sample convergence study in Section~\ref{sec:sampleDivergence}, where its evaluation points are denoted by $\bm{x}_{\mathrm{ds},k}$, $k = 1,\dots,G_{\mathrm{ds}}$.
\newcommand{\uhat}{\hat{\bm{u}}}
\newcommand{\Ndown}{N_{\mathrm{ds}}}
\begin{algorithm}[ht!]
    \caption{Computation of $\|W_1(\nu^{1,N},\nu^{1,N_{\mathrm{ref}}})\|_{L^1(\Omega)}$, i.e., the quantity $W_{1,1}$ of \eqref{eq:W1approx} for $k=1$, based on MC samples.}
    \label{alg:WassersteinDistance1}
    \begin{algorithmic}[1]
        \vspace{0.5em}
        \Procedure{ComputeW1\_1pt}{$\{\uhat^N_m\}_{m=1}^M, \{\uhat^{N_{\mathrm{ref}}}_m\}^M_{m=1}$} 
            \State Initialize distance array $\mathcal{D}_c \gets 0$ for components $c \in \{1,\dots,d\}$
            \For{$i \gets 1$ to $\Ndown^d$}
                \For{$c \gets 1$ to $d$}
                    \State $\mathcal{D}_c \gets \mathcal{D}_c + W_1(\{\hat{u}_{c,m}^N(\bm{x}_i)\}_{m=1}^M, \{\hat{u}_{c,m}^{N_{\mathrm{ref}}}(\bm{x}_i)\}_{m=1}^M)$ \Comment{$W_1$ for component $c$ \eqref{eq:wassersteinPermutations} using \texttt{scipy.stats.wasserstein\_distance} function \cite{virtanenSciPy10Fundamental2020}}
                \EndFor
            \EndFor
            \State \Return $\frac{1}{\Ndown^d} \sum_{c=1}^d \mathcal{D}_c$ \Comment{Average integrated marginal distance \eqref{eq:W1approx}}
        \EndProcedure
        \vspace{0.5em}
    \end{algorithmic}
\end{algorithm}
\begin{algorithm}[ht!]
    \caption{Vector-valued computation of $\|W_1(\nu^{2,N},\nu^{2,N_{\mathrm{ref}}})\|_{L^1(\Omega^2)}$, i.e., the quantity $W_{1,2}^{\mathrm{v}}$ of \eqref{eq:W1_vector} with the ground metric on $\mathbb{R}^{2d}$, based on MC samples.}
    \label{alg:WassersteinDistance2_vector}
    \begin{algorithmic}[1]
        \Procedure{ComputeW1\_2pt\_Vector}{$\{\uhat^N_m\}_{m=1}^M, \{\uhat^{N_{\mathrm{ref}}}_m\}_{m=1}^M$} 
        %\Comment{2-point correlations (Joint velocity vectors)}
            \State Initialize total distance $W\gets 0$
            \State Uniform mass vectors $\bm{a} \gets \frac{1}{M}\bm{1}_M, \quad \bm{b} \gets \frac{1}{M}\bm{1}_M$
            \For{$i \gets 1$ to $\Ndown^d$}
                \For{$j \gets i$ to $\Ndown^d$} \Comment{Exploit symmetry of the upper triangle}
                    \State $w \gets 1$ \textbf{if} $i=j$ \textbf{else} $2$
                    \State $\bm{X}^N \gets [\{\uhat_{m}^N(\bm{x}_i)\}_{m=1}^M, \{\uhat_{m}^N(\bm{x}_j)\}_{m=1}^M]$
                    \Comment{$\bm{X}^{N} \in \mathbb{R}^{M \times 2d}$}
                    \State $\bm{X}^{N_{\mathrm{ref}}} \gets [\{\uhat_{m}^{N_{\mathrm{ref}}}(\bm{x}_i)\}_{m=1}^M, \{\uhat_{m}^{N_{\mathrm{ref}}}(\bm{x}_j)\}_{m=1}^M]$
                    \Comment{$\bm{X}^{N_{\mathrm{ref}}} \in \mathbb{R}^{M \times 2d}$}
                    \State $\mathcal{C} \gets \texttt{ot.dist}(\bm{X}^N, \bm{X}^{N_{\mathrm{ref}}})$ \Comment{Compute cost matrix}
                    \State $W \gets W + w \times \texttt{ot.emd2}(\bm{a}, \bm{b}, \mathcal{C})$ \Comment{\texttt{ot.emd2} returns the loss of the optimal transport plan \cite{flamary2024pot}}
                \EndFor
            \EndFor
            \State \Return $\frac{1}{\Ndown^{2d}} W$ \Comment{Riemann sum approximation \eqref{eq:W1_vector}}
        \EndProcedure
    \end{algorithmic}
\end{algorithm}
\begin{algorithm}[ht!]
    \caption{Component-wise computation of $\|W_1(\nu^{2,N},\nu^{2,N_{\mathrm{ref}}})\|_{L^1(\Omega^2)}$, i.e., the quantity $W_{1,2}$ of \eqref{eq:W1approx} for $k=2$, summed over the velocity components $c$, based on MC samples.}
    \label{alg:WassersteinDistance2}
    \begin{algorithmic}[1]
        \Procedure{ComputeW1\_2pt\_Componentwise}{$\{\uhat^N_m\}_{m=1}^M, \{\uhat^{N_{\mathrm{ref}}}_m\}_{m=1}^M$} 
        %\Comment{2-point correlations (Separate components)}
            \State Initialize distance array $\mathcal{D}_c \gets 0$ for components $c \in \{1,\dots,d\}$
            \State Uniform mass vectors $\bm{a} \gets \frac{1}{M}\bm{1}_M, \quad \bm{b} \gets \frac{1}{M}\bm{1}_M$
            \For{$c \gets 1$ to $d$}
                \For{$i \gets 1$ to $\Ndown^d$}
                    \For{$j \gets i$ to $\Ndown^d$} 
                        \State $w \gets 1$ \textbf{if} $i=j$ \textbf{else} $2$
                        \State $\bm{X}^N \gets [\{{\hat{u}}_{c,m}^N(\bm{x}_i)\}_{m=1}^M, \{\hat{u}_{c,m}^N(\bm{x}_j)\}_{m=1}^M]$
                        \Comment{$\bm{X}^{N} \in \mathbb{R}^{M \times 2}$}
                        \State $\bm{X}^{N_{\mathrm{ref}}} \gets [\{\hat{u}_{c,m}^{N_{\mathrm{ref}}}(\bm{x}_i)\}_{m=1}^M, \{\hat{u}_{c,m}^{N_{\mathrm{ref}}}(\bm{x}_j)\}_{m=1}^M]$
                        \Comment{$\bm{X}^{N_{\mathrm{ref}}} \in \mathbb{R}^{M \times 2}$}
                        \State $\mathcal{C} \gets \texttt{ot.dist}(\bm{X}^N, \bm{X}^{N_{\mathrm{ref}}})$ \Comment{Compute cost matrix 
                        %using \texttt{dist} function of POT python package 
                        \cite{flamary2024pot}}
                        \State $\mathcal{D}_c \gets \mathcal{D}_c + w \times \texttt{ot.emd2}(\bm{a}, \bm{b}, \mathcal{C})$ \Comment{ \texttt{ot.emd2} returns the loss of the optimal transport plan \cite{flamary2024pot}}
                    \EndFor
                \EndFor
            \EndFor
            \State \Return $\frac{1}{\Ndown^{2d}} \sum_{c=1}^d \mathcal{D}_c$
        \EndProcedure
    \end{algorithmic}
\end{algorithm}
\begin{algorithm}[ht!]
    \caption{Evaluation of the time-local Wasserstein distances across resolutions and least-squares computation of the EOWC slopes for each timestep.}
    \label{alg:ComputeSlopes}
    \begin{algorithmic}[1]
        \Require{Set of available resolutions $\mathcal{N}$, reference resolution $N_{\text{ref}} \in \mathcal{N}$, set of timesteps $I_{\triangle t}$, and chosen distance metric $W(\cdot, \cdot)$.}
        
        \vspace{0.5em}
        \Procedure{ComputeDistancesAndSlopes}{$W, \mathcal{N}, N_{\text{ref}}, I_{\triangle t}$}
            \State Initialize output directories and file paths
            \For{$t \in I_{\triangle t}$} \Comment{Loop over chosen timesteps}
                \State $\{\uhat^{N_{\text{ref}}}_m\}_{m=1}^{M} \gets \text{ExtractData}(N_{\text{ref}}, t)$ \Comment{Load reference sample ensemble}
                \State Initialize distance mapping $\mathcal{D} \gets \emptyset$
                
                \vspace{0.5em}
                \For{$N \in \mathcal{N} \setminus \{N_{\text{ref}}\}$} \Comment{Loop over chosen resolutions}
                    \State $\{\uhat^{N}_m\}_{m=1}^{M} \gets \text{ExtractData}(N, t)$ \Comment{Load sample ensemble at resolution $N$}
                    \State $\mathcal{D}_N \gets W(\{\uhat^{N}_m\}_{m=1}^{M}, \{\uhat^{N_{\text{ref}}}_m\}_{m=1}^{M})$ \Comment{Evaluate chosen Wasserstein distance between the correlation marginals $\nu^{k,N}$ and $\nu^{k,N_{\mathrm{ref}}}$ induced by the loaded ensembles (see Alg. \ref{alg:WassersteinDistance1}, Alg. \ref{alg:WassersteinDistance2} and Alg. \ref{alg:WassersteinDistance2_vector})}
                    \State $\mathcal{D} \gets \mathcal{D} \cup \{ (N, \mathcal{D}_N) \}$
                \EndFor
                
                \vspace{0.5em}
                \State $\text{SaveToFile}(\mathcal{D}, t)$ \Comment{Save distances for current timestep}
                \State $\alpha_t \gets \texttt{np.polyfit}(\log N, \log \mathcal{D}_N,1)$ \Comment{Compute slope in EOC plot using \cite{harris2020array}}
                \State $\text{AppendToFile}(t, \alpha_t)$ \Comment{Log convergence rate}
            \EndFor
        \EndProcedure
    \end{algorithmic}
\end{algorithm}

\section{Approximation of reference statistical Euler solutions}\label{appsec:refSol}

The incompressible Euler equations are discretized in Fourier space as \cite{lanthaler2021statistical}
\begin{equation} \label{eq:shv_disc}
    \begin{aligned}
        \partial_t \bm{u}^{\triangle} + \mathcal{P}_N(\bm{u}^{\triangle}\cdot\bm{\nabla}\bm{u}^{\triangle}) + \bm{\nabla} p^{\triangle} &= \vis_{N}|\bm{\nabla}|^{2s_{\mathrm{hv}}}(Q_N*\bm{u}^{\triangle}), \\
        \bm{\nabla}\cdot\bm{u}^{\triangle} &= 0, \\
        \bm{u}^{\triangle}|_{t=0} &= \mathcal{P}_N\bm{u}_0,
    \end{aligned}
\end{equation}
where $\mathcal{P}_N$ is the spatial Fourier projection operator mapping a function $f(\bm{x},t)$ to its first $N$ Fourier modes: 
\begin{align}
    \mathcal{P}_N f = \sum_{|\bm{k}|_{\infty}\leq N} \hat{f}_{\bm{k}}(t)e^{\mathsf{i}\bm{k}\cdot \bm{x}} . 
\end{align} 
Note that the right hand side of the equation includes a viscosity-like term for stabilization. 
This artificial viscosity term consists of a resolution-dependent viscosity $\vis_{N}$ and a Fourier multiplier $Q_N$ controlling the strength at which different Fourier modes are dampened. 
This allows us to avoid dampening the low frequency modes while applying some diffusion to the problematic higher frequencies. 
Additionally, the term includes a hyperviscosity parameter $s_{\mathrm{hv}} \geq 1$ (not to be confused with the regularity index $s$ of Section~\ref{sec:formalWrate}), which can be tuned to strengthen the dissipation in higher modes and thus increase the stability of the method. 
Note that these properties make it difficult to discuss the usual properties of the flow, such as the Reynolds number, because the viscosity is unphysical and purely for stabilization. 
All solutions computed with the above discretization scheme should be regarded as an approximation of the inviscid flow and not as an approximation of the incompressible Navier--Stokes equations with some given small viscosity.
The Fourier multiplier $Q_N$ is of the form
\begin{equation}
    Q_N(\bm{x}) = \sum_{\bm{k}\in\mathbb{Z}^d, |\bm{k}|\leq N} \hat{Q}_{\bm{k}} e^{\mathsf{i}\bm{k}\cdot\bm{x}},
\end{equation}
and its Fourier coefficients $\hat{Q}_{\bm{k}}$ fulfill \cite{Tadmor1989,Tadmor2004,lanthaler2021statistical}
\begin{align}
     & \hat{Q}_{\bm{k}} = 0 \quad \text{ for }~ |\bm{k}|\leq m_N, \\
     & 1-\left(\frac{m_N}{|\bm{k}|}\right)^{\frac{2s_{\mathrm{hv}}-1}{\theta}} \leq \hat{Q}_{\bm{k}} \leq 1,
\end{align}
where we have introduced an additional parameter $\theta > 0$. 
This form of $Q_N$ makes the equation completely dissipation free for all wave numbers smaller than $m_N$. 
Therefore, all modes below this cutoff will evolve exactly according to the Euler equations, while only higher frequency terms need some small artificial viscosity for stabilization. 
With increasing mesh resolution, the cutoff will also increase, enabling the simulation of even higher frequency features without any artificial viscosity. 
The quantities $m_N$ and $\vis_{N}$ are required to scale as
\begin{equation}
    m_N \sim N^{\theta},\quad \vis_{N} = \frac{\vis_{0}}{N^{2s_{\mathrm{hv}}-1}},\quad 0 < \theta < \frac{2s_{\mathrm{hv}}-1}{2s_{\mathrm{hv}}}.
\end{equation}
These requirements also motivate the presence of the hyperviscosity parameter $s_{\mathrm{hv}}$, as increasing it enables the cutoff $m_N$ to be chosen higher, leading to a larger portion of the Fourier modes being dissipation free.
Equation~\eqref{eq:shv_disc} is solved in Fourier space, where the pressure Poisson equation reduces to pointwise operations on each Fourier mode separately. 
This leads to the expression
\begin{equation}
    \partial_t\hat{\bm{u}}_{\bm{k}} = -\left( \mathbf{1} - \frac{\bm{k}\bm{k}^T}{|\bm{k}|^2} \right) \cdot \hat{\bm{b}}_{\bm{k}} - \vis_{N}\left|\bm{k}\right|^{2s_{\mathrm{hv}}}\hat{Q}_{\bm{k}}\hat{\bm{u}}_{\bm{k}},
\end{equation}
where $\hat{\bm{b}}_{\bm{k}} = \mathsf{i}\bm{k}^T \cdot \hat{\bm{B}}_{\bm{k}}$ and $\bm{B} = \bm{u}\otimes\bm{u}$. 
This time derivative can equivalently be viewed as the $L^2$-projection of the nonlinear term onto divergence-free vector fields.
The spectral hyperviscosity scheme was used to compute $M = 992$ Monte Carlo samples of a reference solution for the TGV at a resolution of $N = 256$ Fourier modes in each dimension. 
For each sample, 401 equispaced snapshots were stored.  
The viscosity was chosen to be $\vis_{N} = 0.01/N^2$, and the spectral hyperviscosity parameter was set to $s_{\mathrm{hv}} = 1.5$. 
Furthermore, $\theta$ was set to $\theta = 2/3$, i.e., to the endpoint of the admissible open range $0 < \theta < (2s_{\mathrm{hv}}-1)/(2s_{\mathrm{hv}}) = 2/3$ for $s_{\mathrm{hv}} = 1.5$, which is formally excluded by the convergence theory in \cite{lanthaler2021statistical}. 
We did not observe an effect of this choice on the reported quantities.
These parameters lead to a computational time of approximately 2.7 GPU minutes per sample on a single NVIDIA GH200 Grace Hopper Superchip of the Alps supercomputer at the Swiss National Supercomputing Centre (CSCS). 
Each node contains four superchips, leading to a total of 11.25 node hours required for the full 992 samples.
For details about the high-performance implementation of this numerical scheme, the reader is referred to \cite{rohner2024efficient}. 
The obtained statistical Euler solution was then downsampled to match the spatial grid of the MC LBM in the error computations reported in Section~\ref{sec:numWassersteinRef}.

\section{Further visualizations of statistical flow fields}\label{appsec:flowfields}

We provide the respective velocity visualizations of single sample strains, the mean, and the standard deviation at times \(t=0, 5, 10, 15, 20\) in Figures~\ref{fig:statSolRTGV_00},~\ref{fig:statSolRTGV_05},~\ref{fig:statSolRTGV_10},~\ref{fig:statSolRTGV_15}~\ref{fig:statSolRTGV_20}, respectively. 

\begin{figure}[ht!]
    %\centering
    \subfloat[Sample, \(R\!e = 1280\)]{\includegraphics[width=0.27\textwidth,trim={0cm 1.5cm 0cm 2cm},clip]{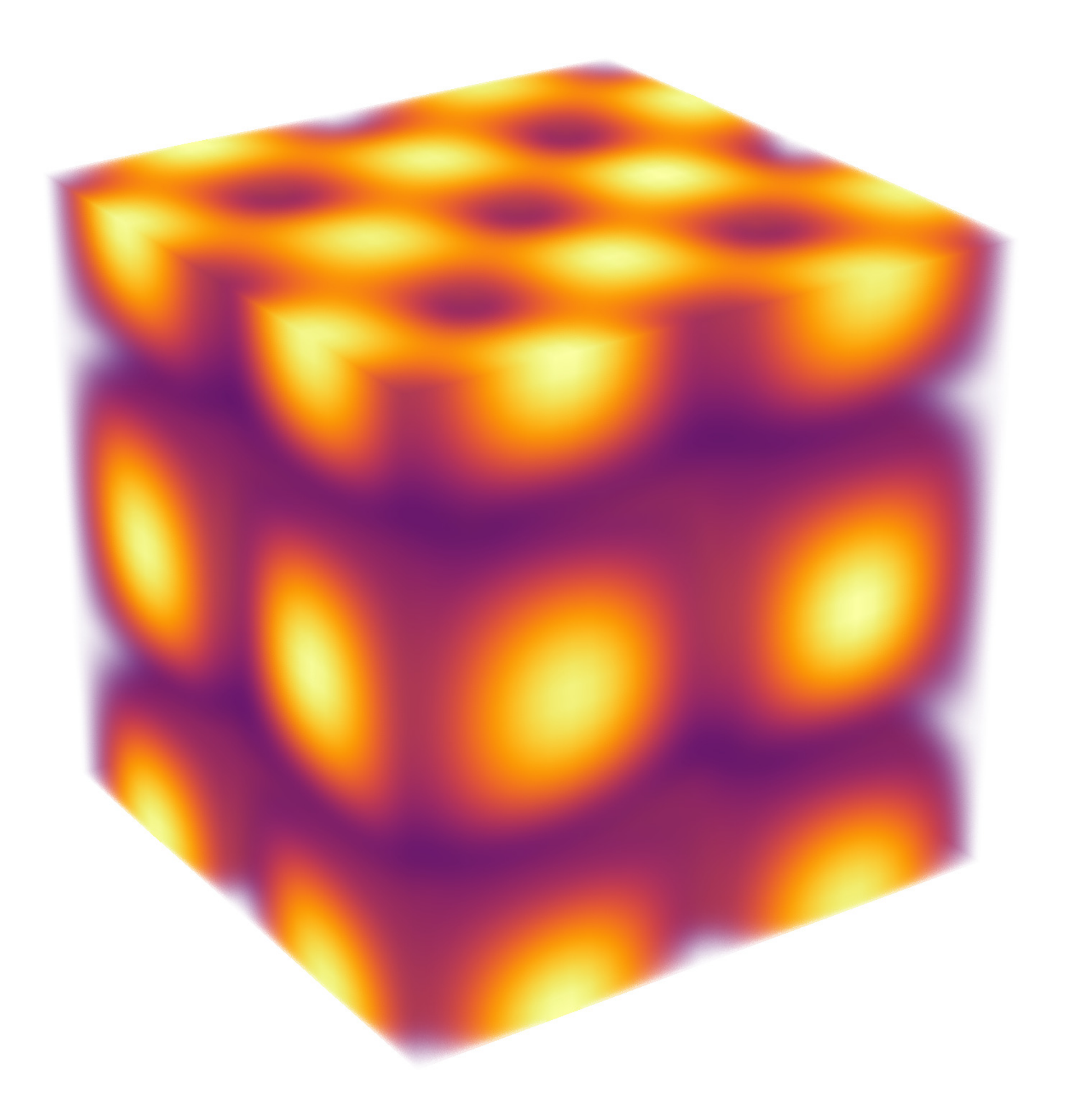}
          \label{subfig:statSolRTGV-sample00-1280}}
    \subfloat[Sample, \(R\!e = 2560\)]{\includegraphics[width=0.27\textwidth,trim={0cm 1.5cm 0cm 2cm},clip]{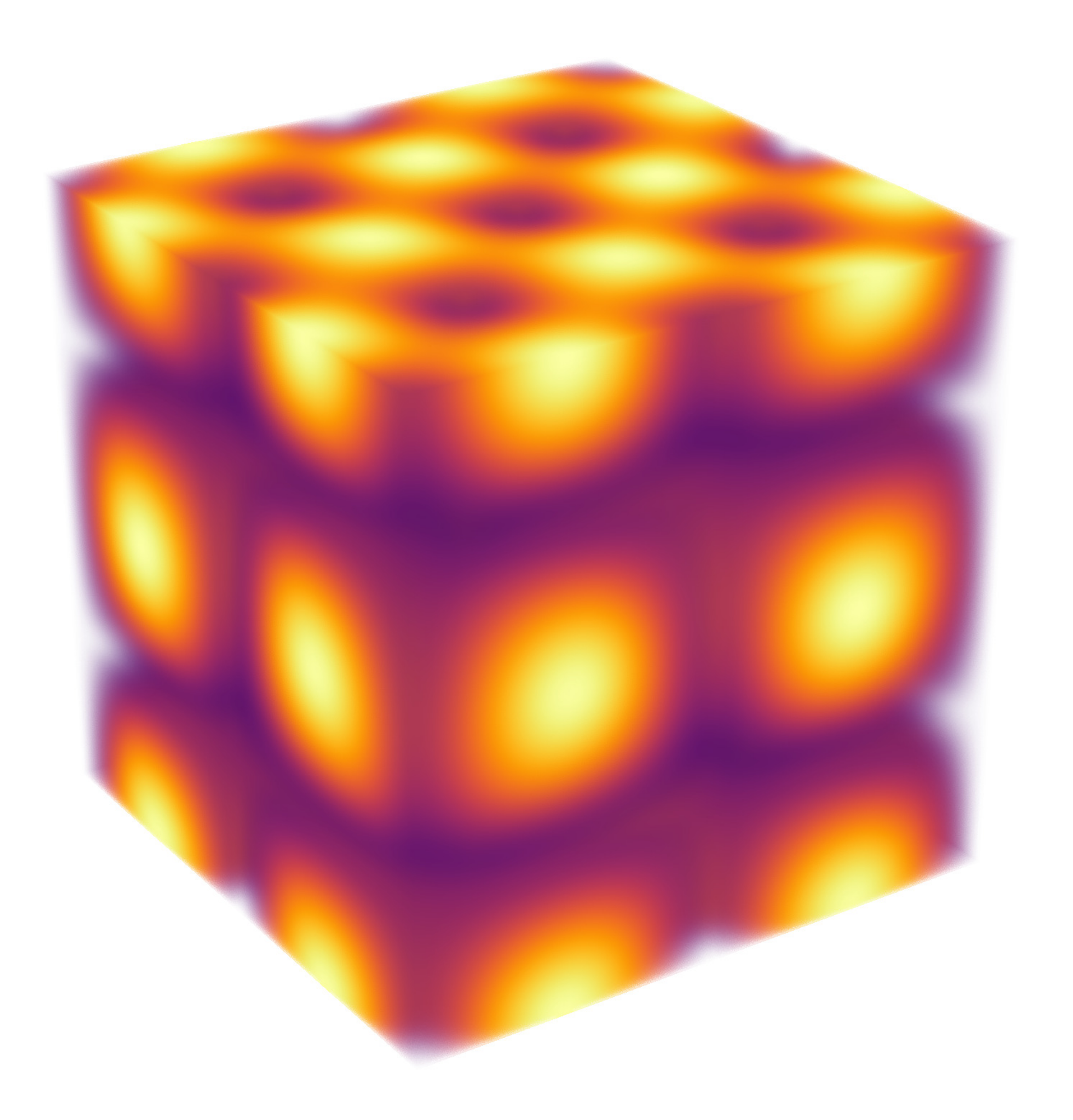}
          \label{subfig:statSolRTGV-sample00-2560}}
   	\subfloat[Sample, \(R\!e = 5120\)]{\includegraphics[width=0.27\textwidth,trim={0cm 1.5cm 0cm 2cm},clip]{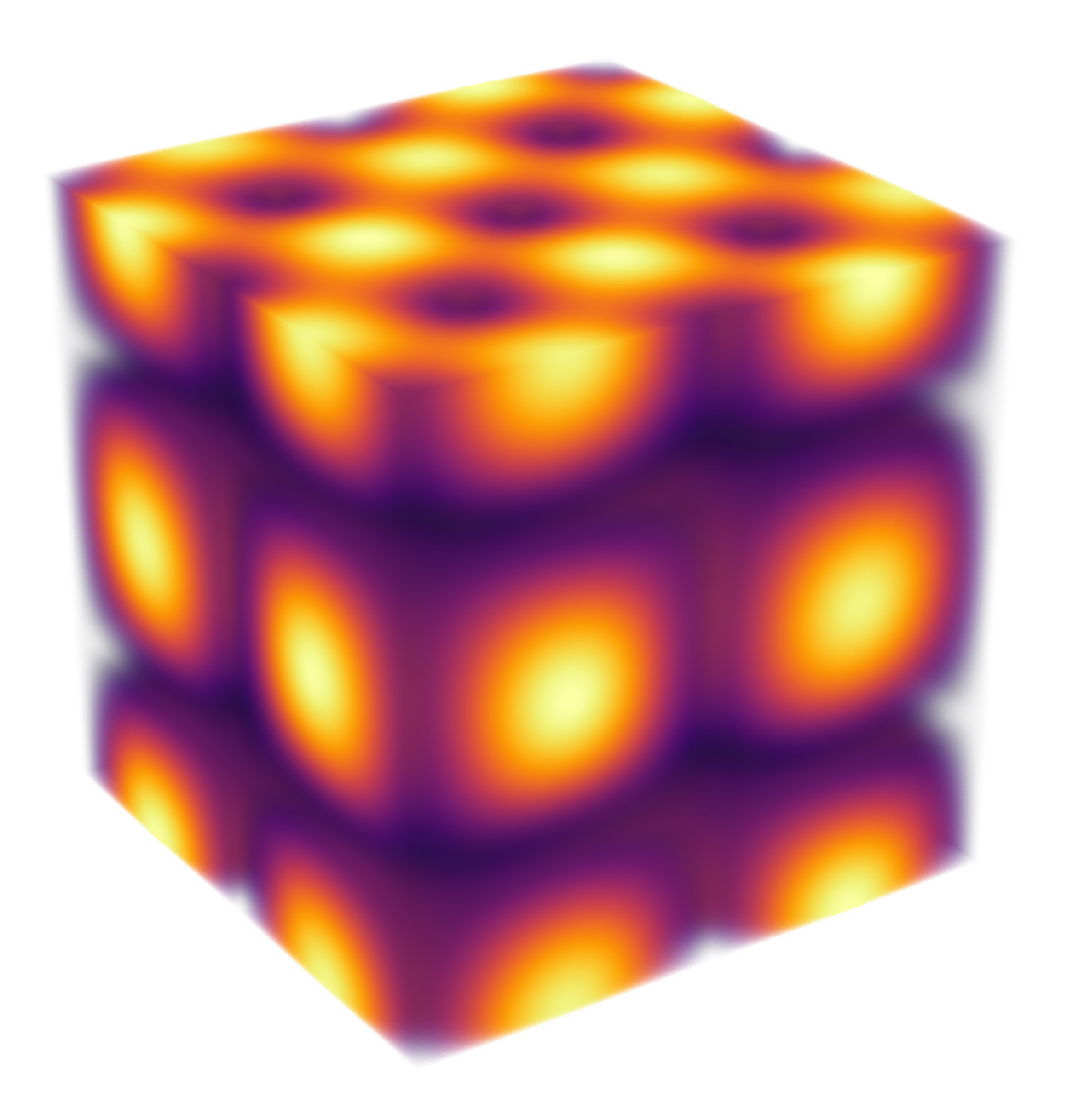}
          \label{subfig:statSolRTGV-sample00-5120}} 
    \hspace{.1em}
    \includegraphics[scale=1]{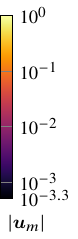}
    \\
    \subfloat[Mean, \(R\!e = 1280\)]{\includegraphics[width=0.27\textwidth,trim={0cm 1.5cm 0cm 2cm},clip]{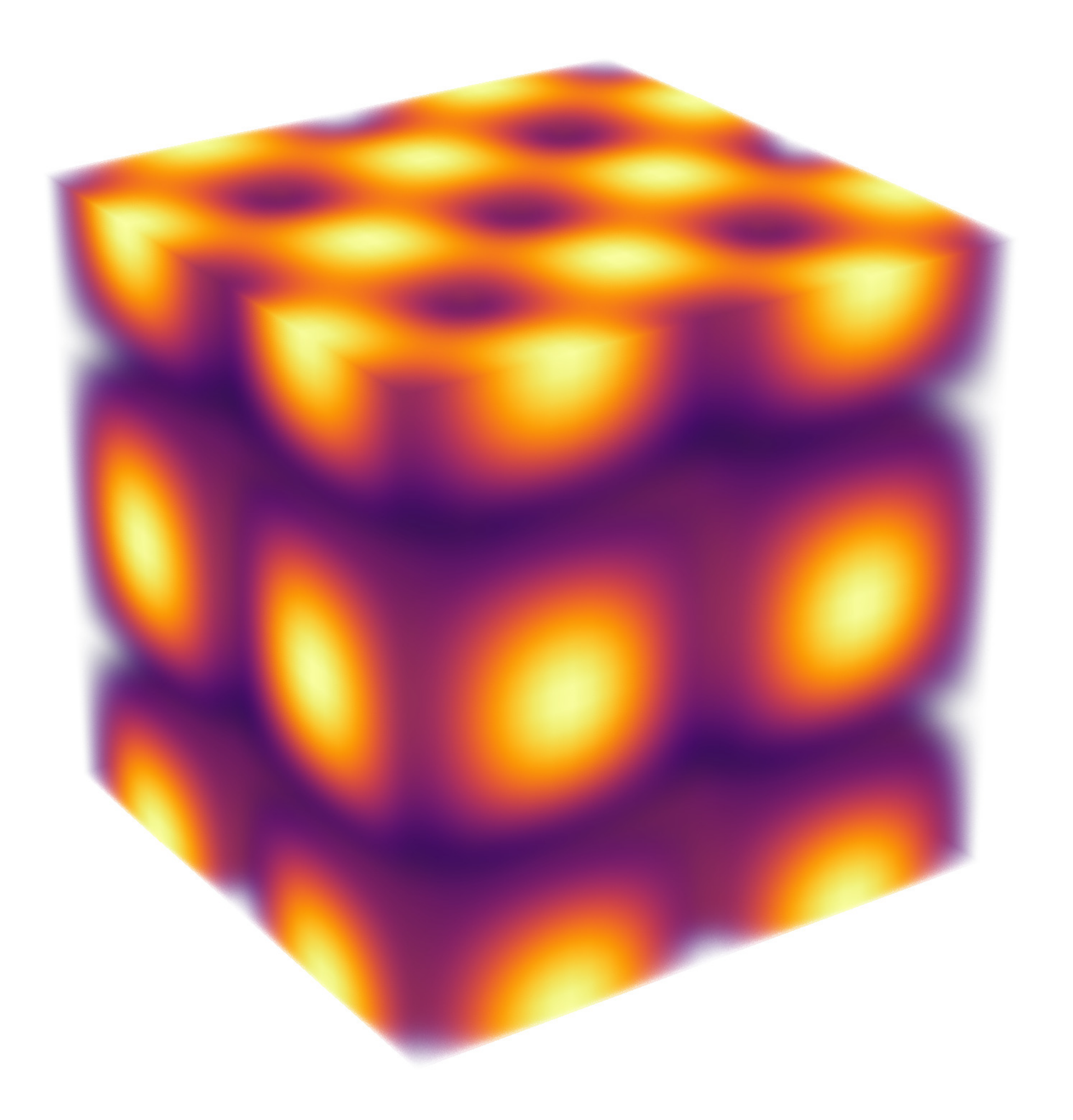}
          \label{subfig:statSolRTGV-mean00-1280}}
	\subfloat[Mean, \(R\!e = 2560\)]{\includegraphics[width=0.27\textwidth,trim={0cm 1.5cm 0cm 2cm},clip]{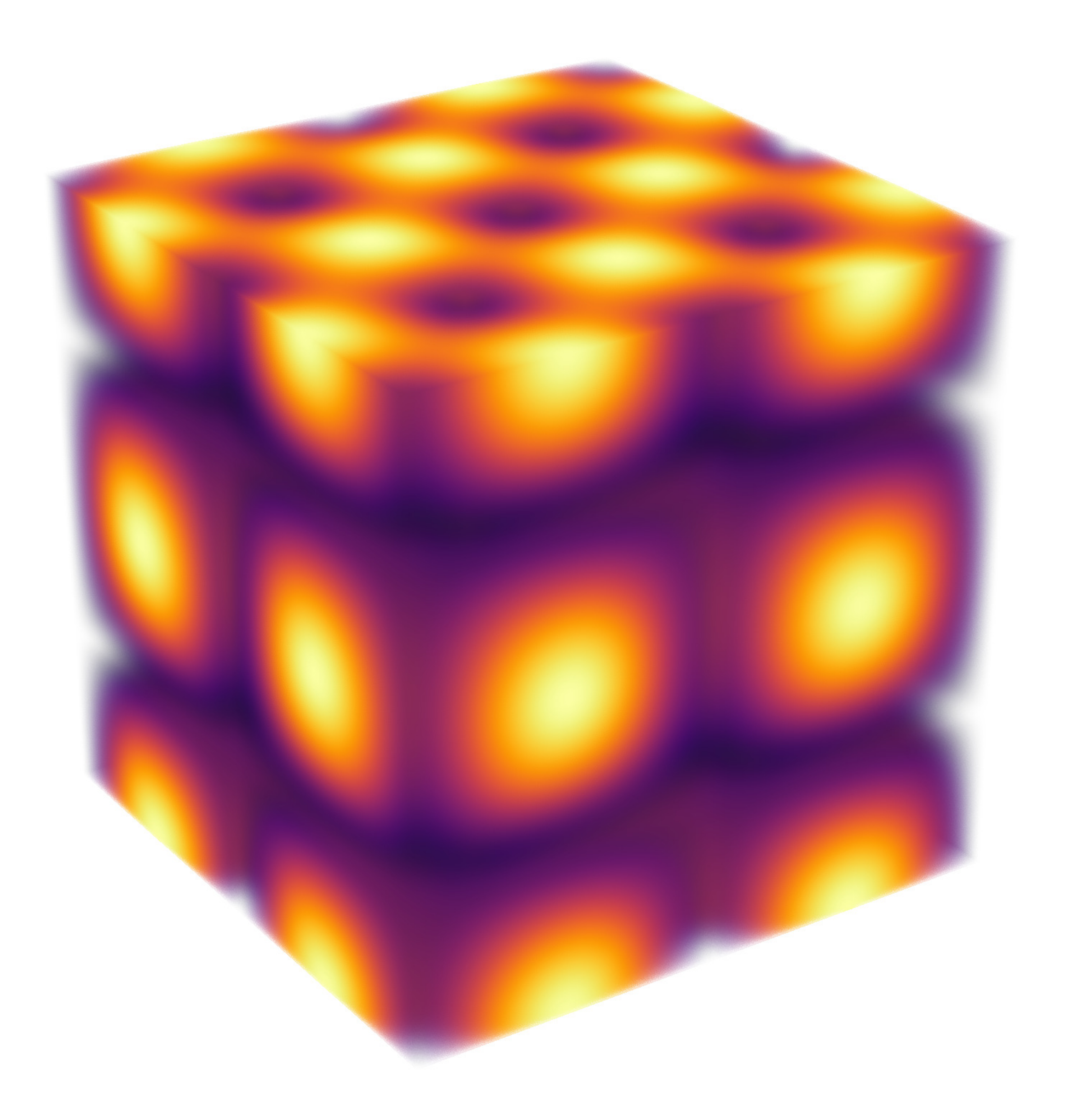}
          \label{subfig:statSolRTGV-mean00-2560}}
    \subfloat[Mean, \(R\!e = 5120\)]{\includegraphics[width=0.27\textwidth,trim={0cm 1.5cm 0cm 2cm},clip]{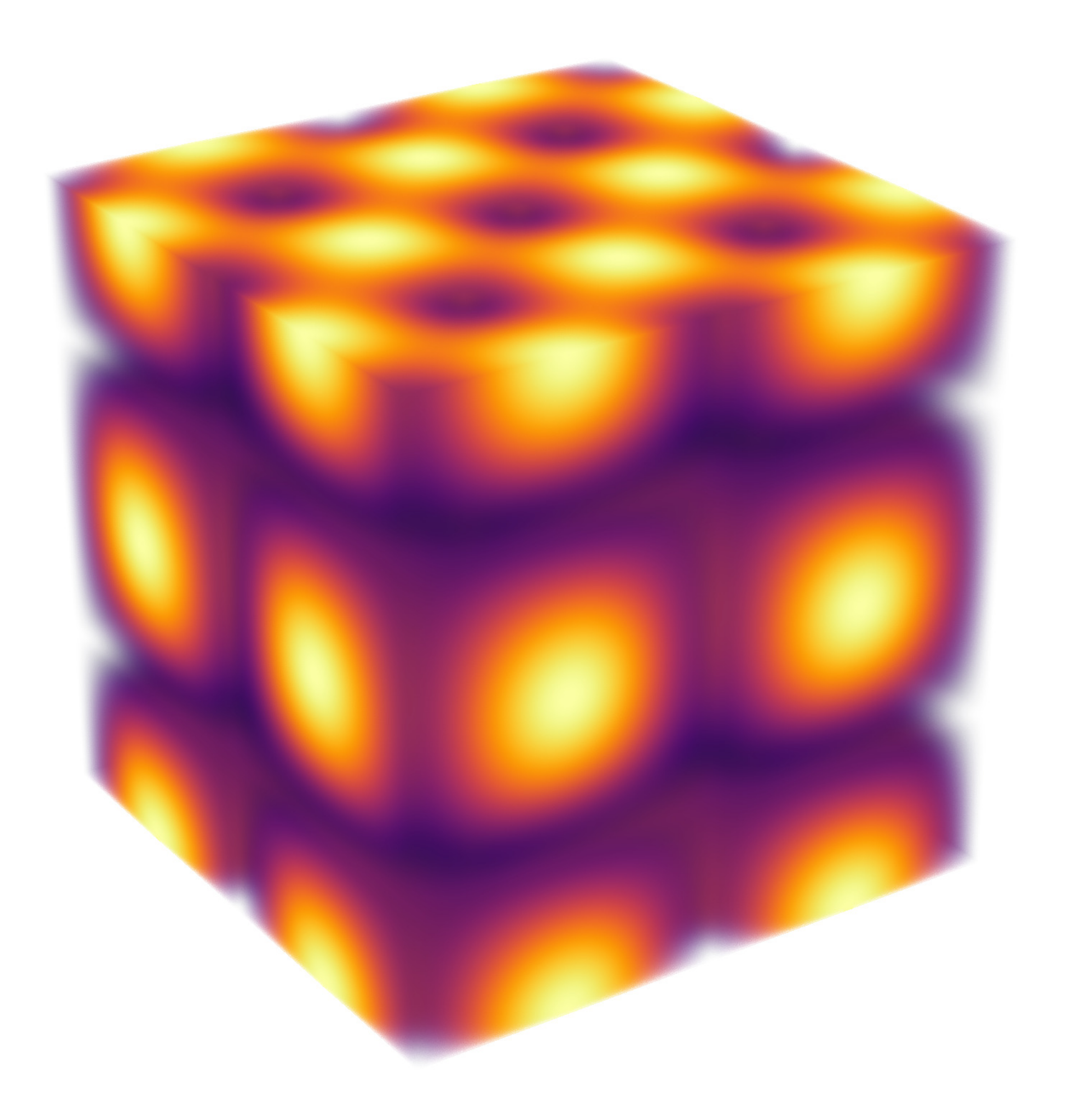}
          \label{subfig:statSolRTGV-mean00-5120}}
    \hspace{.1em}
    \includegraphics[scale=1]{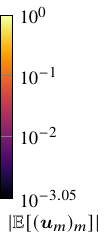}
    \\
    \subfloat[Std, \(R\!e = 1280\)]{\includegraphics[width=0.27\textwidth,trim={0cm 1.5cm 0cm 2cm},clip]{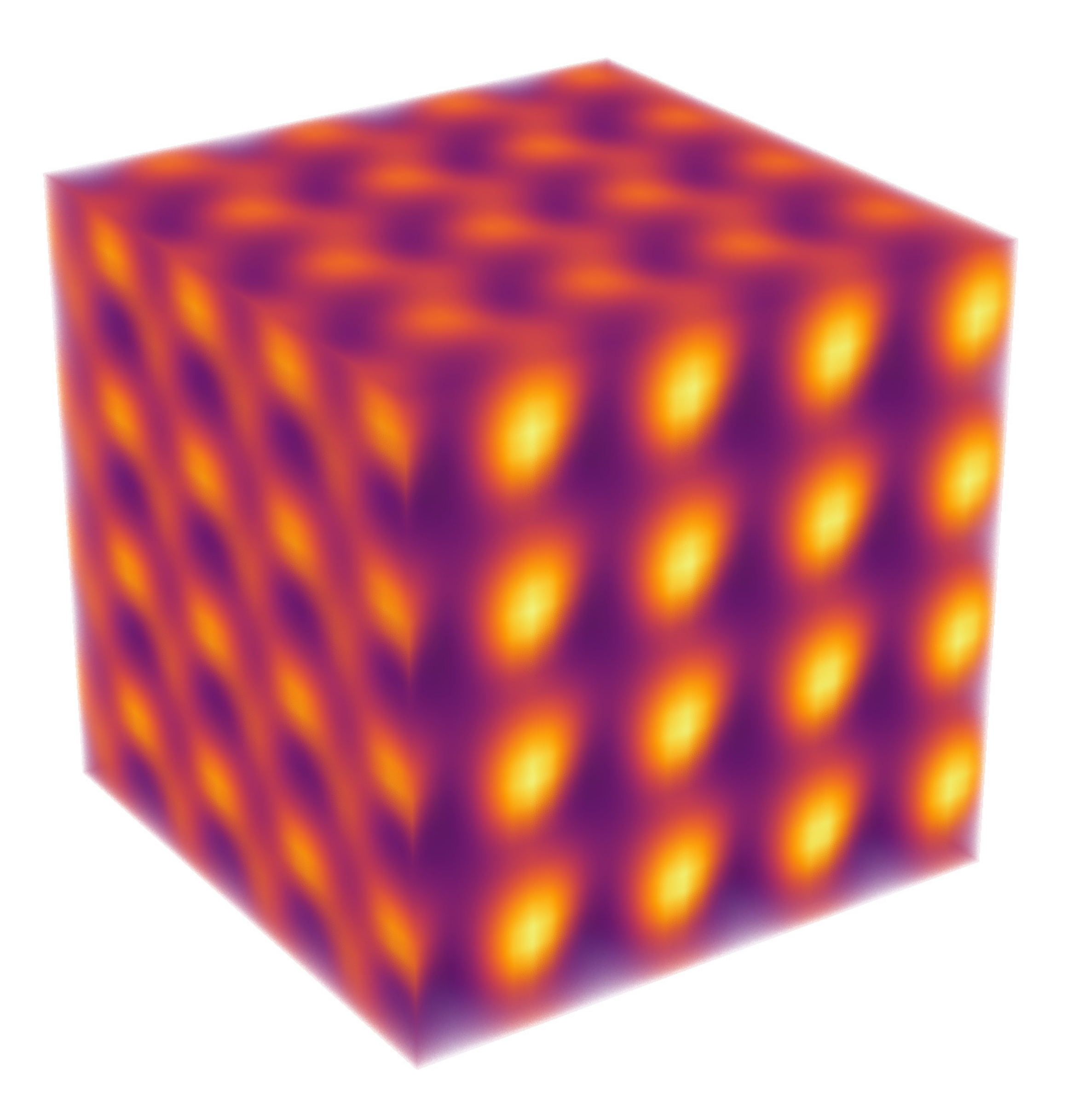}
          \label{subfig:statSolRTGV-std00-1280}}
	\subfloat[Std, \(R\!e = 2560\)]{\includegraphics[width=0.27\textwidth,trim={0cm 1.5cm 0cm 2cm},clip]{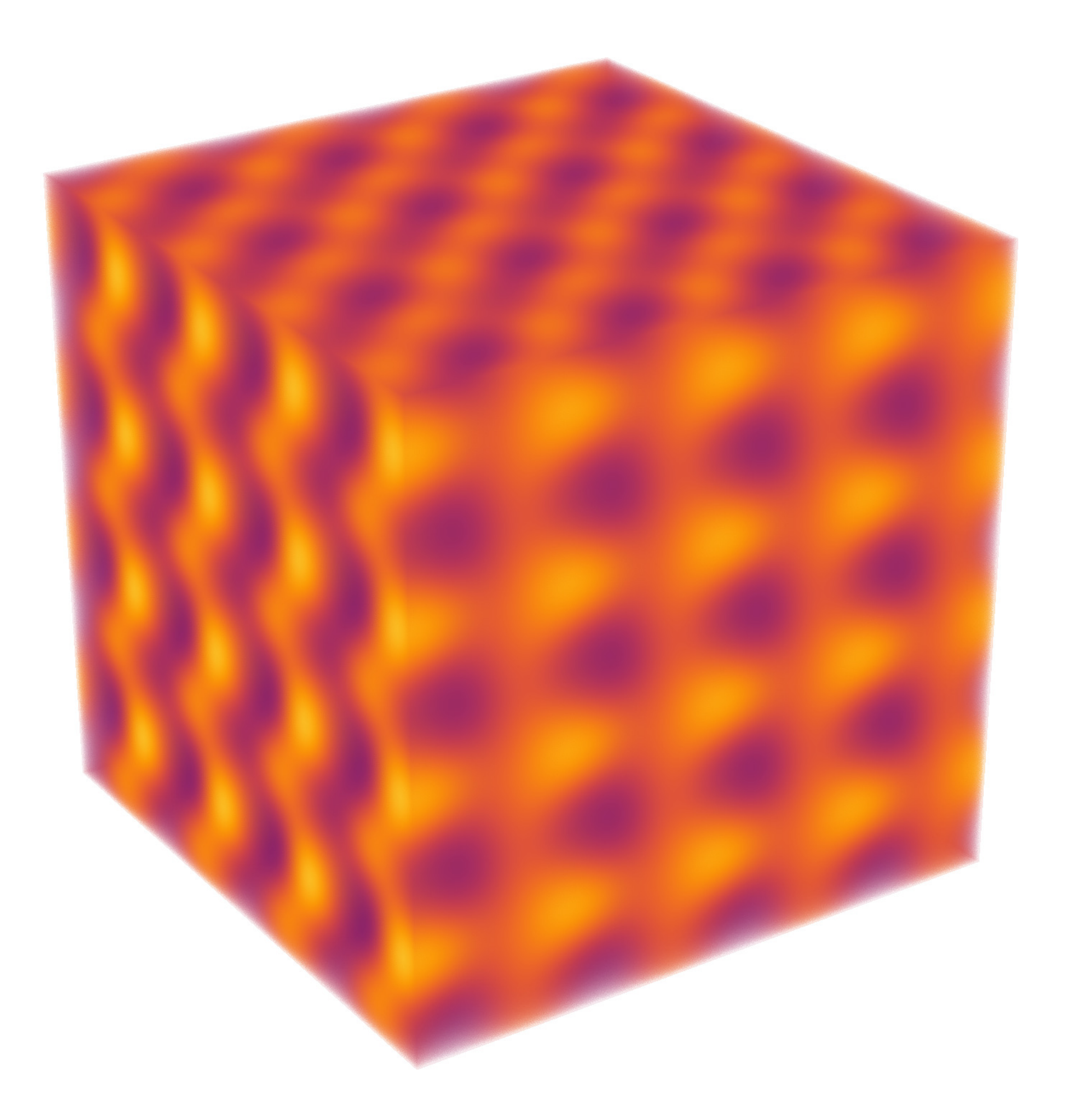}
          \label{subfig:statSolRTGV-std00-2560}}
    \subfloat[Std, \(R\!e = 5120\)]{\includegraphics[width=0.27\textwidth,trim={0cm 1.5cm 0cm 2cm},clip]{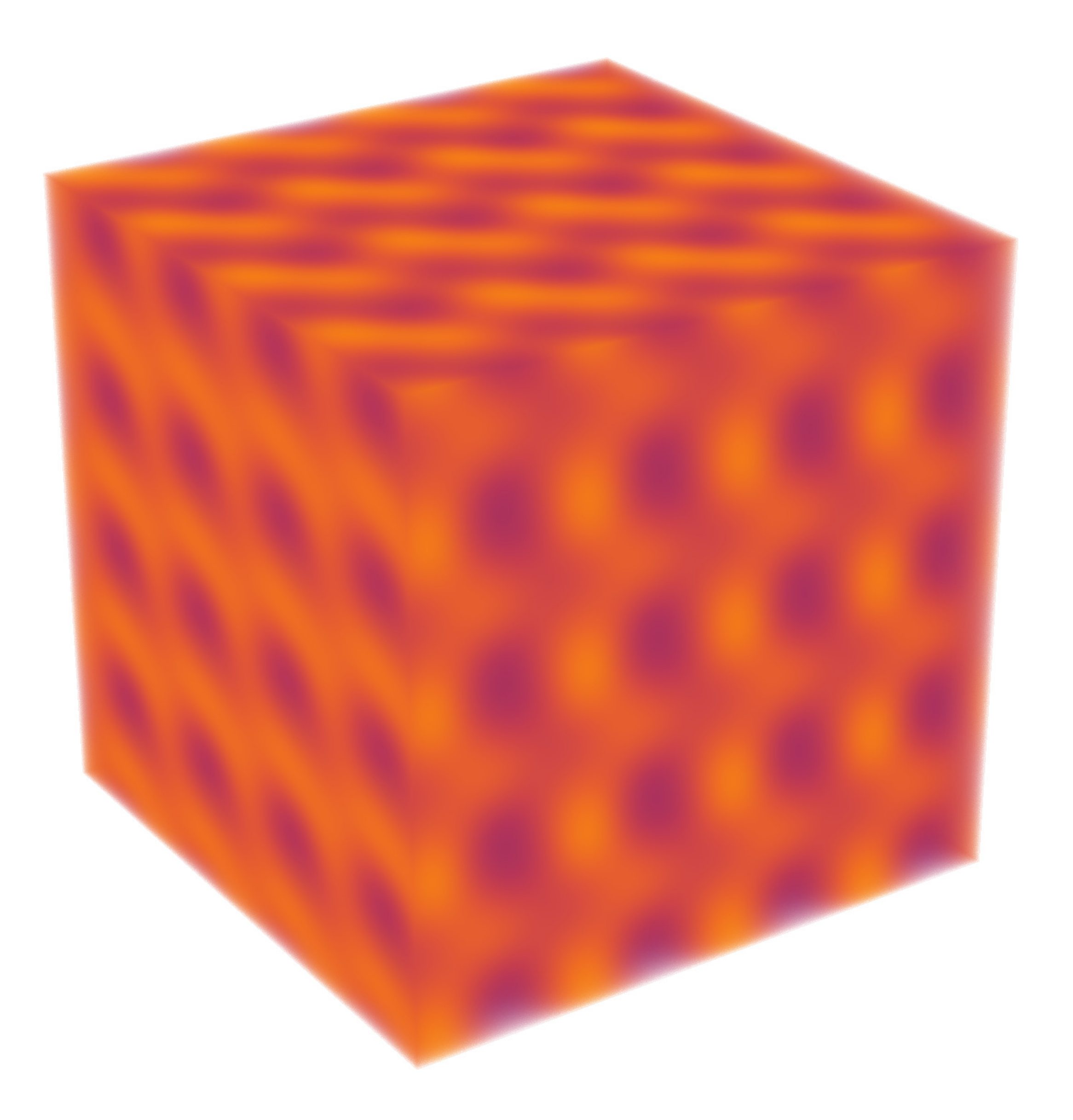}
          \label{subfig:statSolRTGV-std00-5120}}
    \hspace{.1em}
    \includegraphics[scale=1]{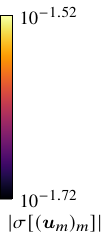}
    \\
	\caption{Iso-volume rendering of the velocity magnitude of the RTGV flow at \(t = 0\mathrm{s}\).  
    Top row: single samples; middle row: mean; bottom row: standard deviation (std); columns from left to right: \(R\!e = 1280, 2560, 5120\). 
    In the renderings, the color map is combined with an opacity transfer function that is linear in the data value, ranging from opacity zero at the minimum to one at the maximum of each colorbar. 
    The colorbars themselves omit this opacity. 
    Discretization parameters from Set~2 (Table~\ref{tab:params2}) are used. 
    }
    \label{fig:statSolRTGV_00}
\end{figure}

\begin{figure}[ht!]
    %\centering
    \subfloat[Sample, \(R\!e = 1280\)]{\includegraphics[width=0.27\textwidth,trim={0cm 1.5cm 0cm 2cm},clip]{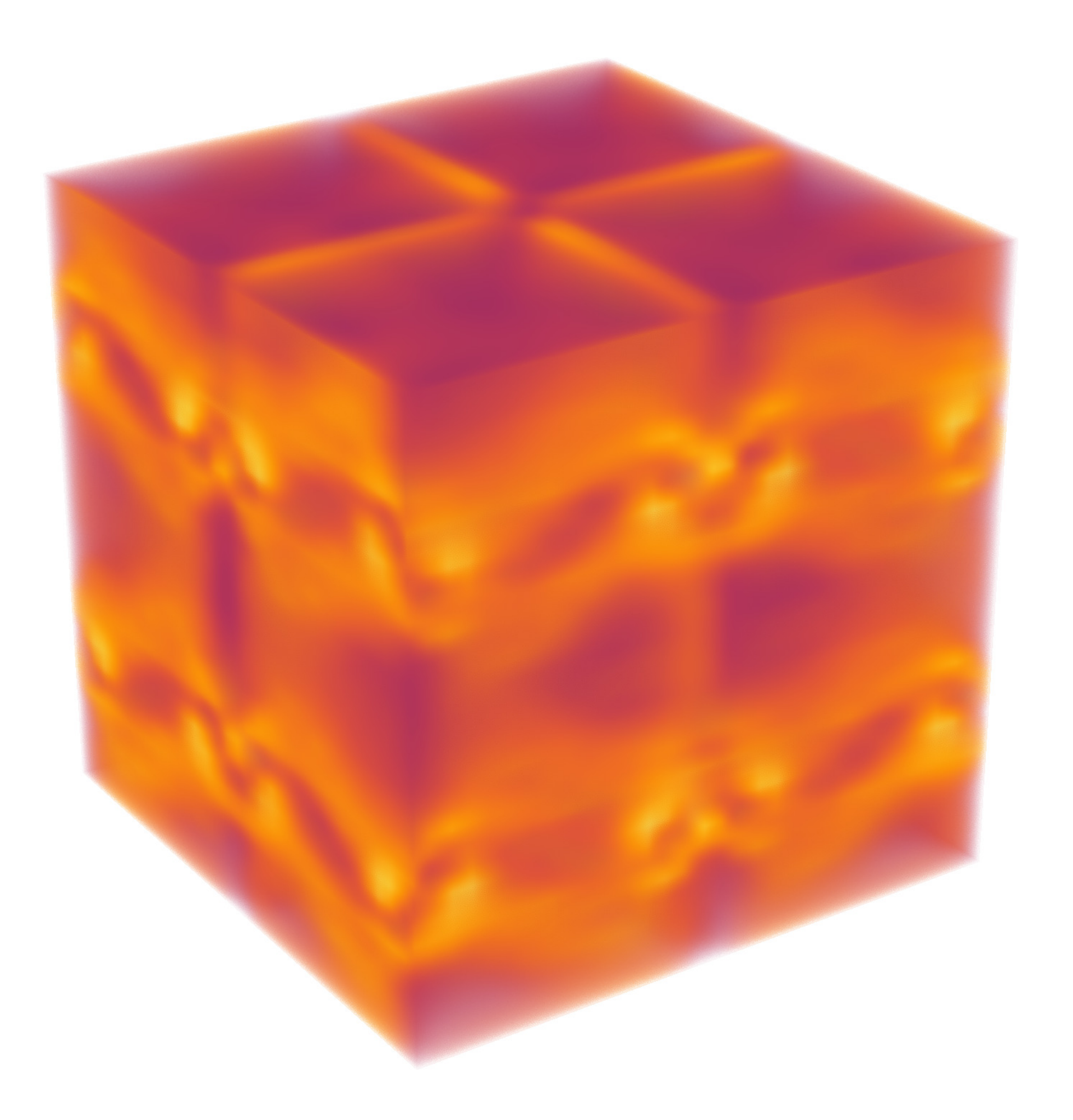}
          \label{subfig:statSolRTGV-sample05-1280}}
    \subfloat[Sample, \(R\!e = 2560\)]{\includegraphics[width=0.27\textwidth,trim={0cm 1.5cm 0cm 2cm},clip]{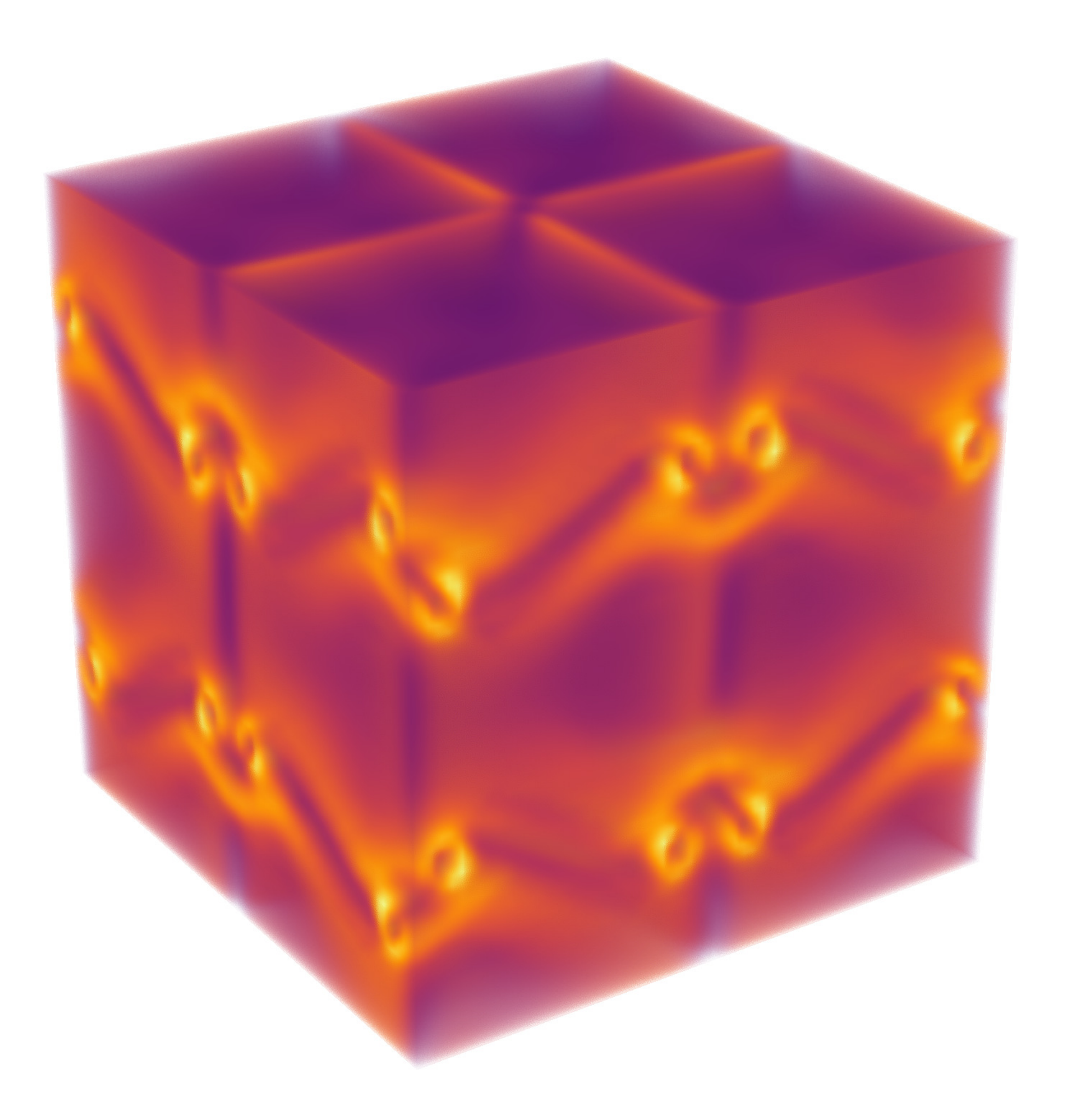}
          \label{subfig:statSolRTGV-sample05-2560}}
   	\subfloat[Sample, \(R\!e = 5120\)]{\includegraphics[width=0.27\textwidth,trim={0cm 1.5cm 0cm 2cm},clip]{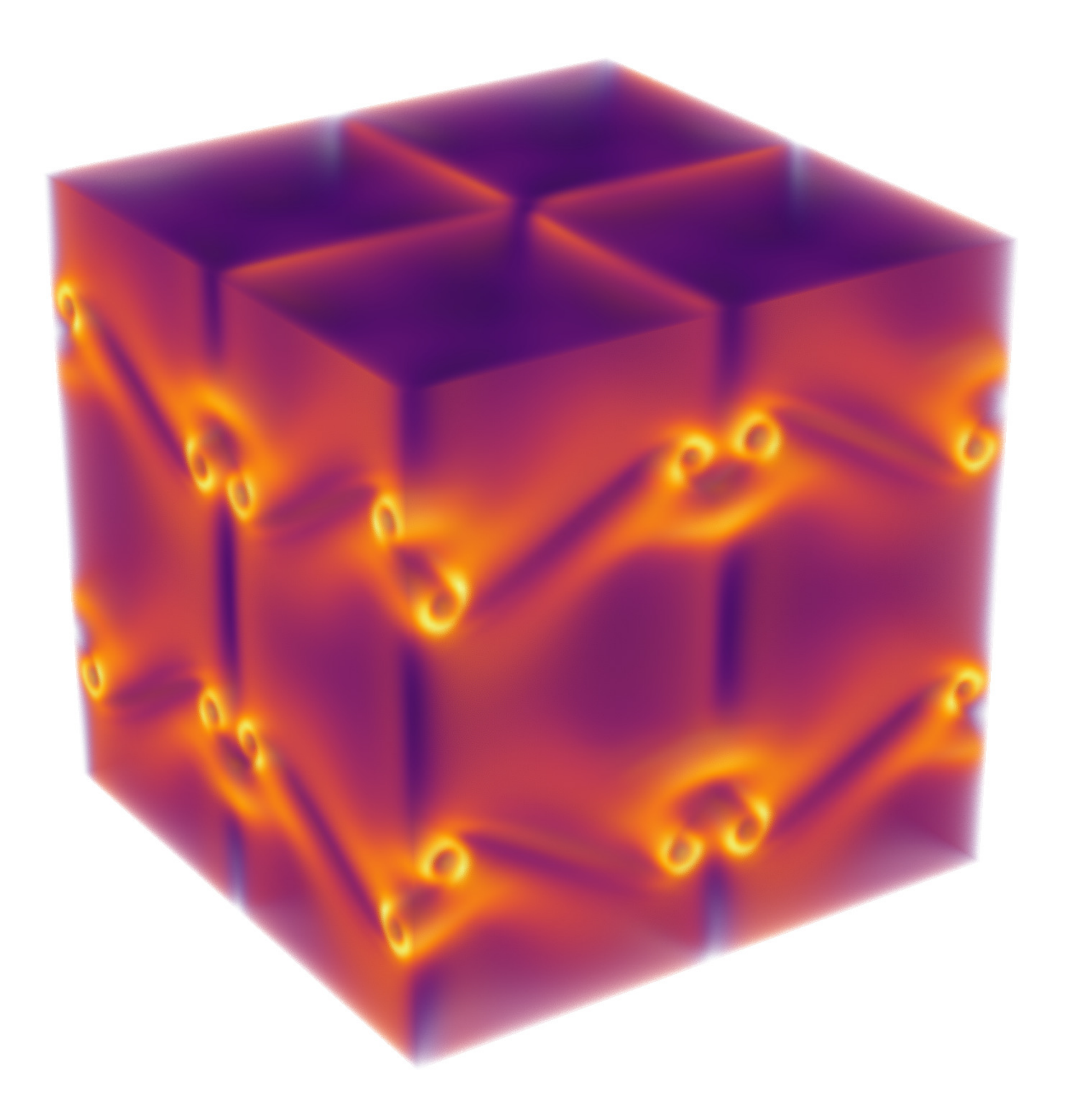}
          \label{subfig:statSolRTGV-sample05-5120}} 
    \hspace{.1em}
    \includegraphics[scale=1]{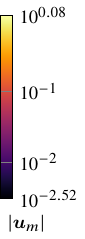}
    \\
    \subfloat[Mean, \(R\!e = 1280\)]{\includegraphics[width=0.27\textwidth,trim={0cm 1.5cm 0cm 2cm},clip]{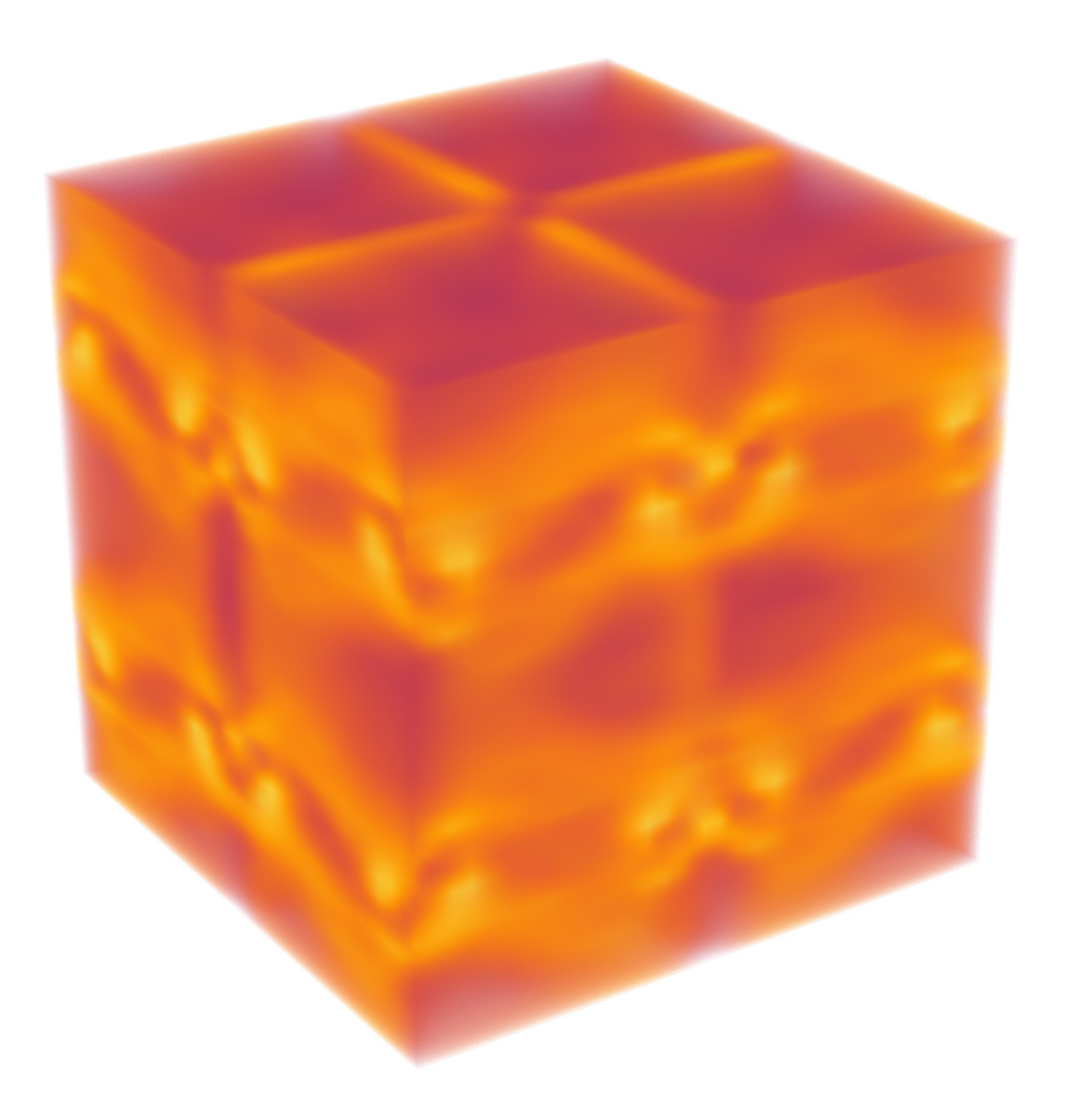}
          \label{subfig:statSolRTGV-mean05-1280}}
	\subfloat[Mean, \(R\!e = 2560\)]{\includegraphics[width=0.27\textwidth,trim={0cm 1.5cm 0cm 2cm},clip]{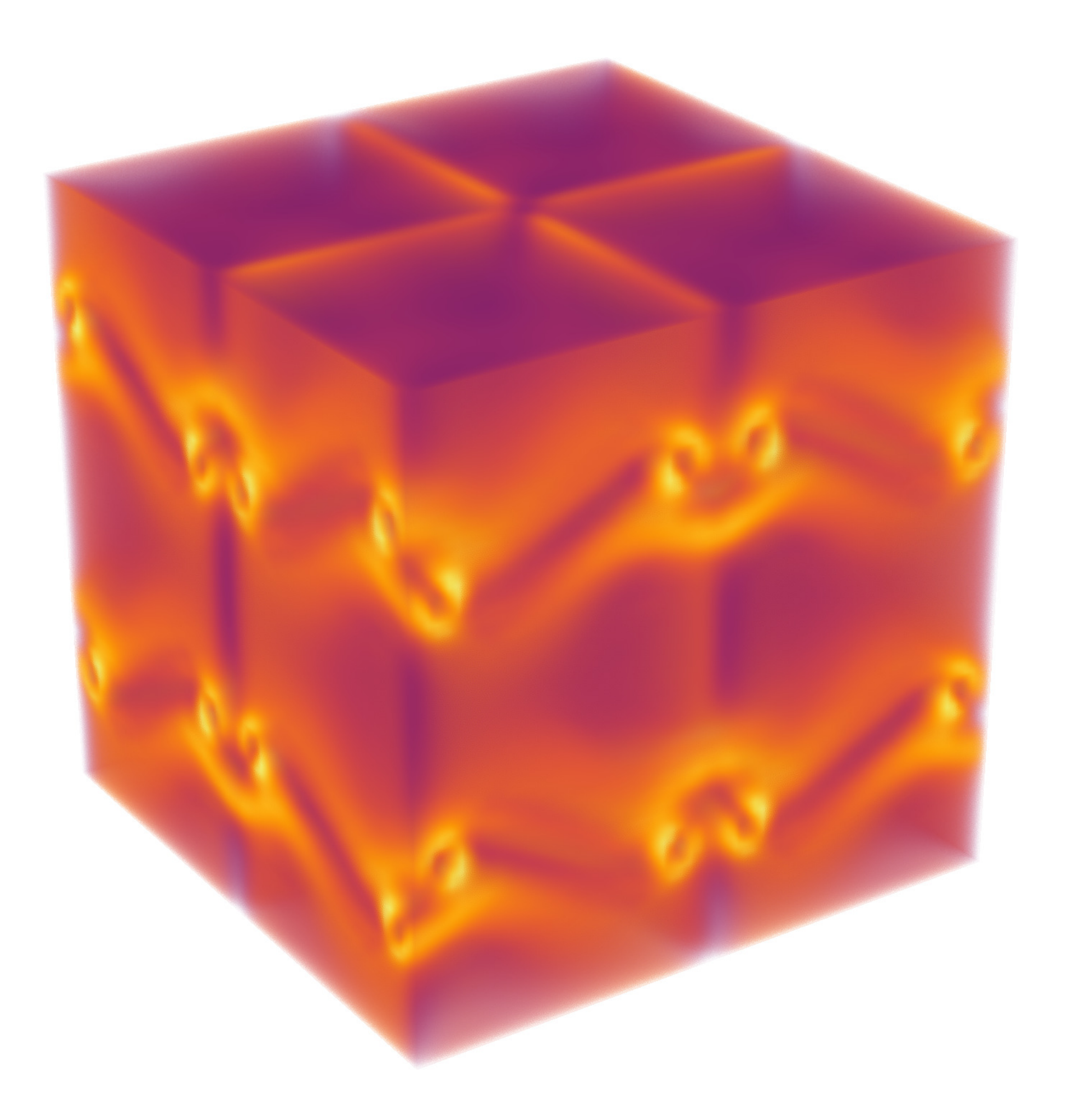}
          \label{subfig:statSolRTGV-mean05-2560}}
    \subfloat[Mean, \(R\!e = 5120\)]{\includegraphics[width=0.27\textwidth,trim={0cm 1.5cm 0cm 2cm},clip]{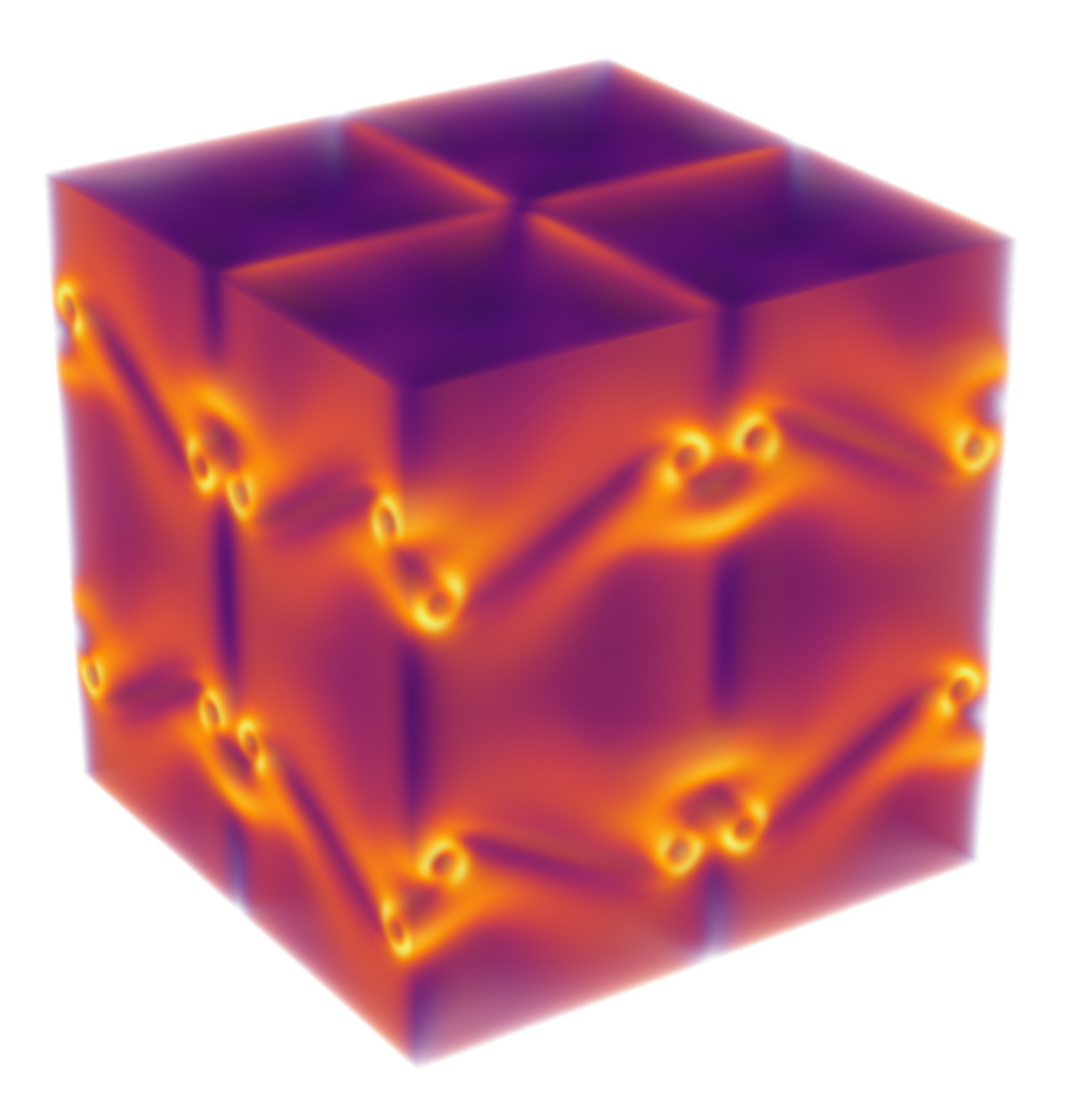}
          \label{subfig:statSolRTGV-mean05-5120}}
    \hspace{.1em}
    \includegraphics[scale=1]{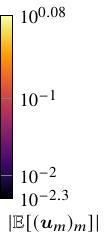}    
    \\
    \subfloat[Std, \(R\!e = 1280\)]{\includegraphics[width=0.27\textwidth,trim={0cm 1.5cm 0cm 2cm},clip]{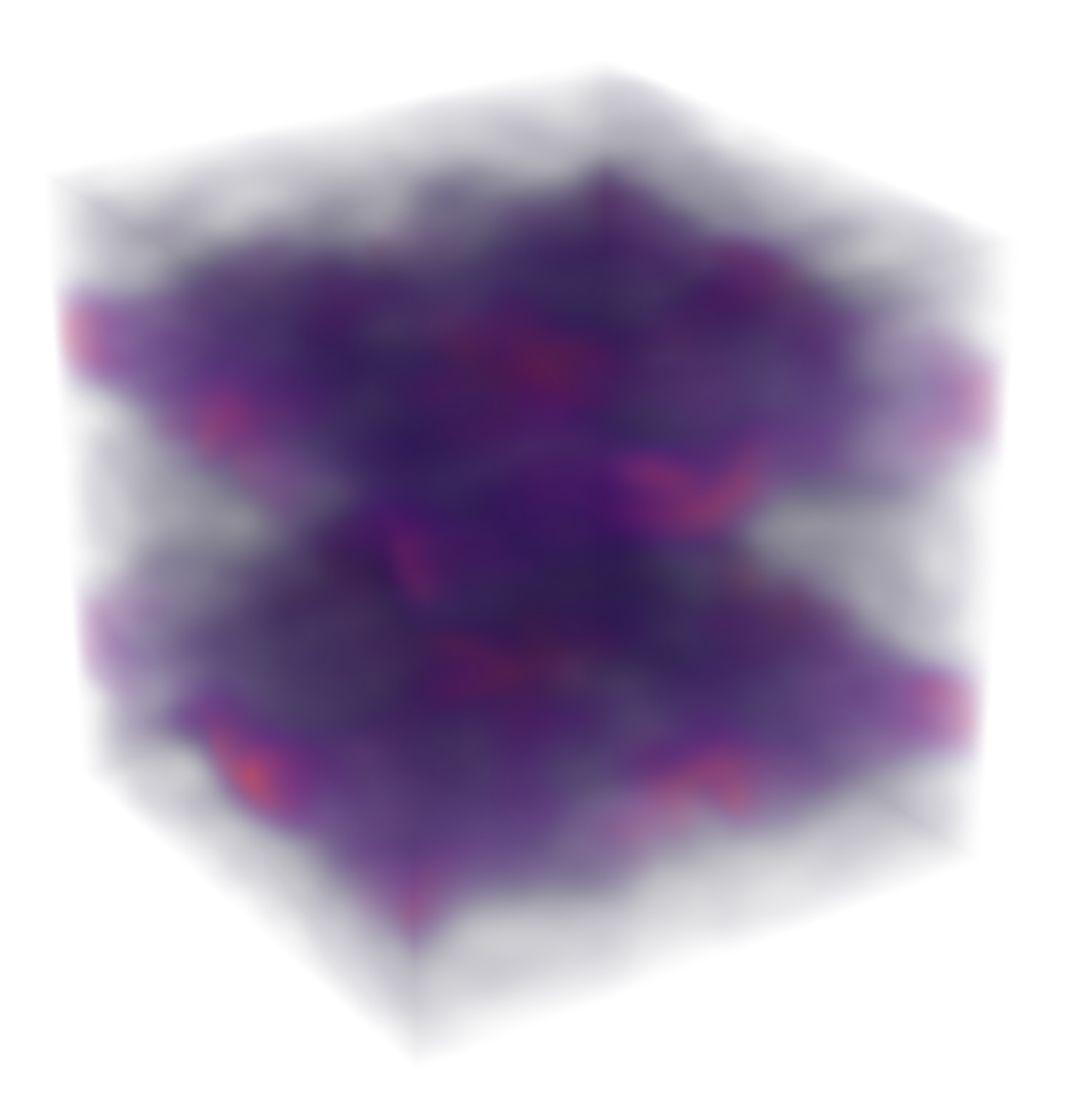}
          \label{subfig:statSolRTGV-std05-1280}}
	\subfloat[Std, \(R\!e = 2560\)]{\includegraphics[width=0.27\textwidth,trim={0cm 1.5cm 0cm 2cm},clip]{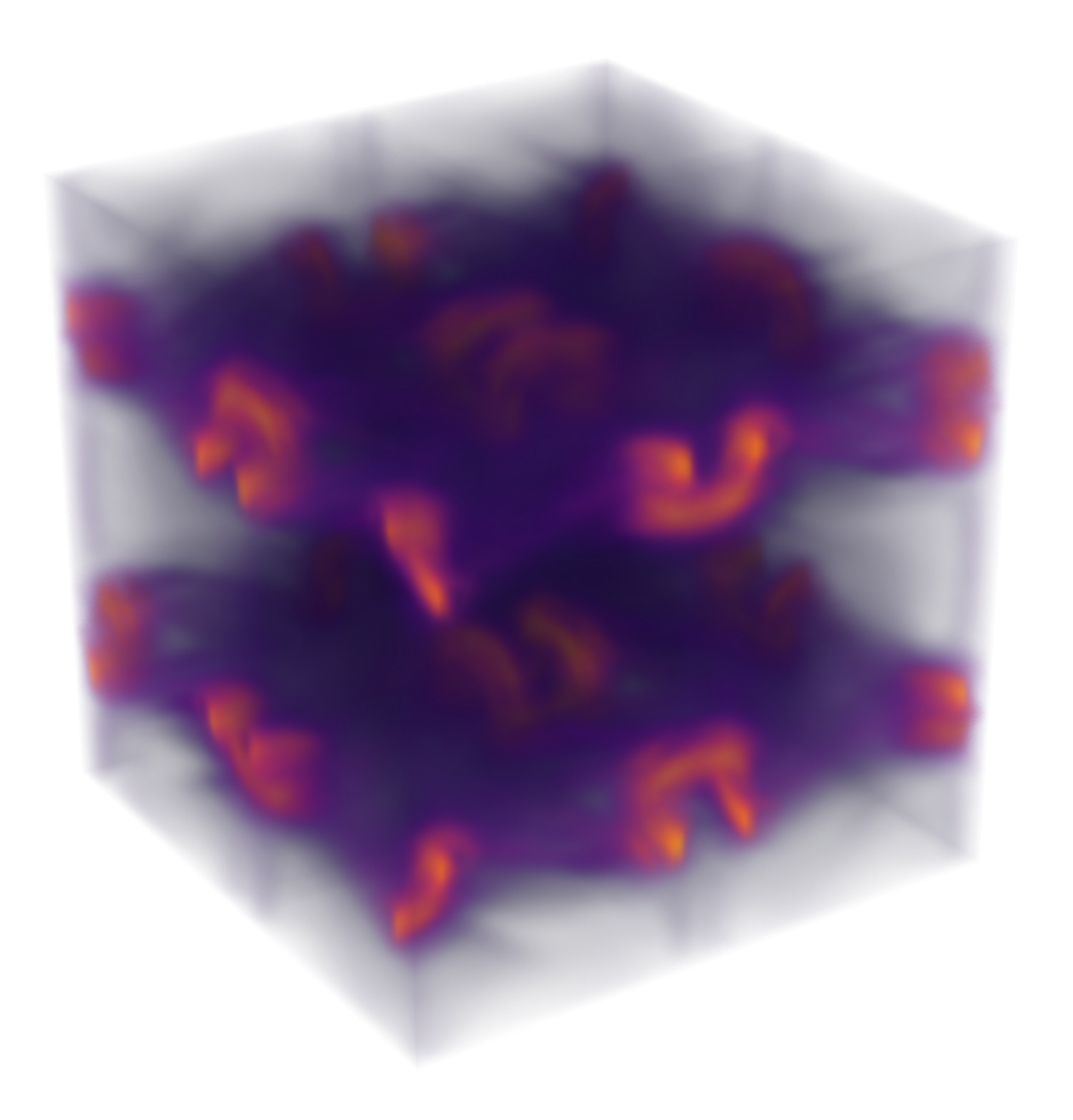}
          \label{subfig:statSolRTGV-std05-2560}}
    \subfloat[Std, \(R\!e = 5120\)]{\includegraphics[width=0.27\textwidth,trim={0cm 1.5cm 0cm 2cm},clip]{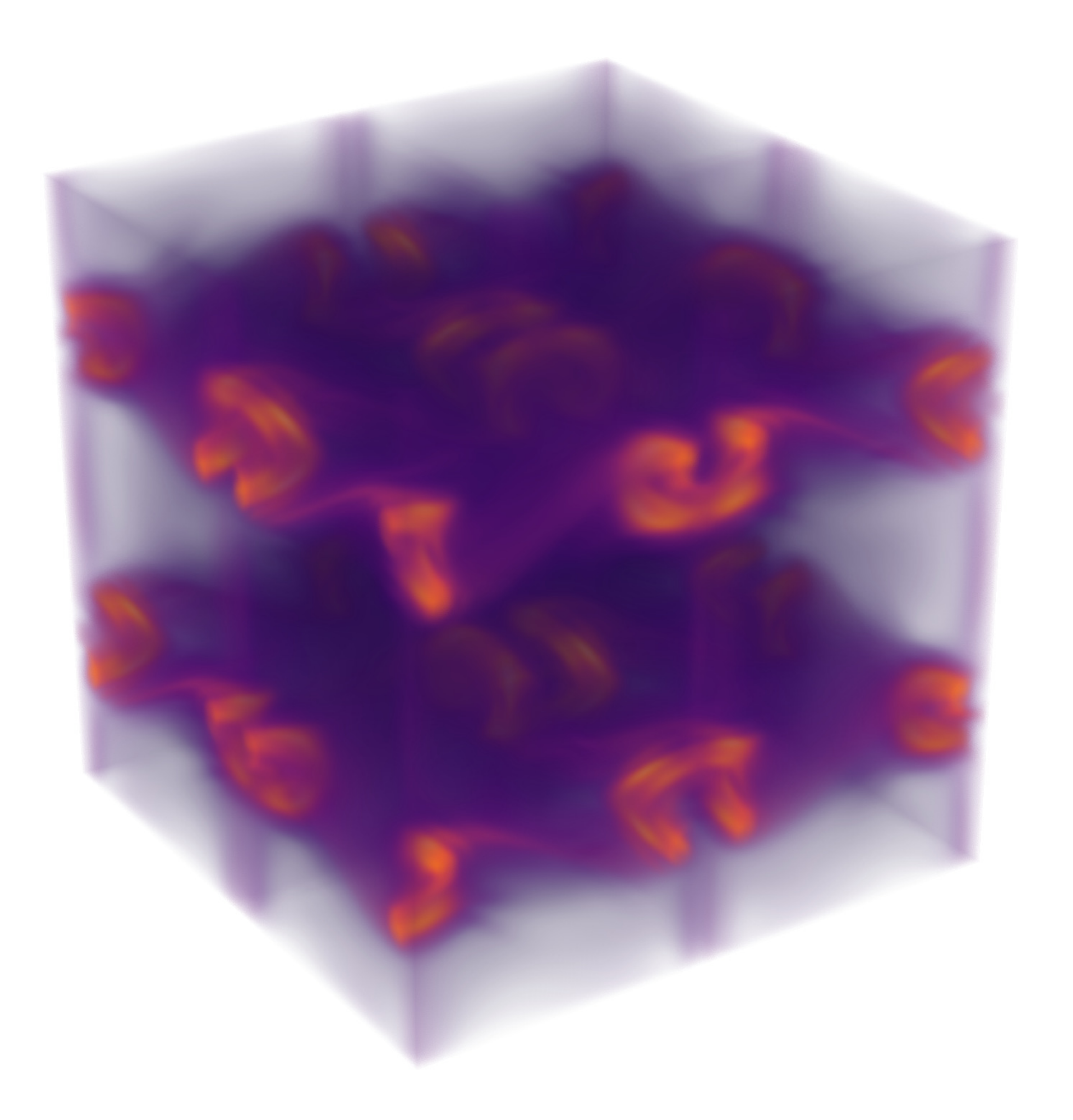}
          \label{subfig:statSolRTGV-std05-5120}}
    \hspace{.1em}
    \includegraphics[scale=1]{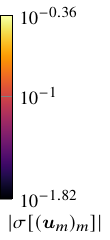}
    \\
	\caption{Iso-volume rendering of the velocity magnitude of the RTGV flow at \(t = 5\mathrm{s}\).  
    Top row: single samples; middle row: mean; bottom row: standard deviation (std); columns from left to right: \(R\!e = 1280, 2560, 5120\). 
    In the renderings, the color map is combined with an opacity transfer function that is linear in the data value, ranging from opacity zero at the minimum to one at the maximum of each colorbar. 
    The colorbars themselves omit this opacity. 
    Discretization parameters from Set~2 (Table~\ref{tab:params2}) are used.}
    \label{fig:statSolRTGV_05}
\end{figure}

\begin{figure}[ht!]
    %\centering
    \subfloat[Sample, \(R\!e = 1280\)]{\includegraphics[width=0.27\textwidth,trim={0cm 1.5cm 0cm 2cm},clip]{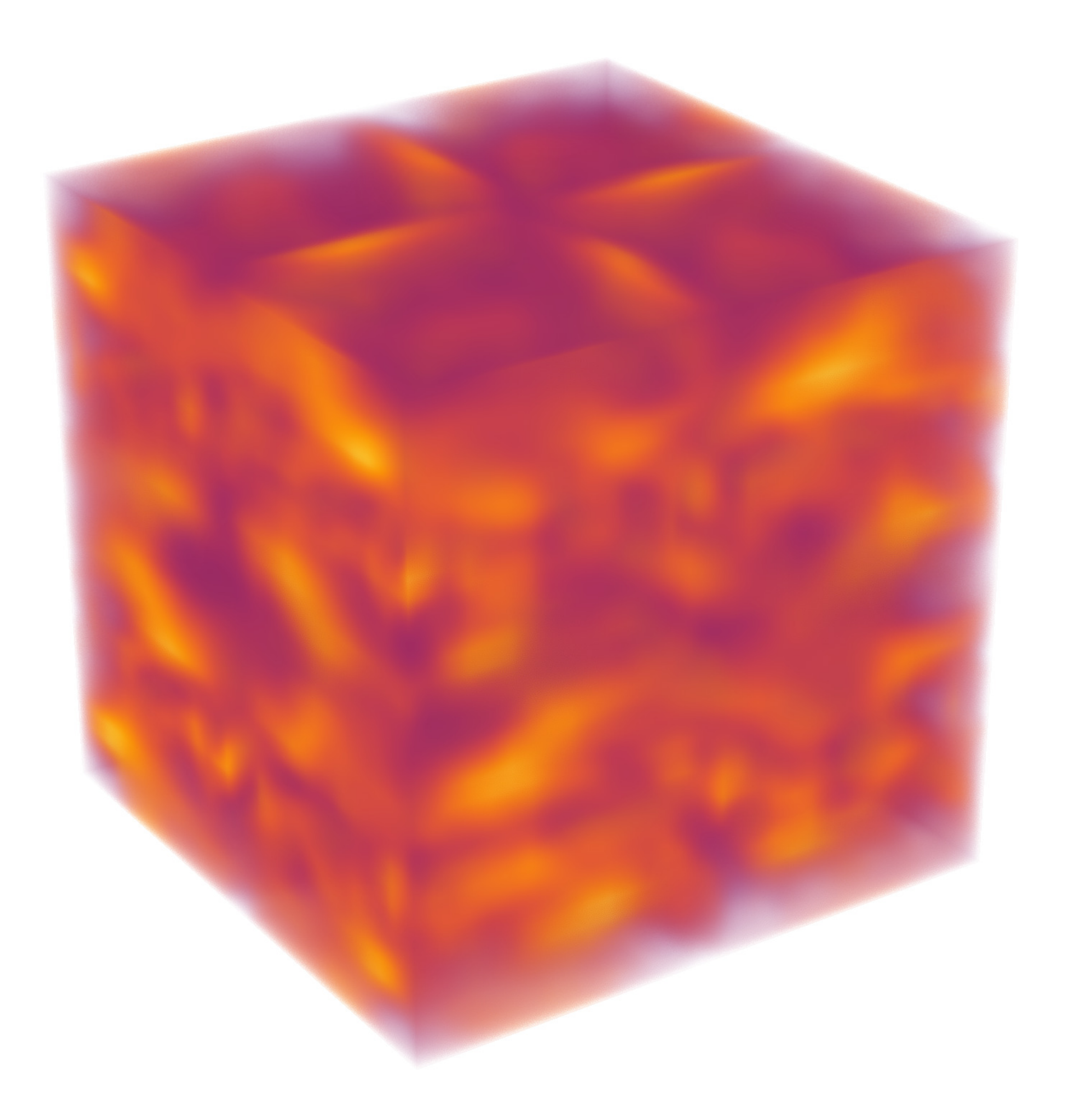}
          \label{subfig:statSolRTGV-sample10-1280}}
    \subfloat[Sample, \(R\!e = 2560\)]{\includegraphics[width=0.27\textwidth,trim={0cm 1.5cm 0cm 2cm},clip]{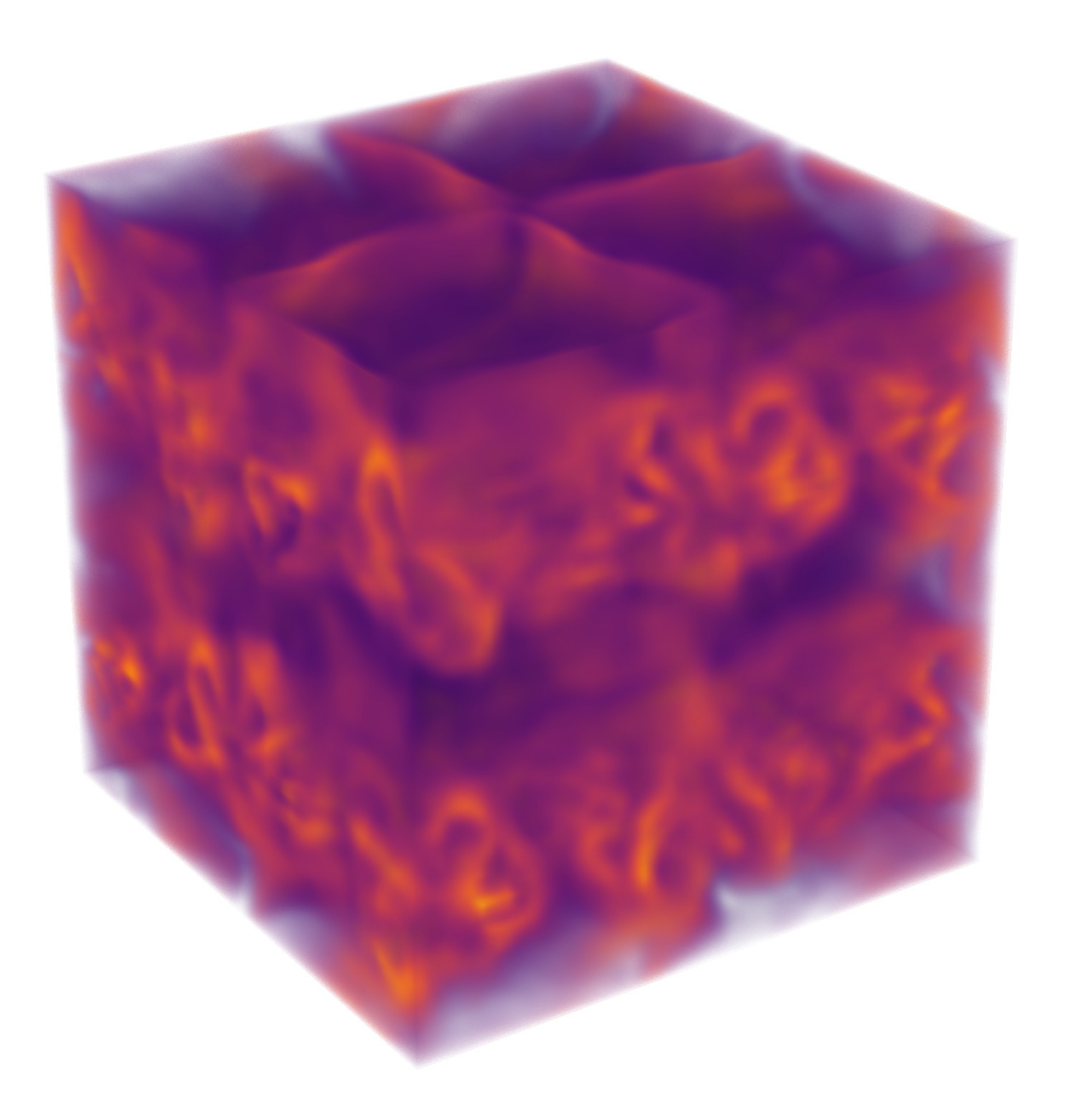}
          \label{subfig:statSolRTGV-sample10-2560}}
   	\subfloat[Sample, \(R\!e = 5120\)]{\includegraphics[width=0.27\textwidth,trim={0cm 1.5cm 0cm 2cm},clip]{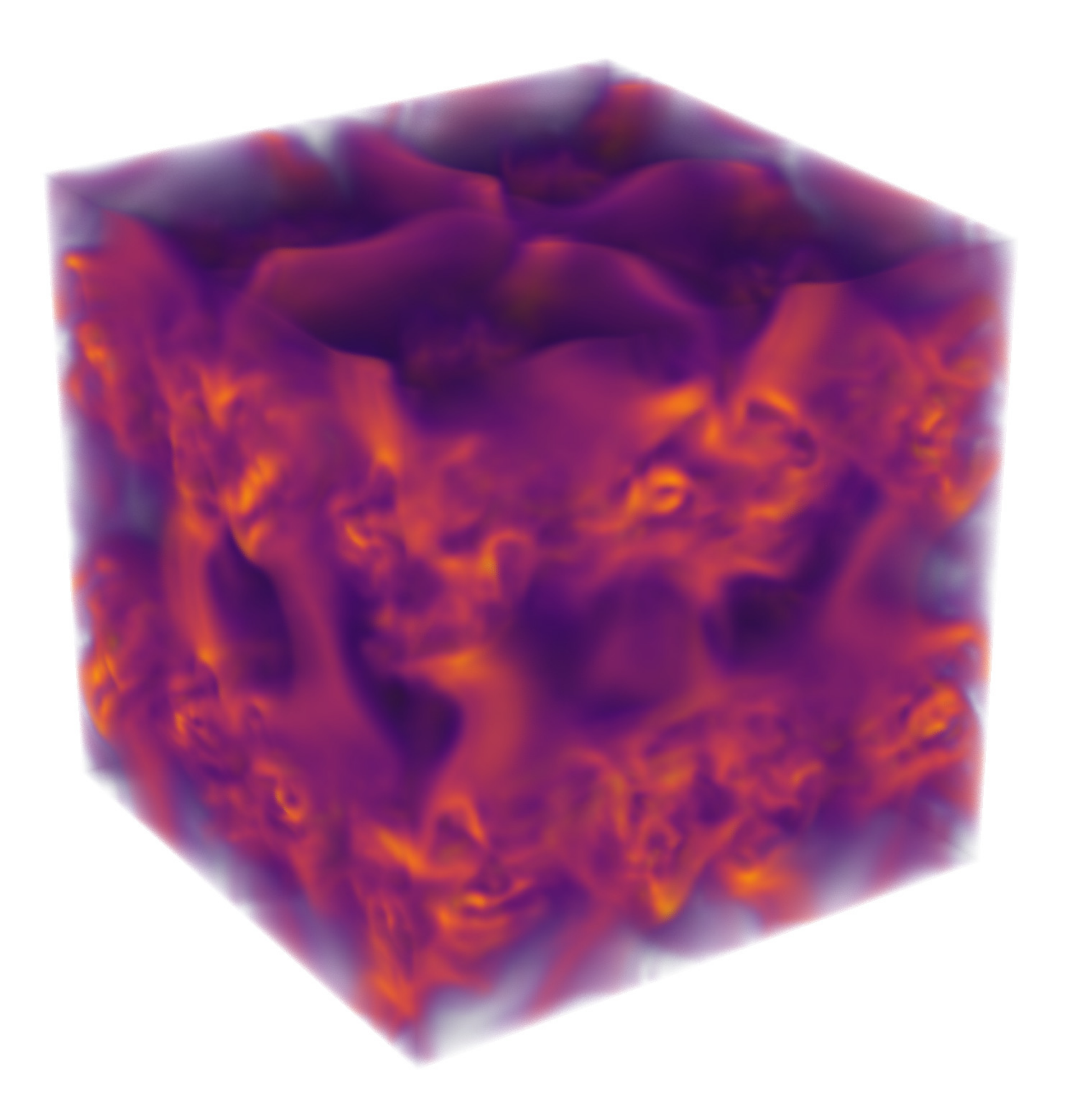}
          \label{subfig:statSolRTGV-sample10-5120}} 
    \hspace{.1em}
    \includegraphics[scale=1]{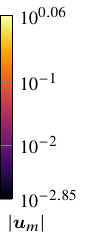}
    \\
    \subfloat[Mean, \(R\!e = 1280\)]{\includegraphics[width=0.27\textwidth,trim={0cm 1.5cm 0cm 2cm},clip]{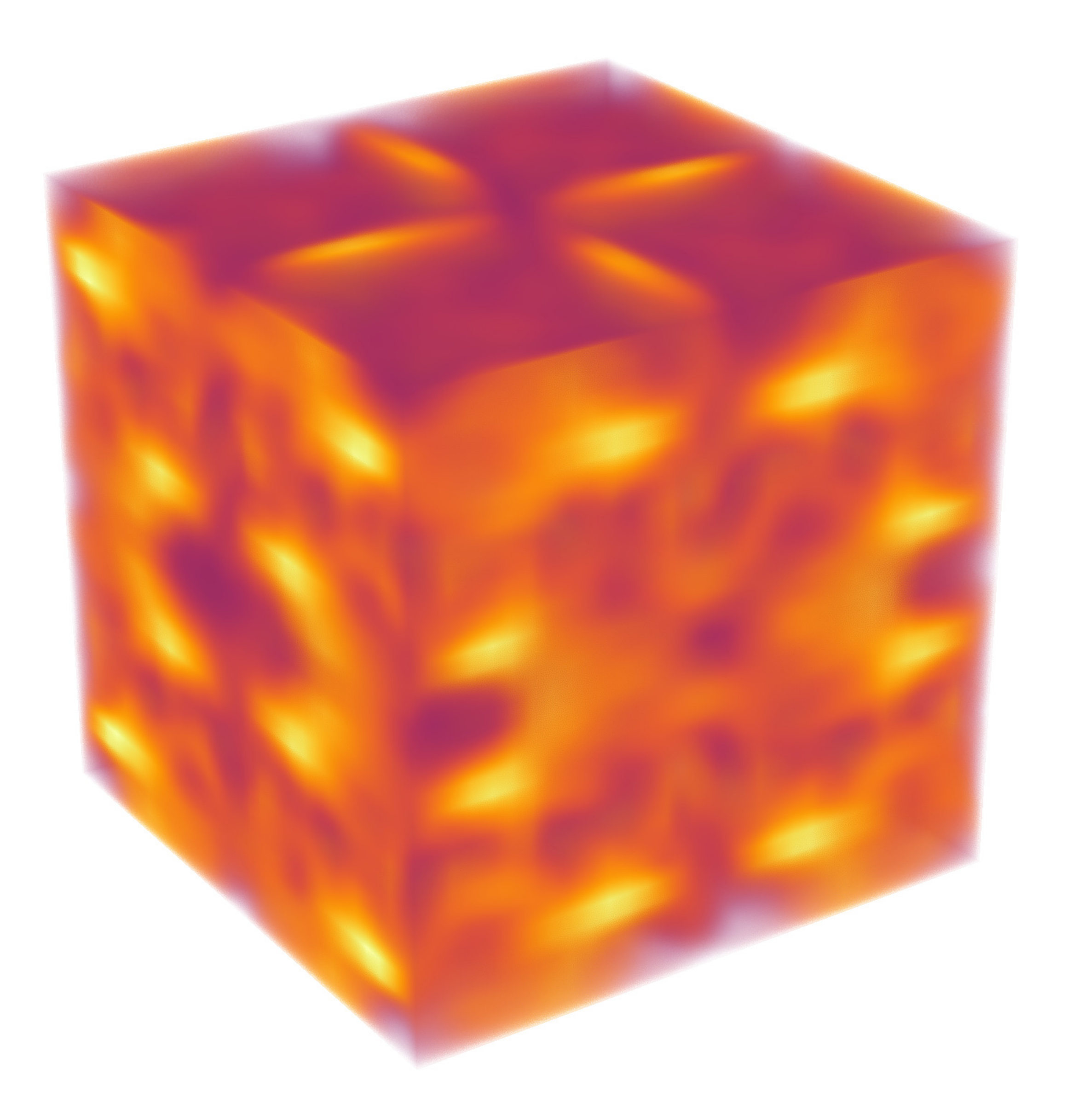}
          \label{subfig:statSolRTGV-mean10-1280}}
	\subfloat[Mean, \(R\!e = 2560\)]{\includegraphics[width=0.27\textwidth,trim={0cm 1.5cm 0cm 2cm},clip]{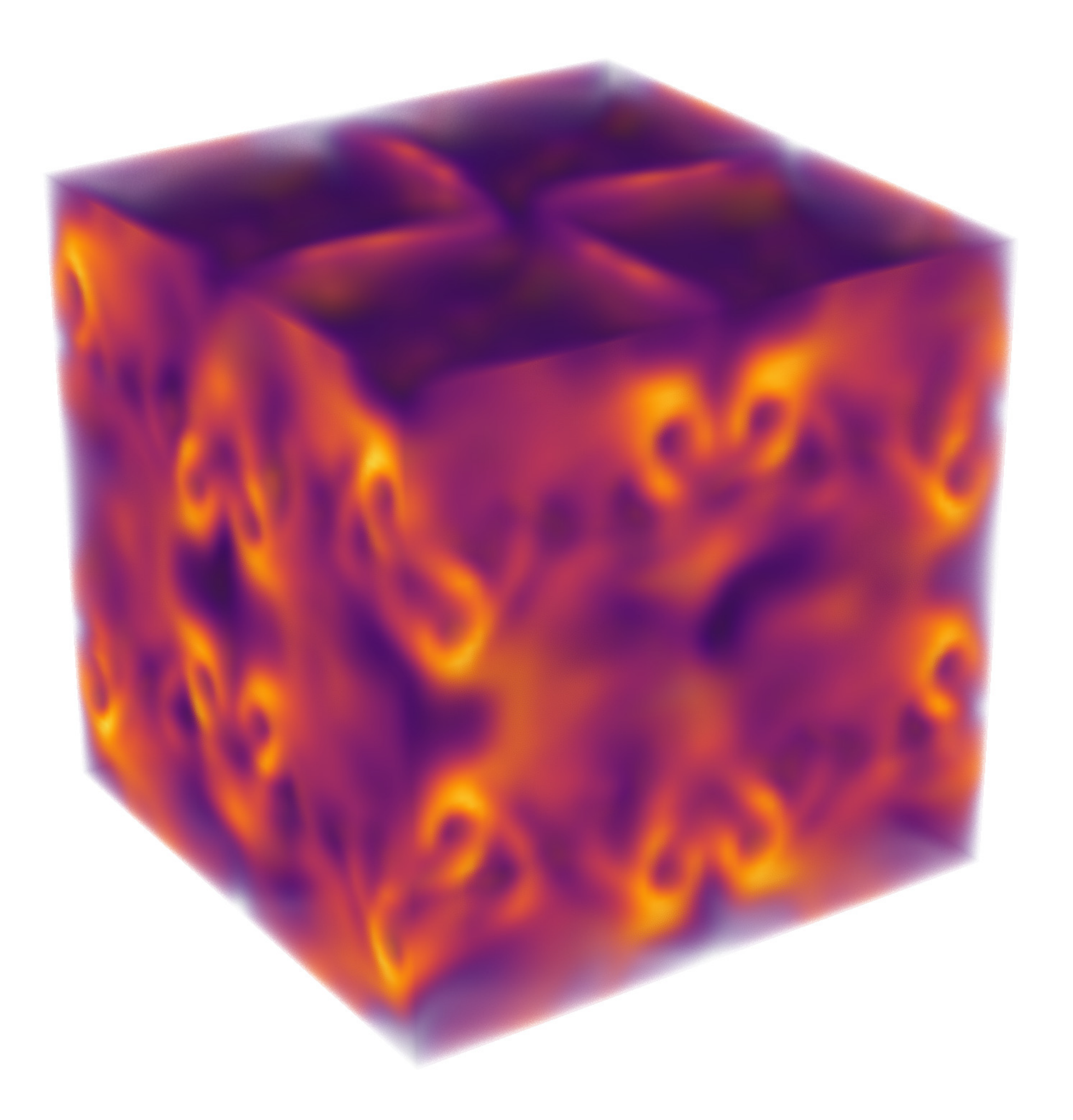}
          \label{subfig:statSolRTGV-mean10-2560}}
    \subfloat[Mean, \(R\!e = 5120\)]{\includegraphics[width=0.27\textwidth,trim={0cm 1.5cm 0cm 2cm},clip]{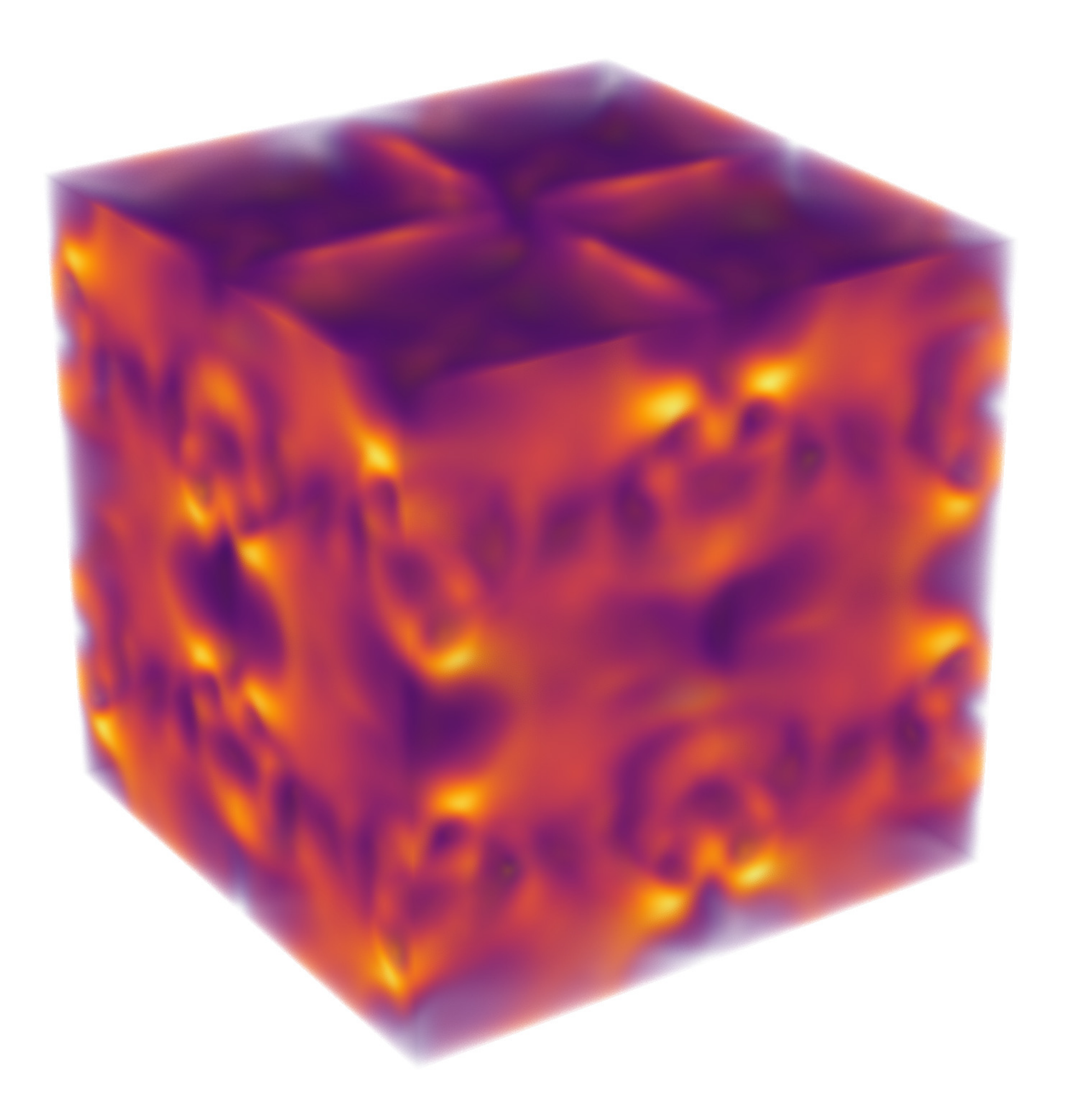}
          \label{subfig:statSolRTGV-mean10-5120}}
    \hspace{.1em}
    \includegraphics[scale=1]{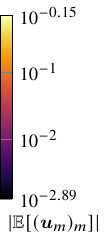}
    \\
    \subfloat[Std, \(R\!e = 1280\)]{\includegraphics[width=0.27\textwidth,trim={0cm 1.5cm 0cm 2cm},clip]{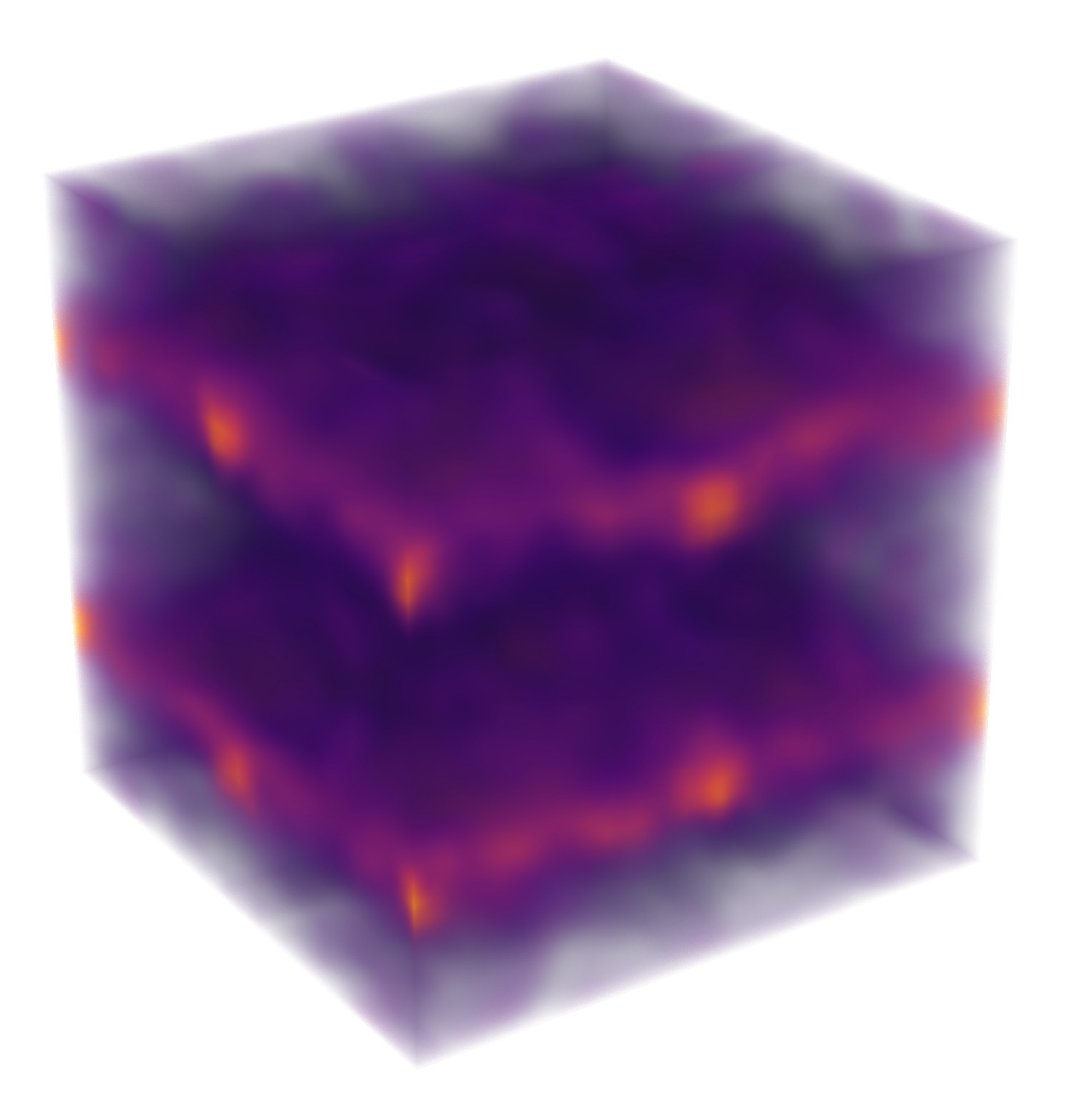}
          \label{subfig:statSolRTGV-std10-1280}}
	\subfloat[Std, \(R\!e = 2560\)]{\includegraphics[width=0.27\textwidth,trim={0cm 1.5cm 0cm 2cm},clip]{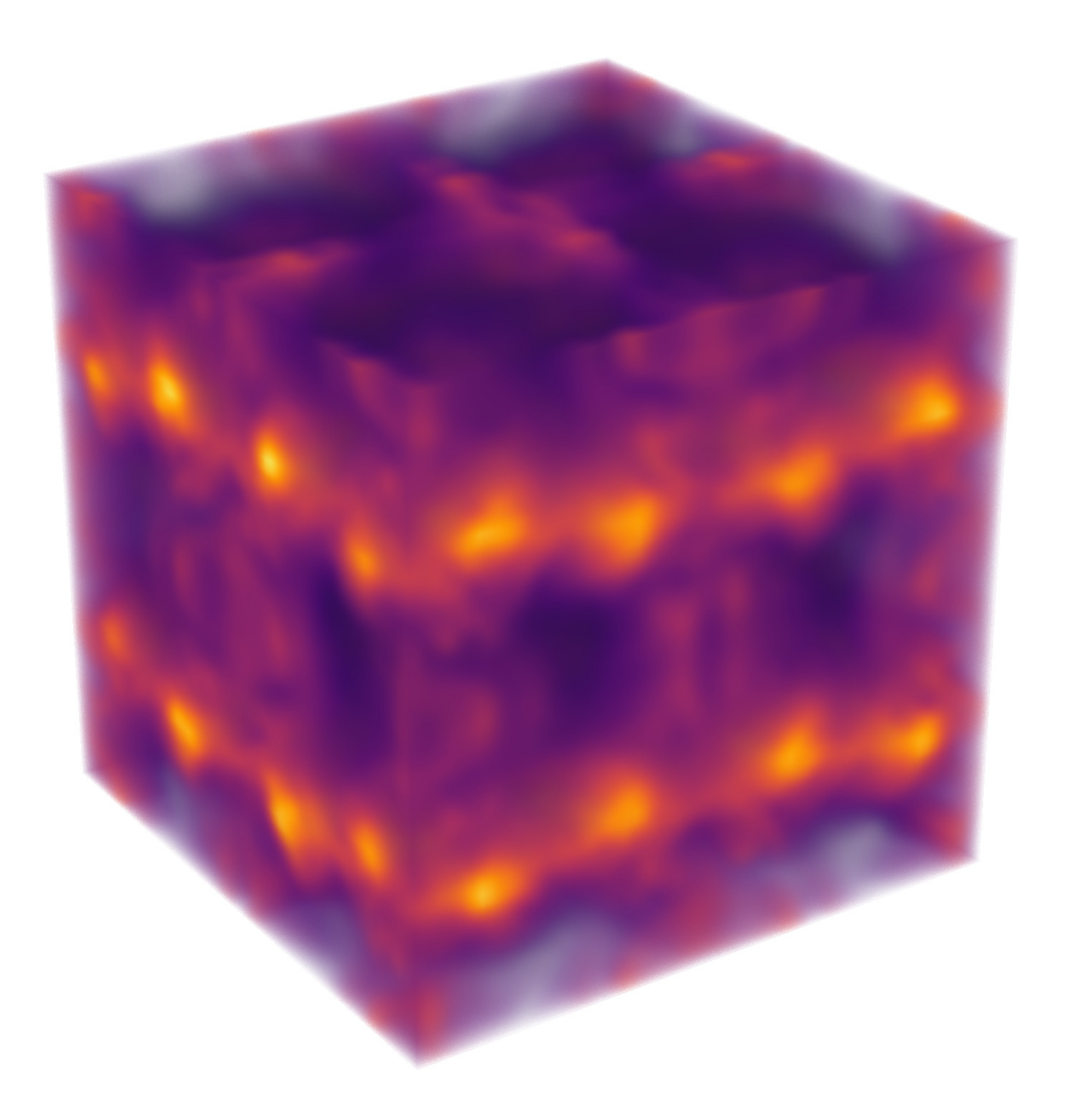}
          \label{subfig:statSolRTGV-std10-2560}}
    \subfloat[Std, \(R\!e = 5120\)]{\includegraphics[width=0.27\textwidth,trim={0cm 1.5cm 0cm 2cm},clip]{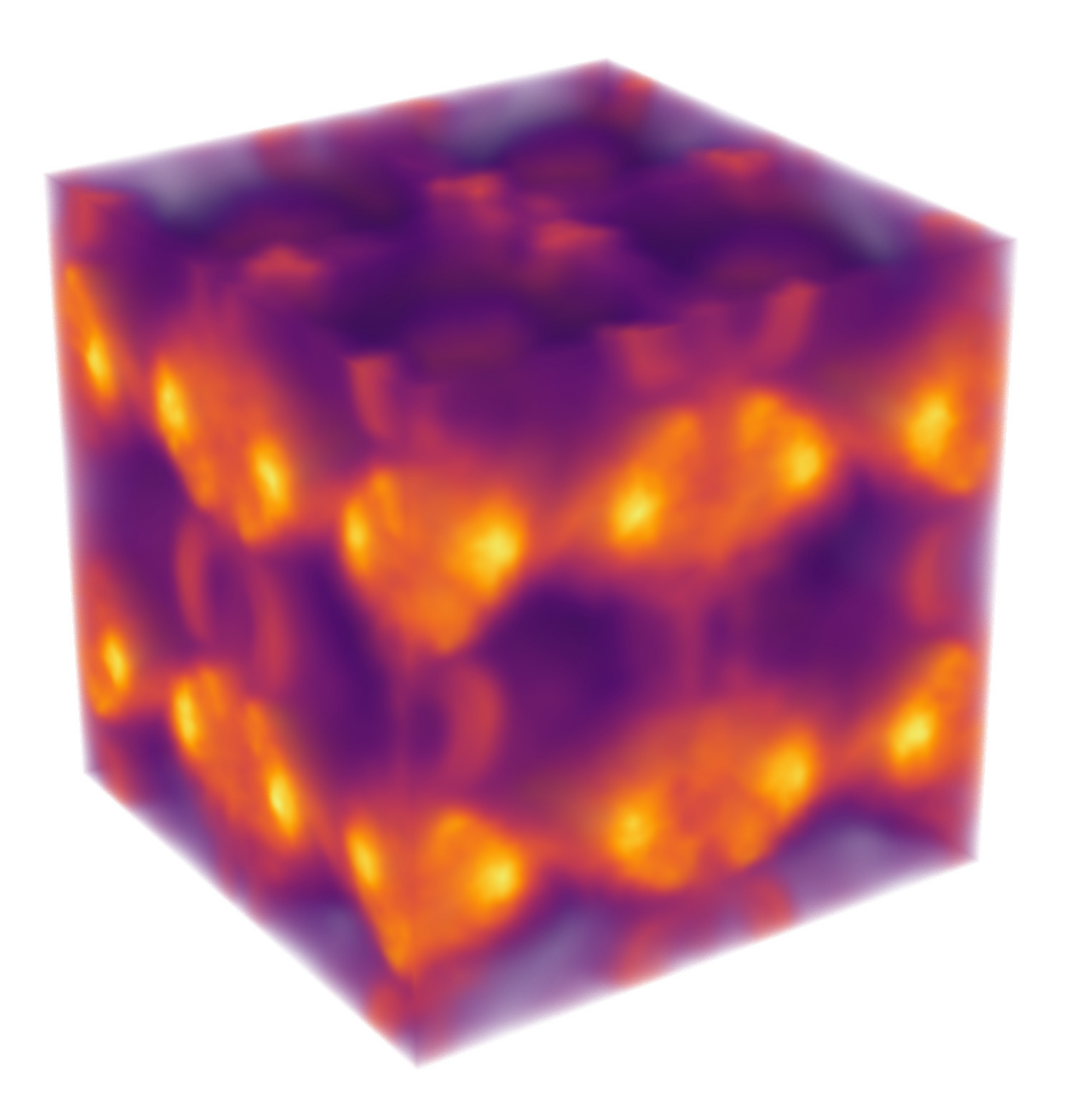}
          \label{subfig:statSolRTGV-std10-5120}}
    \hspace{.1em}
    \includegraphics[scale=1]{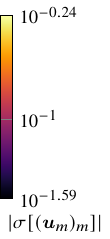}
    \\
	\caption{Iso-volume rendering of the velocity magnitude of the RTGV flow at \(t = 10\mathrm{s}\).  
    Top row: single samples; middle row: mean; bottom row: standard deviation (std); columns from left to right: \(R\!e = 1280, 2560, 5120\). 
    In the renderings, the color map is combined with an opacity transfer function that is linear in the data value, ranging from opacity zero at the minimum to one at the maximum of each colorbar. 
    The colorbars themselves omit this opacity. 
    Discretization parameters from Set~2 (Table~\ref{tab:params2}) are used.}
    \label{fig:statSolRTGV_10}
\end{figure}

\begin{figure}[ht!]
    %\centering
    \subfloat[Sample, \(R\!e = 1280\)]{\includegraphics[width=0.27\textwidth,trim={0cm 1.5cm 0cm 2cm},clip]{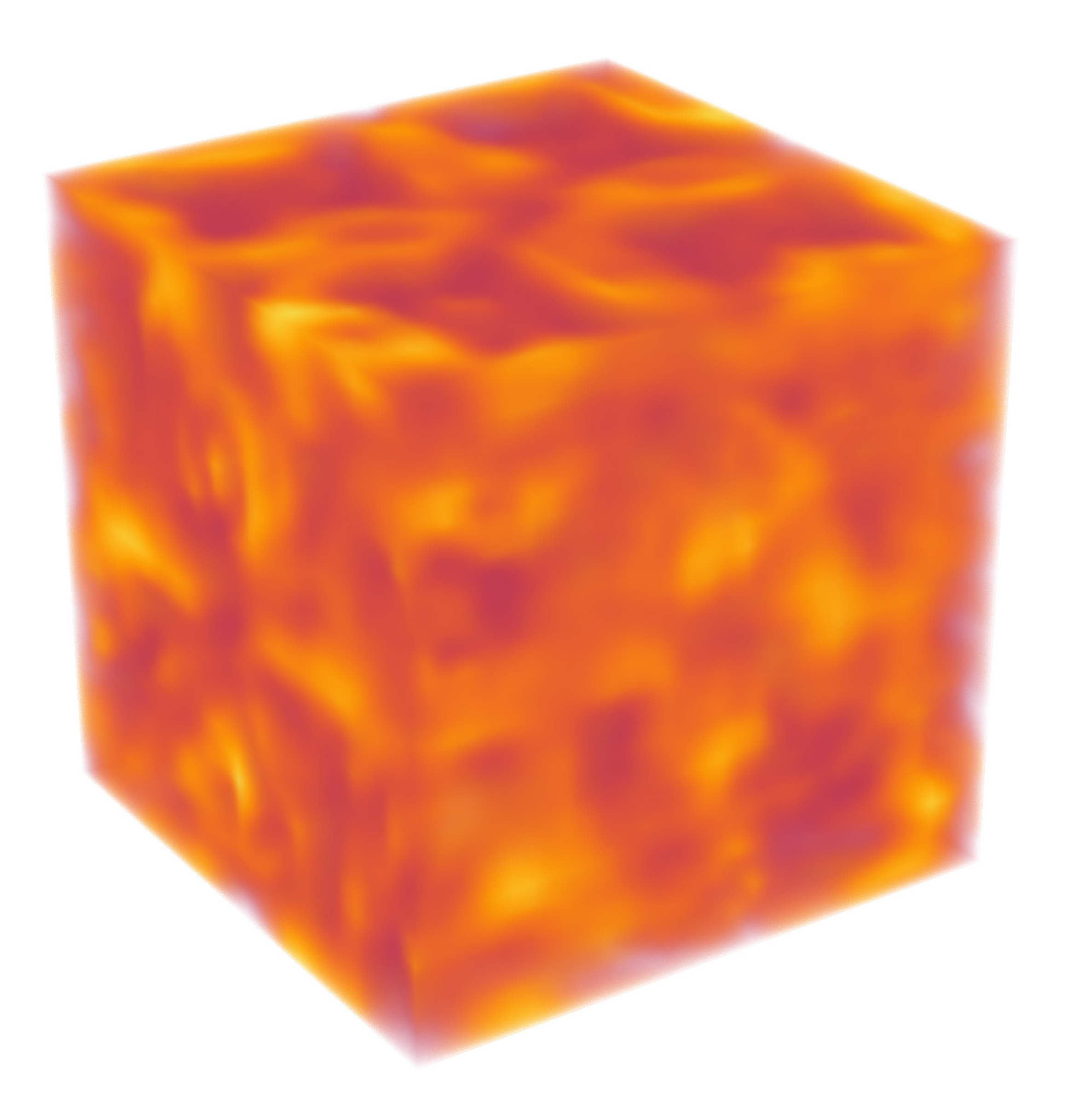}
          \label{subfig:statSolRTGV-sample15-1280}}
    \subfloat[Sample, \(R\!e = 2560\)]{\includegraphics[width=0.27\textwidth,trim={0cm 1.5cm 0cm 2cm},clip]{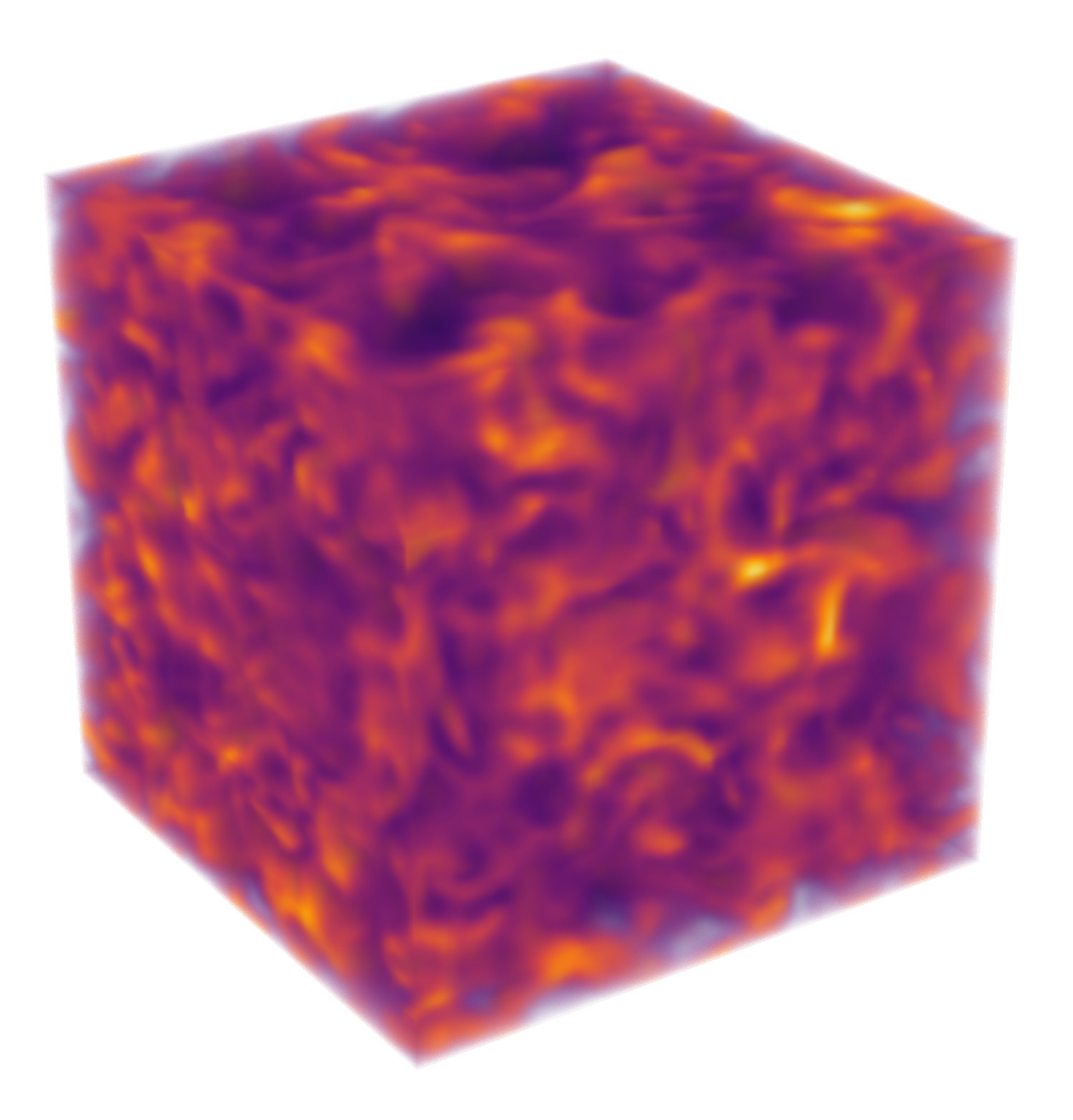}
          \label{subfig:statSolRTGV-sample15-2560}}
   	\subfloat[Sample, \(R\!e = 5120\)]{\includegraphics[width=0.27\textwidth,trim={0cm 1.5cm 0cm 2cm},clip]{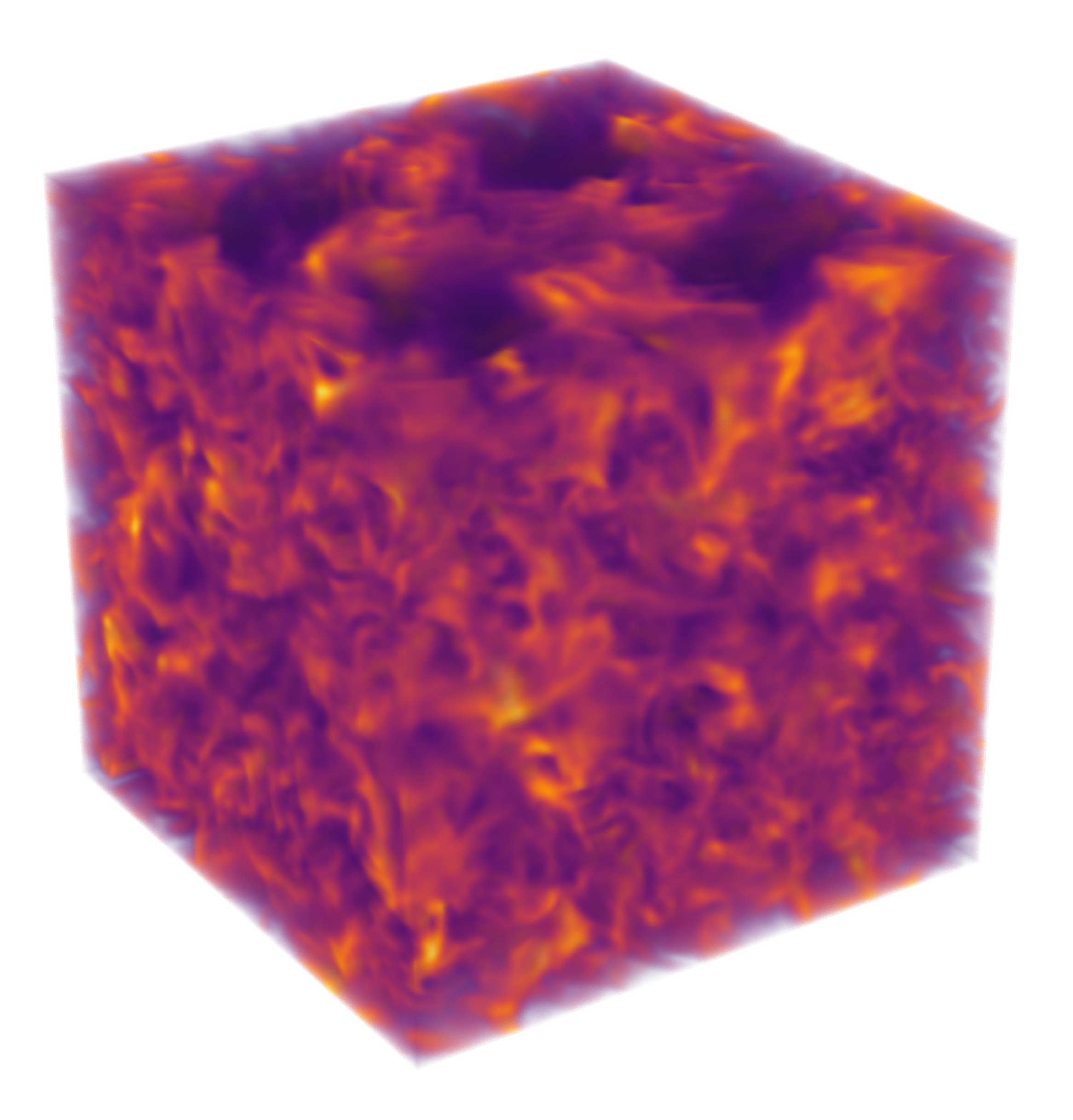}
          \label{subfig:statSolRTGV-sample15-5120}} 
    \hspace{.1em}
    \includegraphics[scale=1]{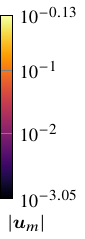}
    \\
    \subfloat[Mean, \(R\!e = 1280\)]{\includegraphics[width=0.27\textwidth,trim={0cm 1.5cm 0cm 2cm},clip]{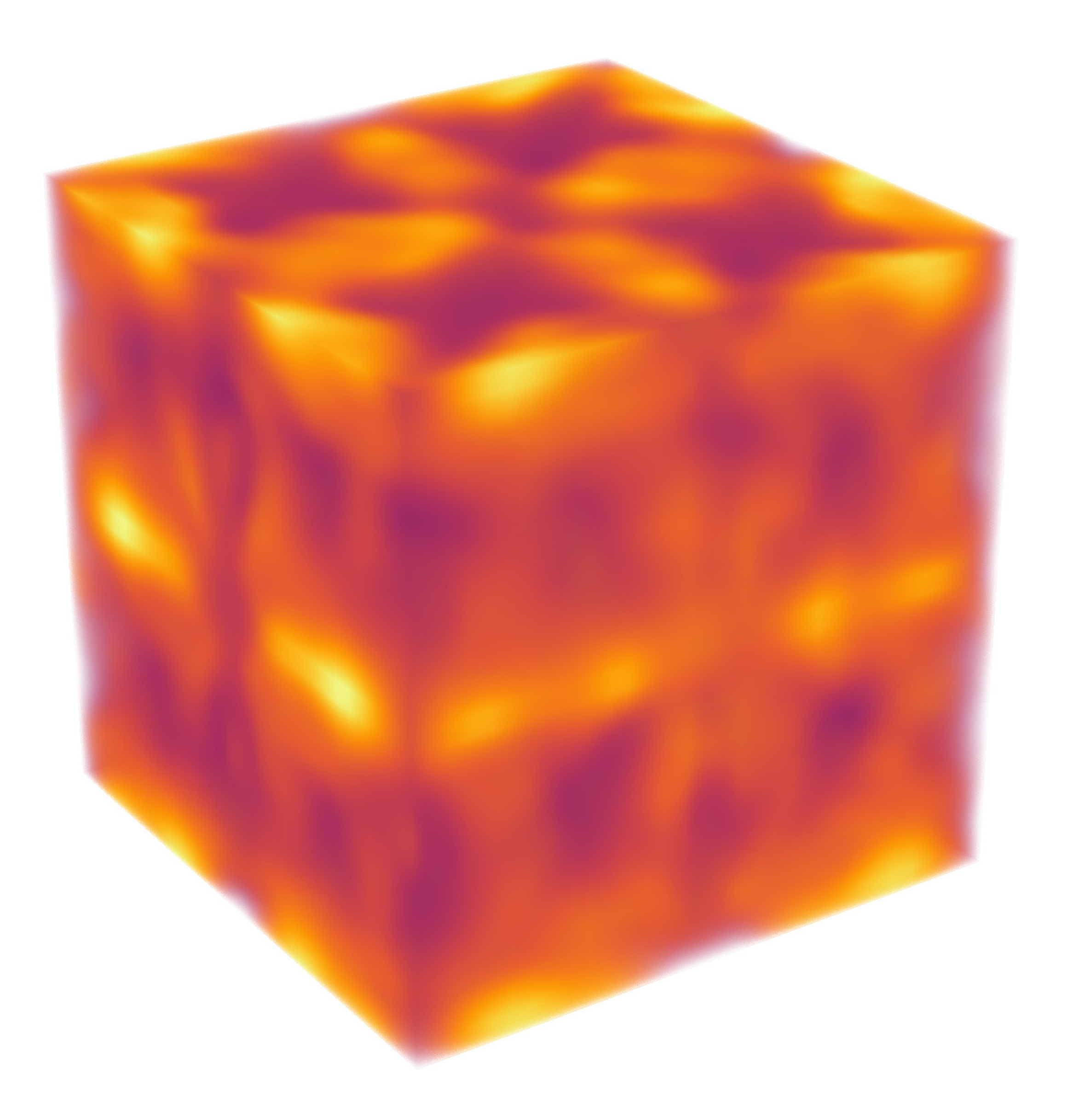}
          \label{subfig:statSolRTGV-mean15-1280}}
	\subfloat[Mean, \(R\!e = 2560\)]{\includegraphics[width=0.27\textwidth,trim={0cm 1.5cm 0cm 2cm},clip]{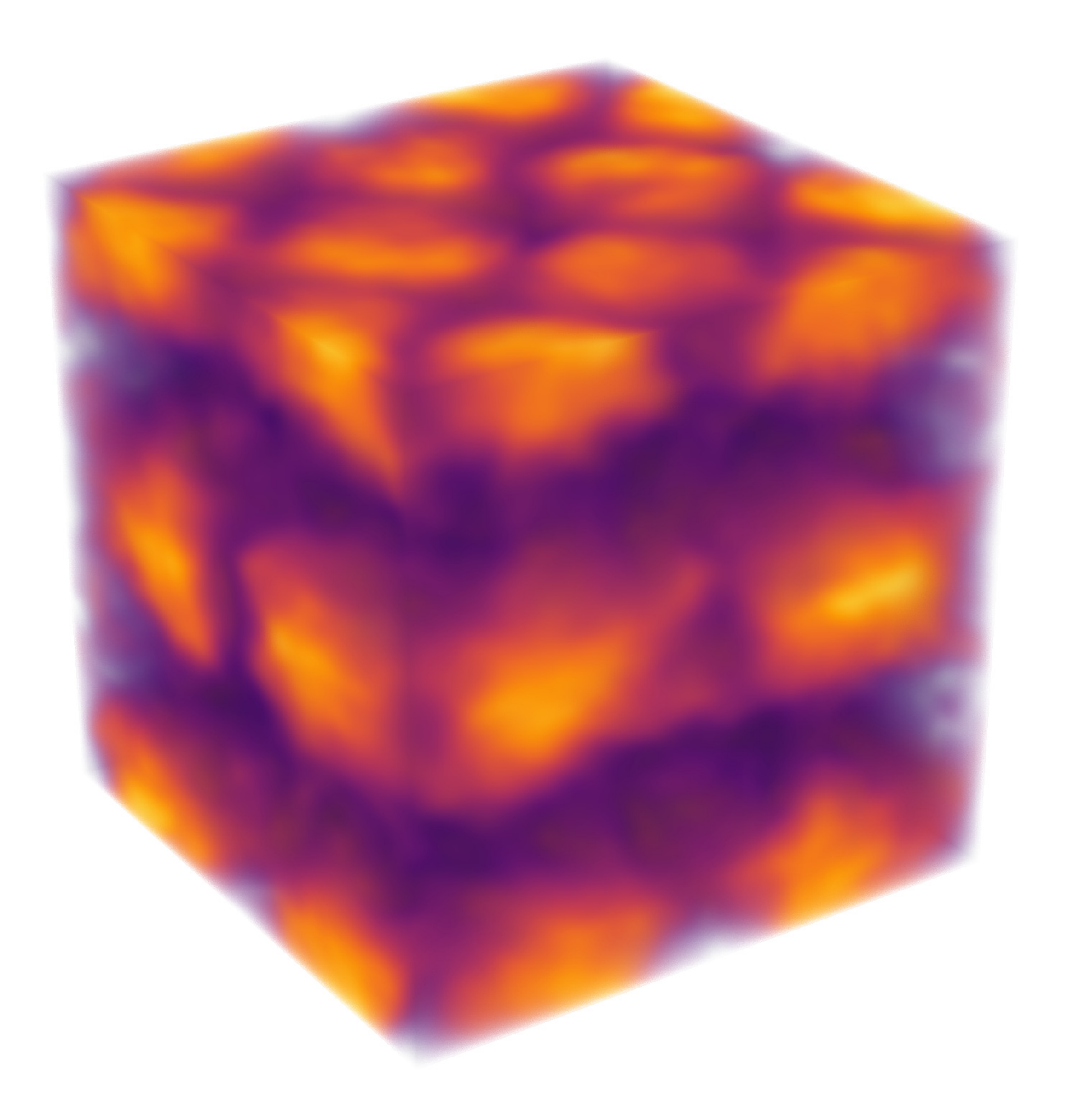}
          \label{subfig:statSolRTGV-mean15-2560}}
    \subfloat[Mean, \(R\!e = 5120\)]{\includegraphics[width=0.27\textwidth,trim={0cm 1.5cm 0cm 2cm},clip]{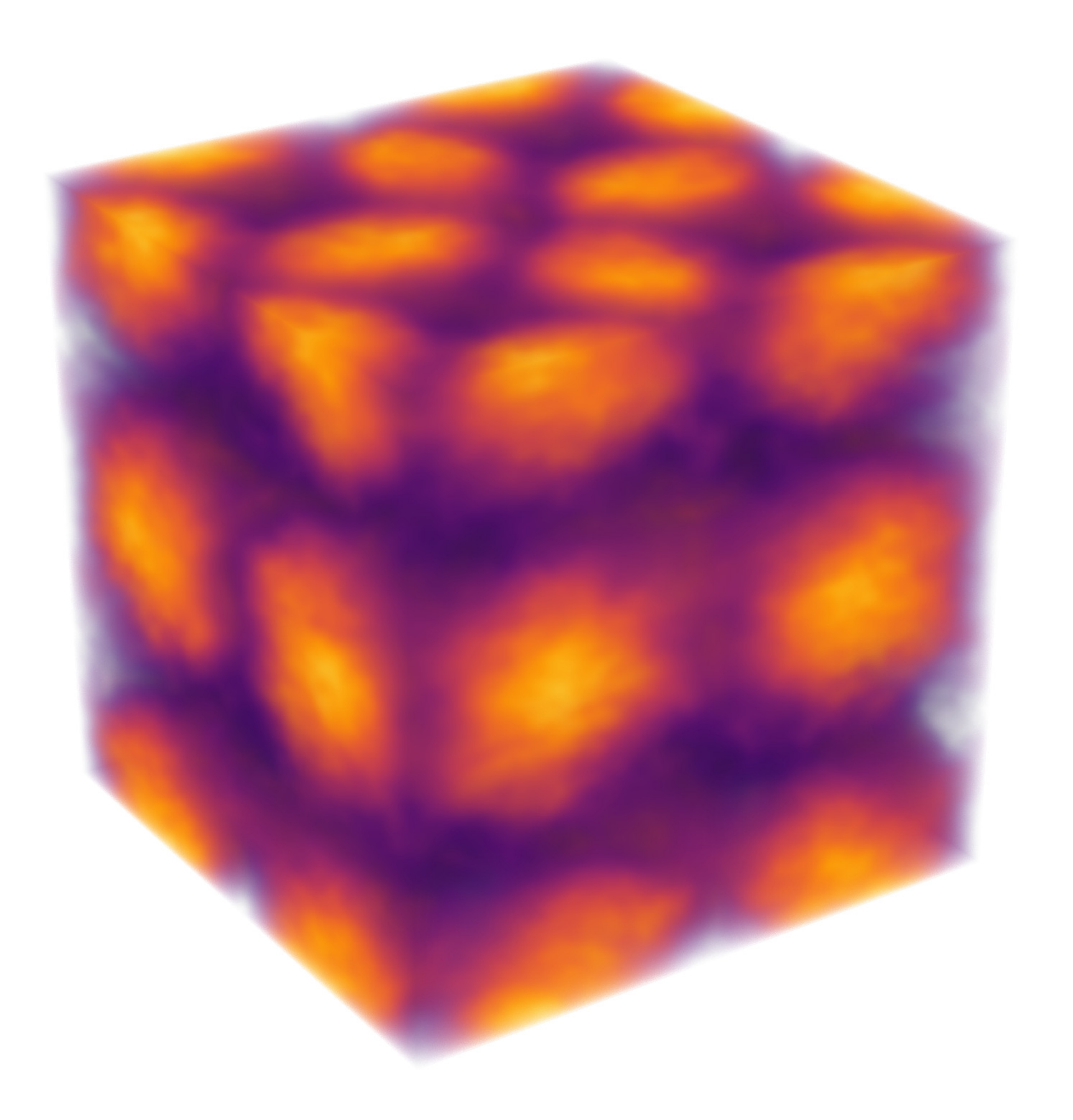}
          \label{subfig:statSolRTGV-mean15-5120}}
    \hspace{.1em}
    \includegraphics[scale=1]{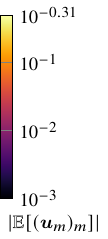}    
    \\
    \subfloat[Std, \(R\!e = 1280\)]{\includegraphics[width=0.27\textwidth,trim={0cm 1.5cm 0cm 2cm},clip]{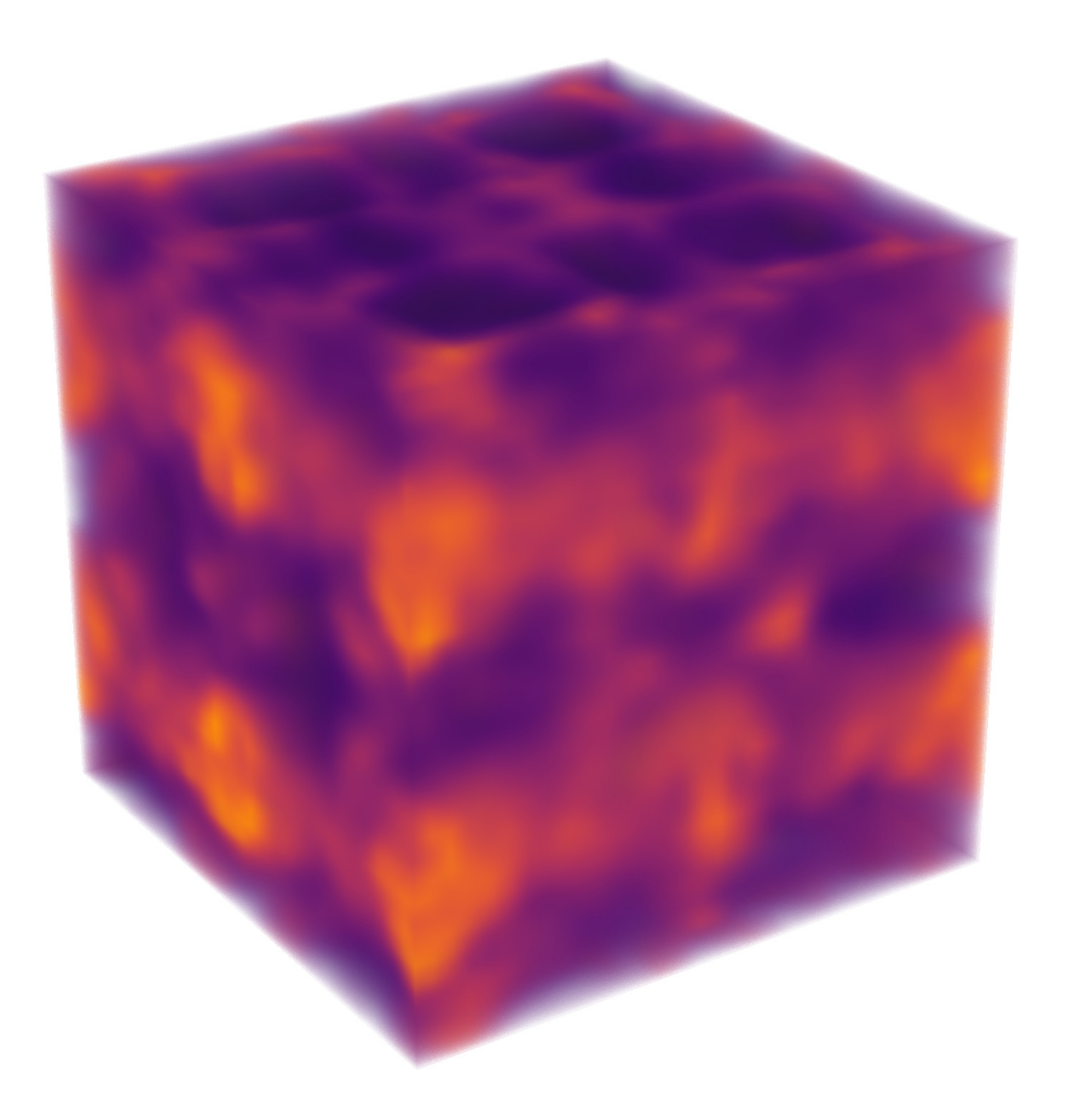}
          \label{subfig:statSolRTGV-std15-1280}}
	\subfloat[Std, \(R\!e = 2560\)]{\includegraphics[width=0.27\textwidth,trim={0cm 1.5cm 0cm 2cm},clip]{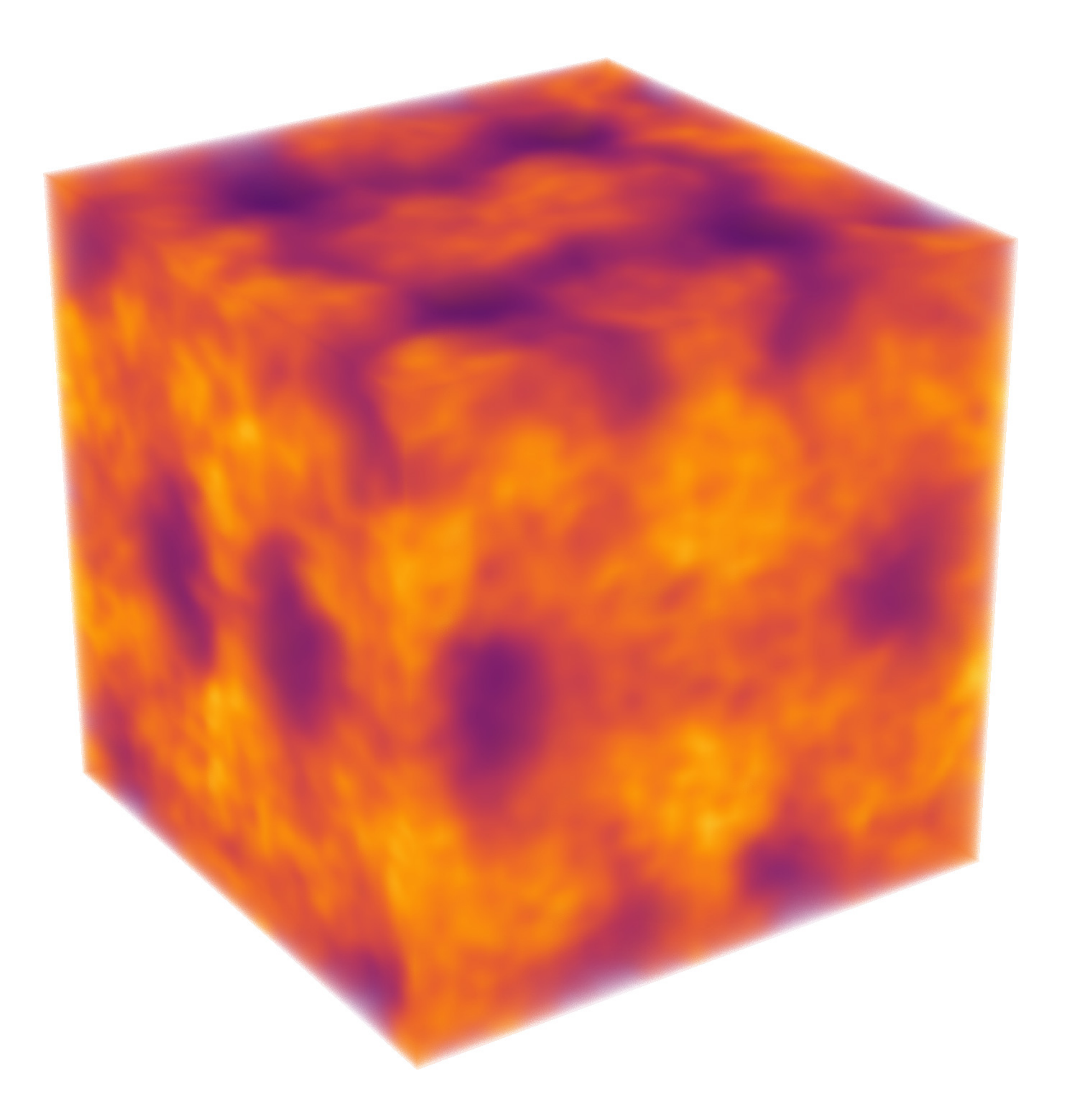}
          \label{subfig:statSolRTGV-std15-2560}}
    \subfloat[Std, \(R\!e = 5120\)]{\includegraphics[width=0.27\textwidth,trim={0cm 1.5cm 0cm 2cm},clip]{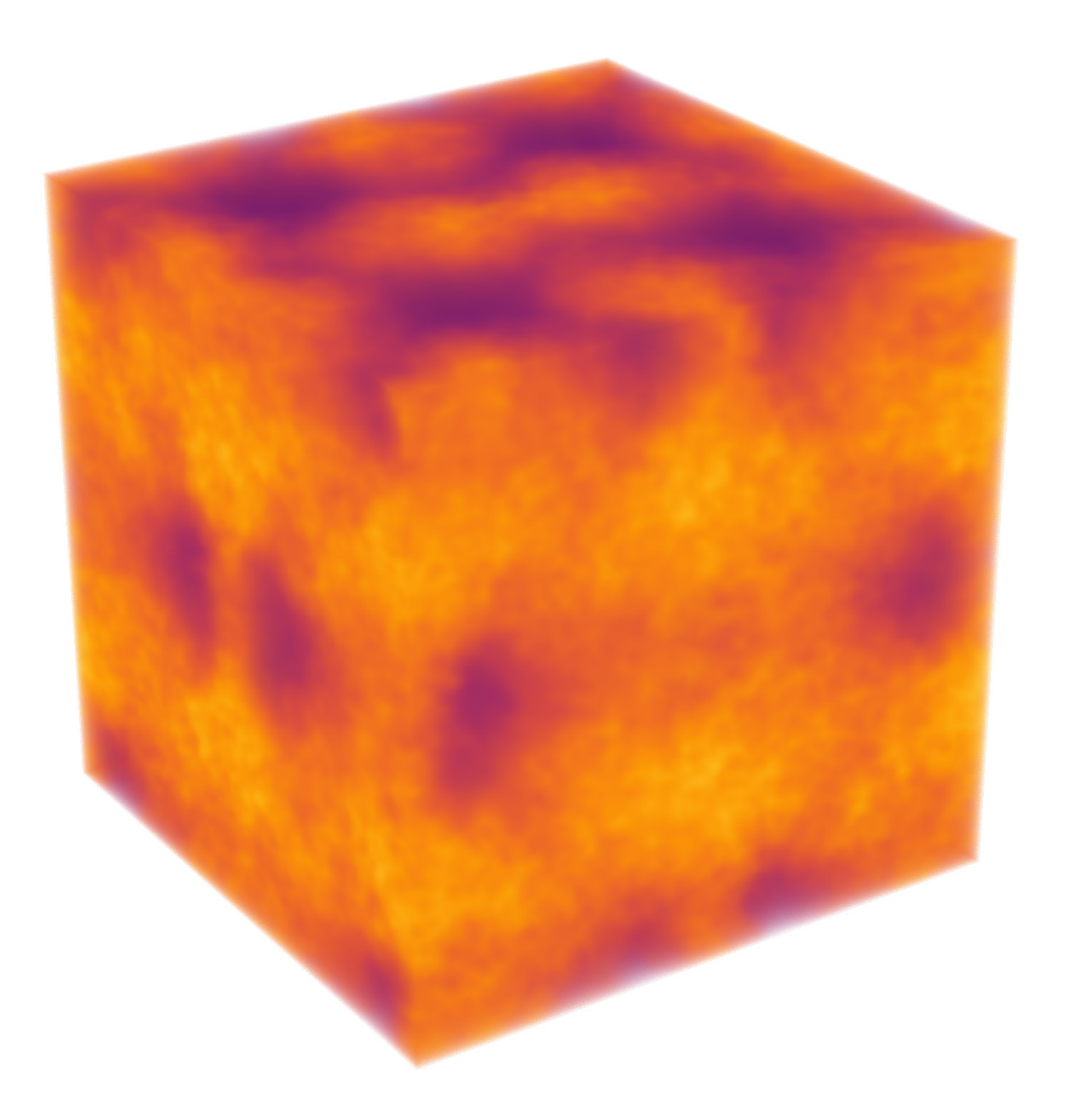}
          \label{subfig:statSolRTGV-std15-5120}}
    \hspace{.1em}
    \includegraphics[scale=1]{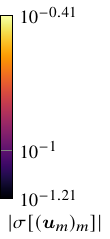}
    \\
	\caption{Iso-volume rendering of the velocity magnitude of the RTGV flow at \(t = 15\mathrm{s}\).  
    Top row: single samples; middle row: mean; bottom row: standard deviation (std); columns from left to right: \(R\!e = 1280, 2560, 5120\). 
    In the renderings, the color map is combined with an opacity transfer function that is linear in the data value, ranging from opacity zero at the minimum to one at the maximum of each colorbar. 
    The colorbars themselves omit this opacity. 
    Discretization parameters from Set~2 (Table~\ref{tab:params2}) are used.}
    \label{fig:statSolRTGV_15}
\end{figure}

\begin{figure}[ht!]
    %\centering
    \subfloat[Sample, \(R\!e = 1280\)]{\includegraphics[width=0.27\textwidth,trim={0cm 1.5cm 0cm 2cm},clip]{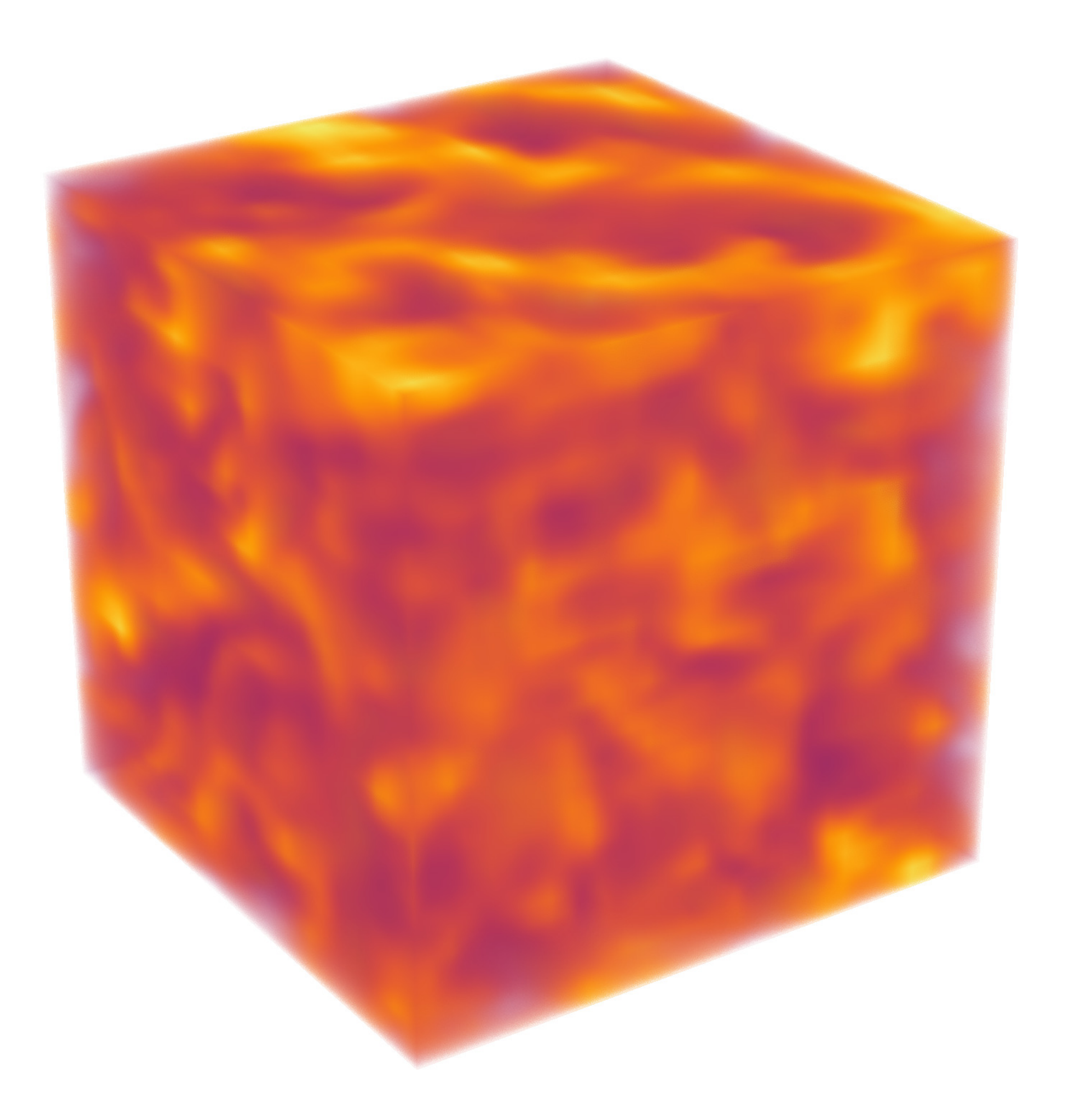}
          \label{subfig:statSolRTGV-sample20-1280}}
    \subfloat[Sample, \(R\!e = 2560\)]{\includegraphics[width=0.27\textwidth,trim={0cm 1.5cm 0cm 2cm},clip]{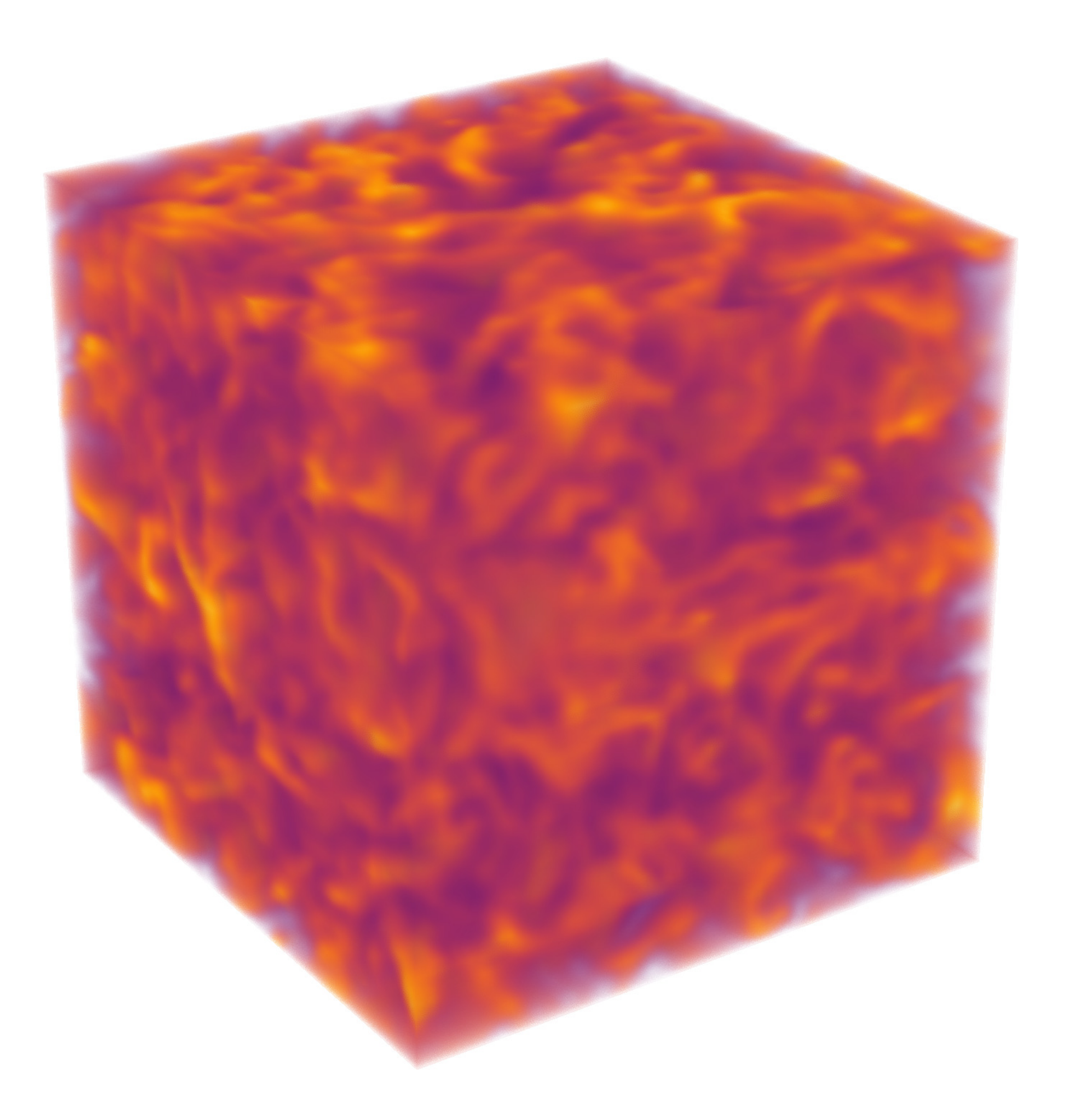}
          \label{subfig:statSolRTGV-sample20-2560}}
   	\subfloat[Sample, \(R\!e = 5120\)]{\includegraphics[width=0.27\textwidth,trim={0cm 1.5cm 0cm 2cm},clip]{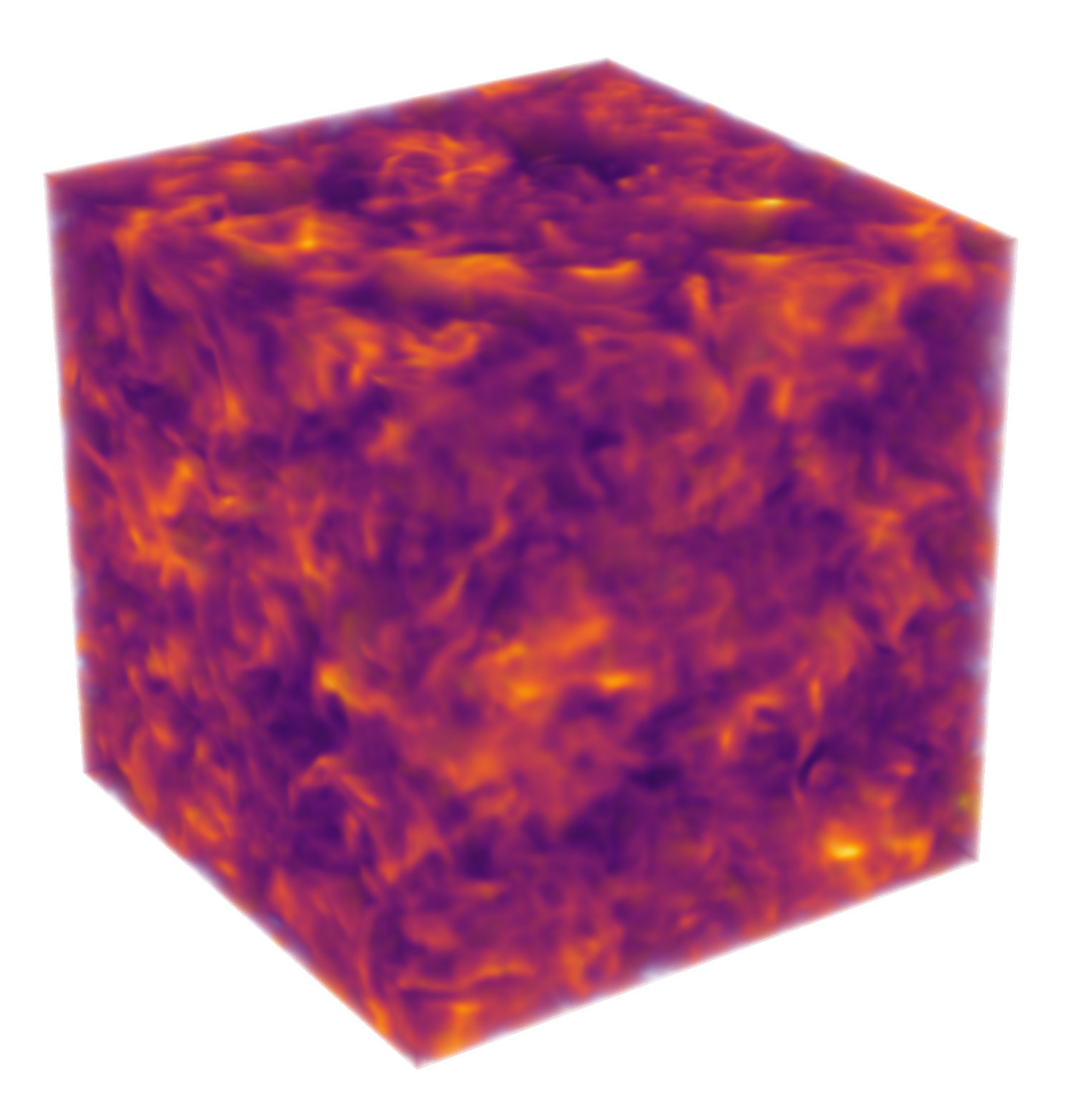}
          \label{subfig:statSolRTGV-sample20-5120}} 
    \hspace{.1em}
    \includegraphics[scale=1]{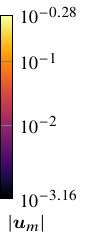}
    \\
    \subfloat[Mean, \(R\!e = 1280\)]{\includegraphics[width=0.27\textwidth,trim={0cm 1.5cm 0cm 2cm},clip]{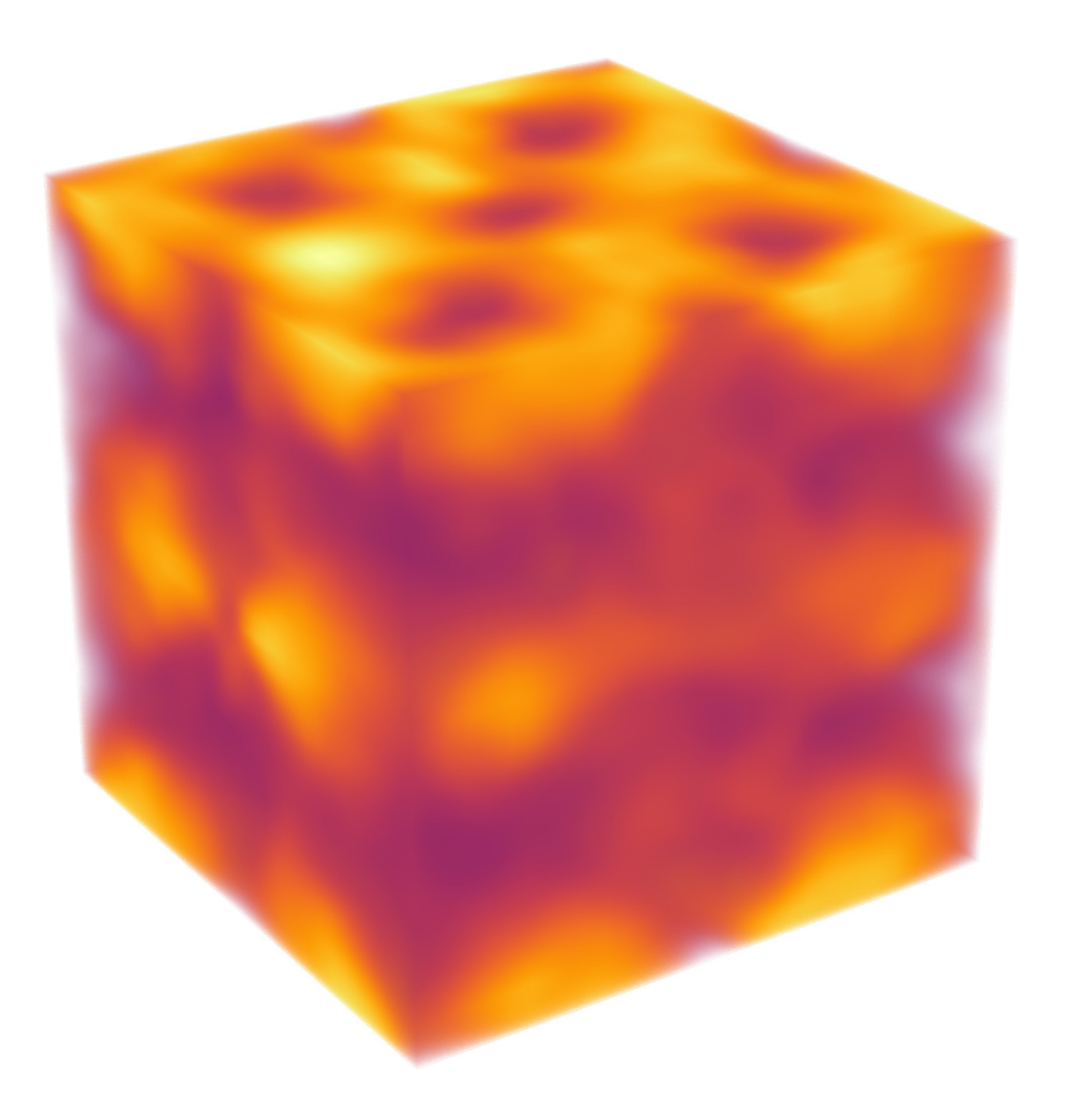}
          \label{subfig:statSolRTGV-mean20-1280}}
	\subfloat[Mean, \(R\!e = 2560\)]{\includegraphics[width=0.27\textwidth,trim={0cm 1.5cm 0cm 2cm},clip]{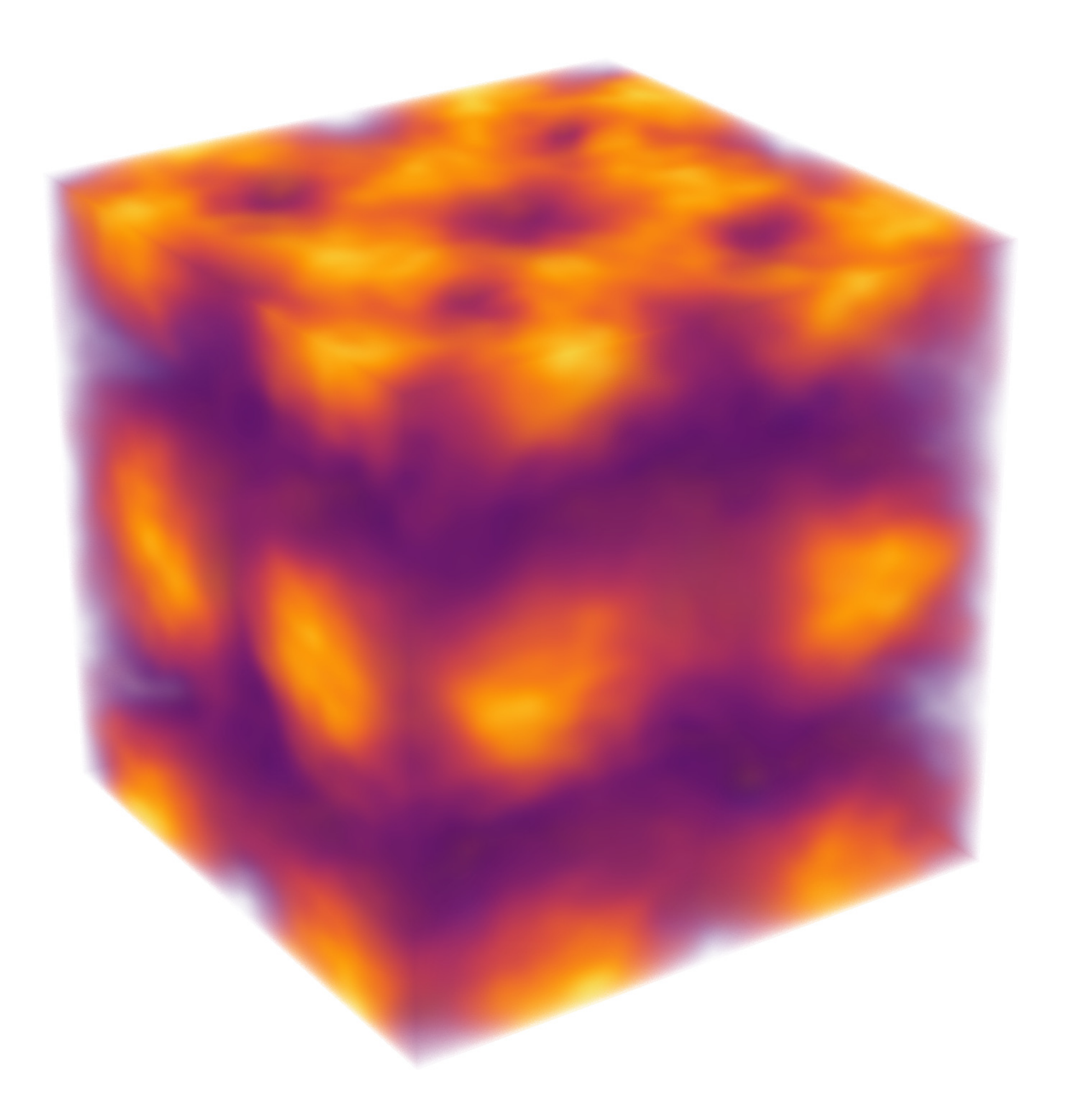}
          \label{subfig:statSolRTGV-mean20-2560}}
    \subfloat[Mean, \(R\!e = 5120\)]{\includegraphics[width=0.27\textwidth,trim={0cm 1.5cm 0cm 2cm},clip]{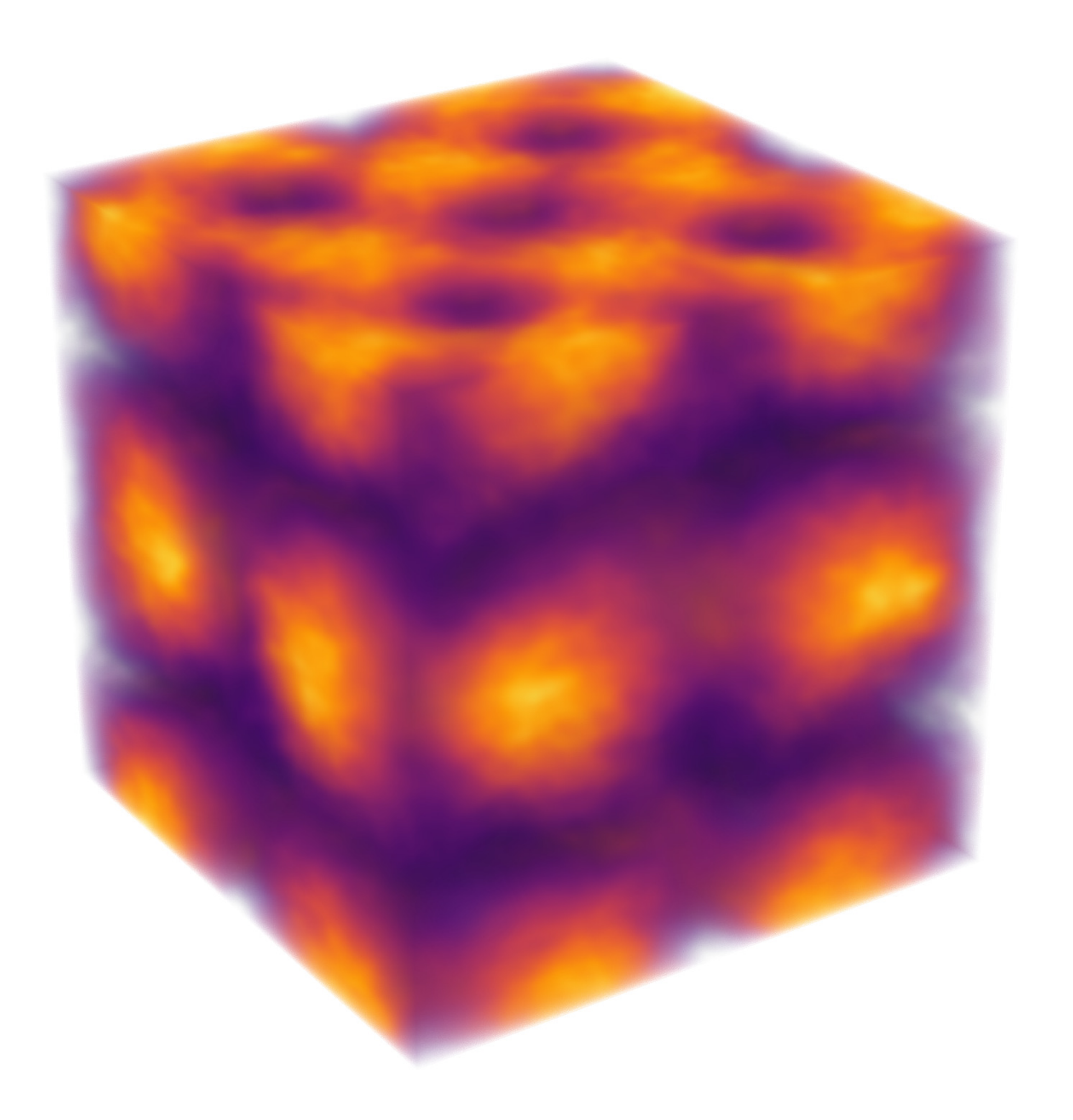}
          \label{subfig:statSolRTGV-mean20-5120}}
    \hspace{.1em}
    \includegraphics[scale=1]{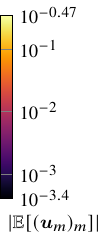}
    \\
    \subfloat[Std, \(R\!e = 1280\)]{\includegraphics[width=0.27\textwidth,trim={0cm 1.5cm 0cm 2cm},clip]{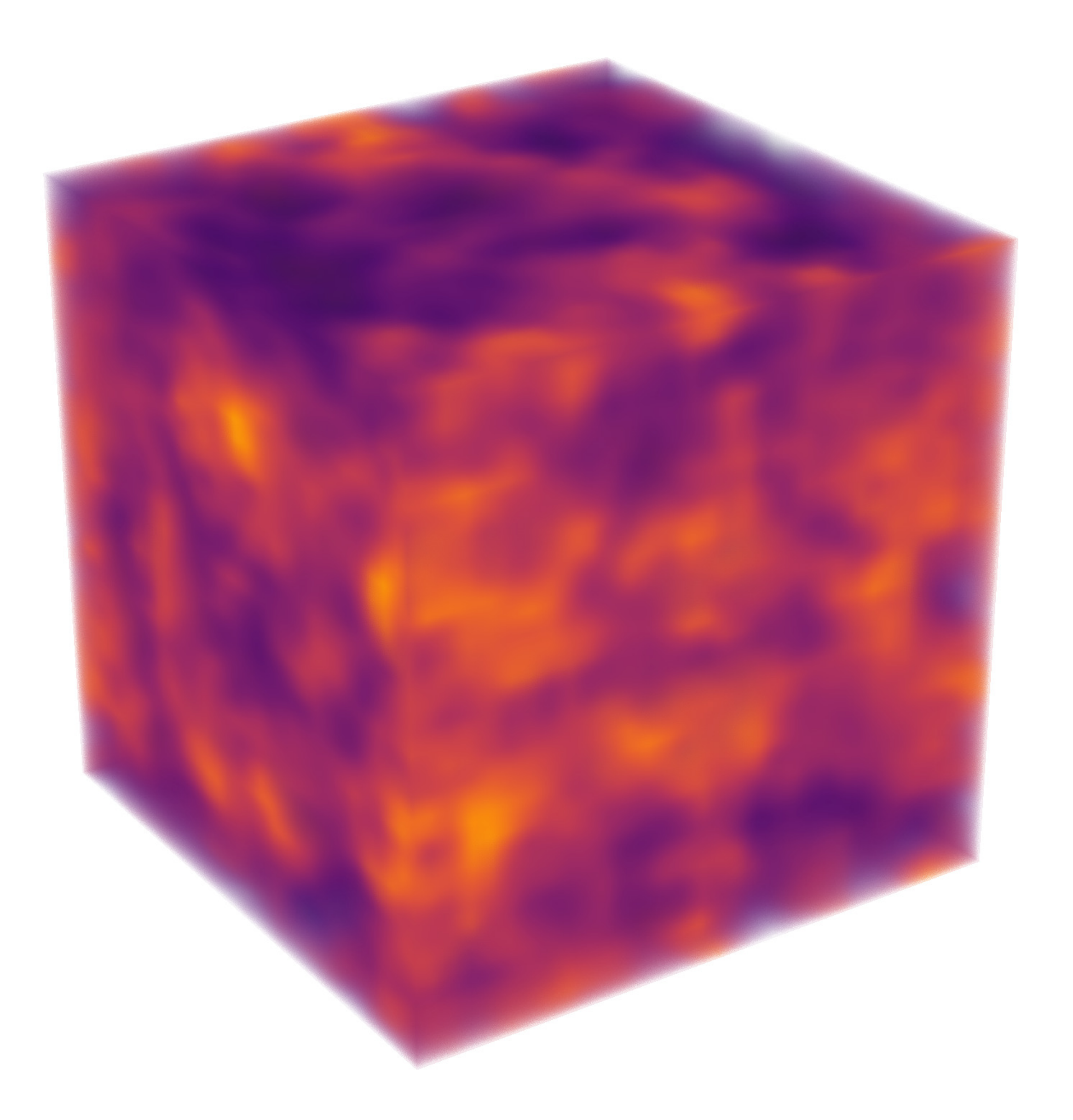}
          \label{subfig:statSolRTGV-std20-1280}}
	\subfloat[Std, \(R\!e = 2560\)]{\includegraphics[width=0.27\textwidth,trim={0cm 1.5cm 0cm 2cm},clip]{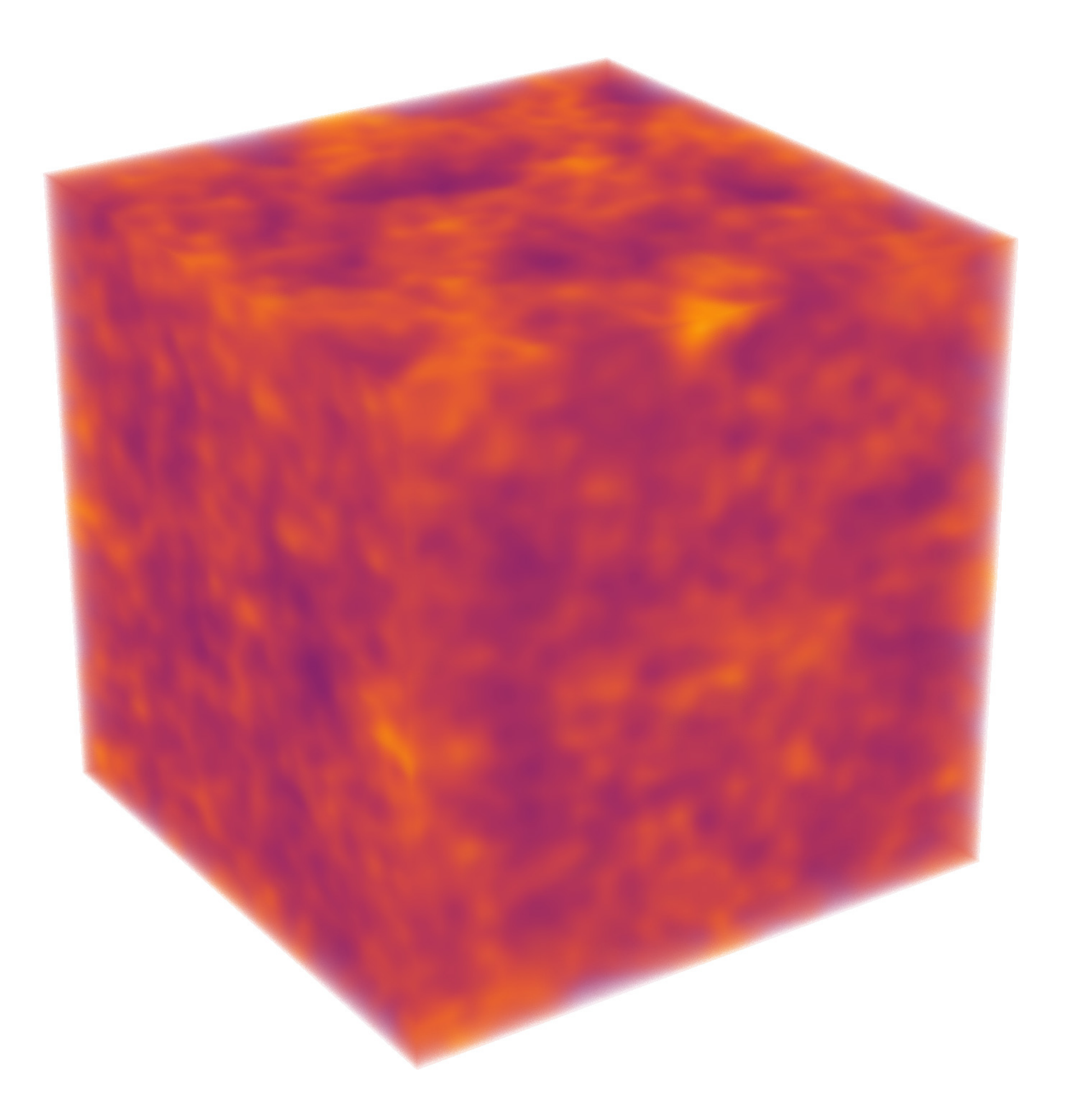}
          \label{subfig:statSolRTGV-std20-2560}}
    \subfloat[Std, \(R\!e = 5120\)]{\includegraphics[width=0.27\textwidth,trim={0cm 1.5cm 0cm 2cm},clip]{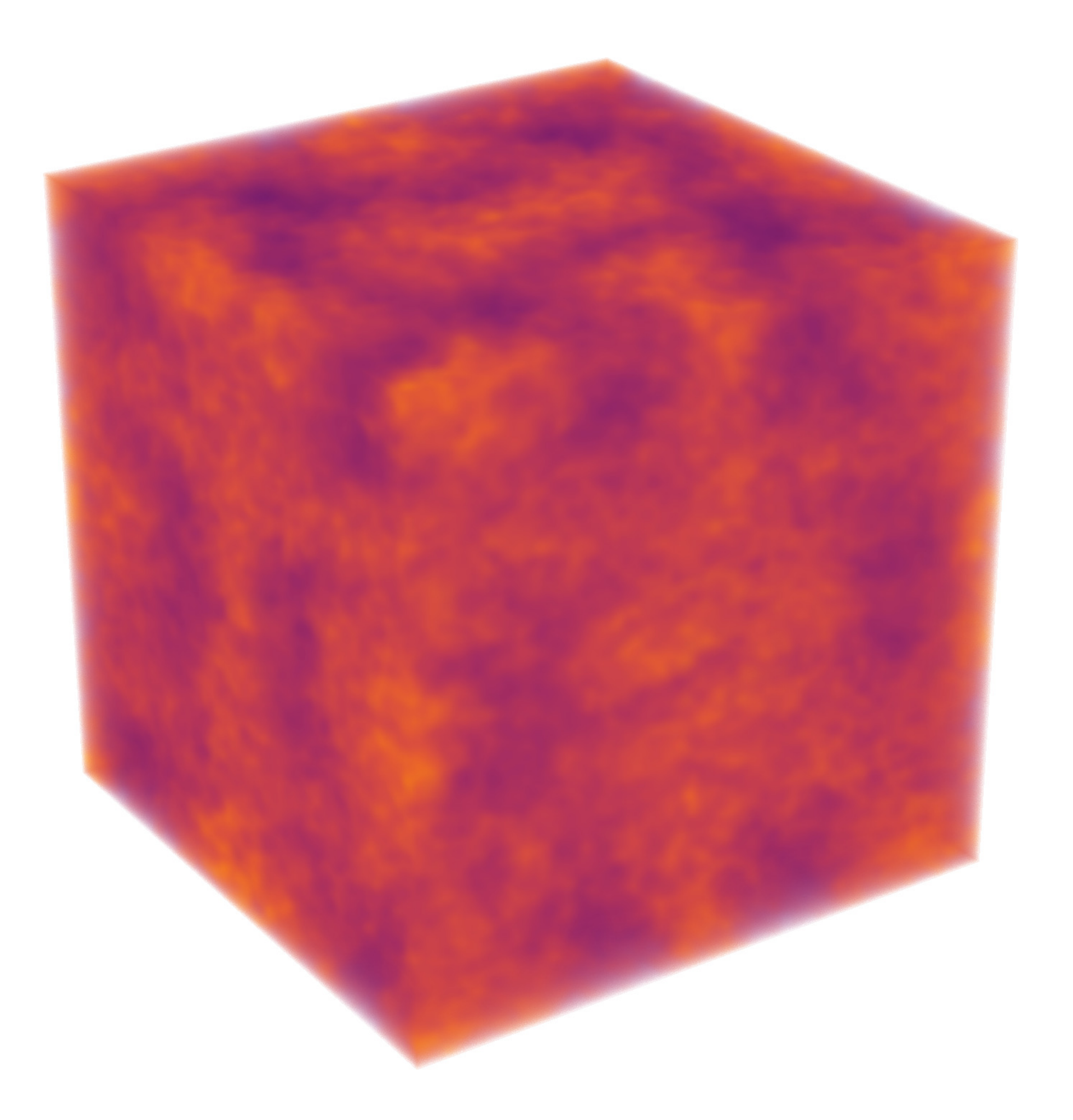}
          \label{subfig:statSolRTGV-std20-5120}}
    \hspace{.1em}
    \includegraphics[scale=1]{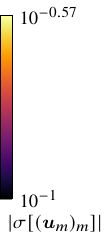}
    \\
	\caption{Iso-volume rendering of the velocity magnitude of the RTGV flow at \(t = 20\mathrm{s}\).  
    Top row: single samples; middle row: mean; bottom row: standard deviation (std); columns from left to right: \(R\!e = 1280, 2560, 5120\). 
    In the renderings, the color map is combined with an opacity transfer function that is linear in the data value, ranging from opacity zero at the minimum to one at the maximum of each colorbar. 
    The colorbars themselves omit this opacity. 
    Discretization parameters from Set~2 (Table~\ref{tab:params2}) are used.}
    \label{fig:statSolRTGV_20}
\end{figure}

\end{document}